\documentclass[11pt]{article}
\usepackage{geometry}
\usepackage{amsmath, amsthm, amssymb}
\usepackage{graphicx, color}
\usepackage{pdfsync} 
\usepackage[english]{babel}
\numberwithin{equation}{section}

\newlength{\wdth}

\newtheorem{lemma}{Lemma}[section]
\newtheorem{corollary}[lemma]{Corollary}
\newtheorem{proposition}[lemma]{Proposition}
\newtheorem{remark}[lemma]{Remark}
\newtheorem{example}[lemma]{Example}
\newtheorem{theorem}[lemma]{Theorem}
\newtheorem{definition}[lemma]{Definition}
\newtheorem{conjecture}[lemma]{Conjecture}
\usepackage{graphicx}
\usepackage{amssymb}
\usepackage{epstopdf}
\newcommand{\e}{\varepsilon}

\begin{document}

\title{Beyond the classical Cauchy-Born rule}


\author{Andrea Braides\\ \small SISSA, via Bonomea 265, Trieste, Italy\\  \\  Andrea Causin and
Margherita Solci \\ \small
DADU, Universit\`a di Sassari\\ \small
 piazza Duomo 6, 07041 Alghero (SS), Italy \\ \\
Lev Truskinovsky\\ \small PMMH, CNRS - UMR 7636 PSL-ESPCI,\\ \small 10 Rue Vauquelin, 75005 Paris, France}
\date{}                                           

\maketitle

\begin{abstract}
Physically motivated variational problems involving non-convex energies are often formulated in a discrete setting  and contain boundary conditions.  The  long-range  interactions in such problems,  combined with   constraints imposed by  lattice discreteness,   can give rise to the phenomenon of   geometric frustration even in a one-dimensional setting.
While  non-convexity  entails the formation of microstructures,  incompatibility between interactions operating at different  scales  can produce nontrivial mixing effects which are exacerbated  in the case of incommensuration between the optimal microstructures and the scale of the underlying lattice. Unraveling the intricacies of the underlying interplay between non-convexity, non-locality and discreteness, represents the main goal of this study.  While in general one cannot expect that  ground states in such problems possess global properties,  such as periodicity, in some cases  the appropriately defined `global' solutions  exist,  and are sufficient to describe the corresponding   continuum (homogenized) limits.  We interpret those cases as complying with a Generalized Cauchy-Born (GCB) rule, and present  a new class of problems with geometrical frustration which comply  with  GCB rule in one  range of (loading) parameters while being strictly outside this class  in a complimentary range. A general approach to problems with such `mixed' behavior is developed.

\end{abstract}


\section{Introduction}


Variational problems emerging from applications are often both discrete and non-convex. Important  examples  include  one-dimensional  boundary-value problems with   trans\-la\-tion-invariant energy densities describing  pairwise interactions.  Such problems  constitute the main subject of this paper. 

The representative  energies for this class of problems can be written in the following generic form 
\begin{equation}\label{eq1} 
F(w;k)=\min\biggl\{\sum_{i,j=0}^k f_{i-j} (u_i-u_j): u_0=0, u_k=w\biggr\},
\end{equation}
where  for every $n$  natural number $f_n$ is a potentially nonconvex energy  governing interactions between the lattice points  at distance $n$, and
 the minimum is searched among $k+1$-arrays $(u_0,\ldots, u_k)$. We may assume that $f_0=0$. As the parameter $k$ increases and more interactions are taken into account, a question arises  about the behavior of minimal arrays $(u^k_0,\ldots, u^k_k)$ and of the corresponding minimal energy. One of the most  important issues concerns  the existence of a continuum limit of the type 
$
F_{\rm hom}(u)=\int_I f _{\rm hom}(u^\prime)\, dt,
$
with $I$ an interval in which the nodes $i$ in \eqref{eq1} are identified as a discrete subset (e.g., $I=[0,1]$ where the discrete subset is ${1\over k}\mathbb Z\cap [0,1]$). 
The single function $f _{\rm hom}$  is expected to carry, in a condensed way, all the  relevant information about the   infinite set of   functions  $f_n$ from \eqref{eq1}. 

To track the asymptotic behavior of the minimum values in \eqref{eq1}  we can  use the average derivative $z=w/k$ as a parameter, and scale the energy  by $k$. Then,  under assumptions on a suitably fast decay of $f_n$ with respect to $n$,  it can be   shown that the limiting energy density $f_{\rm hom}$ exists and can be expressed by the formula
\begin{equation}\label{eq2} 
f_{\rm hom}(z)=\lim_{k\to+\infty} {1\over k} \min\biggl\{\sum_{i,j=0}^k f_{i-j} (u_i-u_j): u_0=0, u_k=kz\biggr\}.
\end{equation}
Moreover, it can be shown that the  function $f_{\rm hom}$ is  convex in the parameter $z$. This  result represents a particular case of a more general variational theory for limits of lattice energies (see e.g.~\cite{AC}); it   can be also seen   as a zero-temperature limit  of the analogous result in Statistical Physics (\cite{Ruelle,RLL}). However, formula \eqref{eq2} is only a formal homogenization result in a discrete-to-continuum setting which is usually  non-constructive. In this paper we are raising the issue  of the actual \emph{computability} of $f_{\rm hom}(z)$.

 Explicit formulas for  $f_{\rm hom}(z)$ in terms of $f_n$ are known only in   few cases, most of which are mentioned below.
In general, it is known that the behavior of minimizing arrays $(u^k_0,\ldots, u^k_k)$ at fixed $z$, may be complex, including equi-distribution (`crystallization'; see e.g.~\cite{LB}), periodic oscillations \cite{BGelli,Giuliani}, development of discontinuities (fracture in lattice models \cite{Trusk,BLO}) or defects (internal boundary layers \cite{BraidesCicalese}). 

A robust  approach to the computation of  $f_{\rm hom}(z)$  is known under the name of  {\em Cauchy-Born} (CB) {\em rule} and  is applicable under some restrictive conditions (\cite{Ericksen,BLBL}).  It is based on the assumption   that  the homogenized energy can be  computed using  the  affine interpolations $u_j=zj$ and relying \emph{exclusively} on problems with finite  $k$. 
Various sufficient conditions for the validity of the Cauchy-Born rule have been obtained  by a number of authors mostly in the context of local minimizers  \cite{EM,HO,MakS,Poz,WCZH,CDKM,Z}. While those results are usually valid only for  subsets of loading parameters, they are  often  applicable for dimensions higher than one. They are of considerable interest, first of all,  for the development of numerical methods because the applicability of the classical CB rule makes such methods extremely efficient, even if for a  limited set of boundary conditions. The difference of our approach to   \eqref{eq2}
is that we  are interested in \emph{global} minimization (viewed as a zero temperature limit of a statistically equilibrium response) and consider the possibility that the conventional  CB rule is operative only in a  \emph{subset} of the loading parameters  while in the complementary  subset   the CB strategy  should be  appropriately \emph{generalized} or  even completely  ruled out.


The main  reason for the failure of the classical Cauchy-Born rule is the geometrical frustration  caused   by   incompatible  optimality demands imposed  by (generically non-convex and long range) potentials $f_n$  with positive integer $n$ and the discreteness of the lattice.  More specifically,  while  non-convexity  entails the formation of microstructures,  incompatibility between interactions operating at different  scales  can produce nontrivial mixing effects which are exacerbated  in the case of incommensuration between the optimal microstructures and the scale of the underlying lattice. Unraveling the intricacies of the underlying interplay between non-convexity, non-locality and discreteness, represents the main goal of this study.


If  the classical Cauchy-Born rule fails, the natural task is  to search for a nontrivial generalization of the Cauchy-Born rule. In this perspective, we pose the problem of finding the conditions for which the minimal arrays in \eqref{eq1} have `global' features in the sense that  solving a  `local' problem on a \emph{finite} domain   opens the way towards  describing the limit in \eqref{eq2}. More specifically,  the question is   whether the limiting energy $f_{\rm hom}(z)$ can be approximately computed by solving a finite set of `cell' problems modeled on \eqref{eq1} and   potentially   producing  \emph{non-affine} optimal configurations.  The validity of the so-interpreted  generalized Cauchy-Born (GCB)  rule would then require that even if the implied  `local' problems could  be   solved only  on  some subsets of   parameters,  the knowledge  of the corresponding  solutions would ensure the recovery   of the  macroscopic (homogenized) energy in the \emph{whole} range of loading parameters.

 Note that in local problems like \eqref{eq1}  the presence of   interactions  $f_n$ with $n\in\{1,\ldots,k\}$  requires $k$ boundary conditions on each side.  By fixing parametrically only the average strain $z$ in \eqref{eq2}  we effectively assume that the remaining boundary conditions are natural. This simplifying assumption may stay on the way  of acquiring, for the given  `local'   problem, the corresponding `global' features. That is why we will  understand the  `local'   GCB problem as having the right boundary conditions to ensure the  recovery of the macroscopic   energy $f_{\rm hom}(z)$. The simplest case is when the value of $f_{\rm hom}(z)$ can be achieved on arrays such that $i\mapsto u_i-zi$ is periodic with a given period, but in general   one  should  be allowed to adjust boundary conditions accordingly while keeping in mind that  these changes should not affect  the minimizers in an asymptotic sense.

 \smallskip

We now   illustrate the main  difficulties  on the way of generalizing the classical CB rule with some  known cases.  We start with the simplest  example where the conventional CB rule works trivially. It  is the case of convex nearest-neighbor  (NN) interactions; i.e., when $f_n=0$ for all $n\ge2$, and $f_1=f$ is a strictly convex function. In this case, the unique minimizer of the problem in \eqref{eq2} is the affine interpolation $u^k_j=zj$. It is independent of $k$ and hence `global': in this case the classical Cauchy-Born rule is applicable  in its simplest form, and $f_{\rm hom}(z)=f(z)$.

If we make the above example only a little more complex considering also convex next-to-nearest-neighbour (NNN) interactions; i.e., $f_n=0$ for all $n\ge3$, with $f_1$ and $f_2$ convex functions, we loose this exact characterization of the minimal arrays. However, the discrepancy between $u^k_j$ and  $zj$ decays fast away from the endpoints $j=0$ and $j=k$ of the array. A slight adjustment of the boundary-value problems, say by imposing  additional  boundary conditions $u_1=z$ and $u_{k-1}= z(k-1)$ (which do not influence the asymptotic value of the minima in \eqref{eq2})   reestablishes the affine interpolations $u^k_j=zj$ as minimizers, so that $f_{\rm hom}(z)=2(f_1(z)+f_2(2z))$. In this case the classical Cauchy-Born rule is applicable,  given that we  modify  boundary conditions in the `cell' problem. 
Note that this analysis extends to any sufficiently fast decaying set of convex potentials $ f_n$, giving $f_{\rm hom}(z)=2\sum_{n=1}^\infty f_n(nz)$.
 
Even if  we abandon the convex setting, we may still easily describe the behavior of minimum problems in \eqref{eq2} in the case of nearest-neighbor  interaction, with $f_1=f$. 
It   can be shown that   $f_{\rm hom}$  in \eqref{eq2}  is given by  the convexification   $f^{**}$ of the NN potential \cite{Braides-Gelli}.  However   the classical Cauchy-Born rule in this case has to be  properly generalized. 
Suppose,  for instance,   that  the  potential $f$  has a   double-well form. In this case the relaxation points towards  configurations containing mixtures of  the   two  energy wells. Since in this setting there are  no obstacles to  simple mixing,   the  relaxation strategy  providing $f_{\rm hom}$ is  straightforward.  
Indeed,  for each $z$ there exist $z_1$, $z_2$, $\theta\in [0,1]$ such that $f^{**}(z)= \theta f(z_1)+(1-\theta)f(z_2)$.  Hence, we can construct a function $u^z:\mathbb Z\to \mathbb R$ with $u^z_i-u^z_{i-1}\in\{z_1,z_2\}$, $u^z_0=0$ and $|u^z_i-iz|\le C$.  
 Such $u^z$ may be chosen periodic, if $\theta$ is rational, or quasiperiodic (loosely speaking, as the trace on $\mathbb Z$ of a periodic function with an irrational period) otherwise. In both cases  we obtain  `local'  minimizers with  `global' properties which allows one to talk about the  applicability of  the GCB rule.

The situation is more complex in the case when  non-convexity is combined with   frustrated (incompatible)  interactions. To show this effect in the simplest setting  it is sufficient to account for  nearest-neighbor and next-to-nearest-neighbor interactions only and we make the simplest nontrivial choice by assuming that  $f_1$ is a `double-well' potential and that $f_2$  is a   convex potential. 
In this case the homogenized potential $f_{\rm hom}$  is also known explicitly   \cite{BGelli,PP}. Its domain can be subdivided in  three zones: two zones of `convexity' where  minimizers are trivial (as for convex potentials) and a zone where (approximate) minimizers in \eqref{eq2} are two-periodic functions with $u^z_i-u^z_{i-1}\in\{z_1,z_2\}$ and $z_1+z_2=z$ (in a sense, a constrained non-convex case as above). Hence, in these three zones we have minimizers with a `global' form because  the macroscopic  energy can be obtained by solving  elementary `cell' problems.

One can say that in the  two zones of `convexity'  the classical CB rule is applicable. In the `two-periodic' third zone we  see  that  the homogeneity of the minimizers is lost but   an appropriately augmented   GCB  rule still holds. 
For the remaining values of $z$  no  `local' GCB rule is  applicable since in those cases the unique (up to reflections) minimizer  is a `two-phase' configuration  with  affine and two-periodic minimizers coexisting while being separated by a single `interface'  \cite{BraidesCicalese}. The frustration (incompatibility) manifests itself in this case  through the impossibility of the penalty-free accommodation of  next-to-nearest interactions  across such an   internal boundary layer. As a consequence, as $k$ diverges, such minimizers tend to an affine interpolation between the `convex' and `oscillating' zones  which delivers the correct  value of $f_{\rm hom}(z)$  without being a solution of any  finite `cell' problem. 
Effectively,  the `representative cell'   in this case has an infinite size and therefore no  GCB-type  `local' description of the macroscopic state is available.  A somewhat similar situation is encountered in continuum homogenization of both random \cite{Kozlov} and strongly nonlinear \cite{BDF,Muller} elastic composites. 

%

\smallskip

In what follows, we interpret the loss of `locality' in homogenization problems, which was illustrated above on the simplest example,   as a failure of the GCB rule. To shed some light on the mechanism of this phenomenon,  we consider below  a class of analytically transparent discrete problems combining nonconvexity with geometrical frustration. 

More specifically, given  the complexity of a general  asymptotic analysis for  even    one-dimensional problems of this type,  we limit our attention  to  a   class of discrete functionals of type \eqref{eq1}  with $f_1(z)= {1\over 2}f(z)+ m_1z^2$, where the function $f(z)$ is non-convex,   and quadratic $f_n(z)= f_{-n}(z)=m_nz^2$ for $n\ge 2$. The coefficients $m_n$ which introduce nonlocality and frustration, are assumed to be non negative and sufficiently integrable.  In other words,  we suppose that the non-convexity is `localized' in the nearest-neighbor interactions, while all other interactions are quadratic. The positivity of the infinite sequence ${\bf m}=\{m_n: n\ge 1\}$ is chosen to ensure that  the  implied quadratic `penalty'  is  a  measure of  the distance of the configuration $u_i$ from the affine   configuration  $L_z(i)=zi$ 
%
and can be then seen as a \emph{non-local version} of the gradient of $u-L_z$.  
 One can also say that such    penalization  brings anti-ferromagnetic interactions;   an alternative, ferromagnetic-type quadratic penalty,    was considered, for instance,   in \cite{rogers-truskinovsky}.  

The advantage of this  choice  of $f_n$  is that the  ensuing problem   can exhibit both `local' (GCB) and `global'  behavior depending on the structure of  the  sequence of scalar parameters $m_n$.  Therefore  our  goal will be  to use the chosen class of functionals to characterize the difference between CB, GCB and  non-GCB  problems  in terms of such sequences. We show that in this  naturally limited but still sufficiently rich  framework one can \emph{precisely}  specify the  factors preventing the  GCB-type  description of the macroscopic energy and pointing instead towards   the  non-GCB nature of the minimizers.  Moreover, the considered example allow us to  \emph{abstract} some  general technical tools which can  facilitate the detection and the characterization of the non-GCB asymptotic behavior in more general minimization problems.


We reiterate that even in the absence of an adequate  `cell' problem, the ensuing value of $f_{\rm hom}(z)$ is  fully determined by the homogenization formula which in our case takes the form
$f_{\rm hom}(z)=\widehat Q_{\bf m} f(z)$ where
\begin{equation}\label{eq3} 
\widehat Q_{\bf m} f(z)=\lim_{k\to+\infty} {1\over k} \min\Bigl\{\sum_{i=1}^k f(u_i-u_{i-1})+\sum_{i,j=0}^k m_{i-j} (u_i-u_j)^2: u_0=0, u_k=kz\Bigr\}\,.
\end{equation}
%
%
  The nontrivial part of the mapping  $\widehat Q_{\bf m} f$, accentuating the nonlinearity of the problem, 
 is carried by  the operator 
$
Q_{\bf m} f(z)=\widehat Q_{\bf m} f(z)-  2\sum_{n\ge 1}m_nn^2z^2.
$
Thus, if $f$  is convex, this mapping,  to which we refer as the  {\em${\bf m}$-transform} of $f$, is the identity;  actually,  the same remains true  even  if $f$ is $2m_1$-convex, in the sense that the function  $z\mapsto f(z)+ 2m_1z^2$ is  convex.  If, however, the function $f$ is not $2m_1$-convex, the   ${\bf m}$-transform of $f$  is nontrivial.  Thus, the function $Q_{\bf m} f(z)$ is  in general non-convex and $Q_{\bf m}f(z)>f^{**}(z)$  for some $z$; 
the  non-convexity of $Q_{\bf m} f(z)$ depends sensitively  and `nonlocally' on the  penalizing sequence  $\bf m$.  

  Indeed, recall that  $\widehat Q_{\bf m} f$  can be viewed as an \emph{operator} acting on the non-convex function $f$ and producing an ${\bf m}$-dependent  function which effectively represents   a  constrained  relaxation of  $f$.  In the same vein, the function $Q_{\bf m} f$ represents a nonlocally \emph{constrained convexification} of $f$.
Interpreted in such a way, the construction of $Q_{\bf m} f$  is reminiscent of energy quasiconvexification in
continuum  elasticity. The latter   deals with  minimization of  the functionals  
$
\int f(  {\bf F}) d \bf x,
$
 where    ${\bf F}$ is a    matrix   field. The role of  nonlocal constraint  in such  problems is played by the condition
    ${\rm curl \,\bf{F}}=0$,
      which is  highly nontrivial in a multidimensional setting \cite{Kristensen}.  In a  one-dimensional  setting  this whole construction can be  imitated  through the introduction of a penalizing kernel $\bf m$  mimicking the  Green's function of the constraint. As in the case of continuum elasticity, such a penalization  can introduce  incompatibility, which in a discrete setting can lead to  geometrical frustration. 
    
One of the  goals of this paper  will be to link the degree  of the non-convexity of the function $Q_{\bf m} f$ with the  breakdown of the GCB  rule. For instance, in the  parametric domain where  periodic microstructures are optimal, one can also expect  the convexity of the function $Q_{\bf m} f$. Topologically different  periodic microstructures  will  exist  in finite intervals of  $z$  where  they  can be `stretched' to secure  the commensurability with the lattice.  In such  intervals  the corresponding minimizers   posses `global' properties and the GCB  rule is  respected. However, in general,  when $z$ is  varied continuously, the  optimal   microstructure   will  change discontinuously and the domain of applicability of the GCB rule can \emph{coexist} with the domains where it breaks down.  The challenge  is to identify the conditions on $\bf m$, when, for instance,  the  knowledge of the   intervals where GCB rule is applicable,  allows one to re-construct  the   ${\bf m}$-transform of a given non-convex function $f$ also for  $z$ where the GCB rule is non-applicable.

In this paper we are not attempting to solve the  problem posed above  in its full generality and instead focus on a  physically interesting    sub-class of non-convex functions $f$ allowing one to construct  \emph{explicit}  solutions of the  minimization problem for several important   classes of penalizing kernels $\bf m$. 

Specifically, we  aim   at the development of  a comprehensive theory for \emph{bi-convex} functions $f$.  More precisely,  we assume that there is a value $z=z^*$ such that the restrictions of $f$ to $(-\infty, z^*]$ and $[z^*,+\infty)$ are both convex;  well-known  examples of bi-convex functions are  the  quadratic double-well potential ($f(z)=(|z|-1)^2$ with $z^*=0$), used for the description of phase transitions, and the  truncated quadratic potential  ($f(z)=z^2$ if $z\le 1$ and $f(z)= 1$ if $z\ge 1$), which is used in Fracture Mechanics.  In what follows we  often refer to the two convex branches of $f$ as  \emph{microscopic phases}. 

An important property of the   bi-convex  functions $f$ is that, independently of the choice of  the kernels $\bf m$,  the mapping  $\widehat Q_{\bf m} f$  is largely characterized by a phase function $\theta=\theta(z)$ which represents  the asymptotic  volume fraction  of one of the `phases' in  the limiting  minimizer,  say the limit of the percentage of indices $i$ for which $u^k_i-u^k_{i-1}\ge z^*$. When $f$ is convex, then  $\theta=0$ or $\theta=1$ and when its is bi-convex, the  central question will be to describe for a given  $\bf m$ the  form of $\theta(z)$. As we show, the applicability of GCB can be related to the emergence of  the  ${\bf m}$-dependent  `steps' on the graph of the function $\theta$ represented by the  values $\overline\theta$ for which $\{z: \theta(z)=\overline\theta\}$ is a non-degenerate interval. In what follows we refer to such intervals  as  \emph{locking} states and to the corresponding GCB-type microstructures  as  \emph{mesoscopic phases}.  This characterization is justified by the fact that  in the locking states the form of minimizers is   \emph{stable} in the sense that the set of indices $i$ at  finite $k$ such that that $u^k_i-u^k_{i-1}\ge z^*$  is  independent of $z$, up to an asymptotically negligible fraction.  Therefore, the implied  `staircase' structure of the function $\theta$   is not a feature of the discrete problem only as it survives    in the continuum limit. As we show,  the  locking states have the desired global  properties, and for such states an appropriate finite  `cell' problem can be formulated and solved. In other words, in such  states   the GCB rule  is operative and the computation of the macroscopic energy energy can be made  explicit.  

In this paper we have chosen to illustrate  all these effects by considering penalization kernels  $\bf m$ amenable to  fully \emph{explicit}   study.  Our analysis shows that a rather comprehensive picture can be obtained based on the analysis of just  two archetypal \emph{classes} of kernels.

 The first class of analytically transparent kernels contains 
 `concentrated'  (compact, localized, narrow banded, etc.) parametric sequences  $\bf m$  defined by the condition that  there exists $M\ge 2$ such that $m_n=0$ if $n\ge 2$ and $n\neq M$; here $M$ plays the role of a parameter.  We prove that for such   kernels (and independently of $f$, as long as it is non-convex)  locking states do exist and correspond to  $\theta_n={n\over M}$ with $n\in\{0,\ldots, M\}$.  Minimizers in this case, representing mesoscopic phases,  are  $M$-periodic. Moreover, we prove that   the associated phase function $\theta$ is piecewise affine, interpolating locally between the locking states $\theta_{n-1}$ and $\theta_n$. Thus, while for $\theta$ that is not a locking state we do not have GCB-type minimizers (with `global'  properties), the whole mapping $\widehat Q_{\bf m}f$ can be recovered from the knowledge of its value at those $z$ corresponding to locking states where the GCB rule is operative.
 
The second  class of analytically transparent  kernels   contains   exponentially decaying sequences $\bf m$  which we write in the parametric form $m_n=e^{-\sigma n}$ with  $\sigma>0$ playing the role of a parameter analogous to $M$ in the first class. Here again  we can give a complete description of the relaxed problem, for instance,  when $f$ is a truncated convex potential ($f$ is constant in $[z^*,+\infty)$). Given this particular structure  of non-convex potentials (describing, for instance,  lattice fracture), locking states  are either    $\theta=0$ or   $\theta\in\{{1\over k}: k\in\mathbb N\}$. In the latter case,  minimizers are  $k$-periodic  and therefore of GCB-type,  which means that they posses `global' properties.  Interestingly, we show that in each  period such minimizers have a single difference $u^k_i-u^k_{i-1}$ exceeding the threshold $z^*$ (single `crack'). Again, we prove that the set of mesoscopic phases is sufficiently rich  to provide the  `building blocks'  whose  simple mixtures  allow one to construct  the whole mapping  $\widehat Q_{\bf m}f$.
 An important difference with the case of  `concentrated' kernels  is that now the  optimal  `simple' mixtures of `global' (or  GCB) states are not unique optimal microstructures.  More precisely, we show that  even for \emph{non-locking} values of $z$ one  can build  optimal    minimizers which are  of GCB-type. For  all values of $z$ such  minimizers are \emph{quasiperiodic}  and therefore posses the desired  `global' properties, thus broadening the spectrum of possible  GCB-type microstructures.

All these explicit results,   which  also include an analytical  study of the intricate role of the  parameters $\sigma$ and $M$,  can be obtained because  for  these two   classes  of  kernels (concentrated and exponential)  one can reformulate  the original non-additive (non-local) minimum problem with presumably complex mixing properties as  an additive (local) problem with no mixing effects at all.  For concentrated kernels this is achieved by  rewriting the  non-additive problem   as a superposition of additive problems.  For exponential kernels the   reduction of complexity is due to the mapping of a  scalar problem with long-range  interactions  on  a  vectorial problem with only nearest-neighbor interactions.



Variational problems with energies  like \eqref{eq3} have been studied extensively in  the physical literature  where they emerged independently in different settings  ranging from conventional  magnetic and  mechanical systems \cite{bak,JJ} to  discotic liquid crystals \cite{degennes,Godreche and L. de Seze,HC}. In such  problems  the optimal periodicity of  a microstructure   representing the ground state (global minimum of the energy)  competes with the periodicity of the  lattice, and the geometrical frustration emerges when the two periodicities  are  incompatible (for instance, incommensurate).  Since the  interactions in  actual physical systems are   very complex, the main focus was on the study of  simplified discrete models  such as Frenkel-Kontorova model \cite{BrKi} or ANNNI model \cite{selke}. A   prototypical Ising model with antiferromagnetic  long-range interactions, which is  the simplest problem of this same type was considered in \cite{BaBru}. Two  explicit solutions for  the  class of problems with exponential kernels studied  in the present  paper, were  found  in  \cite{NT2016,NT2017}.  

In the mathematical literature discrete and continuous variational models with antiferromagnetic interactions were considered  in \cite{Braides-Gelli,BGelli,PP,
RT,Cho,GMS}.  An important  link was established by S.~Aubry and J.~Mather between variational problems of type \eqref{eq3} and the quasiperiodic trajectories  of discrete dynamical systems. 
 Recent mathematical results extending Aubry-Mather theory can be found in  \cite{Bangert,Gomes,Fathi,Garibaldi}.

 In the present   paper we  reformulate   the problems  studied previously in the framework of the theory of dynamical systems,  as problems of the  calculus of variations. This change of   perspective  allows one  to apply powerful homogenization results  providing direct access to  the corresponding  continuum limits. 
The goal is to demonstrate how, already in one-dimensional problems, the the interplay between discreteness and non-convexity compromises the classical Cauchy-Born rule and precludes the use of   conventional `cell' problems for    computation of the relaxed energies. 

In the context of discrete-to-continuum transitions, the obtained results bring new understanding of the role of  the  frustrated non-local interactions in the determination of homogenized energies. While the case of ferromagnetic interactions has been extensively studied  before, here we show that the introduction  of anti-ferromagnetic interactions  brings fundamentally new effects, most importantly  the emergence  of mesoscopic phases resulting in the locking of the minimizers  on lattice-commensurate  microstructures.
While these effects, which are clearly  lattice-induced,  appear to be  `strongly discrete', they affect the structure of the continuum energy and, in this sense, do  not disappear in the course of discrete-to-continuum transition.

Instead of  the focus on  Euler-Lagrange equations,  characteristic of the theory of dynamical systems,  our main tools are  the direct methods of the calculus of variations. In particular, we obtained  our main   results    through the   use of the novel  bounds resulting  either from the judicial choice  of periodic test functions or from  cluster minimization.   In this sense our results   complement  and broaden   the findings made in the dynamical systems framework.

One  result of this type  is the characterization of the continuum limit when non-local  interactions are concentrated on $M$-neighbors.
 The analysis of this case highlights the increasing difficulty of dealing with geometrical frustration and non-commensurability effects as  progressively more distant   interactions are incorporated, and suggests the possibility of  scale-free patterns even  in  the case of finite-range interaction kernels.  It complements  the results of Aubry \cite{aubry80}, who showed that long-range interactions  favor  hyper-uniform solutions. Another result,   allowing one to relate  the  regularity of the relaxed energies  in $\theta$ with  the existence of periodic solutions,  can be viewed as  an extension of  the link between regularity  and  the rotation number established by  Mather in the framework   the dynamical systems approach  \cite{Mather}.

In addition to explicit computations of global minimizers
  we also posed the problem of finding  the   $\Gamma$-equivalent continuum \emph{approximations} of the corresponding  lattice problems \cite{BT}.  Here we imply the construction of the asymptotic  continuum  theories accounting for the lattice scale. We succeed  in constructing  such an approximation  in  the case of an exponential kernel while also showing that the conventional formal asymptotic limit,   which neglects the underlying geometric frustration, underestimates the intricacies of the interplay between non-convexity, non-locality and discreteness and produces only a lower bound for $Q_{\bf m}f$.   This explicit example   serves as a cautionary tale demonstrating  in which form the  finite scale  lattice  effects  can survive homogenization and  affect the  macroscopic variational problem. 

\smallskip

\section{Nonlocal discrete problems and their relaxation}\label{mq:sec}
In this paper we study the asymptotic behaviour 
of particular nonlocal discrete problems parameterized by the number of nodes involved.  
This can be viewed as a discrete-to-continuum homogenization process by introducing a small parameter $\e$ and suitable scalings of the energies. 
However, with an abuse of terminology, we choose to label this process as the computation of a relaxed functional. 

Following the usual terminology, a functional $\overline \Phi$ is the  relaxation of an original functional $\Phi$ if, loosely speaking, infimum problems involving $\Phi$ have the same value as infimum problems involving $\overline \Phi$, and the latter admit solution (given that  the corresponding problem is coercive), see e.g.~\cite{DM,GCB}. 
In the context of the Calculus of Variations, the relaxed functional is usually obtained by a lower-semicontinuous envelope with respect to some topology, it is stable under continuous perturbations, and often (but not always) is stable with respect to closed constraints, such as fixed boundary values or imposed integral constraints. Moreover, if the original functional depends on some energy density, often (but not always) the relaxed functional can be characterized by a new energy density obtained as a transformation (convexification, quasiconvexification, sub-additive or $BV$-elliptic envelope, etc.) of the original energy density, so that relaxation of an energy can be viewed as an operation on  an energy density.
In our case we deal with a sequence of minimum problems, so it would be correct to talk about homogenization or $\Gamma$-convergence rather than relaxation.  
Nevertheless, we would like to highlight properties of the 
homogenized continuum energy in the same spirit of a lower-semicontinuous envelope, and hence we choose the terminology of relaxation. 

We focus on the relaxation of nonlocal discrete functionals of type \eqref{eq3}. 
They involve a non-convex function $f$ and contain a `penalization kernel' $\mathbf m$. 
The idea is to single out the  local (nearest-neighbour) interaction  in the general discrete-to-continuum problem, and consider the corresponding potential $f$ as the function that needs to be `relaxed'.
The  nonlocal (beyond nearest-neighbour) interactions are assumed to be  linear.  The corresponding quadratic term in the energy  brings the simplest penalization  into the relaxation process. We show that even such a simple penalization may still carry incompatibility and may even lead to geometrical frustration. In what follows, with a slight  abuse of terminology, we will be referring to \eqref{eq3} as a
$\mathbf m$-dependent relaxation of a non-convex energy density $f$.  
Before giving the formal definitions, we make some preliminary comments distinguishing  \emph{penalized}  relaxation from \emph{non-penalized} relaxation.  

\subsection{Nearest-neighbour interaction and quadratic penalization}  
As it is well known, the convexification of a function $f$ can be seen as the result of a discrete-to-continuum relaxation process in a local setting involving nearest-neighbour interactions only. To be more specific, 
for any $k\in\mathbb N$ and $z\in \mathbb R$ we introduce the set
\begin{equation}\label{def-Akz}
\mathcal A(k;z)=\{u\colon[0,k]\cap\mathbb N\to\mathbb R \ 
\hbox{ such that } u(0)=0, u(k)=kz\}
\end{equation}
 of admissible test functions satisfying boundary conditions.  
Here the parameter $z$ represents   the affine boundary conditions $u(i)=L_z(i)$, where  $L_z(i)=iz$. 
\begin{proposition}[a characterization of the convex envelope]\label{convex}
Let $f\colon\mathbb R\to\mathbb R$. 
Then, the convex envelope of $f$ is 
\begin{eqnarray*}
&&f^{\ast\ast}(z)=\displaystyle\lim_{k\to+\infty}\frac{1}{k}\inf\bigg\{\sum_{i=1}^k f(u(i)-u(i-1)): \ 
u \in \mathcal A(k; z)
\bigg\}.
\end{eqnarray*}
\end{proposition}
\noindent It is useful in this context to interpret Proposition \ref{convex} as a consequence of 
discrete-to-continuum $\Gamma$-convergence (see e.g.~\cite[Ch.~4.2]{GCB}).
Indeed, define for a given bounded interval $I$ and for any $\e>0$   the set of indices 
$\mathcal I_\e(I)$ 
and the set of discrete functions $\mathcal A_\e(I)$ given by 
\begin{equation}\label{def-ind}
\mathcal I_\e(I)=\{i\in\mathbb Z: \e i\in I\}, \quad \mathcal A_\e(I)=\{u\colon \e \mathcal I_\e(I)\to\mathbb R\},
\end{equation} 
respectively. 
Here and in the sequel, $u_i$ denotes the value $u(\e i)$, and  we identify $u\in\mathcal A_\e(I)$ with its piecewise-constant extension in $I$. 
Having defined 
\begin{equation}
F^0_\e(u;I)=\e\sum_{i,i-1\in \mathcal I_\e(I)}f\Big(\frac{u_i-u_{i-1}}{\e}\Big)
\end{equation}  
for $u\in\mathcal A_\e(I)$,
the $\Gamma$-limit with respect to the $L^2$-convergence of $F^0_\e$ is the functional 
 $F^0(u,I)=\int_I f^{**}(u^\prime)\, dt$ 
for $u\in H^1(I)$.
Then, choosing $\e_k=\frac{1}{k}$, by the convergence of minimum problems we get 
\begin{eqnarray*}
f^{\ast\ast}(z)&=&\min\{F^0(u;(0,1)): u(0)=0, u(1)=z\}\\
&=&\lim_{k\to+\infty}
\min\{F^0_{\e_k}(u;(0,1)): u(0)=0, u(1)=z\}, 
\end{eqnarray*}
which is the desired formula up to a change of variable. 

\begin{remark}[additivity]\label{additivity}\rm Note that the problems defining $f^{\ast\ast}$ are {\em additive}, in the sense that, 
setting 
$$\mu(k,z)=\inf\bigg\{\sum_{i=1}^k\! f(u_i-u_{i-1}): \ 
u \in \mathcal A(k; z)
\bigg\},$$ 
we have 
$\mu(k,z)=\min\big\{\mu(k_1,z_1)+\mu(k_2,z_2): k_1+k_2=k, \ k_1z_1+k_2z_2=kz\big\}$.
\end{remark}

We now add to the nearest-neighbour term, described by a non-convex function $f$,  a quadratic long-range term  which brings the simplest  penalization of global inhomogeneity while promoting uniformity in the sense of averages.

To this end  we introduce a sequence 
$\mathbf m=\{m_n\}_{n\in \mathbb N}$ such that 
\begin{equation}\label{propm}
m_n\geq 0\ \ \hbox{\rm for any } n \ \ \hbox{\rm and } \ \ m_n=o(n^{-\beta})_{n\to+\infty} 
\ \ \hbox{\rm for some }\ \ \beta>3. 
\end{equation}
Such penalization has an `antiferromagnetic' character, in that it in fact favors local oscillations induced by the non-convexity of $f$. 

\begin{figure}[h!]
\centerline{\includegraphics[width=0.7\textwidth]{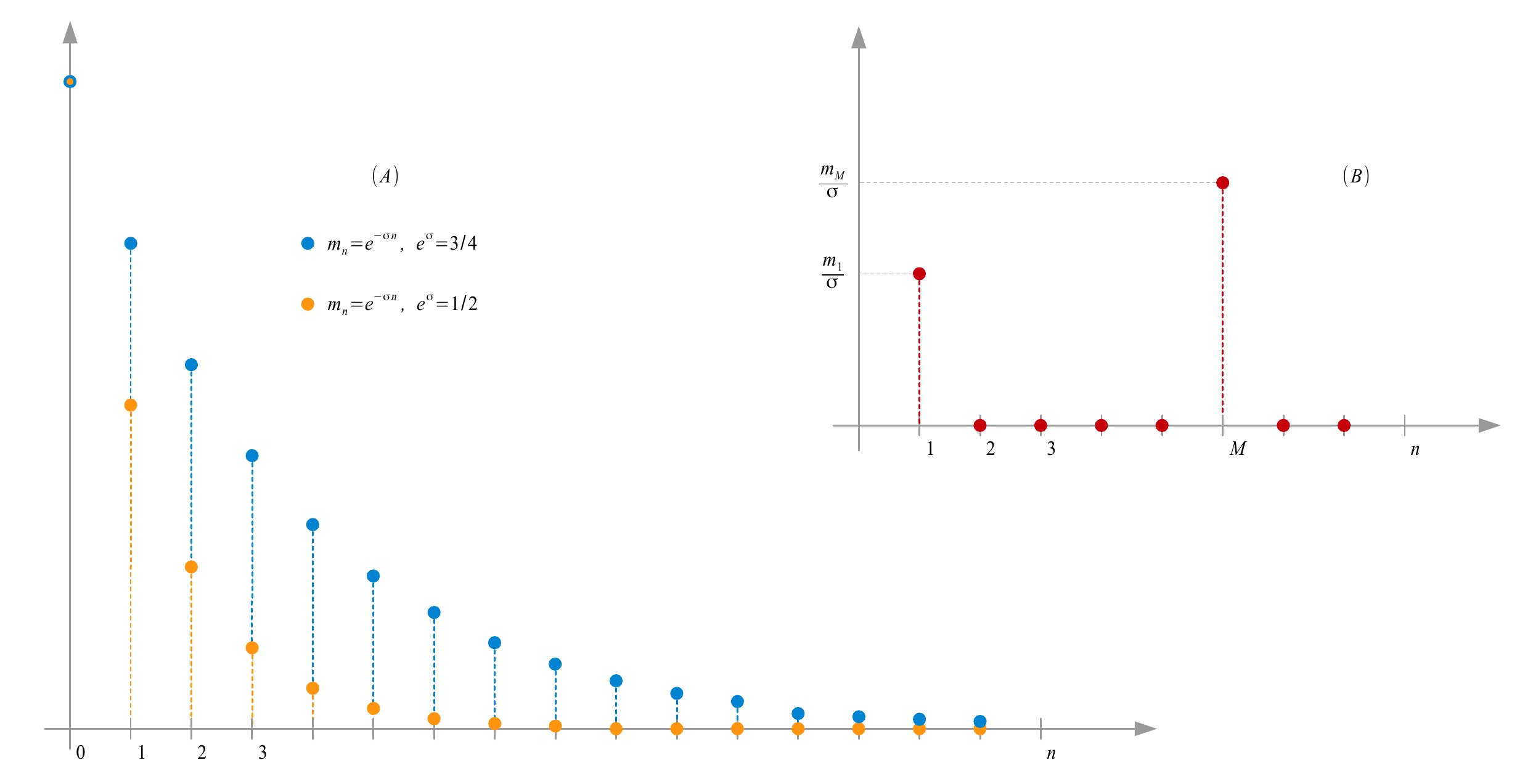}}
\caption{representation of exponential and concentrated kernels.
}
\label{istogram}
\end{figure} 
In the sequel, an important role will be played by the two special families of kernels: exponential,   $m_n=e^{-\sigma n}$, and concentrated at some $M$, $m_n=0$ for all $n$ except $n=1$ and $n=M$ with $M\geq 2$; in the latter example  one can  similarly  account for  a  parameter $\sigma$ by using the new definitions,  $m_1^\sigma= (1/\sigma)m_1$ and $m_M^\sigma= (1/\sigma)m_M$, see Fig.~\ref{istogram}.
\smallskip

Before formally defining the penalized energy, we need to make some assumptions on $f$. These assumptions will be used to obtain the existence of the limit of minimum problems. Note that the hypotheses can be relaxed, but they are stated as follows 
in order to avoid unnecessary technicalities.  Our first simplifying assumption is
that the (non-convex) potential  $f\colon \mathbb R\to [0,+\infty)$  is  non-negative 
and that it satisfy a quadratic growth hypothesis; namely,  
\begin{equation}\label{crescita}
0\le f(z) \leq c(z^2+1) \ \ \hbox{ for some } c>0.  
\end{equation}
In addition to \eqref{crescita}, we will also assume  that the function $f$  satisfies 
\begin{equation}\label{crescitasotto}
{1\over c}z^2\le f(z)+ m_1z^2.
\end{equation}
Note that hypothesis \eqref{crescitasotto} is automatically satisfied if $m_1>0$. We  will  point out specifically in which  of the   cases  assumption  \eqref{crescitasotto} is not  necessary.

\begin{definition}[relaxation with kernel $\bf m$]\label{defQcap}
For all $z\in\mathbb R$ we set 
$$
\widehat Q_{\bf m}f(z) 
=\lim_{k\to+\infty}\frac{1}{k}\inf\bigg\{\sum_{i=1}^k f(u_i-u_{i-1})+
\sum_{i,j=0}^k m_{|i-j|}(u_i-u_j)^2: 
\ u \in \mathcal A(k; z)\bigg\}.$$
\end{definition} 

The function $\widehat Q_{\mathbf m}f$ is well defined since the limit exists by known discrete-to-continuum results (see formula \eqref{limiteQ} below). For this existence the growth condition is essential; however, in some cases we will use this formula also for some degenerate $f$ for which the limit exists.
Note that, except for the case when only nearest-neighbours are involved, 
the minimum problems defining $\widehat Q_{\mathbf m}f$ 
are not additive in the sense of Remark \ref{additivity}.

\subsection{General properties of  $\widehat Q_{\bf m}f(z)$}  
In this section, we list some properties of the relaxation with kernel $\bf m$ derived from its variational nature. 
\begin{remark}[nearest-neighbour interactions]\label{corNN}\rm 
By Proposition \ref{convex}, the convex envelope of $f$ can be viewed as $\widehat Q_{\bf 0}f$, 
where $\bf m=\bf 0$ is the trivial kernel $m_n=0$ for any $n\ge 1$; that is, 
\begin{equation}\label{q0f} 
\widehat Q_{\bf 0}f(z)=f^{\ast\ast}(z). 
\end{equation} 
More in general, again by Proposition \ref{convex}, we obtain 
that $\widehat Q_{\bf m}f(z)=(f(z)+2m_1z^2)^{**}$ 
if $m_n=0$ for any $n\ge 2$. 
Note that in these cases we have no non-additivity effects. 
\end{remark} 

\begin{remark}[$\widehat Q_{\bf m}f$ as a $\Gamma$-limit]\label{queffgamma}\rm
The fact that $\widehat Q_{\bf m}f$ is well defined and some of its key properties follow by the fact that 
the functional $F$ defined by  
 $F(u)=\int_{I} \widehat Q_{\bf m}f(u^\prime)\, dt$ 
for $I$ bounded interval and  $u\in H^1(I)$ 
can be interpreted as the $\Gamma$-limit of a suitable sequence of discrete functionals $F_\e$.    
Indeed,  consider the functionals 
\begin{equation}\label{def-fe-prima}
F_\e(u;I)=\e\sum_{i,i-1\in \mathcal I_\e(I)}f\Big(\frac{u_i-u_{i-1}}{\e}\Big)+
\e \sum_{i,j\in \mathcal I_\e(I)} 
m_{|i-j|}\Big(\frac{u_i-u_j}{\e}\Big)^2  
\end{equation} 
defined in $\mathcal A_\e(I)$, with $\mathcal I_\e(I)$ and $\mathcal A_\e(I)$ as in \eqref{def-ind}.
Such functionals can be rewritten as 
$$F_\e(u;I)=\sum_{h\ge 1}\ \sum_{j,j+h\in \mathcal I_\e(I)}\e\, f^h\Big(\frac{u_{j+h}-u_j}{\e h}\Big)$$
where
 $f^1(z)=f(z)+2 m_1z^2$ and $f^h(z)=2z^2h^2m_h$ if $h>1$. 
With this notation, functionals $F_\e$ satisfy the hypotheses of \cite[Theorem 6.3]{AC}; that is,   
$f^1(z)\geq c_1 z^2$ with $c_1>0$, and $f^h(z)\leq c_hz^2$ with $\sum_{h}c_h<+\infty$. The 
lower bound follows by the growth hypothesis \eqref{crescitasotto}, and 
the upper bound by \eqref{crescita} and by hypothesis \eqref{propm} on $\bf m$.  
Hence, the $\Gamma$-limit of $F_\e$ with respect to the $L^2$-convergence is represented by the functional 
 $F(u,I)=\int_I f_{\rm hom}(u^\prime)\, dt$,  
where $f_{\rm hom}$ satisfies the homogenization formula 
\begin{equation}\label{limiteQ}
f_{\rm hom}(z)=\lim_{k\to+\infty}\frac{1}{k}\inf\bigg\{\sum_{h=1}^{k}\sum_{j=0}^{k-h-1}f^h\big(\frac{u_{j+h}-u_j}{h}\big): u\in \mathcal A(k;z)\bigg\}.
\end{equation}
Rewriting this formula, we get that the function $f_{\rm hom}$ coincides with  
the function $\widehat Q_{\bf m}f$ introduced in Definition \ref{defQcap}, 
which  proves that it 
is well-defined as a limit. 
\end{remark}

\begin{remark}\rm  
Note that, while condition \eqref{crescitasotto} can be relaxed by requiring that $f$ has a superlinear growth (not necessarily quadratic), it cannot be dropped altogether. Indeed, if $f=0$, $m_2\neq 0$ and $m_n=0$ otherwise, then the limit in Definition \ref{defQcap} does not exist. 
\end{remark}

The following proposition states 
the convexity of $\widehat Q_{\bf m}f$, 
which is ensured by the  
lower semicontinuity of the $\Gamma$-limit. 
\begin{proposition}[convexity of 
$\widehat Q_{\bf m}f$] 
\label{propT} 
Let $\bf m$ be as in \eqref{propm} and let $f\colon \mathbb R\to [0,+\infty)$ be a non-negative 
function satisfying \eqref{crescita} and \eqref{crescitasotto}. Then the function $\widehat Q_{\bf m}f$ is convex.
\end{proposition}

In the following remark we highlight that the boundary conditions can be transformed in conditions on 
a boundary layer, which are more convenient for computations. 
\begin{remark}[alternative statements of boundary conditions]\label{boundary-affine}\rm 
The boundary conditions $u_0=0$ and $u_k=kz$ can be replaced by conditions on a boundary layer.
We state two different equivalent possibilities, that will both be used in the proofs.
In the first one the boundary layer is a small portion of the whole domain, parameterized by a small $\delta$, which then we let tend to $0$, as follows
\begin{eqnarray}\label{limdelta}\nonumber
\hskip-6cm\widehat Q_{\bf m}f(z)\!
=\!\lim_{\delta\to 0}\liminf_{k\to+\infty}\frac{1}{k}\inf\bigg\{\sum_{i=1}^k f(u_i-u_{i-1})+
\!\sum_{i,j=0}^k m_{|i-j|}(u_i-u_j)^2\!: 
 u \in \mathcal A_\delta(k; z)\bigg\}\hskip-1cm\ \\
=\!\lim_{\delta\to 0}\limsup_{k\to+\infty}\frac{1}{k}\inf\bigg\{\sum_{i=1}^k f(u_i-u_{i-1})+
\!\sum_{i,j=0}^k m_{|i-j|}(u_i-u_j)^2\!: 
 u \in \mathcal A_\delta(k; z)\bigg\},
\end{eqnarray}  where
$$\mathcal A_\delta(k;z)=\{u\in\mathcal A(k;z): u_i=iz
\ \hbox{\rm if }\ i\leq \delta k \ \hbox{\rm and }\ i\geq (1-\delta)k\}.$$
In the second one the double limit is replaced by a $k$-depending boundary layer at a mesoscopic scale, as follows
\begin{eqnarray}\label{limscala}
\widehat Q_{\bf m}f(z)
=\lim_{k\to+\infty}\frac{1}{k}\inf\bigg\{\sum_{i=1}^k f(u_i-u_{i-1})+
\sum_{i,j=0}^k m_{|i-j|}(u_i-u_j)^2: 
\ u \in \mathcal A_{k^{\alpha}}(k; z)\bigg\},
\end{eqnarray} with $\alpha\in (-1,0)$. 

These formulas can be proved by an argument which is customary to variational treatments of homogenization problems (see e.g.~\cite{AC}). 
In proving formulas \eqref{limdelta} and \eqref{limscala}, it is necessary to use the growth hypothesis \eqref{crescitasotto}. 
In case it does not hold, the limits in formulas 
\eqref{limdelta} and \eqref{limscala} may be different from the limit in the definition 
of $\widehat Q_{\bf m}f$. 
\end{remark}

We now give some general estimates on $\widehat Q_{\bf m}f$. 

\begin{remark}[estimates by decomposition for $\widehat Q_{\bf m}f$]\label{lowerrem}
\rm If $\bf m=\bf m^\prime+\bf m^{\prime\prime}$; that is, $m_n=m^{\prime}_n+ m^{\prime\prime}_n$ for all $n$, and $f=g+h$,  then  we have  
$$\widehat Q_{\bf m}f(z)\geq \widehat Q_{\bf m^\prime}g(z)+
\widehat Q_{\bf m^{\prime\prime}}h(z).$$ 
\end{remark}

In Remark \ref{corNN} we have examined the case when ${\bf m}={\bf 0}$. It may be of interest to consider the case when conversely $f=0$ as in the following lemma. If $\bf m$ is as in \eqref{propm}, then we set 
\begin{equation}\label{defam}
a_{\mathbf m}=2\sum_{n=1}^{+\infty}m_n n^2. 
\end{equation}

\begin{lemma}[minimization of the quadratic part]\label{ubquic} 
Let $m_1>0$, so that \eqref{crescitasotto} is satisfied with $f=0$. 
Then we have $\widehat Q_{\bf m}0(z)= a_{\bf m}z^2$.
\end{lemma}

\begin{proof}
By using $u_i=iz$ as a test function in the definition of $\widehat Q_{\bf m}0(z)$  
we get the inequality $\widehat Q_{\bf m}0(z)\leq a_{\bf m}z^2$, 
after noting that 
$$
\lim_{k\to+\infty}\frac{1}{k}\sum_{i,j=0}^k m_{|i-j|}(i-j)^2=
\lim_{k\to+\infty}\frac{2}{k}\sum_{n=1}^k (k-n+1)m_{n}n^2=
2\sum_{n=1}^{+\infty}m_n n^2=a_{\mathbf m}.$$ 

It then suffices to prove that for all fixed $N$ we have 
$$
\widehat Q_{\bf m}0(z)\ge 2\sum_{n=1}^{N}m_n n^2z^2.
$$
With fixed $\alpha\in(-1,0)$, let $u$ be a test function for the problem in \eqref{limscala} with $f=0$ for $k^{1+\alpha}>N$. We then have 
\begin{equation}\label{epv}
{1\over k} \Bigl(\sum_{i=1}^k2 m_1(u_i-u_{i-1})^2\Bigr)\ge 2m_1 z^2.
\end{equation}

If $n\in\{2,\ldots,N\}$ and $\ell\in\{0,\ldots, n-1\}$, let $i_\ell=\lceil\frac{k-\ell}n\rceil$. We can rewrite the energy due to interactions at distance $n$ as 
\begin{eqnarray*}
{1\over k}2m_n\sum_{\ell=0}^{n-1}\sum_{i=1}^{i_\ell} (u_{\ell+in}-u_{\ell+(i-1)n})^2
&\ge& {1\over k}2m_n\sum_{\ell=0}^{n-1}i_\ell\Bigl({1\over i_\ell}\sum_{i=1}^{i_\ell}(u_{\ell+in}-u_{\ell+(i-1)n})\Bigr)^2\\
&=&{1\over k}2m_n\sum_{\ell=0}^{n-1}i_\ell\Bigl(\frac{u_{\ell+i_\ell n}-u_{\ell}}{i_\ell}\Bigr)^2={1\over k}2m_n\sum_{\ell=0}^{n-1}i_\ell n^2 z^2\\
&\ge&2m_n \frac{n}k \Bigl\lceil\frac{k-n}n\Bigr\rceil n^2 z^2= 2m_n (1+o_k(1)) n^2 z^2,
\end{eqnarray*}
where we have used the convexity inequality and the boundary condition $u_j=jz$ close to the boundary.
Summing up for $n\in\{2,\ldots,N\}$ and using \eqref{epv}, we prove the claim.
\end{proof}

In the following proposition we compare $\widehat Q_{\bf m}f$ with the convex envelope of $f$ and with $f$ itself (to be more accurate, taking into account the case that $f$ is not lower semicontinuous, with the lower-semicontin\-uous envelope of $f$).

\begin{proposition}[trivial bounds for $\widehat Q_{\bf m}f$]\label{disug}
Let $\bf m$ be as in \eqref{propm} and let $f\colon \mathbb R\to [0,+\infty)$ be a non-negative 
function satisfying \eqref{crescita} and \eqref{crescitasotto}. 
The inequalities 
\begin{equation}\label{estimates}
f^{\ast\ast}(z)+a_\mathbf m
 z^2
\leq
 \widehat Q_{\bf m}f(z)\leq \big(f(z)+a_\mathbf m
 z^2\big)^{\ast\ast}
 \le \overline f(z)+a_\mathbf m 
 z^2 
\end{equation}
hold, where
$\overline f$ denotes the lower-semi\-continuous envelope of $f$; i.e., the largest lower-semiconti\-nuous function not larger than $f$. 
\end{proposition} 
\begin{proof}
By using $u_i=iz$ as a test function in the definition of $\widehat Q_{\bf m}f(z)$  
we get the inequality $\widehat Q_{\bf m}f(z)\leq f(z)+a_{\bf m}z^2$ as in the first part of the proof of Lemma  \ref{ubquic}.
Since $\widehat Q_{\bf m}f$ is continuous by Proposition \ref{propT}, this ensures that
$\widehat Q_{\bf m}f(z)\leq \overline f(z)+a_{\bf m}z^2$.  
Since $\widehat Q_{\bf m}f(z)$ is convex, we also obtain $
\widehat Q_{\bf m}f(z)\leq (f(z)+a_\mathbf m z^2)^{\ast\ast}$.   
The lower bound is obtained by using Remark \ref{lowerrem} with the choice $g=f$, $h=0$,
$\bf m^\prime=\bf 0$ and $\bf m^{\prime\prime}=\bf m$. This gives 
$$\widehat Q_{\bf m}f(z)\geq  \widehat Q_{\bf 0}f(z)+\widehat Q_{\bf m}0(z)  = f^{\ast\ast}(z)+a_{\bf m}z^2$$ 
since $\widehat Q_{\bf m}0(z)= a_{\bf m}z^2$ by Lemma  \ref{ubquic}, and $\widehat Q_{\bf 0}f(z)=f^{\ast\ast}(z)$.
\end{proof}

\begin{corollary}\label{convcor}
If $f$ is convex, then $\widehat Q_{\bf m}f(z)=f(z)+a_{\bf m}z^2$. 
\end{corollary}

\begin{remark}[non-sharpness of lower bounds by decomposition]\rm 
If we apply Corollary \ref{convcor} to the estimate 
in Remark \ref{lowerrem} with  
$h$ convex and ${\bf m}^{\prime\prime}\neq {\bf 0}$, then the estimate gives an equality 
only if $\widehat Q_{\bf m}f(z)=f(z)+a_{\bf m}z^2$.  
\end{remark}

\subsection{Lower bound: optimization on nearest-neighbour clusters}
Rather remarkably, one can explicitly  compute  $\widehat Q_{\bf m}f$ when there is only one non-zero coefficient $m_M$ of $\bf m$ beside nearest neighbours.
The computation is obtained by optimizing on clusters of nearest neighbours of length $M$. 
 As a consequence one can obtain lower bound for a general $\bf m$, which are in general not sharp but however useful.
 
For any given $\lambda\geq 0$, we set  
\begin{equation}\label{flambda}
f_{\lambda}(z)=f(z)+\lambda z^2.
\end{equation}
In particular $f_{2m_1}(z)= f(z)+2m_1 z^2$ describes the total energy due to nearest-neighbour interactions.
We first rewrite Corollary \ref{convcor} in terms of  the effect of the convexity of this contribution.

\begin{proposition}[convex nearest-neighbour interactions]\label{coimqc}
Let $f$ be such that $f_{2m_1}$ is convex. Then
$$
\widehat Q_{\bf m}f(z)= f(z)+a_{\mathbf m} z^2.
$$
More in general, for an arbitrary $f$ this equality holds at all $z$ such that $f_{2m_1}(z)=f_{2m_1}^{**}(z)$.
\end{proposition}

\begin{proof} Applying Remark \ref{lowerrem} with $g=f$, $h=0$ and $\bf m'$ defined as $m'_1=m_1$ and $m'_n=0$ if $n\ge 2$, for all $z$ such that $f_{2m_1}(z)=f_{2m_1}^{**}(z)$ we have 
$$
\widehat Q_{\bf m}f(z)\ge \widehat Q_{\bf m'}f(z)+ \widehat Q_{\bf m''}0(z)= f^{**}_{2m_1}(z)+ a_{\bf m''}z^2= f_{2m_1}(z)+ a_{\bf m''}z^2=f(z)+a_{\mathbf m} z^2,
$$
where we have used Remark \ref{corNN}, Lemma \ref{ubquic} and the convexity hypothesis.
The converse inequality holds by Proposition \ref{disug}.
\end{proof}

Now, we can define nearest-neighbour cluster energies. More precisely, for any integer $M\geq 2$ we define 
\begin{equation}\label{lowerPM}
P^M\!f(z)=\frac{1}{M}\min\Big\{\sum_{j=1}^Mf_{2m_1}(z_j): \sum_{j=1}^{M}z_j=Mz
\Big\}+2m_M M^2z^2.
\end{equation}
For completeness of notation, we also set
$P^1\!f(z)=f_{2m_1}(z)$.

Note that if $M\ge2$ and $f_{2m_1}$ is convex then $P^M\!f(z)=f(z)+2m_1z^2+2m_M M^2z^2$.

\begin{definition}[concentrated kernels]\label{def-concentrated} 
Let $M\geq 1$. We say that a kernel $\bf m$ is {\em concentrated at $M$} if $m_n=0$ if $n\not\in\{1,M\}$.
\end{definition}

\begin{proposition}[relaxation with concentrated kernel]\label{MniPM}
If $\bf m$ is concentrated at $M$, then $\widehat Q_{\bf m}f=(P^M\!f)^{\ast\ast}$.
\end{proposition}

\begin{proof} Remark \ref{corNN} proves the claim for $M=1$. Now, assume $M\geq 2$. 
We can use formula \eqref{limscala} for the computation of $\widehat Q_{\bf m}f(z)$; in particular, we may suppose that test functions satisfy $u_i=zi$ if $i\le M$ and $i\ge k-M$. Let $u$ be a minimizer; using the notation in the proof of  Lemma \ref{ubquic}  with $i_\ell=\lceil\frac{k-\ell}M\rceil$,  we can write,
\begin{eqnarray*} 
&&\sum_{i=1}^k\Bigl(f(u_i-u_{i-1})+ 2 m_1(u_i-u_{i-1})^2\Bigr)+ 2m_M\sum_{i=M}^{k}(u_{i}-u_{i-M})^2\\
&=&\sum_{\ell=0}^{M-1}\sum_{i=1}^{Mi_\ell}{1\over M}\Bigl(f(u_i-u_{i-1})+ 2 m_1(u_i-u_{i-1})^2\Bigr) 
+2m_M\sum_{\ell=0}^{M-1}\sum_{i=1}^{i_\ell} (u_{\ell+iM}-u_{\ell+(i-1)M})^2 +C_z,
\end{eqnarray*}
where $C_z$ is a constant taking into account extra boundary interactions, with $|C_z|\le M C(1+z^2)$ independent of $k$.
We then estimate 
\begin{eqnarray*} 
&&\hskip-1cm\sum_{\ell=0}^{M-1}\sum_{i=1}^{Mi_\ell}{1\over M}\Bigl(f(u_i-u_{i-1})+ 2 m_1(u_i-u_{i-1})^2\Bigr) +2m_M\sum_{\ell=0}^{M-1}\sum_{i=1}^{i_\ell} (u_{\ell+iM}-u_{\ell+(i-1)M})^2 
\\
&\ge&\sum_{\ell=0}^{M-1}\sum_{i=1}^{i_\ell} P^Mf\Bigl({u_{\ell+iM}-u_{\ell+(i-1)M}\over M}\Bigr) 
\ge\sum_{\ell=0}^{M-1}\sum_{i=1}^{i_\ell} (P^Mf)^{**}\Bigl({u_{\ell+iM}-u_{\ell+(i-1)M}\over M}\Bigr) 
\\
&\ge&\sum_{\ell=0}^{M-1}i_\ell(P^Mf)^{**}\Bigl(\frac{u_{\ell+i_\ell M}-u_{\ell}}{i_\ell M}\Bigr) 
=\sum_{\ell=0}^{M-1}i_\ell(P^Mf)^{**}(z) \\
&\ge& M \Bigl\lceil\frac{k-M}M\Bigr\rceil (P^Mf)^{**}(z)\,.
\end{eqnarray*}
Dividing by $k$ and taking the limit as $k\to +\infty$ we obtain the lower bound.

To prove that the lower bound is sharp it suffices to choose a minimizer $z_1,\ldots, z_M$ for $P^Mf(z)$, extend it by $M$-periodicity and define a test function $u$ on $\{0,\ldots, k\}$ with $k=nM$ by setting $u_0=0$, $u_i-u_{i-1}=z$ if $i\in\{1,\ldots, M\}\cup \{k-M+1,\ldots, k\}$, and $u_i-u_{i-1}=z_i$ otherwise. Using this test function and letting $k\to+\infty$, we obtain $\widehat Q_{\bf m}f\le P^Mf$. Since $\widehat Q_{\bf m}f$ is convex, we finally get $\widehat Q_{\bf m}f\le (P^Mf)^{**}$.
\end{proof} 

\begin{remark}[general concentrated interactions]\label{gecoin}\rm In the previous proposition we have considered quadratic interactions between $M$th neighbours. Actually, it is not necessary to assume quadraticity or even convexity of these interactions, and the same proof shows that
\begin{equation}
\lim_{k\to+\infty}{1\over k}\min\Bigl\{\sum_{i=1}^kf(u_i-u_{i-1})+ \sum_{i=M}^kg(u_i-u_{i-M}): u_0=0, u_k=kz\Bigr\}=  \psi^{**}(z),
\end{equation}
where $f,g\colon\mathbf R\to[0,+\infty)$ are such that
 $f$ is of quadratic growth and $g$ satisfies a quadratic bound from above, and  $\psi$ is defined by
\begin{equation}\label{lowerPM-gen}
\psi(z)=\frac{1}{M}\min\bigg\{\sum_{j=1}^Mf(z_j): \sum_{j=1}^{M}z_j=Mz\bigg\}+g(z).
\end{equation}
\end{remark}

\begin{remark}[periodic recovery sequences and multiplicity of minimizers]\rm
Note that if $P^Mf(z)=(P^Mf)^{**}(z)$ and $\{z_i\}$ is a minimizer for $P^Mf(z)$ extended by $M$-periodicity, a function $u$ with $u_0=0$, $u_i-u_{i-1}=z_i$ gives a recovery sequence for the $\Gamma$-limit of the functionals \eqref{def-fe-prima} at $u(x)=zx$. Note that  $u_i-zi$ is $M$-periodic.

We also observe that if $\{z_1,\ldots, z_M\}$ is a minimizer, then any permutation of its values gives a minimizer.
\end{remark}

\begin{proposition}[a lower bound for general $\bf m$]  Let $\bf m$ be any kernel; then for any $M$ the following estimate holds
\begin{eqnarray}\label{lbpM}
\widehat Q_{\bf m}f(z)\geq(P^M\!f)^{\ast\ast}(z)+2\sum_{\substack{n\ge 2\\ n\neq M}}n^2m_n z^2,
\end{eqnarray} 
and in particular we have
$\displaystyle 
\widehat Q_{\bf m}f(z)\geq\sup_{M\ge1}\biggl((P^M\!f)^{\ast\ast}(z)+2\sum_{\substack{n\ge 2\\ n\neq M}}n^2m_n z^2\biggr)$.
\end{proposition}

\begin{proof} Inequality \eqref{lbpM} is obtained by using Remark \ref{lowerrem} with ${\bf m^\prime}=(m_1,0,\dots, 0, m_M, 0, \dots)$, Proposition \ref{MniPM}, and the fact that $\widehat Q_{\bf m^{\prime\prime}}0(z)=2\sum_{n\not\in\{1,M\}}n^2m_n z^2$.
If $M=1$, the estimate is an immediate consequence of Remarks \ref{corNN} and \ref{lowerrem}. 
\end{proof} 

\subsection{Upper bound: optimization over periodic patterns}
In order to give an upper bound for $\widehat  Q_{\bf m}f$, it is of interest to consider minimum problems on sets of $N$-periodic functions. We will see that when the value $\widehat  Q_{\bf m}f(z)$ is obtained 
by this periodic minimization, which can be interpreted as a Cauchy-Born approach, it is possible to deduce further structural properties of the relaxed functional. 

For $N\in\mathbb N$ we define 
\begin{equation}\label{phiminus1}
\widehat R^N_{\bf m}f(z)=\frac{1}{N}\inf\big\{F^\#(u;[0,N]):\  i\mapsto u_i-zi \hbox{ is } N\hbox{-periodic}\big\}, 
\end{equation}
where
\begin{equation*}
F^\#(u;[0,N])= \sum_{i=1}^{N} f(u_i-u_{i-1})+\sum_{i=1}^N\sum_{j\in\mathbb Z}m_{|i-j|}(u_i-u_j)^2.
\end{equation*} 
Note that each site $i\in\{1,\dots, N\}$ interacts with all $j\in\mathbb Z$. 
Using periodic functions as test functions in the $\Gamma$-limit, we see that $\widehat R^N_{\bf m}f(z)\geq \widehat Q_{\bf m}f(z)$ for all $N$, 
so that, setting 
\begin{equation*}
\widehat R_{\bf m}f(z) 
=\Big(\inf_{N} \widehat R^N_{\bf m}f(z)\Big)^{\ast\ast}, 
\end{equation*}
 we obtain a bound for the $\bf m$-relaxation of $f$.  More specifically, we can write 
\begin{equation}\label{boundR}
f(z)+a_{\bf m}z^2\geq \widehat R^N_{\bf m}f(z) \geq \widehat R_{\bf m}f(z)\geq \widehat Q_{\bf m}f(z) 
\geq a_{\bf m}z^2\end{equation}
where  $N$ is arbitrary; the first estimate is obtained by 
taking $u_i=iz$. 

An application of Remark \ref{boundary-affine} to boundary conditions allows one to show that in \eqref{phiminus1} 
we can asymptotically neglect 
the interaction terms with sites outside $[0,N]$. Then, we have the following proposition. 
\begin{proposition}\label{perest}
For all $z\in\mathbb R$ we have $\widehat R_{\bf m}f(z)=\lim\limits_{N\to+\infty}\widehat R^N_{\bf m}f(z)=\widehat Q_{\bf m}f(z)$. 
\end{proposition}
\noindent Accordingly, the $\bf m$-relaxation can be alternatively defined as a limit of minimum problems constructed on periodic functions.

\begin{remark}[global periodic solutions]\rm Note that in general the equality in Proposition \ref{perest}  is not attained at finite $N$. However, in some cases the knowledge of $\widehat R^N_{\bf m}f$ for some finite $N$ is sufficient for the description of $\widehat Q_{\bf m}f$. A notable case is that of nearest and next-to-nearest neighbor interactions, for which a general formula for $\widehat Q_{\bf m}f$ can be proven using this approach. In the notation above that formula simply reads $\widehat Q_{\bf m}f= (\widehat R^2_{\bf m}f)^{**}$ \cite{BGS,PP}. In particular, if $f$ is a double-well energy with minimum value $0$ attained for $z\in\{-1,1\}$ then in a neighbourhood of $0$ we have $\widehat Q_{\bf m}f(z)= \widehat R^2_{\bf m}f(z)$; that is, the minimum for $\widehat Q_{\bf m}f(z)$ is reached on functions with $u_i-zi$ 2-periodic, up to an error due to the boundary conditions and vanishing as $k\to+\infty$. In this sense, such problems have `global' solutions and are therefore solvable by the application of the GCB rule.\end{remark}

\subsection{The $\bf m$-transform of $f$}  
In view of Proposition \ref{disug},  in order to compare $\widehat Q_{\bf m}f$ with $f$ we can subtract the quadratic term.
In this way, the bounds in \eqref{estimates} are rewritten as
\begin{equation}\label{estimates2}
f^{\ast\ast}(z)\leq \widehat Q_{\bf m}f(z)-a_\mathbf m z^2 \le \overline f(z).
 \end{equation} 
 This suggests to interpret the function $\widehat Q_{\bf m}f(z)-a_\mathbf m z^2$ as an independent operator acting on $f$. 
 We then give the following definition. 
 \begin{definition}[$\mathbf m$-transform of $f$]\label{defin} 
Let $\bf m$ be as in \eqref{propm} and let $f\colon\mathbb R\to[0,+\infty)$ satisfy \eqref{crescita} and \eqref{crescitasotto}.  
The {\em ${\mathbf m}$-transform of $f$}  is the function $Q_{\bf m}f\colon\mathbb R\to[0,+\infty)$ defined as 
\begin{equation}\label{quet}
Q_{\bf m}f(z)=\widehat Q_{\bf m}f(z)-a_{\mathbf m}z^2.
\end{equation}
\end{definition}

 Given that, by \eqref{estimates2}, 
 \begin{equation}\label{estimates3} 
 f^{\ast\ast}(z)\leq Q_{\bf m}f(z) \le \overline f(z),
 \end{equation} 
the $\mathbf m$-transform of $f$ can be viewed as an $\mathbf m$-dependent  interpolation between $f$ and $f^{**}$. 
 
We start the study of the $\mathbf m$-transform  with the observation that at $z$ fixed the construction of $Q_{\bf m}f(z)$ can be interpreted in a variational sense as a minimization problem with a penalization term involving a distance from the affine function $L_z$.  This claim is justified by Remarks \ref{vadeqm} and \ref{penaltyrem} below.  

\begin{remark}[variational definition of $Q_{\bf m}f$]\label{vadeqm}\rm 
Note that, when $u_i=iz$, then 
$$a_{\mathbf m}z^2=\lim_{k\to+\infty} 
{1\over k} \sum_{i,j=0}^k m_{|i-j|}(u_i-u_j)^2.$$ 
Hence, we have the equality 
\begin{equation}\label{quet1}
\displaystyle Q_{\mathbf m}f(z)=\displaystyle\!\lim_{k\to+\infty}\frac{1}{k}\inf\bigg\{\sum_{i=1}^k\! f(u_i-u_{i-1})+\!
\sum_{i,j=0}^k \! m_{|i-j|}\big((u_i-u_j)^2-(i-j)^2z^2\big)\!: 
\!u \in \mathcal A(k; z)
\bigg\}.
\end{equation}
\end{remark}

\begin{remark}[interpretation of the penalty term as a distance]\label{penaltyrem}\rm 
If $m_1>0$, then the last sum in \eqref{quet1}  is a  measure of the distance from $u_i$ to the affine function $L_z(i)=iz$. 
To show this, we first note that by Remark \ref{boundary-affine} we can restrict to test functions $u$ such that 
$u_{i}=iz$ for $i\leq k^{\alpha+1}$ and $i\geq k-k^{\alpha+1}$ for some $\alpha\in (-1,0)$. 

Now, for any $\ell\in\{1,\dots, k\}$ we consider the sum of the terms with $|i-j|=\ell$, obtaining

\begin{eqnarray*}
&&\hspace{-1cm}\sum_{|i-j|=\ell}\big((u_i-u_j)^2-(i-j)^2z^2\big) \\
&&= 
\sum_{|i-j|=\ell}\big((u_i-iz)-(u_j-jz)\big)^2 +2z\sum_{|i-j|=\ell} ((u_i-iz)-(u_j-jz))(i-j)\\ 
&&= 
\sum_{|i-j|=\ell}\big((u_i-iz)-(u_j-jz)\big)^2 +4z\ell\sum_{i-j=\ell} ((u_i-iz)-(u_j-jz))\\
&&= 
\sum_{|i-j|=\ell}\big((u_i-iz)-(u_j-jz)\big)^2 +4z\ell
\sum_{r=0}^{\ell-1} \big((u_{k_{r,\ell}}-k_{r,\ell}z)-(u_{r}-rz)\big), 
\end{eqnarray*} 
where $k_{r,\ell}=r+\ell \lfloor \frac{k-r}{\ell}\rfloor$.

If $\ell\leq k^{\alpha+1}$, then 
$r\leq k^{\alpha}k$ and  
$k_{r,\ell}=r+\ell \lfloor \frac{k-r}{\ell}\rfloor\geq   k-\ell \geq (1-k^{\alpha})k$, 
so that $u_{k_{r,\ell}}-u_{r}=\ell \lfloor \frac{k-r}{\ell}\rfloor=\ell \lfloor \frac{k}{\ell}\rfloor$, and 
the last term in the sum vanishes, 
so that 
\begin{eqnarray*}
\frac{1}{k}\sum_{|i-j|\leq k^{\alpha+1}}\big((u_i-u_j)^2-(i-j)^2z^2\big) 
= \frac{1}{k}\sum_{|i-j|\leq k^{\alpha+1}}\big((u_i-iz)-(u_j-jz)\big)^2.
\end{eqnarray*} 

Now, we fix $\delta>0$. Recalling the decay condition \eqref{propm} on $\bf m$, there exists 
$\ell_\delta$ such that for $\ell>\ell_\delta$ 
we have $m_\ell <\delta \ell^{-\beta}$. If $k$ is such that $k^{\alpha+1}>\ell_\delta$, then 
\begin{eqnarray*}
&&\frac{1}{k}\sum_{|i-j|>k^{\alpha+1}}^k \! m_{|i-j|}\big((u_i-u_j)^2-(i-j)^2z^2\big)\leq \frac{2}{k}\sum_{\ell>k^{\alpha+1}}^k\sum_{i=\ell}^k \! m_{\ell}(u_i-u_{i-\ell})^2\\
&&\leq \frac{2}{k}\sum_{\ell>k^{\alpha+1}}\ell^2m_{\ell}\sum_{i=1}^k \! (u_i-u_{i-1})^2
\leq \frac{2\delta}{k}\sum_{\ell>k^{\alpha+1}}\ell^{2-\beta}\sum_{i=1}^k \! (u_i-u_{i-1})^2\,.
\end{eqnarray*}
Note that in our computations we limit to $u$ satisfying $\sum_{i=1}^k \! (u_i-u_{i-1})^2\le Ck$ by \eqref{crescitasotto}, so that this term is negligible as $k\to+\infty$. 
Likewise, we obtain
\begin{eqnarray*}
{2m_1\over k} \sum_{i=1}^k \! (u_i-u_{i-1}-z)^2&\le&
{1\over k}\sum_{\ell=1}^k\sum_{|i-j|=\ell}m_{|i-j|}\big((u_i-u_j)^2-(i-j)^2z^2\big) \\
&\le &{2\over k}\Bigl(\sum_{\ell=1}^\infty \ell^2 m_\ell\Bigr)\sum_{i=1}^k(u_i-u_{i-1}-z)^2\,. \\
\end{eqnarray*}
This double inequality shows that the quadratic part is equivalent to the square of the $L^2$ norm of the derivative of 
$u-L_z$, where $u$ is identified with the piecewise-affine function on $(0,1)$ with $u'=u_i-u_{i-1}$ on $({i-1\over k}, {i\over k})$.
\end{remark}

Some general algebraic properties deriving from the definition of $Q_{\bf m}f$ are the following.

\begin{remark}[properties of $Q_{\bf m}$]\label{properties}\rm  

\ 

{\rm (i)} $Q_{\bf m} (f+g)\ge Q_{s{\bf m}}f+ Q_{(1-s){\bf m}} g$ for all $s\in(0,1)$;

{\rm (ii)} if $g$ is convex $Q_{\bf m} (f+g)\ge (Q_{\mathbf m}f)+ g$;

{\rm (iii)} if $g$ is affine then  $Q_{\bf m} (f+g)= (Q_{\bf m}f)+ g$;

{\rm (iv)} if $r\geq 0$, then 
 $
Q_{\bf m}(rf)(z)= rQ_{{\bf m}/r}f(z);
 $

{\rm (v)} if $r\in \mathbb R$ and $(f\circ L_r)(z)= f(rz)$ then 
 $
Q_{\bf m}(f\circ L_r)(z)= Q_{{\bf m}/r^2}f(rz);
 $

{\rm (vi)} if $\lambda \in \mathbb  R$ and we denote $(f\circ T_\lambda)(z)= f(z-\lambda)$ then
 $
Q_{\bf m}(f\circ T_\lambda)(z)= Q_{\bf m}f(z-\lambda).
 $
\smallskip 

Properties (i)-(v) follow directly from the definition of $Q_{\bf m} f$. 
We give some details for the proof of (vi), since for this  
we have to modify the boundary condition of the test functions, 
using \eqref{limscala} in Remark \ref{boundary-affine}. 
For any test function $u$ for $Q_{\bf m} (f\circ T_\lambda)(z)$ we consider $u^\lambda$ given by 
$u^\lambda_i=u_i-\lambda i$, which is a test function for 
$Q_{\bf m} f(z-\lambda)$ obtaining 
\begin{eqnarray*}
&&\sum_{i=1}^k f(u_i-u_{i-1}-\lambda)+
\sum_{i,j=0}^k m_{|i-j|}(u_i-u_j)^2-\sum_{i,j=0}^k m_{|i-j|}(i-j)^2z^2\\
&&\hspace{5mm}=
\sum_{i=1}^k f(u^\lambda_i-u^\lambda_{i-1})+
\sum_{i,j=0}^k m_{|i-j|}(u^\lambda_i-u^\lambda_j)^2-\sum_{i,j=0}^k m_{|i-j|}(i-j)^2(z-\lambda)^2\\
&&\hspace{10mm}+
2\lambda\sum_{i,j=0}^k m_{|i-j|}(u_i-u_j-z(i-j))(i-j).
\end{eqnarray*}
Then, (vi) holds if we show that 
$$
\lim_{k\to+\infty}\frac{1}{k}\sum_{i,j=0}^k m_{|i-j|}(u_i-u_j)(i-j)=a_{\bf m}z. 
$$ 
Now, we note that in the sum $ \sum_{i,j=0}^k m_{|i-j|}(u_i-u_j)(i-j)$ we can regroup the terms with $|i-j|=\ell$ 
and obtain a telescopic sum whose ending terms are in the boundary layer. Hence, since for each $\ell$ these sums are exactly 
$\ell$, we have   
\begin{eqnarray*}\lim_{k\to+\infty}\frac{1}{k}\sum_{i,j=0}^k m_{|i-j|}(u_i-u_j)(i-j)&=&
\lim_{k\to+\infty}\frac{2}{k}\sum_{\ell=1}^k \ell (m_{\ell}(u_k-u_0)\ell)
\\
&=&
\lim_{k\to+\infty}\frac{2}{k} kz\sum_{\ell=1}^k m_{\ell}\ell^2=a_{\bf m}z, 
\end{eqnarray*}
concluding the proof of (vi). 
\end{remark}

\begin{definition}[stability under $\bf m$-transform] 
We say that $z$ is a {\em point of $\bf m$-stability for $f$} if 
$Q_{\bf m}f(z)=f(z)$.
If this equality holds for all $z$, we say that $f$ is {\em ${\bf m}$-stable}. 
\end{definition} 
\begin{remark}[global properties of points of stability]\rm 
Let $z$ be a point of $\bf m$-stability for $f$. Then, the value of 
$\widehat Q_{\bf m}f(z)$ is realized by choosing the affine function $u\in\mathcal A(k;z)$ given by $u_i=iz$ in each minimum problem in Definition \ref{defQcap}. 
\end{remark}

We recall that $f_{\lambda}(z)=f(z)+\lambda z^2$ as in \eqref{flambda}.

\begin{proposition}[$\bf m$-stability and convexity]\label{comst}\ 

{\rm \ (i)} if $f$ is $\bf m$-stable then $f_{a_{\bf m}}$ is convex;

{\rm (ii)} if $f_{2m_1}$  is convex then $f$ is $\bf m$-stable. 
\end{proposition}

\begin{proof} Claim (i) follows from the definition of $\bf m$-stability since $f_{a_{\bf m}}=\widehat Q_{\bf m}f$. Claim (ii) is given by Proposition \ref{coimqc}.
\end{proof}

\begin{remark}[`moderately' non-convex functions are  $\bf m$-stable]\label{nocopeco}\rm
The proposition above implies that if $f$ is  `moderately non-convex' then it is also $\bf m$-stable.  This is valid in particular if $f$ is twice differentiable and  \begin{equation}\label{modnconv}
\inf_z f^{\prime\prime}(z)>-4 m_1.
\end{equation} 
\end{remark}

\begin{figure}[h!]
\centerline{\includegraphics[width=0.6\textwidth]{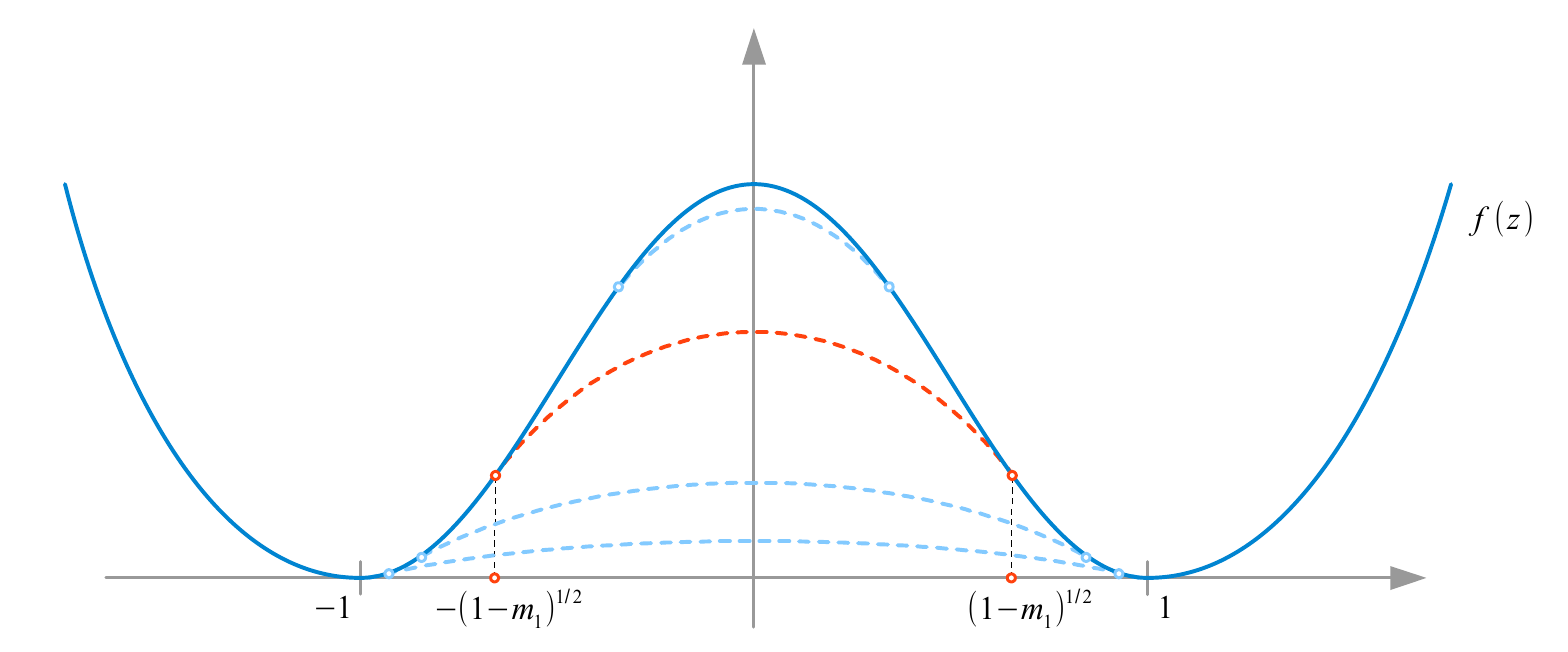}}
\caption{the function $Q_{\bf m}f$ in Remark \ref{corNN2} 
with $f(z)=(1-z^2)^2$ 
for different values of $m_1<1$.}
\label{cor-fig}
\end{figure} 

\begin{remark}[nearest-neighbour interactions]\label{corNN2}\rm 
By Remark \ref{corNN} we get that 
\begin{enumerate}
\item[{\rm(i)}] if $m_n=0$ for any $n\ge 1$, then 
$Q_{\bf m}f(z)=\widehat Q_{\bf m}f(z)=f^{\ast\ast}(z);$
\item[{\rm(ii)}] if $m_n=0$ for any $n\ge 2$, then 
$Q_{\bf m}f(z)=(f(z)+2m_1z^2)^{**}-2m_1z^2.$
\end{enumerate}
In the second case, we note that 
in general if $m_1\neq 0$ both inequalities in \eqref{estimates3} may be strict for some values of $z$. 
For example, if $f(z)=(1-z^2)^2$ and $m_1\leq 1$, then 
$$Q_{\bf m}f(z)=\begin{cases}
(1-z^2)^2 & \hbox{\rm if }\ z\leq -\sqrt{1-m_1}\\
m_1(2-m_1)   -2m_1z^2 & \hbox{\rm if }\  
|z| \leq \sqrt{1-m_1}\\
(1-z^2)^2 & \hbox{\rm if }\ z\geq \sqrt{1-m_1},
\end{cases} $$
and both inequalities are strict for $|z|<\sqrt{1-m_1}$ (see Fig.~\ref{cor-fig}).  
Conversely, if 
$m_1\geq 1$ then 
$Q_{\bf m}f(z)=f(z)$ for any $z$;  in particular in this case $f$ is $\bf m$-stable (but not convex).
\end{remark}

\begin{remark}[regularity properties]\label{regqf}
\rm From equality \eqref{quet} we deduce that for any $\bf m$ the operator $Q_{\bf m}$ has the same regularity properties of $\widehat Q_{\bf m}$; that is, $Q_{\bf m}f$ has the regularity properties of a convex function. In particular, 
$Q_{\bf m}f$ is locally Lipschitz, which is then a necessary condition for $f$ to be $\bf m$-stable.   
Note that by \eqref{estimates2} the convexity of $f$ is a sufficient condition for the stability with respect to any $\bf m$. 
\end{remark} 

\begin{proposition} Let $Q_{\bf m}^0f=f$ and define iteratively $Q_{\bf m}^nf=Q_{\bf m}(Q_{\bf m}^{n-1}f)$. Then the sequence $Q_{\bf m}^nf$ is non-increasing and its limit $Q_{\bf m}^{\infty}f$ is $\bf m$-stable.
\end{proposition}

\begin{proof}
The sequence is non-increasing by \eqref{estimates2}. Moreover $Q_{\bf m}^nf\ge f^{**}$ for all $n$. Since the functions $Q_{\bf m}^nf$ are  equi-Lipschitz continuous by Remark \ref{regqf}, they converge uniformly on compact sets to their limit $Q_{\bf m}^{\infty}f$ by Ascoli-Arzel\`a's Theorem. Since $Q_{\bf m}$ is continuous with respect to the uniformly convergence on compact sets, we have $Q_{\bf m}^{\infty}f= \lim_n Q_{\bf m}^nf =  Q_{\bf m}(\lim_n Q_{\bf m}^{n-1}f)=  Q_{\bf m}(Q_{\bf m}^{\infty}f)$.
\end{proof}

The following proposition states that for non-trivial kernel concentrated at $M\geq 2$ stable functions are only $f$ such that  $f_{2m_1}$ is convex, which is a trivial condition implying stability by Proposition \ref{comst}(ii). Moreover, iteration of the $\bf m$ transform gives a strictly decreasing sequence. 

\begin{proposition} \label{stab1M}
Let $\bf m$ be a non-trivial kernel concentrated at $M\geq 2$; that is,  with $m_M\neq 0$.  
In this case: 

{\rm \ \ (i)} $f$ is ${\bf m}$-stable if and only if $f_{2m_1}$ is convex; 

{\rm \ (ii)}  if $f_{2m_1}$ is not convex then for any $n$, there exists $z$ such that 
$Q_{\bf m}^nf(z)>Q_{\bf m}^{n+1}f(z)$; 

{\rm (iii)}  $Q_{\bf m}^{\infty}f(z)=f_{2m_1}^{\ast\ast}(z)-2m_1z^2$. 
\end{proposition}
\begin{proof}
{\rm (i)}  By Proposition \ref{comst} we only have to prove that the convexity of $f_{2m_1}$ is necessary for the ${\bf m}$-stability of $f$. We then suppose that $f$ is $\bf m$-stable and $f_{2m_1}$ is not convex, and show that there exists $\overline z$ such that $f(\overline z)>Q_{\bf m}f(\overline z)$, contradicting the ${\bf m}$-stability of $f$. 
 
From Proposition \ref{coimqc} we have that $\widehat Q_{\bf m}f(z)=f_{a_{\bf m}}(z)$ for all $z$ such that $f_{2m_1}(z)=f_{2m_1}^{\ast\ast}(z)$. We consider a maximal interval where $f_{2m_1}>f_{2m_1}^{\ast\ast}$. By the growth conditions on $f$ and its continuity (since we suppose that it is $\bf m$-stable) this interval is a bounded open interval $(S_0,S_M)$, and we have 
$\widehat Q_{\bf m}f(S_0)=f_{a_{\bf m}}(S_0)$ and $\widehat Q_{\bf m}f(S_M)=f_{a_{\bf m}}(S_M)$.

Note that, upon setting
$$
r(z)=f_{2m_1}(S_0)+\frac{f_{2m_1}(S_M)-f_{2m_1}(S_0)}{S_M-S_0}(z-S_0),
$$
 for $z\in [S_0,S_M]$ we have 
 \begin{equation*}
\frac{1}{M}\min\Big\{\sum_{j=1}^Mf_{2m_1}(z_j): \sum_{j=1}^M z_j=Mz\Big\}\geq
f_{2m_1}^{\ast\ast}(z)=r(z), 
\end{equation*}
with equality if and only if minimal $z_j$ belong to $\{S_0,S_M\}$ for all $j$,
which implies that  $z\in\{S_h:h\in\{0,\dots, M\}\}$, where
$$S_h 
=S_0+h\frac{S_M-S_0}{M}.$$
We then have
\begin{equation*}
P^M\!f(S_h)=r(S_h)+2m_MM^2S_h^2.  
\end{equation*}
Since
\begin{equation*}
P^M\!f(z)\geq(P^M\!f)^{\ast\ast}(z)\geq f_{2m_1}^{**}(z)+2m_M M^2z^2=
r(z)+2m_M M^2z^2,
\end{equation*}
and $Q_{\bf m}f=(P^M\!f)^{\ast\ast}$  by Proposition \ref{MniPM}, in particular we 
have
\begin{equation*}
\widehat Q_{\bf m}f(S_h)=(P^M\!f)^{\ast\ast}(S_h)=P^M\!f(S_h)=r(S_h)+2m_MM^2S_h^2,  
\end{equation*}
from which we get
$$Q_{\bf m}f(S_h)=r(S_h)-2m_1S_h^2.$$
If $h\in\{1,\ldots,M-1\}$ we have
$$f(S_h)+2m_1S_h^2=f_{2m_1}(S_h)>f_{2m_1}^{\ast\ast}(S_h)=r(S_h)$$ 
which implies
$$f(S_h)>r(S_h)-2m_1S_h^2 =Q_{\bf m}f(S_h),$$
which contradicts the stability of $f$. 

Note that indeed 
\begin{equation}\label{striu}
\widehat Q_{\bf m}f(z) > r(z)+2m_MM^2z^2\hbox{ if }  z\in (S_h,S_{h+1}).
\end{equation}
To check this observe that, since $\widehat Q_{\bf m}f(z)= (P^M\!f)^{\ast\ast}(z)$, there exist 
$z_1,z_2\in [S_h,S_{h+1}]$ and $t\in [0,1]$ such that $z=tz_1+(1-t)z_2$ and 
\begin{eqnarray*}\nonumber
\widehat Q_{\bf m}f(z)&=& t\,P^M\!f(z_1)+ (1-t) \,P^M\!f(z_2)\\
&\ge&  t\,r(z_1)+ (1-t) r(z_2)+ t\,2m_MM^2z_1^2+ (1-t) \,2m_MM^2z_2^2,\nonumber
\end{eqnarray*}
and we get \eqref{striu} unless $z_1=z_2=z$. The latter case is ruled out, as we would have $\widehat Q_{\bf m}f(z)= t\,P^M\!f(z)$;
that is, $z\in\{S_h, S_{h+1}\}$.
\smallskip

{\rm (ii)} 
We fix $h\in \{0,\ldots, M-1\}$ and consider any interval $(S_h,S_{h+1})$ as defined in the proof of claim (i) above. 
If  $z\in (S_h,S_{h+1})$, by \eqref{striu} we have
\begin{eqnarray}\label{stur}
Q_{\bf m}f(z)+2m_1z^2&=& \widehat Q_{\bf m}f(z)- 2m_MM^2z^2\nonumber
\\
&>&r(z)= f^{**}_{2m_1}(z)=(f(z)+2m_1z^2)^{**}\nonumber\\
&\ge&
(Q_{\bf m}f(z)+2m_1z^2)^{**}\,.
\end{eqnarray}
Hence, each $(S_h,S_{h+1})$ is an interval of non-convexity of $Q_{\bf m}f(z)+2m_1z^2$ and we may repeat the argument of the proof of claim (i) to show that $Q_{\bf m}^f(z)> Q^2_{\bf m}f(z)$ in $M-1$ equi-spaced points in $(S_h,S_{h+1})$.
The argument can be then used iteratively.

{\rm (iii)} If $f_{2m_1}$ is convex the claim is trivial. Suppose otherwise. By \eqref{stur} we have 
$$(f(z)+2m_1z^2)^{\ast\ast}=(Q_{\bf m}f(z)+2m_1z^2)^{\ast\ast}=r(z)$$ 
for $z\in [S_0,S_M]$, and, iterating the argument also
$$r(z)=(f(z)+2m_1z^2)^{\ast\ast}=(Q^n_{\bf m}f(z)+2m_1z^2)^{\ast\ast},$$ 
for $z\in [S_0,S_M]$ and $n\ge1$.
As in the proof of claim (ii) above we have 
$$Q^n_{\bf m}f(z)+2m_1z^2=r(z) \ \ \hbox{\rm if }\ z=S_0+\frac{k}{M^n}(S_M-S_0)$$ 
for all $k\leq M^n$, and then  
$$Q^\infty_{\bf m}f(z)+2m_1z^2=r(z) \ \ \hbox{\rm if }\ z=S_0+\frac{k}{M^n}(S_M-S_0)$$
for some $n$ and for all $k\leq M^n$. By density, the equality then extends to all $z\in  [S_0,S_M]$.
Arguing in this way in each interval of non-convexity of $f_{2m_1}$ we conclude.
\end{proof} 

\begin{corollary}\label{cormult}
The same claims of the previous proposition hold if $\bf m$ is such that $M\geq 2$ exists such that 
$m_n=0$ if $n\not\in\{1, M\mathbb N\}$. 
\end{corollary}

\begin{proof} The proof follows by noting that 
$$
\sum_{i,j=0}^k m_{|i-j|}(u_i-u_j)^2\le 2m_1\sum_{i=1}^k(u_i-u_{i-1})^2+ \tilde m_M\sum_{{i,j=0,\ |i-j|=M}}^k(u_i-u_j)^2,
$$
where $\tilde m_M= \sum_{j=1}^\infty j^2 m_{jM}$, and arguing by comparison, applying the previous proposition to the kernel $\bf \tilde m$ where $\tilde m_1=m_1$ and $\tilde m_n=0$ if $n\not\in\{1, M\}$
\end {proof} 

Proposition \ref{stab1M} does not hold for `incommensurate' kernels; i.e., such that there are interactions not multiple of a common $M>1$. In the example below we treat a paradigmatic case.

\bigskip 

\begin{example}[incommensurability and non-trivial $\bf m$-stability]
\label{exincomm}\rm 
Let $\bf m$ be such that $m_n\neq 0$ if and only if $n\in\{2,3\}$. 

Let $k\in \mathbb N$ and consider the quadratic function 
$$G(z_1,\dots, z_k)=2m_2\sum_{i=1}^k(z_i+z_{i+1})^2+
2m_3\sum_{i=1}^k (z_i+z_{i+1}+z_{i+2})^2$$
defined on $k$ periodic sequences $\{z_i\}_{i\in\mathbb Z}$. 
Noting that 
$$2m_2(z_{i+1}+z_{i+2})^2+2m_3(z_i+z_{i+1}+z_{i+2})^2\geq \min\{m_2,m_3\}z_i^2,$$ 
we obtain that 
$$H_c(z_1,\dots, z_k)=G(z_1,\dots, z_k)-2c\sum_{i=1}^kz_i^2\geq 0$$
for any $c\in(0,\frac{\min\{m_2,m_3\}}{2})$.  Hence, since $H_c$ is a symmetric non-negative $2$-homogeneous polynomial of degree $2$, it is convex. Then 
\begin{equation}
\frac{1}{k}\min\Big\{H_c(z_1,\dots, z_k): \sum_{i=1}^kz_i=kz\Big\}=8m_2z^2+18m_3z^2-2cz^2.
\end{equation} 

Now, we suppose that $f(z)+2cz^2$ is convex for some $c\in(0,\frac{\min\{m_2,m_3\}}{2})$. 
Then for any $k$ 
\begin{eqnarray*}
&&\frac{1}{k}\min\Big\{\sum_{i=1}^k f(z_i)+G(z_1,\ldots, z_k) : \sum_{i=1}^k z_i=kz\Big\}\\
&&=\frac{1}{k}\min\Big\{\sum_{i=1}^k (f(z_i)+2cz_i^2)+H_c(z_1,\dots, z_k) : \sum_{i=1}^k z_i=kz\Big\}\\
&&=f(z)+2cz^2+8m_2z^2+18m_3z^2-2cz^2=f(z)+a_{\bf m}z^2.
\end{eqnarray*} 
Note that by Remark \ref{boundary-affine} in the definition 
of $Q_{\bf m}f$ 
we can take $u_i-u_{i-1}=z$ for $i=1,2,3$ and $i=k,k-1,k-2$, and consider 
the function $u_i-u_{i-1}$ extended by $k$-periodicity.
Indeed, the minimum problem in \eqref{limscala} is estimated from below by the periodic problem 
up to a term $O(\frac{1}{k})$. 
Hence, 
$\widehat Q_{\bf m}f(z)\geq f(z)+a_{\bf m}z^2$,
and $f$ is $\bf m$-stable, since the other inequality is true by \eqref{estimates2}. 
Note that this implies that in general the condition $f_{2m_1}$ convex is not necessary for the $\bf m$-stability of $f$, since in this case it suffices that $f_{2m_1+2c}$ be convex.
\end{example}

\begin{definition}[effective strength of nearest-neighbour interaction] 
Let $$G_k(z_1,\dots,z_k)=\sum_{n=1}^{+\infty}2m_n\sum_{i=1}^k \Big(\sum_{j=i}^{i+n}z_j\Big)^2$$
defined on $k$-periodic sequences $\{z_j\}_{j\in\mathbb Z}$. 
We define the {\em effective strength of nearest-neighbour interaction} $m_1^{\rm eff}$ for $\bf m$ as the supremum 
of all constant $c$ such that 
$$G_k(z_1,\dots,z_k)\geq 2c\sum_{i=1}^kz_i^2$$ 
for all $k\in \mathbb N$ and for all $\{z_j\}_{j\in\mathbb Z}$. 
\end{definition}

\begin{remark}[lower bound with $m_1^{\rm eff}$]\rm 
Note that $m_1^{\rm eff}\geq m_1$. 
The two values coincide if and only if $\bf m$ satisfies the generalized concentration hypothesis of Corollary \ref{cormult}. 
Repeating the argument in Example \ref{exincomm}, we obtain that 
a sufficient condition for the $\bf m$ stability of a function $f$ is the convexity of $f_{2m_1^{\rm eff}}$. 
Moreover, we have the estimate 
\begin{equation}
Q_{\bf m}f(z)\geq f^{\ast\ast}_{2m_1^{\rm eff}}(z)-2m_1^{\rm eff}z^2. 
\end{equation}
This can be achieved again following Example \ref{exincomm}, estimating $f_{2m_1^{\rm eff}}$ with its convex envelope.  
\end{remark} 

\subsubsection{Interpolation by parameterized kernels} 
The penalization kernel $\bf m$ may depend  on a scale parameter $\sigma$, measuring either the range or the scale of incompatibility. Of particular interest are kernels that tend to $0$ as $\sigma\to +\infty$, while they loose their summability as $\sigma\to0$.  
Kernels $\bf m$ with such a dependence on a scale parameter $\sigma$ can be used to interpolate between the extreme bounds in \eqref{estimates3}. 

A suitable class of such kernels is  constructed as follows.  
Let $m\colon[0,+\infty)\to[0,+\infty)$ be a continuous non-increasing function such that $m$ is strictly positive up to some $\overline x>0$, and 
$$\int_{0}^{+\infty}x^2m(x)\, dx<+\infty.$$ 
These conditions are satisfied by $m(x)=e^{-x}$; in this case, by setting $m_n=m^{\sigma}_n=m(\sigma n)$, we obtain the exponential kernels   $m_n=e^{-\sigma n}$ studied in more detail in Section \ref{exp-sec}. 

The following proposition holds. 
\begin{proposition}\label{interpolation-sigma} 
Let $m\colon[0,+\infty)\to[0,+\infty)$ be as above, and for all $\sigma>0$ consider the kernel  ${\bf m}^\sigma=\{m(\sigma n)\}_n$. Let $f\colon \mathbb R\to [0,+\infty)$ satisfy growth assumptions \eqref{crescita} and $\eqref{crescitasotto}$.  
Then,   
\begin{equation}\label{interpsigma}
\lim_{\sigma\to+\infty}
Q_{{\bf m}^\sigma} f(z)=f^{\ast\ast}(z) \ \ \hbox{\rm and }\ \ 
\lim_{\sigma\to 0^+} Q_{{\bf m}^\sigma} f(z)= \overline f(z).
\end{equation}
\end{proposition}

\begin{proof} 
Setting $\displaystyle a_{{\mathbf m}^\sigma}=2\sum_{n=1}^{+\infty}m(\sigma n) n^2$, 
we obtain that 
$$a_{{\mathbf m}^\sigma}\leq 2\int_{0}^{+\infty} m(\sigma(x+1)) (x+1)^2\, dx \leq 
\frac{2}{\sigma}\int_{0}^{+\infty} m(y) y^2\, dy=\frac{C}{\sigma}\to 0 \ \ \hbox{\rm as }\ \sigma\to+\infty.$$
Then, the first equality in \eqref{interpsigma} 
follows directly from Proposition \ref{disug}, as we have
$$f^{\ast\ast}(z)
\leq
 Q_{{\bf m}_\sigma}f(z)\leq \psi_\sigma^{\ast\ast}(z)-a_{{\mathbf m}^\sigma} z^2\leq f(z)$$
where $\psi_\sigma(z)=
f(z)+  a_{{\mathbf m}^\sigma}z^2$. Since $a_{{\mathbf m}^\sigma}$ decreases to $0$ as $\sigma\to+\infty$, 
then there exists a convex function $\psi$ such that 
$$
f^{\ast\ast}(z)\leq \psi(z)
=\lim_{\sigma\to +\infty} \psi^{\ast\ast}_\sigma(z)
=\lim_{\sigma\to +\infty} \Big(\psi^{\ast\ast}_\sigma(z)-a_{{\mathbf m}^\sigma}z^2\Big)\leq f(z).
 $$
Hence, $\psi(z)=f^{\ast\ast}(z)=\lim\limits_{\sigma\to+\infty}Q_{{\bf m}^\sigma}f(z)$.  

Now, we prove the second limit in  \eqref{interpsigma}. 
Since \eqref{estimates} holds, it is sufficient to show that 
$\lim\limits_{\sigma\to 0} Q_{{\bf m}^\sigma} f(z)\geq\overline f(z)$. 
Up to scaling, we can suppose that $\overline x=1$ and $m(1)=1$. 
Since $m$ is non-increasing, it is sufficient to prove the desired equality for 
$m=\chi_{[0,1]}$. The function $a_{{\mathbf m}^\sigma}$ is non-increasing with respect to $\sigma$; hence, for any $z$ there exists the limit of $Q_{{\bf m}^\sigma} f(z)$ as $\sigma\to 0$.
Let $\sigma_k\to 0$ as $k\to+\infty$ and let $u^k$ be a minimizer  in $[0,k]$ for the minimum problem  in the formula of $Q_{{\bf m}^{\sigma_k}}$ in Remark \ref{vadeqm}; that is, $u^k$ is an admissible minimizer for $G^{\sigma_k}(u)$  
defined by 
$$
G^{\sigma_k}(u)=\sum_{i=1}^k\! f(u_i-u_{i-1})+\! \sum_{i,j=0}^k \! m(\sigma_k |i-j|)
\big((u_i-u_j)^2-(i-j)^2z^2\big).
$$ 
 Let $N_k=\lfloor\frac{1}{\sigma_k}\rfloor$. By Remark \ref{boundary-affine}, 
we can assume that 
 the test functions $u$, defined for $i\in\mathbb Z$, satisfy $u_i=iz$ for 
$i\leq N_k$ and $i\geq k-N_k$.  
Reasoning as in Remark \ref{penaltyrem}, 
for any $\ell=1,\dots, N_k$ and $r=1,\dots, \ell$ we have 
$$\sum_{i=1}^{\lfloor k/l\rfloor}\Big((u^k_{i\ell+r}-u^k_{(i-1)\ell+r})^2-z^2\ell^2\Big)
=\sum_{i=1}^{\lfloor k/l\rfloor}\Big(u^k_{i\ell+r}-u^k_{(i-1)\ell+r}-z\ell\Big)^2\geq 0.$$  
We now define a discrete function $w^k$ by setting 
$w^k_i=u^k_i-u^k_{i-1}-z$. For any $1\leq n\leq N_k$, 
we can write 
$$w^k_i=\sum_{j=i}^{i+n} w^k_j-\sum_{j=i+1}^{i+n} w^k_j,$$ 
so that, by summing over $n$ 
\begin{eqnarray*}
N_k\frac{1}{k}\sum_{i=1}^k(w^k_i)^2&\leq&\frac{2}{k}\sum_{n=1}^k\sum_{i=1}^k\big((u^k_{i+n}-u^k_{i-1}-(n+1)z)^2+(u^k_{i+n}-u^k_{i}-nz)^2\big)\\
&\leq& \frac{4}{k} G^{\sigma_k}(u^k),
\end{eqnarray*}
recalling that $m(\sigma_k n)=1$ if $\sigma_k n\leq 1$ and $0$ otherwise. 
Now, let $\tilde u^k$ denote the piecewise-affine extension to $[0,1]$ of the discrete function defined by $\tilde u^k(\frac{i}{k})=\frac{1}{k}u^k_i$, 
so that $(\tilde u^k)^\prime-z=w^k_i$ in each interval $(\frac{i-1}{k},\frac{i}{k})$.
Since $\frac{4}{k} G^{\sigma_k}(u^k)$ is equibounded, 
we obtain that 
$$\int_0^1 ((\tilde u^k)^\prime-z)^2\, dt=\frac{1}{k}\sum_{i=1}^k(w^k_i)^2\leq \frac{C}{N_k}\to 0 \ \ \hbox{\rm as }\ \ k\to+\infty.$$
Hence, $\tilde u^k\to z x$ in $H^1(0,1)$.
We get 
\begin{eqnarray*}
\lim_{k\to+\infty}Q_{{\bf m}^{\sigma_k}}f(z)&\geq&\liminf_{k\to+\infty} \frac{1}{k}\sum_{i=1}^k f(u^k_i-u^k_{i-1}) 
\geq \liminf_{k\to+\infty} \frac{1}{k}\sum_{i=1}^k \overline f(u^k_i-u^k_{i-1}) \\
&=&\liminf_{k\to+\infty} \frac{1}{k}\sum_{i=1}^k \overline f( w_i^k+z )=\liminf_{k\to+\infty} \int_0^1\overline f(\tilde u_k^\prime)\, dt \geq \overline f(z),  
\end{eqnarray*}
by the lower-semicontinuity of the functional $w\mapsto \int_0^1\overline f(w^\prime)\, dt$ with respect to the strong $H^1$-convergence. 
\end{proof}

\begin{remark}[`singular' kernels depending on $\sigma$]\label{sigmasing}\rm 
If $\bf m$ is a kernel concentrated at some $M\geq 2$, with $m_M\neq 0$, we consider a different type of parameter dependence. 
In this case, we can 
set ${\bf m}^\sigma=\{m_n^{\sigma}\}_n=\{\phi(\sigma) m_n\}_n$, with $\phi$ decreasing and such that $\lim_{\sigma\to 0^+}\phi(\sigma)=+\infty$ and $\lim_{\sigma\to+\infty}\phi(\sigma)=0$; for instance, we may consider  
$$m_n^\sigma=\frac{1}{\sigma}m_n.$$ Since $m_n$ is not decreasing, this case cannot be treated directly by applying the result 
of the above  proposition.
However, the same argument as in the proof of Proposition \ref{interpolation-sigma} can be used as well giving
\begin{equation}\label{sigmasinginf}
\lim_{\sigma\to+\infty}Q_{{\bf m}^\sigma}f(z)=f^{\ast\ast}(z).
\end{equation}
As for the limit as $\sigma\to 0^+$, we can follow the proof up to the definition of $w_i^k$, obtaining 
\begin{eqnarray*}
\frac{1}{k}\sum_{i=1}^k(w^k_i)^2&\leq&\frac{2}{k}\sum_{n=1, M}^k\sum_{i=1}^k\big((u^k_{i+n}-u^k_{i-1}-(n+1)z)^2+(u^k_{i+n}-u^k_{i}-nz)^2\big)\\
&\leq& \max\Big\{\frac{\sigma_k}{m_1},\frac{\sigma_k}{m_M}\Big\}\frac{4}{k} G^{\sigma_k}(u^k),
\end{eqnarray*}
and we can conclude exactly as above, proving that 
\begin{equation}\label{sigmasing0}\lim_{\sigma\to0^+}Q_{{\bf m}^\sigma}f(z)=\overline f(z).
\end{equation} 
Note that if $m_1=0$ equality \eqref{sigmasing0} in general does not hold (while \eqref{sigmasing0} is always valid). As an example, we refer to 
Remark \ref{casem10}. 
\end{remark}

In general, for $\sigma$-dependent kernels equalities \eqref{interpsigma} are achieved only asymptotically. However, in some cases they are reached for some finite values of $\sigma>0$. To highlight this fact, we give the following definition. 
\begin{definition}[critical transition value of $\sigma$] 
Let $f\colon \mathbb R\to [0,+\infty)$ be a continuous function satisfying growth assumptions \eqref{crescita} and $\eqref{crescitasotto}$. 
Let $\{\bf m^{\sigma}\}_{\sigma>0}$ be a family of parameterized kernels.  
We define the critical transition value of $\sigma$ by setting 
$$\sigma_c=\sigma_c(f)=\sup\{\sigma>0: Q_{\bf m^\tau}f=f \ \hbox{\rm for all } \tau<\sigma\}.$$ 
We set $\sigma_c=0$ if $Q_{\bf m^\sigma}f<f$ for any $\sigma$. 
\end{definition}
 
\begin{example}[existence of positive critical transition values]\label{excriticalsigma}\rm 
Let $f(z)=(1-z^2)^2$ and let $\bf m$ be concentrated at some $M\geq 2$, as in Remark \ref{sigmasing}. We set 
$m_1^\sigma=\frac{m_1}{\sigma}$ and $m_M^\sigma=\frac{m_M}{\sigma}$. 
By Remark \ref{nocopeco}, we have that $\sigma_c=m_1$ is the critical transition value. 
Note that, conversely, for any $\sigma$ we have that 
$Q_{{\bf m}^{\sigma}}f(z)>f^{\ast\ast}(z)$ 
for some values of $z$, hence the limit in \eqref{sigmasinginf} is only reached at $+\infty$. 
\end{example}

In the sequel, an important role will be played by the two special families of kernels depending on $\sigma$, exponential and 
concentrated at some $M$, introduced above and already illustrated in Fig.~\ref{istogram}, for
which the computation of $Q_{\bf m} f$ can be performed analytically. 
In both cases, we will be able  trace explicitly the role of  the scale parameter $\sigma$  characterizing the range/strength of the penalization kernel. 
 
\section{Description of minimizers by a phase function}
\label{constraint-sec} 
In this section, we focus on the important case of `generalized double well' 
potentials  when  the domain of a function $f$ can be  subdivided  in two sub-domains of convexity, which in what follows we refer to as $A$ and $\mathbb R\setminus A$. We call such  potentials  {\em bi-convex} and refer to the two convex branches of $f$ as  {\em phases}.  
Given that some  microstructures in such models  can be interpreted  as a `phase mixtures', it will be convenient  to introduce the `volume-fraction parameter' $\theta$ representing the percentage of indices $i$ such that $u_i-u_{i-1}$ is in
  the set $A$.  
The computation of minima with prescribed volume fraction $\theta$ gives an upper bound for  $\widehat Q_{\bf m}f$. 

With the introduction of $\theta$, one can proceed in two steps. The first step involves the computation of   the  function $\widehat Q_{\bf m}f(\theta,z)$ which is obtained  by a constrained minimization  with prescribed  $\theta$.  Then the function  $\widehat Q_{\bf m}f$ can be  obtained by a one-dimensional optimization of $\widehat Q_{\bf m}f(\theta,z)$ over $\theta$, which  also defines the  {\em phase function} $\theta(z)$ such that  $\widehat Q_{\bf m}f(\theta(z),z)=\widehat Q_{\bf m}f(z)$. In the problems of interest the function $\theta(z)$ will  have  a complex `staircase' structure reflecting the existence of the  {\em locking states} at the values of $\theta$ that are stable under   variation  of $z$. 

\begin{remark}[constrained minimization and the structure of the phase function]\label{const_min}\rm
To understand the role of the constrained minimization producing the function $\widehat Q_{\bf m}f(\theta,z)$ and to reveal the link between the shape of the phase function $\theta(z)$ and  the structure of the  relaxed energy  $\widehat Q_{\bf m}f$, it will be instructive  to consider first the case when 
only $M$-neighbour interactions are taken into account. We recall that in this case  there exists $M\geq 2$ such that $m_M\neq 0$ and $m_n=0$ for any $n\neq M$. 

Proposition \ref{MniPM} gives a formula for $\widehat Q_{\bf m}f(z)$. 
If $f$ is bi-convex, we can subdivide its computation 
by introducing  a dependence on  the fraction $\theta$ of $z_i=u_i-u_{i-1}$ belonging to the convexity region $A$. 
More precisely, for any $n=0,\dots, M$ we can first compute the minimum at  a fixed fraction 
$\theta_n=\frac{n}{M}$ of $z_i$ belonging to $A$.  Using the convexity, 
such  minimum problems reduce to the computation of   
\begin{equation}\label{pieMMeenne}
\left. \begin{array}{ll}
P^{M,n}(z)=&\displaystyle\min\Big\{(1-\theta_n)f(z^-)+\theta_n f(z^+): z^-\leq z^\ast, z^+\geq z^\ast, \\
&\hspace{2cm}\displaystyle(1-\theta_n)z^-+\theta_nz^+=z\Big\}
+2m_M(Mz)^2. 
\end{array}
\right. 
\end{equation} 
The optimal bounds are then completely characterized by the functions $P^{M,n}$, in the sense that 
\begin{equation*}
\widehat Q_{\bf m}f(z)=\big(\min_n P^{M,n}(z)\big)^{\ast\ast}.
\end{equation*}  
We will show that all the $M+1$ values $\theta_n$ are locking states in the sense above. These values of $\theta$ are particularly relevant since the shape of $\widehat Q_{\bf m}f(z)$ will be shown to depend exclusively on `phase mixtures' with `volume fraction' $\theta_n$. Another property enjoyed by $\theta_n$ is that the minimum problems corresponding to values of $z$ for which $\theta(z)=\theta_n$ admit periodic solutions.
\end{remark}  

\subsection{Phase-constrained relaxation and related properties}
We now give some precise definitions, and obtain some general bounds valid for any choice of $f$ and $\bf m$. 

\medskip
Let $z^\ast\in\mathbb R$ and let 
$A=[z^\ast,+\infty)$.  
For a given $\theta\in\mathbb Q\cap[0,1]$ and $N\in\mathbb N$ 
we consider the set of test functions $u$ with a 
percentage $\theta$ of indices $i$ such that $u_i-u_{i-1}\in A$. 
Since we need a closed condition, 
the form of the constraint is given as follows: 
\begin{eqnarray}\label{vtheta}
&&\mathcal V(N;\theta)=\{u\colon [0,N]\cap\mathbb Z\to\mathbb R: 
\#\{i: u_i-u_{i-1}>z^\ast\}\leq\theta N,\nonumber \\
&&\hspace{5cm}\#\{i: u_i-u_{i-1}<z^\ast\}\leq(1-\theta) N\}. 
\end{eqnarray} 
For any 
$z\in\mathbb R$ we can then define the function 
\begin{eqnarray}\label{defqtheta}
\widehat Q_{\bf m}f(\theta,z)=\liminf_{\substack{N\to+\infty\\ \theta N\in\mathbb N}}\frac{1}{N}
\inf\Big\{F_1(u;[0,N]): u\in\mathcal A(N;z)\cap\mathcal V(N;\theta) \Big\}, 
\end{eqnarray} 
where $F_1$ is 
the (non-scaled) functional defined for $u\colon[0,N]\cap\mathbb Z\to\mathbb R$ by 
\begin{equation}\label{defF1}
F_1(u;[0,N])=\sum_{i=1}^{N} f(u_i-u_{i-1})+\sum_{i,j=0}^Nm_{|i-j|}(u_i-u_j)^2
\end{equation} 
(see \eqref{def-fe-prima} with $\e=1$ and $I=[0,N]$). 
In the notation $\widehat Q_{\bf m}f(\theta,z)$ we omit the dependence on $z^\ast$.  
Note that a corresponding definition could be given also for a more general set $A$. 

In order to obtain bounds for $Q_{\bf m}f$, we also define 
\begin{equation}\label{defqtetazeta}
Q_{\bf m}f(\theta,z)=\widehat Q_{\bf m}f(\theta,z)-a_{\mathbf m}z^2. 
\end{equation}
\begin{theorem}[optimization over the  phase fraction]\label{minpro}
The following equality holds:
$$\inf_{\theta\in\mathbb Q\cap[0,1]}\widehat Q_{\bf m}f(\theta,z)=\widehat Q_{\bf m}f(z).$$
\end{theorem}
\begin{proof}
It is sufficient to prove that $\widehat Q_{\bf m}f(z)\geq \inf_{\theta\in\mathbb Q\cap[0,1]}\widehat Q_{\bf m}f(\theta,z)$. To this end, with $\eta>0$ fixed we choose $\delta>0$, $k \in \mathbb N$
and $u$ an admissible test function for  the minimum in \eqref{limdelta} such that 
$$\frac{1}{k}\Bigl(\sum_{i=1}^k f(u_i-u_{i-1})+
\sum_{i,j=0}^k m_{|i-j|}(u_i-u_j)^2\Bigr) \leq  \widehat Q_{\bf m}f(z)+\eta.$$
Setting 
$$\theta=\frac{\#\{i: \ u_i-u_{i-1}\geq z^\ast\}}{k},$$ 
we extend $u$ to $\mathbb Z$ so that $u_i-zi$ is $k$-periodic. 
Since $u\in \mathcal A(Nk;z)\cap\mathcal V(Nk;\theta)$, we can use 
it as a test function for 
\begin{equation}\label{defqntheta}
\frac{1}{Nk}
\inf\Big\{F_1(v;[0,Nk]): v\in\mathcal A(Nk;z)\cap\mathcal V(Nk;\theta) \Big\}
\end{equation}
in the computation of $\widehat Q_{\bf m}f(\theta,z)$. 

We subdivide the estimate of $F_1(v;[0,Nk])$ by grouping interactions in three (partially overlapping) different subsets taking into account the location of the interacting sites in the subintervals $[(r-1)k, rk]$ for  $r\in\{1,\ldots, N\}$.

\medskip 
i) (interactions within a single subinterval) $i, j\in [(r-1)k, rk]$ for some $r\in\{1,\ldots, N\}$. Summing over all $i,j$ and $r$ gives the contribution 
\begin{equation}\label{conto1}
{1\over k} F_1(u;[0,k])
\end{equation}
to \eqref{defqntheta}.

\medskip
ii) (interactions between different intervals, but not close to the endpoints) $i\in I_r^\delta= [(r-1)k+k\delta, rk-k\delta]\cap \mathbb Z$, $j\in I_s =[(s-1)k, sk]\cap \mathbb Z$ for some $r,s\in\{1,\ldots, N\}$ with $r\neq s$.

Let $i'=i-(r-1)k$ and $j'=j-(s-1)k$.  We can write
\begin{eqnarray*}
(u_i-u_j)^2&=&(u_{i'}-u_{j'}+z(r-s)k)^2\le 2(u_{i'}-u_{j'})^2+2z^2(r-s)^2k^2\\
&\le&  2(i'-j')\sum_{l=j'+1}^{i'}(u_{l}-u_{l-1})^2+2z^2(r-s)^2k^2
\end{eqnarray*}
(we can suppose for simplicity that $j'<i'$).
By \eqref{crescitasotto} we have that 
$$
\sum_{l=j'+1}^{i'}(u_{l}-u_{l-1})^2\le \sum_{l=1}^{k}(u_{l}-u_{l-1})^2\le c\, F_1(u; [0,k])\le Ck,
$$
so that $(u_i-u_j)^2\le  2Ck^2+2z^2(r-s)^2k^2$.

We may suppose that $k$ is large enough, so that 
$
m_l\le {\eta\over l^\beta} 
$ 
if $l\ge k\delta$, where $\beta$ is the decay exponent of $\bf m$. Note that $|i-j|\ge\big||s-r|+\delta-1\big|k\ge \delta k$. Hence,
summing over such $i,j$, $r$ and $s$ we obtain 
\begin{eqnarray}\label{conto2} \nonumber
{1\over Nk}\sum_{r\neq s}\sum_{i\in I^\delta_r}\sum_{j\in I_s}m_{|i-j|}(u_i-u_j)^2&\le&  {1\over N}\sum_{r\neq s}2k (C+z^2(r-s)^2)\sum_{i\in I^\delta_r}\sum_{j\in I_s}m_{|i-j|}\\  \nonumber
&\le&  2{1\over N} \sum_{r\neq s}(C+z^2(r-s)^2){\eta \over \big||s-r|+\delta-1\big|^\beta}k^{3-\beta}\\
&\le&  2 \sum_{n=1}^\infty(C+z^2n^2){\eta \over |n+\delta-1|^\beta}k^{3-\beta}\le \widetilde C \eta\,.
\end{eqnarray}

\medskip
iii) (interactions between different intervals, close to the endpoints) $i,j\in J^\delta_r =[ rk-k\delta, rk+k\delta]\cap\mathbb Z$ for some $r\in\{1,\ldots, N-1\}$.

For such $i,j$ we have $u_i-u_j= z(i-j)$. Hence, we have
\begin{eqnarray}\label{conto3} \nonumber
{1\over Nk}\sum_{r=1}^{N-1} \sum_{i,j\in J^\delta_r}m_{|i-j|} (u_i-u_j)^2
&=&{z^2\over Nk}\sum_{r=1}^{N-1} \sum_{i,j\in J^\delta_r}m_{|i-j|} (i-j)^2
\\ 
&\le& {z^2\over k}\sum_{-k\delta\le l\le k\delta} \sum_{n\in \mathbb Z} m_n n^2
\le \widetilde C \delta.
\end{eqnarray}

\medskip
By \eqref{conto1}--\eqref{conto3} we obtain the estimate
$$
{1\over Nk} F_1(u;[0,Nk])\le {1\over k} F_1(u;[0,k])+ \widetilde C(\eta+\delta)\le \widehat Q_{\bf m} f(z)+C(\eta+\delta).
$$
Taking the liminf as $N\to+\infty$, by the arbitrariness of $\eta$ and $\delta$ we obtain the claim.
\end{proof} 

We now study the general properties of $\widehat Q_{\bf m}f(\theta,z)$ as a function  of $\theta$. To that end, we write $\theta$ as the quotient of (coprime) integer numbers $p$ and $q$, so that   
\begin{eqnarray}\label{qthetapq}
\widehat Q_{\bf m}f(\theta,z)=\liminf_{k\to+\infty}\frac{1}{kq}\inf\Big\{F_1(u;[0,kq]): u\in\mathcal A(kq;z)\cap\mathcal V(kq;\theta) \Big\}. \nonumber
\end{eqnarray}

\medskip 
We will need to develop some technical ideas related to the possibility of modifying boundary conditions. 
We note that the usual cut-off argument as in Remark \ref{boundary-affine} cannot be directly followed, since forcing the test function to satisfy the affine condition $u_i=iz$ near the boundary may be incompatible with the constraint. Still, we can modify the argument with a compatible condition remaining   close to the affine function near the boundary. 

To make this precise, for any $\delta>0$ we introduce the set 
$$\widetilde{\mathcal A}_\delta(N;z)=\{u\in\mathcal A(N;z): \ |u_i-u_{i-1}|\leq |z^\ast|+2|z|
\ \hbox{\rm if } i\leq \delta N \ \hbox{\rm and }\ i\geq (1-\delta)N\}$$ 
and state the following result. 
\begin{lemma}[compatible boundary conditions]\label{boundary-Q}  
The following equality holds 
\begin{eqnarray*} 
\widehat Q_{\bf m}f(\theta,z)=\lim_{\delta\to 0}
\liminf_{k\to+\infty}\frac{1}{kq}\inf\Big\{F_1(u;[0,kq]): u\in \mathcal V(kq;\theta) \cap \widetilde{\mathcal A}_\delta(kq;z)  
 \Big\} \nonumber
\end{eqnarray*}
for any $\theta=\frac{p}{q}\in \mathbb Q\cap[0,1]$ and $z\in\mathbb R$. 
\end{lemma}
\begin{proof} 
Let $z\in\mathbb R$; we may suppose without loss of generality $z\le z^\ast$. 
Let $u\in \mathcal V(kq;\theta) \cap {\mathcal A}(kq;z)$ be a test function. 
We modify $u$ separately close to the two endpoints $i=0$ and $i=kq$. Let $u^z$ be a function with $u^z_0=0$, and such that $u^z_i-u^z_{i-1}=z^\ast$ if $u_i-u_{i-1}\ge z^\ast$ and $u^z_i-u^z_{i-1}=2z-z^\ast$ if $u_i-u_{i-1}\le z^\ast$.
By a cut-off argument as in Remark \ref{boundary-affine} we can modify $u$ on $[0, 2kq{\delta}]$ in a function $\widetilde u$ in such a way that $\widetilde u_i= u^z$ on $[0, kq{\delta}]$, and $\widetilde u_i-\widetilde u_{i-1}\not \in \{u_i-u_{i-1}, u^z_i-u^z_{i-1}\}$ except for at most $ kq{\delta}/N$ for a given arbitrary $N$. Since $u^z_i-u^z_{i-1}=z^\ast$ on a strictly positive percentage of points in $[0, kq{\delta}]$ (hence, we can suppose larger than $ kq{\delta}/N$), up to slightly modifying $\widetilde u$ on such points we have that $\widetilde u$ satisfies the constraint; i.e., $\widetilde u\in\mathcal V(kq;\theta)$.
The same argument can be repeated close to $i=kq$. Note that the energy of $u^z$ is comparable to that of the affine function $zi$, so that we obtain an estimate for the energy of $\widetilde u$.
\end{proof}

This lemma allows to prove the convexity of $\widehat Q_{\bf m}f$ in both variables.

\begin{proposition}[convexity of $\widehat Q_{\bf m}f$]\label{convex-q}
The function 
$$(\theta,z)\mapsto \widehat Q_{\bf m}f(\theta,z)$$ 
is convex; more precisely, 
$$(1-t)\widehat Q_{\bf m}f(\theta_1,z_1)
+t\,\widehat Q_{\bf m}f(\theta_2,z_2)\geq \widehat Q_{\bf m}f((1-t)\theta_1+t\theta_2,(1-t)z_1+tz_2)$$ 
for any $t\in[0,1]\cap \mathbb Q$, $\theta_h=\frac{p_h}{q_h}\in [0,1]\cap \mathbb Q$ and $z_h\in\mathbb R$. 
\end{proposition}

\begin{proof}
For any $k\in\mathbb N$ and $\delta>0$ 
we define 
\begin{eqnarray*}
\widehat Q_{\bf m}f_k^\delta(\theta,z)=
\frac{1}{kq}\inf\Big\{F_1(u;[0,kq]): u\in \mathcal V(kq;\theta)\cap \widetilde{\mathcal A}_\delta(kq;z)
\Big\}. 
\end{eqnarray*} 
With fixed $\delta>0$, we choose sequences $k^1_N, k^2_N\to+\infty$ (omitting the dependence on $\delta$) such that 
$$\liminf_{k\to+\infty}\widehat Q_{\bf m}f_k^\delta(\theta_h,z_h)=
\lim_{N\to+\infty}\widehat Q_{\bf m}f_{k_N^h}^\delta(\theta_h,z_h)$$
for $h=1,2$. We set $M_N^h=k_N^h q_h$. 
Recalling Lemma \ref{boundary-Q}, for any fixed $\eta>0$ we find $\delta_\eta>0$ such that for $0<\delta<\delta_\eta$ small enough there exists a test function $u^h\in\widetilde{\mathcal A}_\delta(M_N^h; z_h)$ (again omitting the dependencies) such that 
\begin{equation}\label{kh}
\widehat Q_{\bf m}f(\theta_h,z_h)\geq
\lim_{N\to+\infty}\widehat Q_{\bf m}f_{k_N^h}^\delta(\theta_h,z_h)-\eta 
= \lim_{N\to+\infty} \frac{1}{M_N^h} F_1(u^h;[0,M_N^h])-\eta.
\end{equation}
Setting $M_N=nM_N^1M_N^2$, we define a test function $u$ in $[0,M_N]\cap\mathbb N$ by means of  
suitable translations of $u^1$ and $u^2$. 
More precisely,  
we set $t=\frac{m}{n}$ and 
$$u_i=
\left\{ \begin{array}{ll}
\widehat u^1_i & \hbox{\rm if } i\in [0,(n-m)M_N^1M_N^2]\\
\widehat u^2_{i-(n-m)M_N^1M_N^2}+\widehat u^1_{(n-m)M_N^1M_N^2}
 & \hbox{\rm if } i\in ((n-m)M_N^1M_N^2, M_N]
\end{array}
\right. 
$$
where $\widehat u^h\colon [0,M_N^h]\cap\mathbb N\to\mathbb R$
is given by 
$$\widehat u^h=u^h_{i-(j-1)M_N^1} + (j-1)M_1z_1 \ \ \hbox{\rm if } i\in [(j-1)M^h_N, j M^h_N]$$
with $j=1,\dots, (n-m)M_N^2$ if $h=1$ and $j=1,\dots, mM_N^1$ if $h=2$. 
The function $u$ is an admissible test function for $\widehat Qf_{k_N^1k_N^2}(\theta,z)$, 
where 
$$\theta=(1-t)\theta_1+t\theta_2= \frac{(n-m)q_2 p_1 + mq_1p_2}{nq_1q_2}=\frac{p}{q}
\ \ \hbox{\rm and } \ z=(1-t)z_1+tz_2.$$ 
Indeed, $M_N=k_N^1k_N^2q$, and 
\begin{eqnarray*}
\frac{\#\{i: u_i-u_{i-1}\geq z^\ast\}}{M_N}=\frac{(n-m)N_2 k^1_Np_1 + mN_1k^2_Np_2}{M}=\theta; 
\end{eqnarray*}
the boundary conditions are satisfied since $u_{M_N}=M_Nz$. 
We get 
\begin{equation}\label{test}
\frac{1}{M_N} F_1(u; [0,M_N])\geq \widehat Q_{\bf m}f_{k_N^1k_N^2}(\theta, z). 
\end{equation}
Since $u^h\in\widetilde{\mathcal A}_\delta(M_N^h; z)$, recalling that  
$m_{|i-j|}=o(|i-j|^{-\beta})$ with $\beta>3$ we obtain  
\begin{eqnarray*}
\frac{1}{M_N} 
F_1(u; [0,M_N]) & \leq & \frac{(n-m)M_N^2 M_N^1}{M_N} F_1(u^1;[0,M_N^1])
+
\frac{m M_N^1M_N^2}{M_N} F_1(u^2;[0,M_N^2])\\
&&+c(\delta)o(1)_{N\to+\infty}+C\delta\\
& = &\frac{n-m}{n}\widehat Qf_{k_N^1}^\delta(\theta_1,z_1)+
\frac{m}{n}\widehat Qf_{k_N^2}^\delta(\theta_2,z_2)+c(\delta)o(1)_{N\to+\infty}+C\delta.
\end{eqnarray*}
Taking the $\liminf$ as $N\to+\infty$ and recalling \eqref{kh} and \eqref{test} we get 
\begin{eqnarray*} 
\widehat Q_{\bf m}f(\theta,z)&\leq&\liminf_{N\to+\infty} \widehat Q_{\bf m}f_{k_N^1k_N^2}(\theta,z)\\
&\leq&\liminf_{N\to+\infty} \frac{1}{M_N} F_1(u; [0,M_N])\\
&\leq&\liminf_{N\to+\infty}  \Big(  \frac{n-m}{n}\widehat Q_{\bf m}f_{k_N^1}^\delta(\theta_1,z_1)+
\frac{m}{n}\widehat Q_{\bf m}f_{k_N^2}^\delta(\theta_2,z_2)\Big)+C\delta \\
&\leq&\frac{n-m}{n}\widehat Q_{\bf m}f(\theta_1,z_1)+
\frac{m}{n}\widehat Q_{\bf m}f(\theta_2,z_2))+\eta+C\delta .
\end{eqnarray*}
Since $\eta>0$ is arbitrary and $\delta\in (0,\delta_\eta)$, this concludes the proof. 
\end{proof} 

\subsection{Phase function and locking states}
By the convexity of the function $\theta\mapsto \widehat Q_{\bf m}f(\theta,z)$,   
we can extend it (and consequently also $Q_{\bf m}f(\theta,z)$) 
 to the irrational values of $\theta\in(0,1)$ by continuity. 
 This naturally leads to a definition which singles out some critical values for $\theta$ remaining `stably optimal' for a range of values of the loading parameter $z$. 
\begin{definition}[locking states]
We say that $\theta$ is a {\rm locking state} for $f$ and ${\bf m}$ if 
the set $\{z : Q_{\bf m}f(\theta,z)=Q_{\bf m}f(z)\}$ contains an open interval.  
\end{definition}

The  special values of $\theta$, for which the relaxed energy $\widehat Q_{\bf m}f(\theta,z)$ can be  obtained by considering periodic minimizers,  play a particular role in the construction  of $\widehat Q_{\bf m}f$. Usually, the arrangements of such minimizers remain optimal over an interval of the values of $z$ and the corresponding $\theta$ are locking states (see Remark \ref{emepe}). 
The analysis of some model examples from this standpoint will show 
how the knowledge of such special values of $\theta$ can allow one to 
compute  the whole relaxed energy $\widehat Q_{\bf m}f(z)$ (for instance for concentrated kernels).

\smallskip 
We can now introduce a `phase function' as follows. 

\begin{definition}[phase function]\label{phasefunctiondef} We define the {\rm phase (multi)function}
$\Theta(z)$ by 
$$\Theta(z)=\big\{\theta\in[0,1]: \hbox{\rm sc} (Q_{\bf m}f)(\theta,z)
=Q_{\bf m}f(z)\big\},$$
where {\rm sc}$(Q_{\bf m}f)$ denotes the lower semicontinuous envelope of $Q_{\bf m}f(\theta,z)$ with respect to $\theta$. 
In order to define a {\rm phase function} $\theta(z)$, we select $\theta(z)$ as 
the minimum of the set $\Theta(z)$. 
\end{definition} 

\begin{remark}[a selection issue] \rm
Note that in order to have $\theta$ well defined we have made a choice of $\theta(z)$ as a minimum in the case when $\Theta(z)$ is not a singleton. This is an arbitrary choice and may lead to some difficulty in the interpretation of this value, for example in cases where the dependence on $\theta\in[0,1]$ is symmetric, or in degenerate cases (see for instance items (b) and (c) with the corresponding examples in Section \ref{sec-degene}). 
\end{remark}

\begin{remark}[locking states as the `steps' (constancy intervals) developed by $\theta(z)$]\rm 
The definition of the phase function $\theta(z)$ allows one to to interpret locking states as the values $\overline\theta$ for which $\theta^{-1}(\overline\theta)$ contains an open interval. 
\end{remark}

\begin{remark}[possible non-semicontinuity at the extreme points]\label{remnlsc}\rm 
Note that $\hbox{\rm sc} (Q_{\bf m}f)(\theta,z)$ differs from $Q_{\bf m}f(\theta,z)$ only at most for $\theta\in\{0,1\}$, by the continuity of $Q_{\bf m}f(\theta,z)$ in $(0,1)$. 
If the function $\theta\mapsto Q_{\bf m}(\theta,z)$ is lower semicontinuous 
in $0$ and $1$, then the multi-function $\Theta(z)$ coincides with the set  
$$\overline\Theta(z)=\{\theta\in[0,1]: Q_{\bf m}f(\theta,z)=Q_{\bf m}f(z)\}.$$ 
In general, the set $\overline\Theta(z)$ can be empty, in which case, by Definition \ref{phasefunctiondef}, 
$\Theta(z)$ is a singleton and $\theta(z)=0$ (or $1$) if 
there exists $\theta_n\to 0$  (or $1$, respectively) 
such that $Q_{\bf m}f(\theta_n,z)\to Q_{\bf m}f(z)$ 
(see Example \ref{tcm0} below). 
\end{remark} 

\begin{proposition}\label{phasefunctionaffine} 
If $\widehat Q_{\bf m}f$ is affine in an open interval $I$ and $\Theta(z)=\{\theta(z)\}$ for all $z\in I$,   then $\theta$ is affine in $I$. 
\end{proposition}

\begin{proof}
Let $z_1,z_2\in I$, $\theta_1=\theta(z_1)$, and  $\theta_2=\theta(z_2)$. For $t\in (0,1)$, Proposition \ref{minpro}, the convexity of $\widehat Q_{\bf m}f(\theta,z)$ and the hypothesis imply that
\begin{eqnarray*}
\widehat Q_{\bf m}f(tz_1+(1-t)z_2)&\le& \widehat Q_{\bf m}f(t \theta_1+(1-t)\theta_2,tz_1+(1-t)z_2)\\
&\le &
t \widehat Q_{\bf m}f(\theta_1,z_1)+ (1-t)\widehat Q_{\bf m}f(\theta_2,z_2)\\
&= &t \widehat Q_{\bf m}f(z_1)+ (1-t)\widehat Q_{\bf m}f(z_2)\\ & =
&\widehat Q_{\bf m}f(tz_1+(1-t)z_2),
\end{eqnarray*}
and the claim follows.\end{proof}

\begin{remark}[locking states and periodic microstructures]\rm
The definition of locking state is formally disconnected from the periodicity of the associated minimizers. However, the two notions are perhaps related. Indeed,  if the value of the minimum energy $\widehat Q_{\bf m}f(z)$ is reached by some periodic minimizer with a given `pattern' or microstructure (describing the arrangement of $u_i-u_{i-1}$ in the two energy wells), then one can  expect the same pattern to be optimal also for small perturbations of $z$ (with of course, a small variation of the values of $u$). This would  then  entail that the corresponding $\theta$ is a locking state, however, the formalization of this statement remains  unproven even if   it   holds in all our examples.
\end{remark}

\subsection{Phase-constrained analysis for decoupled interactions}
In this section we focus on the two extreme cases when the effects of $f$ and $\bf m$ can be decoupled;  namely, either when $\bf m$ vanishes or when $f$ is convex. 
A comparison with these cases will highlight how for general $f$ and $\bf m$ the interplay between non-convexity and non-locality gives rise to complex superposition effects. Such effects will be analyzed in the following sections in two particularly meaningful examples. 

\subsubsection{Convexification as an envelope of phase-constrained problems}\label{sec-degene}
We start by considering the case when the kernel $\bf m$ vanishes. 
We know that in this case 
$$Q_{\bf m}f(z)=\widehat Q_{\bf m}f(z)=f^{\ast\ast}(z)$$
for any $z$. We can still  focus on the dependence of the partially relaxed energy on the volume fraction $\theta$, which is already non trivial. Moreover, it  shows some features that  we  will  later encounter  in  more complex examples.  

In this section, we will use ${\bf 0}$ instead of ${\bf m}$ in the notation. 
Suppose that while $f\colon\mathbb R\to[0,+\infty)$ is not convex,  there exists $z^\ast\in\mathbb R$ such that the restrictions of $f$ to $(-\infty,z^\ast]$ and $[z^\ast,+\infty)$ are convex. 
For such  $f$, we now compute both  $Q_{\bf 0}f(\theta,z)$ and $\Theta(z)$.

\begin{remark}[growth condition]\rm 
The growth condition from below on $f(z)+2m_1z^2$ assumed in the previous sections, in this case would imply 
a growth condition on $f$. 
Nevertheless, for the results of this section it is not necessary, and below we also treat cases where it is not satisfied, showing some non-continuity effects. 
\end{remark}

Let $f_0$ and $f_1$ denote the restrictions of $f$ to $(-\infty,z^\ast]$ and to $[z^\ast,+\infty)$, respectively. 
For $\theta\in(0,1)$, by using the convexity of $f_0$ and $f_1$  we get
\begin{eqnarray*}
Q_{\bf 0}f(\theta,z)=\inf\{(1-\theta)f_0(t)+\theta f_1(s): t\leq z^\ast, \ s\geq z^\ast,\ (1-\theta)t+\theta s=z\}.  
\end{eqnarray*} 
As for the limit cases $\theta=0$ and $\theta=1$, we have   
$$Q_{\bf 0}f(0,z)=\begin{cases}f_0(z) &\hbox{\rm if } z\leq z^\ast\\
+\infty &\hbox{\rm if } z > z^\ast
\end{cases} 
\quad \hbox{\rm and }\quad 
Q_{\bf 0}f(1,z)=\begin{cases}+\infty &\hbox{\rm if } z < z^\ast\\
f_1(z) &\hbox{\rm if } z \geq z^\ast.
\end{cases}$$ 

We subdivide the subsequent analysis in dependence of the shape of the function $ f^{\ast\ast} (z)$ representing the convex envelope of $f$; more precisely, on whether the `non-convexity set' $\{z: f^{\ast\ast}(z)<f(z)\}$ is a bounded interval, a half line or the whole line. Note that in this set $ f^{\ast\ast}$ is affine.

\paragraph{Case (a): the non-convexity set is a bounded interval.} 
We suppose that there exist $z_0\in(-\infty,z^\ast]$ and 
$z_1\in[z^\ast,+\infty)$ such that 
\begin{equation}\label{convenv}
f^{\ast\ast}(z)=
\begin{cases}
f(z) & \hbox{\rm if } \ z\in\mathbb R\setminus(z_0,z_1)\\
r(z) & \hbox{\rm if } \ z\in[z_0,z_1], 
\end{cases}
\end{equation}
where $r$  is affine and $r(z)<f(z)$ in $(z_0,z_1)$, then 
$Q_{\bf 0}f(z)$ is obtained as a minimum of $Q_{\bf 0}f(\theta,z)$. 
In this case, $\Theta(z)$ is a single value $\theta(z)$ for any $z$, and 
$$\theta(z)=\begin{cases}
0 & \hbox{\rm if } z\leq z_0\\
\displaystyle\frac{z-z_0}{z_1-z_0} & \hbox{\rm if } z_0\leq z \leq z_1\\
1 & \hbox{\rm if } z\geq z_1. 
\end{cases}$$
Note that trivially $Q_{\bf 0}f(z)$ is the convex envelope of the minimum of the two 
functions $Q_{\bf 0}f(0,z)$ and $Q_{\bf 0}f(1,z);$ that is, of $\min\{Q_{\bf 0}f(\theta,z): \theta \ \hbox{\rm is a locking state}\}$, since the only locking states are $0$ and $1$.

Note moreover that, if $\lim\limits_{z\to+\infty}\frac{f(z)}{z}=+\infty$ and 
$f^\prime_-(z^\ast)$ is finite, then the formula giving $Q_{\bf 0}f(\theta,z)$ can be simplified for $z$ large enough. 
Indeed, there exists 
$z^+$ such that for any $\theta\in(0,1)$
\begin{eqnarray*}
Q_{\bf 0}f(\theta,z)=(1-\theta)f_0(z^\ast)+\theta f_1\Big(\frac{z-(1-\theta)z^\ast}{\theta}\Big) \ \ \hbox{\rm if}
\ z\geq z^+. 
\end{eqnarray*} 
Correspondingly, if 
$\lim\limits_{z\to-\infty}\frac{f(z)}{|z|}=+\infty$ and $f^\prime_+(z^\ast)$ is finite then, for any $\theta\in(0,1)$,  
\begin{eqnarray*}
Q_{\bf 0}f(\theta,z)=(1-\theta)f_0\Big(\frac{z-\theta z^\ast}{1-\theta}\Big)+\theta f_1(z^\ast) \ \ \hbox{\rm if} \ z\leq z^-
\end{eqnarray*}
for $|z^-|$ large enough.  
\begin{figure}[h!]
\centerline{\includegraphics[width=0.8\textwidth]{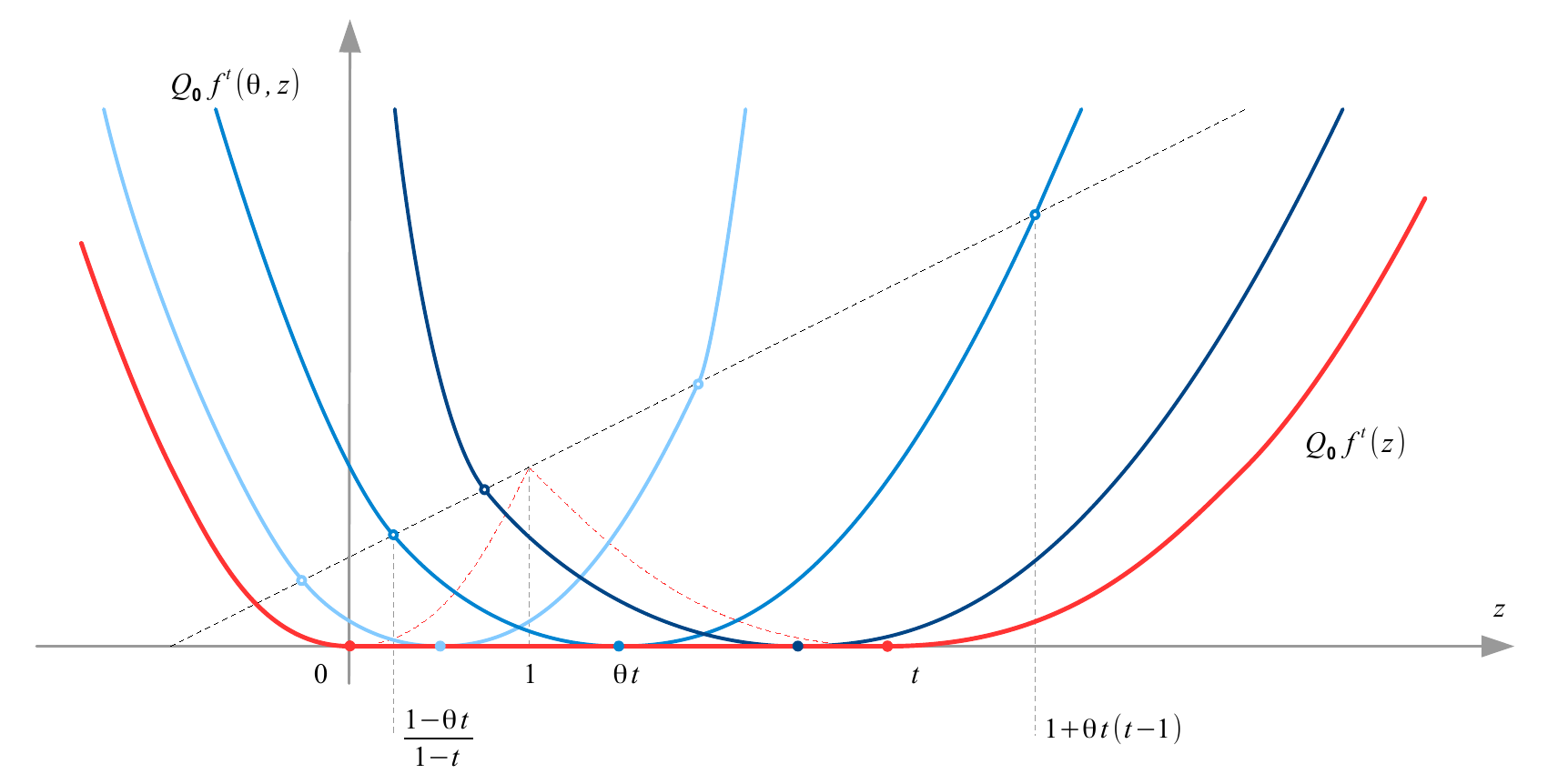}}
\caption{Graph of $Q_{\bf 0}f^t(\theta,z)$ for different values of $\theta$.}
\label{fig_fb1}
\end{figure} 
\begin{figure}[h!]
\centerline{\includegraphics[width=0.8\textwidth]{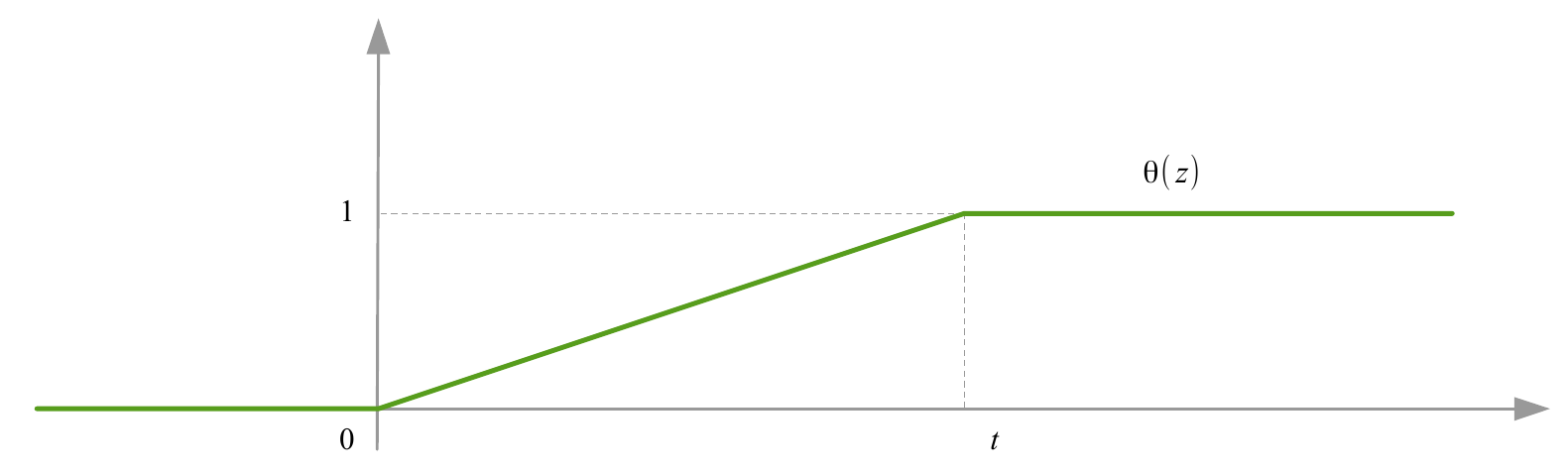}}
\caption{Graph of the phase function $\theta(z)$ for the function $f^t$ in Example \ref{dwm0}.}
\label{fig_fb2}
\end{figure}

\begin{example}[double-well bi-quadratic potential]\label{dwm0}\rm 
For any $t>1$ we define 
$$f^{t}(z)=\begin{cases}
z^2 & \hbox{\rm if }\ z\leq 1 \\
\displaystyle\Big(\frac{z-{t}}{1-{t}}\Big)^2 & \hbox{\rm if }\ z\geq 1.
\end{cases}$$
If $\theta\in(0,1)$, we get 
$$Q_{\bf 0}f^{t}(\theta,z)=\begin{cases}
\vspace{2mm}
\displaystyle\frac{(z-\theta)^2}{1-\theta}+\theta & \displaystyle\hbox{\rm if }\ z\leq \frac{1-\theta {t}}{1-{z_0}}\\
\vspace{2mm}
\displaystyle\frac{(z-\theta {t})^2}{1-\theta+\theta(1-{t})^2} & \displaystyle\hbox{\rm if }\ \frac{1-\theta {t}}{1-{t}}\leq z\leq 
1+\theta {t}({t}-1)\\
\displaystyle\frac{(z-1+\theta(1-{t}))^2}{\theta(1-{t})^2}+1-\theta 
& \displaystyle\hbox{\rm if }\ z\geq 
1+\theta {t}({t}-1) 
\end{cases}$$
(see Fig.~\ref{fig_fb1} and Fig.~\ref{fig_fb2} for the graph of $Q_{\bf 0}f^{t}(\theta,z)$ and $\theta(z)$, respectively, with different values of 
$\theta$ and ${t}$ fixed). 
\begin{figure}[h!]
\centerline{\includegraphics[width=1\textwidth]{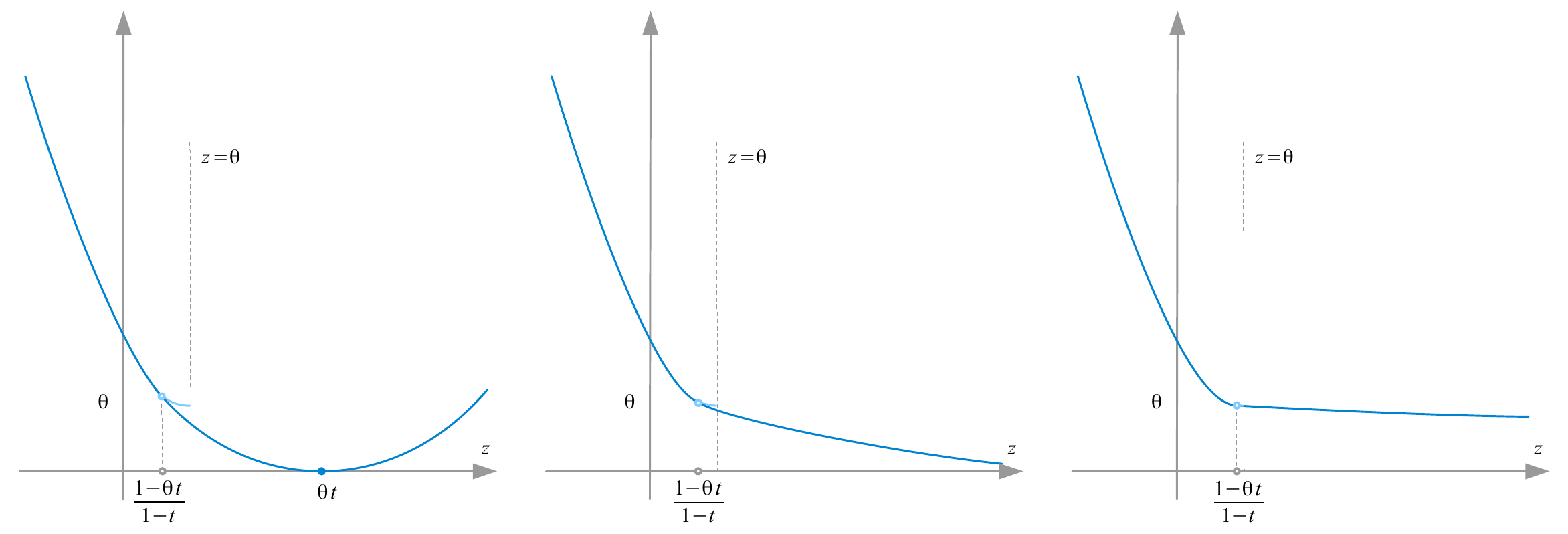}}
\caption{Graph of $Q_{\bf 0}f^{t}(\theta,z)$ with $\theta$ fixed and increasing values of ${t}$.}
\label{fig_fb3}
\end{figure} 

In Fig.~\ref{fig_fb3} we picture the graph for a fixed $\theta$ and increasing values of ${t}$. 

\end{example}

\begin{remark}[fracture as limit of phase transitions]\rm 
If $f^{t}$ is defined as in Example \ref{dwm0}, then for any fixed $\theta\in(0,1)$ 
$$\lim_{{t}\to+\infty}Q_{\bf 0}f^{t}(\theta,z)
=\begin{cases}
\displaystyle\frac{(z-\theta)^2}{1-\theta}+\theta & \hbox{\rm if }\ z\leq \theta\\
\theta & \hbox{\rm if }\ z\geq \theta.
\end{cases}$$ 
This limit function is $Q_{\bf 0}f(\theta,z)$ for $f$ the truncated parabola (see Example \ref{tcm0} below with $\tilde f(z)=z^2$). 
This asymptotic behaviour is illustrated in Fig.~\ref{fig_fb3} above. 

From a mechanical standpoint, in the limit as $t \to +\infty$ we can recover the case fracture as limit of phase-transitions problems as the second well gets moves to the right and its curvature diminishes. 
For a mechanical interpretation of this phenomenon, we refer to \cite{Trusk}.  
In that perspective, also the energies at fixed $\theta$ are of interest, because it is the case when something prevents cracks from localization. The resulting constrained material becomes `tension free'. 
\end{remark}

\medskip 
\paragraph{Case (b): the non-convexity set is a half line.}
Let $f^{\ast\ast}(z)<f(z)$ on a half-line, and 
assume that the half-line is bounded from below, the other case being symmetric. 

By the convexity properties of $f_0$ and $f_1$, up to the subtraction of the affine function asymptotic to $f_1$ at $+\infty$, it is not restrictive 
to assume that 
$\lim\limits_{z\to+\infty}f_1(z)\in [\min f_0, +\infty)$, so that 
$f^{\ast\ast}=\min f_0$ in $[z_0^{\rm min},+\infty)$, where $z_0^{\rm min}$ is the largest minimizer of $f_0$ in $(-\infty,z^\ast]$.

For any $\theta\in(0,1)$ and $z\geq (1-\theta)z_0^{\rm min}+\theta z^\ast$, 
we can use $z_0^{\rm min}$ and $\frac{z-(1-\theta)z_0^{\rm min}}{\theta}$ as test values for $Q_{\bf 0}f(\theta,z)$. If $z>z_0^{\rm min}$, taking the limit as $\theta\to 0$ we get 
$$
Q_{\bf 0}f(z)=f_0(z_0^{\rm min})=
\lim_{\theta\to 0}\Big((1-\theta)f_0(z_0^{\rm min})+\theta f_1\Big(\frac{z-(1-\theta)z_0^{\rm min}}{\theta}\Big)
\Big)=\lim_{\theta\to 0}Q_{\bf 0}f(\theta,z).
$$
Since $Q_{\bf 0}f(\theta,z)=+\infty$ for $z>z^\ast\geq z_0^{\rm min}$, the function $\theta\mapsto Q_{\bf 0}f(\theta,z)$ is not lower semicontinuous in $0$. If we also assume  that $\lim\limits_{z\to+
\infty}f_1(z)>\min f_0$, then  
$$Q_{\bf 0}f(\theta,z)>Q_{\bf 0}f(z) 
\ \ \hbox{\rm for any }\ \ \theta\in[0,1] \ \ \hbox{\rm and }\ \ z>z_0^{\rm min}$$ 
and $\overline\Theta(z)=\emptyset$ (see Remark \ref{remnlsc}). 
Since for $z\leq z^\ast$ 
we have 
$Q_{\bf 0}f(0,z)=Q_{\bf 0}f(z)$, it follows that  
$ 
\theta(z)=0$ for any $z\leq z^\ast$.

\begin{example}[truncated convex potentials]\label{tcm0}\rm 
\begin{figure}[h!]
\centerline{\includegraphics[width=1\textwidth]{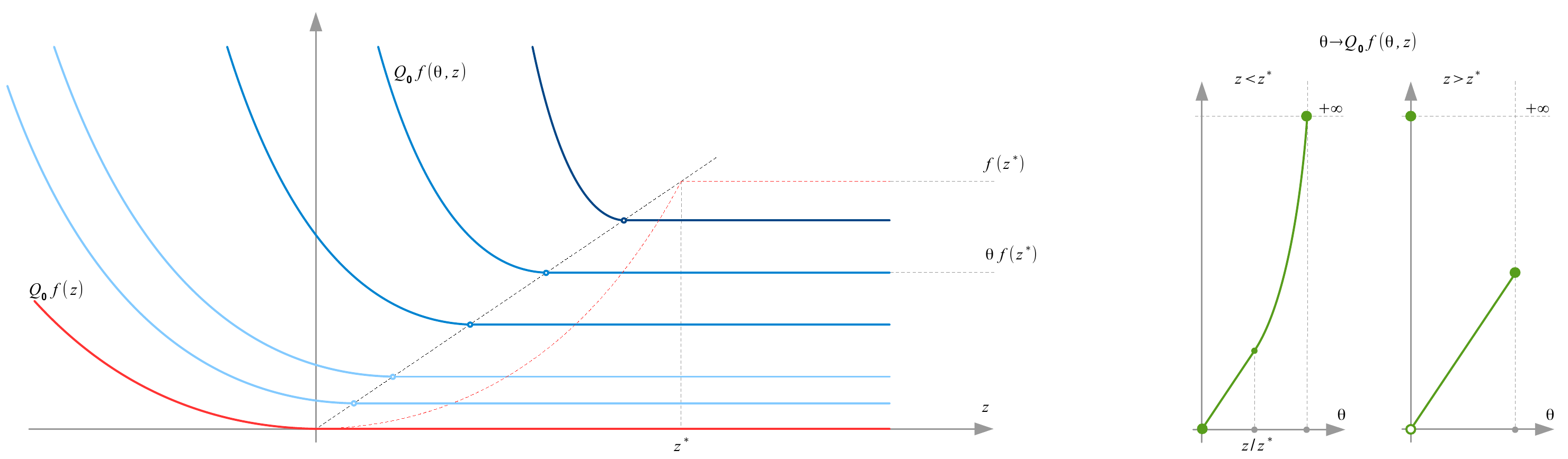}}
\caption{(a) $Q_{\bf 0}f(\theta,z)$ for a truncated convex potential and (b) $\theta\mapsto Q_{\bf 0}f(\theta,z)$ for different values of $z$.}
\label{fig_frattura_0}
\end{figure} 
Let $f$ be the truncated convex given by  
\begin{equation}\label{def-ftrunc1}
f(z)=\begin{cases} \tilde f(z) & \hbox{\rm if } z\leq z^\ast\\
\tilde f(z^\ast) 
& \hbox{\rm if } z\geq z^\ast, 
\end{cases}
\end{equation}
where $\tilde f$ is a convex function such that 
the only minimum point of $\tilde f$ is $0$ with $\tilde f(0)=0$, and $z^\ast>0$. 
In particular, we can take $\tilde f(z)=z^2$, in which case $f$ is called a truncated quadratic potential. 
For $\theta\in(0,1)$ we get 
\begin{eqnarray*}
&&Q_{\bf 0}f(\theta,z)=\begin{cases}
\vspace{1mm}
\displaystyle\theta f(z^\ast)+(1-\theta)
f\Big(\frac{z-\theta z^\ast}{1-\theta}\Big) & \hbox{ if } z <  \theta z^\ast\\
\displaystyle\theta f(z^\ast)  & \hbox{ if } z \geq \theta z^\ast. 
\end{cases} 
\end{eqnarray*} 
For all such $f$ the graphs of $Q_{\bf 0}f(\theta,z)$ and of $Q_{\bf 0}f(z)$ have the form as those pictured in Fig.~\ref{fig_frattura_0}(a). 
In Fig.~\ref{fig_frattura_0}(b) the function $\theta\mapsto Q_{\bf 0}f(\theta,z)$ is represented for two different values of $z$, highlighting the lack of lower semicontinuity in $0$ if $z>0$.   
Note that for any $\theta\in(0,1)$ we have $Q_{\bf 0}f(\theta,z)> 
f(z)$ in $(-\infty,\theta z^\ast]$.  
Moreover, 
the optimal volume fraction $\theta(z)$ is always equal to zero, 
even though $Q_{\bf 0}f(\theta,z)=Q_{\bf 0}f(0,z)$ only if $z\leq 0$ 
(see Remark \ref{lockingrem} below). 
\end{example}

\paragraph{Case (c): the non-convexity set is the whole line.}
If $f^{\ast\ast}<f$ in the whole $\mathbb R$, then in our hypothesis it is constant, and as in case (b) it is not restrictive to suppose that both 
$\lim\limits_{z\to-\infty}f(z)$ and 
$\lim\limits_{z\to+\infty}f(z)$ are finite, so that $$Q_{\bf 0}f(z)=\min\Bigl\{\lim\limits_{z\to-\infty}f(z), \lim\limits_{z\to+\infty}f(z)\Bigr\}.$$ For $\theta\in(0,1)$, 
$$Q_{\bf 0}f(\theta,z)=(1-\theta)\lim_{z\to-\infty}f(z)+\theta\lim_{z\to+\infty}f(z).$$ 
The function $\theta\mapsto Q_{\bf 0}f(\theta,z)$ is not lower semicontinuous 
in $0$ if $z> z^\ast$ and in $1$ if $z< z^\ast$.  

If $\lim\limits_{z\to-\infty}f(z)=\lim\limits_{z\to+\infty}f(z)$, then $Q_{\bf 0}f(\theta,z)=Q_{\bf 0}f(z)$ for any 
$\theta\in(0,1)$ and $z\in\mathbb R$, hence $\Theta(z)=[0,1]$ for any $z$ 
and $\theta(z)=0$ for any $z$.

If $\lim\limits_{z\to+\infty}f(z)<\lim\limits_{z\to-\infty}f(z)$, then for any $z$ we have $\Theta(z)=\{1\}$ and $\theta(z)=1$.  
Conversely, if $\lim\limits_{z\to+\infty}f(z)>\lim\limits_{z\to-\infty}f(z)$, then for any $z$ we have $\Theta(z)=\{0\}$ and $\theta(z)=0$. Note that $\overline\Theta(z)=\emptyset$ 
at least for any $z<z^\ast$ in the first case, and 
at least for any $z>z^\ast$ in the second.

We give some simple examples of case (c), highlighting the difference between $\Theta$ and $\overline\Theta$ due to the lack of semicontinuity at the endpoints. 
 
\begin{example}\label{ex-deg-2}\rm 
\noindent $\bullet$ If $f(z)=\min\{1, e^{-z}\}$, then $Q_{\bf 0}f(z)=0$ and $Q_{\bf 0}f(\theta,z)=1-\theta$ for $\theta\in(0,1)$. Since $Q_{\bf 0}f(0,z)$ and $Q_{\bf 0}f(1,z)$ are strictly positive, then 
$\overline\Theta(z)=\emptyset$ for any $z$. In this case, there is no locking state. 

\medskip 

\noindent $\bullet$ If $f(z)=\max\{\min\{1,2e^{-z}-1\}, 0\}$, then $Q_{\bf 0}f(z)=0$ and $Q_{\bf 0}f(\theta,z)=1-\theta$ for $\theta\in(0,1)$ as in the previous case. In this case, $Q_{\bf 0}f(1,z)=0$ if $z\geq \log 2$, hence 
$\overline\Theta(z)=\emptyset$ for any $z<\log 2$ and $\overline\Theta(z)=\{1\}$ if $z\geq \log 2$. 
The only locking state is $\theta=1$. 

\medskip 

\noindent $\bullet$ If $f(z)=e^{-|z|}$, then $Q_{\bf 0}f(z)=Q_{\bf 0}f(\theta,z)=0$ for any $\theta\in(0,1)$ and $z\in\mathbb R$. The set $\overline\Theta(z)=(0,1)$ for any $z$, while $\Theta(z)=[0,1]$. The only locking state is $\theta=0$. 
\end{example}

\begin{remark}[locking states in the degenerate cases]\label{lockingrem}\rm 
While we still have that trivially $Q_{\bf 0}f(z)$ is the convex envelope of 
$\min\{Q_{\bf 0}f(0,z), Q_{\bf 0}f(1,z)\}$, in the examples of cases (b) and (c) nor both values $\theta=0$ and  $\theta=1$ are regarded as locking states. 
In the last of Examples \ref{ex-deg-2} this is due to the arbitrary choice of defining $\theta(z)$ as an infimum. 
As a consequence, the notion of locking state is not relevant in the computation of $Q_{\bf 0}f$, in the sense that we cannot recover $Q_{\bf 0}f(z)$ from the only knowledge of  $Q_{\bf 0}f(\theta,z)$ for $\theta$ locking states. 
In Example \ref{tcm0}, indeed we have the only locking state $\theta=0$ but $Q_{\bf 0}f(0,z)=+\infty$ for $z>0$. 
\end{remark}

\subsubsection{Convex potentials: phase-constrained interpolation}\label{trconm01}
We now consider the second extreme case; that is, when the function $f$ is convex on all $\mathbb R$ and the kernel $\bf m$ is arbitrary. As we noticed in Proposition \ref{disug}, in this  case  the  function $Q_{\bf m}f$ is trivially equal to $f$ for any choice of $\bf m$. Nevertheless, the results  of the constrained minimization producing the functions $Q_{\bf m}f(\theta,z)$ 
are non-trivial even  in this case. They provide  further information regarding the general structure of the dependence of $Q_{\bf m}f(\theta,z)$  on the phase variable $\theta$. Moreover, such examples can  serve as   comparison limit cases for non-convex energies $f$. 

\medskip  

Let $f\colon\mathbb R\to\mathbb R$ be a convex function while  ${\bf m}$ can be arbitrary. In this case, we would need  the growth 
hypothesis 
$\lim\limits_{z\to\pm\infty} f(z)+2m_1z^2=+\infty$  only to use some technical result concerning the variation of the boundary conditions. 
We fix an arbitrary $z^\ast\in\mathbb R$ and define $A=[z^\ast,+\infty)$. 

As for $\theta=0$ and $\theta=1$, by   definition  we have 
\begin{eqnarray*}
		Q_{\bf m}f(0,z)=\begin{cases}  
			f(z)& {\rm if }\ z\leq z^\ast\\
			+\infty & {\rm if }\ z>z^\ast
		\end{cases} 
		\quad  {\rm and } \quad 
		Q_{\bf m}f(1,z)=\begin{cases} 
			+\infty & {\rm if }\ z<z^\ast\\
			f(z)& {\rm if }\ z\geq z^\ast.
		\end{cases}
	\end{eqnarray*} 
	In particular, $Q_{\bf m}f(z)=Q_{\bf m}f(0,z)$ for $z\leq z^\ast$ and  
$Q_{\bf m}f(z)=Q_{\bf m}f(1,z)$ for $z>z^\ast$. 
Moreover, the following proposition holds. 	
\begin{proposition}
For $\theta\in(0,1)$, we have  
	\begin{equation}\label{qtetaconvessa}
	Q_{\bf m}f(\theta,z)=
	\begin{cases}\vspace{1mm}
	\displaystyle \theta f(z^\ast)+(1-\theta)f\Big(\frac{z-\theta z^\ast}{1-\theta}\Big) 
	+a_{\bf m}\frac{\theta}{1-\theta}(z-z^\ast)^2
	& {\rm if }\ z<z^\ast \\
	\displaystyle (1-\theta) f(z^\ast)+\theta f\Big(\frac{z-(1-\theta) z^\ast}{\theta}\Big) 
	+a_{\bf m}\frac{1-\theta}{\theta}(z-z^\ast)^2
	& {\rm if }\ z\geq z^\ast. 
		\end{cases}
	\end{equation}
\end{proposition} 

\begin{proof} 
We fix $z<z^\ast$. Let $z_i$ be such that $\sum_{i=1}^{kq}z_i= kqz$, and
$$
\overline z=\frac{1}{\#I}\sum_{i\in I}(f(z_i)+2m_1z_i^2),
$$ where $I=\{i: z_i\ge z^*\}$ and $\#I=\theta kq$.
Since $f$ is convex, we get
\begin{eqnarray}\label{primconv}\nonumber
&&\hspace{-2cm}\frac{1}{kq}\sum_{i=1}^{kq}(f(z_i)+2m_1z_i^2)
=\frac{1}{kq}\Bigl(\sum_{i\in I}(f(z_i)+2m_1z_i^2)+\sum_{i\not\in I}(f(z_i)+2m_1z_i^2)\Bigr)\\ \nonumber
&&\geq\theta (f(\overline z)+2m_1(\overline z)^2)
\nonumber
+(1-\theta)\Big(f\big(\frac{z-\theta \overline z}{1-\theta}\big)+2m_1\big(\frac{z-\theta \overline z}{1-\theta}\big)^2\Big)\\ 
&&\geq\theta (f(z^\ast)+2m_1(z^\ast)^2)+(1-\theta)\Big(f\big(\frac{z-\theta z^\ast}{1-\theta}\big)
+2m_1\big(\frac{z-\theta z^\ast}{1-\theta}\big)^2\Big),
\end{eqnarray}
since $\overline z\ge z^*$ and $\frac{z-\theta \overline z}{1-\theta}\le \frac{z-\theta z^\ast}{1-\theta}$. 

Let $M\in \mathbb N$ and  $n\leq M$ be fixed. We define $n$ partitions of the interval $[0,kq]$ given by the set of points 
$$P_j=\Big\{hn+j: h=0,\dots, \Big\lfloor \frac{kq-j}{n}\Big\rfloor-1\Big\}, \ j=0,\dots, n-1.$$ 
Let $u$ be an admissible test function for $\widehat Q_{\bf m}f(\theta,z)$. Recalling Lemma \ref{boundary-Q},  we can suppose $u\in\mathcal V(kq;\theta)\cap \widetilde{\mathcal A}_\delta(kq;z)$. 
With fixed $n$ and $j$, let $\widetilde z$ and $\widetilde\theta$ be defined by 
\begin{eqnarray*}
	&&u_{\lfloor \frac{kq-j}{n}\rfloor n+j}-u_j=
	\Big\lfloor \frac{kq-j}{n}\Big\rfloor n \widetilde z \\ 
	&&\widetilde\theta\Big\lfloor \frac{kq-j}{n}\Big\rfloor n=\#\Big\{i\in \Big[j, \Big\lfloor \frac{kq-j}{n}\Big\rfloor n+j\Big]\cap
	\mathbb Z: u_i-u_{i-1}\geq z^\ast\Big\}.
\end{eqnarray*} 
Since $u\in \widetilde{\mathcal A}_{\delta}(kq;z)$ and $n\leq \delta kq$, we obtain $(kq-2n) |z - \widetilde z|
\leq 4n|z| +2n |z|$.   
Moreover 
$(kq-2n)|\theta-\widetilde\theta|\leq 2n+2n \theta$, 
so that (uniformly with respect to $n$ and $j$)
\begin{equation}\label{varbordo} 
\widetilde z=z+o(1)_{k\to+\infty}, \ \  \widetilde \theta=\theta+o(1)_{k\to+\infty}.
\end{equation}
In particular, if $k$ is large enough then $\widetilde z<z^\ast$.
By substituting to any $z_i\geq z^\ast$ the value $z^\ast$ and 
to any $z_i<z^\ast$ the value $\frac{\widetilde z-\widetilde\theta z^\ast}{1-\widetilde\theta}$, 
the convexity of the square gives   
$$\frac{1}{kq}\sum_{i\in P_{j}}(z_{i+1}+\dots+z_{i+n})^2\geq\frac{1}{kq} \Big\lfloor\frac{\theta kq}{n}\Big\rfloor (nz^\ast)^2+\frac{1}{kq}\Big\lfloor\frac{(1-\widetilde\theta)kq}{n}\Big\rfloor 
\Big(n\frac{\widetilde z-\widetilde\theta z^\ast}{1-\widetilde\theta}\Big)^2.$$
Hence, recalling \eqref{varbordo} 
\begin{eqnarray*}
	\sum_{i,j=0}^{kq}m_{|i-j|}(u_i-u_j)^2&\geq&  2 \sum_{n=1}^M m_n n \Big(  \Big\lfloor\frac{\theta kq}{n}\Big\rfloor (nz^\ast)^2+\Big\lfloor\frac{(1-\theta)kq}{n}\Big\rfloor (n \frac{z-\theta z^\ast}{1-\theta})^2\Big)\\
	&&+o(1)_{k\to+\infty}
\end{eqnarray*}
which, together with \eqref{primconv}, gives the estimate
\begin{eqnarray*}
	\frac{1}{kq}F_{1}(u;[0,kq])&\geq& 
	\theta f(z^\ast)+(1-\theta)f\Big(\frac{z-\theta z^\ast}{1-\theta}\Big)\\
	&&+
	2 \sum_{n=1}^M m_n \Big( \theta(nz^\ast)^2+(1-\theta) (n \frac{z-\theta z^\ast}{1-\theta})^2\Big)+o(1)_{k\to+\infty}.
\end{eqnarray*} 
We obtain that 
\begin{equation*}
\widehat Q_{\bf m}f(\theta,z)\geq
\theta f(z^\ast)+(1-\theta)f\Big(\frac{z-\theta z^\ast}{1-\theta}\Big)
+2  \sum_{n=1}^M m_n \Big( \theta(nz^\ast)^2+(1-\theta) \big(n \frac{z-\theta z^\ast}{1-\theta}\big)^2\Big).
\end{equation*} 
Since $M$ is arbitrary, we conclude that 
\begin{equation*}
Q_{\bf m}f(\theta,z)\geq
\theta f(z^\ast)+(1-\theta)f\Big(\frac{z-\theta z^\ast}{1-\theta}\Big)+a_{\bf m}\Big(\theta(z^\ast)^2 
+ (1-\theta) \big(\frac{z-\theta z^\ast}{1-\theta}\big)^2\Big)-a_{\bf m}z^2, 
\end{equation*} 
which gives the lower bound for \eqref{qtetaconvessa} in the case $z<z^\ast$. 

As for the upper estimate, we define a test function $\overline u$ by setting 
$$\overline u_i=\begin{cases}\vspace{1mm}
z^\ast i  & {\rm if }\ i\leq \theta kq\\
\displaystyle z^\ast \theta kq + \frac{z-\theta z^\ast}{1-\theta}(i-\theta kq)  & {\rm if }\ i > \theta kq; 
\end{cases}$$
since $m_n=o(n^{-\beta})_{n\to+\infty}$ with $\beta>3$ 
we obtain 
\begin{eqnarray*}
	\frac{1}{kq}F_1(\overline u;[0,kq])&=&\theta f(z^\ast)+(1-\theta)f\Big(\frac{z-\theta z^\ast}{1-\theta}\Big)\\
	&&+
	2 \sum_{n=1}^{\infty} m_n\Big( \theta(nz^\ast)^2+(1-\theta) (n \frac{z-\theta z^\ast}{1-\theta})^2\Big)+o(1)_{k\to+\infty}, 
\end{eqnarray*}
which gives the upper bound for $k\to+\infty$. 
Similar arguments allow one to prove 
\eqref{qtetaconvessa} for $z>z^\ast$ or $z=z^\ast$. 
\end{proof}
\begin{figure}[h!]
	\centerline{\includegraphics[width=0.6\textwidth]{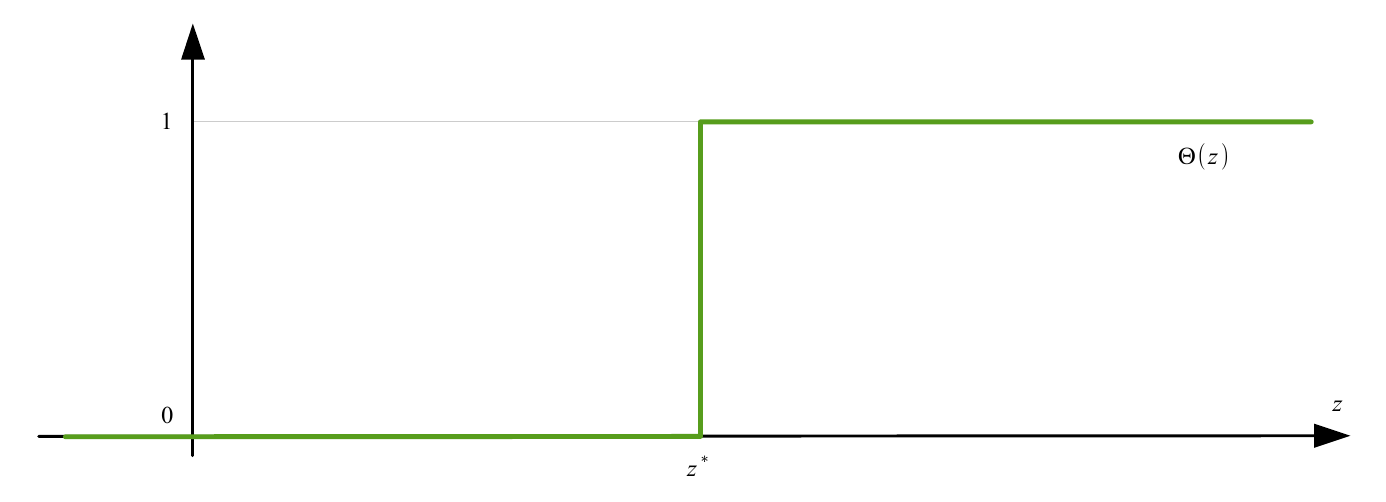}}
	\caption{the phase multifunction $\Theta(z)$ in the convex case.}
	\label{figure_thetacost}
\end{figure} 

Note that the phase multifunction $\Theta(z)$ is given by (see Fig.~\ref{figure_thetacost})  
\begin{equation} 
\Theta(z)=\begin{cases} 
\{0\} & \hbox{\rm if } z<z^\ast\\
[0,1] & \hbox{\rm if } z=z^\ast\\
\{1\} & \hbox{\rm if } z>z^\ast. 
\end{cases}
\end{equation} 
Here $0$ and $1$ are the only locking states.
\medskip 

A particular  interesting  sub-case in this general class of problems is represented by semi-degenerate  quadratic-affine functions, often used in theories of plasticity. Assume for instance that for all $\tau\in\mathbb R$  the function 
$\ell^\tau\colon\mathbb R\to\mathbb R$ is defined as 
\begin{equation}\label{quadraticalfa}
\ell^{\tau}(z)=\begin{cases}
z^2 & \hbox{\rm if }\ z\leq 1\\
2\tau(z-1)+1& \hbox{\rm if }\ z>1. 
\end{cases}
\end{equation} 
Using the general expression for $Q_{\bf m}f(\theta,z)$ in \eqref{qtetaconvessa}, we can now obtain an explicit formula for $Q_{\bf m}
\ell^\tau(\theta,z)$ in the convex case $\tau\geq 1$, with the natural choice $A=[1,+\infty)$. 

\begin{example}[convex-affine potentials]\label{plasticity-ex2}\rm 
	\begin{figure}[h!]
		\centerline{\includegraphics[width=0.9\textwidth]{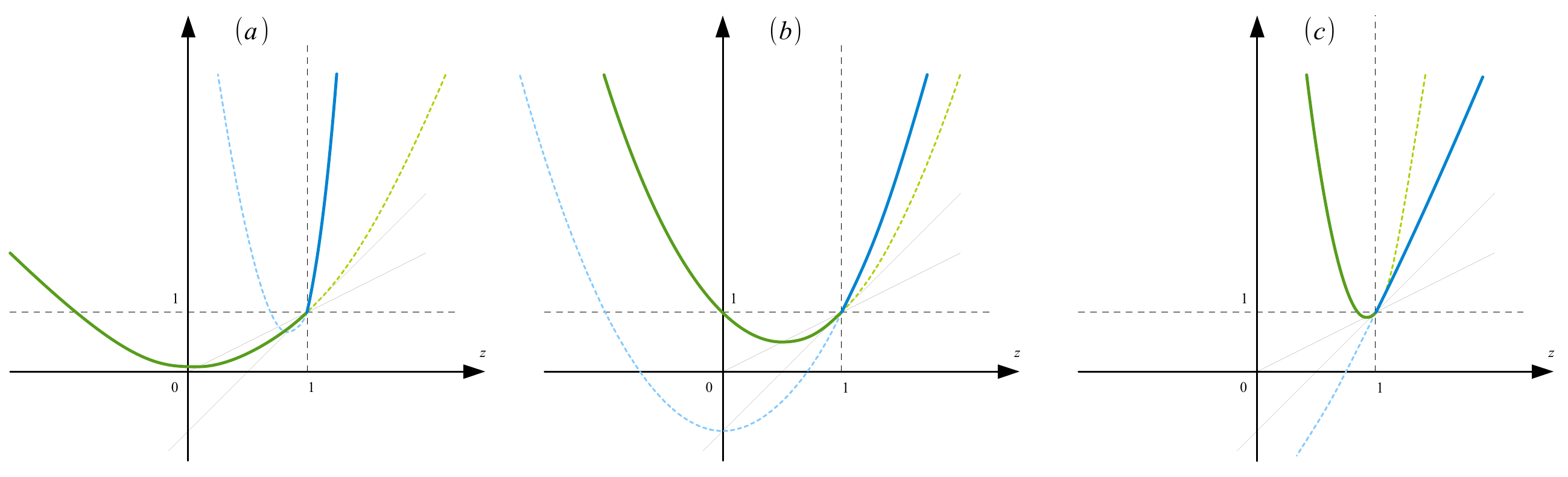}}
		\caption{$Q_{\bf m}
		\ell^\tau(\theta,z)$ for increasing values of $\theta$.}
		\label{figure_convexaffinege1}
	\end{figure} 
	Let  
	$\ell^\tau$ 
	be defined as in \eqref{quadraticalfa}. In the convex case $\tau\geq 1$, for any $\theta\in(0,1)$ we have  
	\begin{equation*}
	Q_{\bf m}
	\ell^\tau(\theta,z)=\begin{cases}
	\vspace{2mm}
	\displaystyle\frac{1+a_{\bf m}\theta}{1-\theta}z^2
	-\frac{2(1+a_{\bf m})\theta}{1-\theta}z
	+\frac{(1+a_{\bf m})\theta}{1-\theta}&\hbox{\rm if } z\leq1\\
	\displaystyle\frac{a_{\bf m}(1-\theta)}{\theta} z^2
	-\Big(\frac{2a_{\bf m}(1-\theta)} {\theta} -\tau\Big)z+
	1-2\tau+\frac{a_{\bf m}(1-\theta)}{\theta}  
	&\hbox{\rm if } z\geq1
	\end{cases}
	\end{equation*}
\noindent These constructions are illustrated in Fig.~\ref{figure_convexaffinege1}. 	
\end{example}
In Sections \ref{mneigh-sec} and \ref{exp-sec} we will also treat the non-convex case of 
$\ell^\tau$; that is, $\tau<1$, with particular choices of the interaction kernel $\bf m$.  
Note that all the general results concerning $\ell^\tau$ 
still hold if we take a convex $\tilde f$ instead of the quadratic term. 

\subsection{Spin representation and optimal microstructures} 
We observe that for bi-convex problems 
a more detailed way to describe the behaviour of extremal functions is by using a two-value function which labels the position of the strain variable, whether in one or in the other of the two convex zones of $f$. Such `spin function' can be viewed as a characteristic function of  the microstructure of an  extremal. 
Note that periodic spin functions determine a corresponding rational volume fraction $\theta$. 

To illustrate the geometry of microstructures  we restate periodic minimum problems for bi-convex functions  in terms of a spin representation. This will allow us  to rewrite non-convex minimum problems as minima of a family of convex problems, and to obtain a better control of the geometry of minimizers.  We will use this formulation in some explicit examples in the next sections, to characterize  
optimal periodic geometries. 

We begin by formally introducing the {\em spin variable} $\underline s\in\{-1,1\}^N$  parameterizing the location of the argument of a bi-convex function $f$. The corresponding  volume fraction is then $$\theta={1\over 2N}\sum_{j=1}^N(1-s_j).$$ 

Let $f_{-1},f_{1}\colon\mathbb R\to\mathbb R$ be such that 
\begin{equation}\label{fmin}
f(z)=\min\{f_{-1}(z), f_1(z)\}=\begin{cases} 
f_{-1}(z) & \hbox{\rm if }\ z\in \mathbb R\setminus A\\
f_1(z) & \hbox{\rm if }\ z\in A. 
\end{cases}
\end{equation}   
The slight difference in the notation with respect to previous sections, where the two functions were denoted by $f_0$ and $f_1$, is due to the focus on individual components of the  spin vector  taking the values  $-1$ and $1$.
While the definitions and properties will hold without any further assumptions, in the applications 
we will consider the `natural' case when $A$ is a half-line and the functions $f_{-1},f_{1}$ are convex. 

Omitting the dependence on $A$, for $N\in\mathbb N$  
we define $\widehat R^N_{\bf m}f\colon \{-1,1\}^N\times \mathbb R\to\mathbb R$ by setting 
\begin{equation}\label{phiminus}
\widehat R^N_{\bf m}f(\underline{s},z)=\frac{1}{N}\inf\big\{F^\#(u,\underline s;[0,N]):\  i\mapsto u_i-zi \hbox{ is } N\hbox{-periodic}\big\}, 
\end{equation}
where 
\begin{equation*}
F^\#(u,\underline s;[0,N])= \sum_{i=1}^{N} f_{s_i}(u_i-u_{i-1})+\sum_{i=1}^N\sum_{j\in\mathbb Z}m_{|i-j|}(u_i-u_j)^2.
\end{equation*}
Note that $\widehat R^N_{\bf m}f$ depends on the choice of $f_1$ and $f_{-1}$ and not only on their  minimum $f$.  

\begin{remark}[regularity with respect to $z$]
\rm If $f_1$ and $f_{-1}$ are of class $C^1(\mathbb R)$ then the function 
$z\mapsto \widehat R^N_{\bf m}f(\underline{s},z)$ 
is of class $C^1(\mathbb R)$ for any fixed $\underline{s}\in \{-1,1\}^N$. 
This is a direct consequence of the Euler-Lagrange equations characterizing the minumum points of  $F^\#$.
\end{remark}

Now we add the phase constraint, minimizing over all $\underline s$ corresponding to a given volume fraction, which eventually will give an alternative chatacterization of $\widehat Q_{\bf m}f(\theta,z)$.  More precisely, fixed  
$\theta=\frac{p}{q}\in \mathbb Q\cap[0,1]$, for any $N\in q\mathbb Z$ 
we define the  function  
\begin{equation}\label{deffin}
\Phi^N_{\bf m}f(\theta,z)=\min\{\widehat R^N_{\bf m}f(\underline s, z): \ \underline s\in \mathcal S_{N}(\theta)\},
\end{equation}
where $\mathcal S_{N}(\theta)$ is the
set of admissible spin vectors  
$$\mathcal S_{N}(\theta)=\{\underline s\in \{-1,1\}^{N}: \ \#\{i: s_i=1\}=\theta N \}$$
and again we omit the dependence on $A$. 
Moreover, we define $$\Phi_{\bf m} f(\theta,z)=\liminf_{N\to+\infty}\Phi^N_{\bf m}f(\theta,z).$$ 
The following proposition states that the analysis of $\widehat Q_{\bf m}f(\theta,z)$ can be 
reduced to the periodic spin formulation 
giving $\Phi_{\bf m} f(\theta,z)$.

\begin{proposition}[periodic spin characterization of $\widehat Q_{\bf m} f(\theta,z)$]\label{equalityoffandphi} The following equality holds:
$$
\Phi_{\bf m} f(\theta,z)= \widehat Q_{\bf m} f(\theta,z)
$$
In particular the function $(\theta,z)\mapsto \Phi_{\bf m} f(\theta,z)$
is convex.
\end{proposition}

\begin{proof}
The inequality $\Phi_{\bf m} f(\theta,z)\ge \widehat Q_{\bf m} f(\theta,z)$ directly follows by definition. Conversely, given a minimum point $u$ for 
\begin{eqnarray}\nonumber
\widehat Q^{\delta, Nq}_{\bf m}f(\theta,z)=\frac{1}{Nq}\inf\Big\{F_1(u;[0,Nq]): u\in\widetilde{\mathcal A}_\delta(Nq;z)\cap\mathcal V(Nq;\theta) \Big\}, \nonumber
\end{eqnarray} we can extend it to $\mathbb Z$ so that $u_i-zi$ is $Nq$-periodic. Using this extended test function in the definition of $\Phi^{Nq}_{\bf m} f(\theta,z)$, with the same computations as in the proof of Lemma \ref{boundary-Q} we obtain
$$
\Phi^{Nq}_{\bf m} f(\theta,z)\le \widehat Q^{\delta, Nq}_{\bf m}f(\theta,z) +o(1)
$$ as $N\to+\infty$ and $\delta\to 0$. 
\end{proof}

We are interested in those $\theta$ for which the constrained relaxation $\widehat Q_{\bf m}f(\theta,z)$ is characterized by periodic minimization; that is, for which there is an interval of $z$ such that the corresponding optimal spin function $\underline s$ is periodic. Such $\underline s$ will be locally $z$-independent, and this will allow to derive regularity properties for $\widehat Q_{\bf m}f(\theta,z)$. 
For those special values of $\theta$, we think of such functions $\widehat Q_{\bf m}f(\theta,\cdot)$ as describing {\em energy meta-wells}. For brevity of notation, we directly say that the corresponding value of $\theta$ is an {\em energy well}. As we are going to show below, this concept is closely related  to that of a locking state. 
\begin{definition}[energy meta-wells]\label{energiuel} Let $f$ be as in \eqref{fmin} and let $\Phi^N_{\bf m}$ be as in \eqref{deffin}. The value $\theta\in [0,1]\cap\mathbb Q$ 
is an {\rm energy   well of $f$  at $z$} (related to the sequence ${\bf m}$) if there exists $N$ such that $N\theta\in\mathbb Z$ and
\begin{equation}\label{enwelldef}
\Phi^N_{\bf m} f(\theta,z)=\Phi_{\bf m} f(\theta,z).
\end{equation}
We say 
that $\theta$ is an {\rm energy well of $f$ in an open interval $I$} 
if there exists $N$ such that \eqref{enwelldef} holds for all $z\in I$; 
if such $I$ exists, we say that $\theta$ is a {\em non-degenerate energy well of $f$}. 
If $I=\mathbb R$, we simply say that 
$\theta$ is an {\rm energy well of $f$}. 
\end{definition}
Note that the definition {\it a priori} depends on $f_1$ and $f_{-1}$. However, the condition that $f=\min\{ f_1,f_{-1}\}$ implies that in the minimization procedure we may assume $f_1=+\infty$ outside $A$ and  $f_{-1}=+\infty$ inside $A$, which shows that the definition indeed only depends on $f$. 

\begin{remark}[energy meta-wells and periodic solutions]\label{emepe} \rm 
By Proposition \ref{equalityoffandphi} we also have that if $\theta\in [0,1]\cap\mathbb Q$ 
is an {\rm energy well} of $f$ at $z$ then $$\Phi^N_{\bf m} f(\theta,z)=\widehat Q_{\bf m} f(\theta,z).$$
This implies the existence of periodic minimizers; that is, of test function $u_i$  minimizing $\widehat Q_{\bf m} f(\theta,z)$ with $u_i-zi$ $N$-periodic.
\end{remark}

\begin{remark}\rm 
If $\theta$ is an energy well of $f$ at $z$, then there exists $N$ such that 
$$\Phi^{kN}_{\bf m} f(\theta,z)=\Phi^{N}_{\bf m}f(\theta,z)=\Phi_{\bf m} f(\theta,z)$$ 
for any $k\geq 1$. 
\end{remark}

We now examine the regularity of $\Phi_{\bf m} f$ at fixed $\theta$.

\begin{proposition}[differentiability with respect to $z$]\label{diffinz}
If $\theta$ is an energy well of $f$ in an open interval $I$, then the function
$z\mapsto \Phi_{\bf m} f(\theta,z)$ 
is differentiable at any $z\in I$.   
\end{proposition}
\begin{proof} 
Given $\theta$ an energy well in $I$ and the corresponding $N$ as in Definition \ref{energiuel}, note that $z\mapsto \Phi_{\bf m} f(\theta,z)$ is the minimum of a finite number of $C^1$ functions, corresponding to  $\underline s\in \mathcal S_{N}(\theta)$. Since $z\mapsto \Phi_{\bf m} f(\theta,z)$ is convex the derivatives of these functions must agree at the intersections.
\end{proof}

A central question in the description of $\widehat Q_{\bf m}f$ is the reduction to a set $X$ of $\theta$ such that the claim of Theorem \ref{minpro} holds taking the infimum only on $X$ and such that the computation of $\widehat Q_{\bf m}f(\theta,z)$ can be carried on for $\theta\in X$. This is the case for concentrated kernel. We will see in the examples that $\theta$ in these $X$ are often energy wells. The following proposition shows that if such an energy well is `essential'  then it is a locking state.

\begin{proposition}[energy wells and locking states]
Let $X\subset[0,1]\cap \mathbb Q$ be such that 
\begin{equation}\label{inftheta}
\Big(\inf_{\theta\in X}\{\widehat Q_{\bf m}f(\theta,z)\}\Big)^{\ast\ast}
= \widehat Q_{\bf m}f(z)
\end{equation} 
for all $z$, and let $\theta^\ast\in X$ be an energy well that is {\em essential} in \eqref{inftheta}; that is, such that
\begin{equation}\label{minimality}
\Big(\inf_{\theta\in X\setminus\{\theta^\ast\}}\{\widehat Q_{\bf m}f(\theta,z)\}\Big)^{\ast\ast}
> \widehat Q_{\bf m}f(z)
\end{equation} 
 for some $z$. Then, $\theta^\ast$ is a locking state. 
\end{proposition}

\begin{proof} 
We recall that $\widehat Q_{\bf m}f(\theta,z)=\Phi_{\bf m} f(\theta,z)$ by Proposition \ref{equalityoffandphi}. Since $\theta^\ast$ is an energy well, by Proposition \ref{diffinz}, the function $z\mapsto\Phi_{\bf m}f(\theta^\ast,z)$ is differentiable. By the essentiality condition \eqref{minimality}, that function cannot be tangent to $\widehat Q_{\bf m}f$ in an isolated point, nor can be transversal to it. Hence, it must coincide with  $\widehat Q_{\bf m}f$ in an interval.\end{proof}

As for regularity properties of $Q_{\bf m}f$ with respect to $\theta$, we note that in general locking states are points where the characterization of the energy changes.  
This suggests that we may have a jump in the derivative at these points.

\begin{conjecture}[Non differentiability at the energy wells]
If $X\subset[0,1]\cap \mathbb Q$ is such that \eqref{inftheta} holds for all $z$ and $\theta^\ast\in X$ is an energy well satisfying \eqref{minimality}, 
hen the function
$\theta\mapsto Q_{\bf m}f(\theta,z) $
is not differentiable in $z$ at $\theta^\ast$. 
\end{conjecture}

This conjecture is reminiscent of regularity properties in dynamical systems, where the global structure of minimizers can be used in the proofs, as in the work of J.~Mather \cite{Mather}.  
Anyway, we will prove that it holds in the case studies (see Remark \ref{nondiff} for the $M$-th neighbour case, and Remark \ref{nondiff-exp} for the truncated convex potential and exponential kernel).

\begin{remark}[Generalized Cauchy-Born (GCB) states]\rm
The spin representation of a  microstructure allows one to effectively parametrize   periodic minimizers.  Such a representation can be  expected to exist for locking states which can be viewed as examples of   `global' solutions. We can also interpret  such states as respecting the  generalized Cauchy-Born (GCB) rule. To  make the notion of the GCB rule more general we may  refer to the possibility of computing the macroscopic energy by solving an appropriate boundary value problem on a finite   representative `cell'. 
The question arises in which cases  any minimizer   can be viewed as a GCB state in the above sense or as a simple mixture (a convex combination) of such  states. We will see in the next sections that for broad classes  of  physically interesting  non-convex energies $f$ and the penalization kernels $\bf m$ only GCB states are relevant.
\end{remark}

\section{Relaxation with concentrated-kernel penalization}
\label{mneigh-sec} 
In this section, we analyze the relaxation of a general bi-convex function $f$ 
with a concentrated kernel $\bf m$. 
We recall that in this case  there exists $M\neq 2$ such that 
$m_n=0$ for all $n\geq 2$ except for $n=M$ and that such penalization leads to a  non-additive problem (see Definition \ref{def-concentrated}). 
We show that the optimal microstructures in this case are   restricted to periodic states, corresponding to a fraction $\theta_n= {n\over M}$ for $n\in\{0,\dots, M\}$, and  compatible mixtures of such periodic states corresponding to neighbouring values of the phase fractions $\theta_n$ and $\theta_{n+1}$, in other words,  to first and second order laminates. 

\smallskip 
Following the notation of Section \ref{constraint-sec}, let $z^\ast\in \mathbb R$, $A=[z^\ast,+\infty)$, and let 
$f\colon\mathbb R\to\mathbb R$ be such that 
the restrictions of $f$ to $(-\infty,z^\ast]$ and $[z^\ast, +\infty)$ are convex. 
In this section, we again use the notation 
\begin{equation}\label{def-psi-NN}
f_{2m_1}(z)=f(z)+2m_1z^2
\end{equation}
for the overall nearest-neighbour interactions.

We assume that growth hypothesis 
\eqref{crescitasotto} holds, so that 
$f_{2m_1}(z)\to+\infty$ as $z\to\pm\infty$. 
Note that the analysis can also cover the degenerate case when this condition is not satisfied. 
As a model, in Remark \ref{casem10} we will consider the case of a truncated quadratic potential $f$ with $m_1=0$,  highlighting the effect of degeneracy.

\subsection{Formulas for the relaxation}   
In the case of a bi-convex $f$, formula 
of Proposition \ref{MniPM} describing $\widehat Q_{\bf m}f$ can be further specified as follows 
\begin{equation}\label{QMn}
\widehat Q_{\bf m}f(z)=\big(\min_n P^{M,n}(z)\big)^{\ast\ast},
\end{equation}  
where for any $n\in\{0,\dots, M\}$ we let $\theta_n=\frac{n}{M}$ and introduce
\begin{equation}\label{PMn}
\left. \begin{array}{ll}
P^{M,n}(z)=&\displaystyle\min\Big\{(1-\theta_n)f_{2m_1}(z^-)+\theta_n f_{2m_1}(z^+): z^-\leq z^\ast, z^+\geq z^\ast, \\
&\hspace{2cm}\displaystyle(1-\theta_n)z^-+\theta_nz^+=z\Big\}
+2m_M(Mz)^2. 
\end{array}
\right.
\end{equation} 
Now we prove that for any rational $\theta$ the constrained function $\widehat Q_{\bf m}f(\theta,z)$, defined in \eqref{defqtheta}, can be also characterized in terms of the functions $P^{M,n}$, which themselves correspond 
to particular values of $\theta$, 
in the sense that $P^{M,n}(z)=\widehat Q_{\bf m}f(\theta_n,z)$.

\begin{theorem}
[shape of $\widehat Q_{\bf m}f$ and of the phase function $\theta$]\label{structureqm} There exists an ordered family of disjoint intervals 
$(s_n^-,s_n^+)$, where $s_0^-=-\infty$ and $s_M^+=+\infty$, such that  

\smallskip 
{\rm (i)}\  
$\widehat Q_{\bf m}f(z)=P^{M,n}(z)$ in $(s_n^-,s_n^+)$ and it is affine in each of the remaining intervals; that is, between $s_n^+$ and $s_{n+1}^-$ for each $n$;

\smallskip 
{\rm (ii)}\ $\theta(z)=\theta_n$ in $(s_n^-,s_n^+)$ and it is affine in each of the remaining intervals. 

\smallskip 
{\rm (iii)}\ the set of the locking states of $f$ is $\{\theta_n\}$ and 
$$
\widehat Q_{\bf m}f(z)= \big(\min\{\widehat Q_{\bf m}f(\theta,z): \theta \hbox{ is a locking state}\} \big)^{\ast\ast}\,.
$$
\end{theorem} 
\begin{proof}
The proof of {\rm (i)} and {\rm (ii)} will follow from Lemma \ref{formula-n-prop} below, while {\rm (iii)} is obtained by \eqref{QMn}. 
\end{proof}
\begin{remark}\label{rem-shapeM}\rm  Note that if $\theta(z)=\theta_n$ then the value of $\widehat Q_{\bf m}f(z)$ is attained on periodic minimizers. 
The phase function $\theta$ 
can be explicitly written as 
\begin{equation}\label{tetaM}
\theta(z)=\begin{cases}\vspace{2mm}
0 & \hbox{\rm if } \ z\leq s_0^+\\
\displaystyle\theta_n+\frac{1}{M}\frac{z-s_n^+}{s_{n+1}^--s_n^+} &  \hbox{\rm if } \ s_n^+ \leq z \leq s_{n+1}^-\\
\displaystyle\theta_n & \hbox{\rm if } \ s_n^- \leq z \leq s_{n}^+\\
 1 & \hbox{\rm if } \ s_M^- \leq z.  
\end{cases}
\end{equation} 
Moreover, if we write the convex envelope of the minimum of $P^{M,n}$ and $P^{M,n+1}$  
as 
\begin{equation}\label{minpnm} 
\min\big\{P^{M,n},P^{M,n+1}\big\}^{\ast\ast}(z)=\begin{cases}
P^{M,n}(z) &  \hbox{\rm if } \ z \leq s_{n}^+\\
r^{M,n}(z) &  \hbox{\rm if } \ s_n^+ \leq z \leq s_{n+1}^-\\
P^{M,n+1}(z) &  \hbox{\rm if } \  s_{n+1}^-\leq z, 
\end{cases}
\end{equation} 

where $r^{M,n}$ is the interpolating affine function 
\begin{eqnarray*}
r^{M,n}(z)&=&P^{M,n}(s_n^+)+\frac{P^{M,n+1}(s_{n+1}^-)-P^{M,n}(s_n^+)}{s_{n+1}^--s_n^+}(z-s_n^+),
\end{eqnarray*}   
then $\widehat Q_{\bf m}f(z)=r^{M,n}(z)$ if $z\in[s_n^+,s_{n+1}^-]$.   
Note that this characterization of $\widehat Q_{\bf m}f$ holds under assumption \eqref{crescitasotto}, while it may fail if this condition is dropped, as we show in Remark \ref{casem10} below. 
\end{remark}

The main technical point of this section is Lemma \ref{formula-n-prop} giving an explicit formula for the constrained minimizations involving only pairs of successive locking states.  
The proof of this fact relies on the following algebraic lemma. 

\begin{lemma}[an algebraic lemma]\label{alphaklemma} 
Let $n\in[0,M-1]\cap \mathbb N$.  If $\theta\in \big[\frac{n}{M},\frac{n+1}{M}\big]$, then there exist coefficients  $\alpha^n_k$, $k=0,\dots, M$, such that $\alpha^n_k\in [0,I_k]$ for any $k$ 
and 
\begin{equation}\label{alphak}
\begin{cases}
\displaystyle\sum_{k=0}^M\alpha^n_k=\frac{M\theta-n}{M}Nq\\
\displaystyle\sum_{k=0}^Mk\alpha^n_k=(n+1)\frac{M\theta-n}{M}Nq. 
\end{cases}
\end{equation} 
\end{lemma}
\begin{proof}
The linear system \eqref{alphak} has infinitely many solutions depending on $M-1$ parameters. 
We have to show that there exists one solution in $\Pi_{k=0}^M[0,I_k]$. 
To this end, it is sufficient to show that the hyperplane given by the equation  
$$H_\lambda(\alpha_0^n, \dots, \alpha_M^n)=
\sum_{k=0}^M (\lambda+k)\alpha^n_k-(\lambda+n+1)\frac{M\theta-n}{M}Nq=0$$
intersects $\Pi_{k=0}^M[0,I_k]$ for any $\lambda\in\mathbb R$, which happens if for any $\lambda\in \mathbb R$ there exist 
two points $v,w\in \mathbb R^{M+1}$ such that 
$H_\lambda(v) \ H_\lambda(w)\leq 0.$ 
Since $n\leq M\theta\leq n+1$ and 
\begin{equation*}
H_\lambda(0,\dots, 0)=-(\lambda+n+1)\frac{M\theta-n}{M}Nq, \quad 
H_\lambda(I_0^n, \dots, I_M^n)=
(\lambda+n)\frac{(n+1-M\theta)}{M}Nq, 
\end{equation*}
we get 
$H_\lambda(0,\dots, 0)\ H_\lambda(I_0^n, \dots, I_M^n) \leq 0$  
if  $\lambda\leq -(n+1) \ \hbox{ or } \lambda \geq -n$. 
 
For the remaining cases, we note that by \eqref{teta} 
$$(n+1)\sum_{k=0}^M(M-k)I_k-(M-(n+1))\sum_{k=0}^MkI_k=\big((n+1)-M\theta\big) Nq.$$
Since $(n+1)(M-k)- (M-(n+1))k=M(n+1-k)\leq 0$ if $k\geq n+1$, 
we get  
$$\sum_{k=0}^n (n+1-k)I_k\geq \frac{(n+1)-M\theta}{M} Nq.$$
If we choose 
$v=(v_0,\dots, v_k)$ with $v_k=0$ if $k\leq n$ and $v_k=I_k$ if $k>n$
we obtain 
\begin{eqnarray*}
H_{-(n+1)}(v)&=&-(n+1)\Big(\sum_{k=0}^MI_k-\sum_{k=0}^nI_k\Big)+\sum_{k=0}^M kI_k-\sum_{k=0}^n kI_k \\
&=&\frac{M\theta-(n+1)}{M}Nq+\sum_{k=0}^n(n+1-k)I_k\geq 0.
\end{eqnarray*}
Noting that 
$$H_{-n}(v)=\frac{M\theta-(n+1)}{M}Nq+\sum_{k=0}^n(n-k)I_k-\frac{M\theta-(n+1)}{M}Nq\geq 0,$$
it follows that $H_{\lambda}(v)\geq 0$ for any $\lambda\in (-(n+1),-n)$. Since $H_{\lambda}(0, \dots, 0)\leq 0$, this concludes the proof of Lemma \ref{alphaklemma}.  
\end{proof}

Now, we state the interpolation lemma.  
\begin{lemma}[interpolation between locking states]\label{formula-n-prop}
Let $\theta\in[\theta_n,\theta_{n+1}]\cap \mathbb Q$, 
with $n$ integer such that $0\leq n<M$, and $\theta_n=\frac{n}{M}$ as above.  
Then the following formula holds: 
\begin{eqnarray}\label{formula-n}\nonumber
&&\widehat Q_{\bf m}f(\theta,z)=\displaystyle\min\Big\{M(\theta_{n+1}-\theta) P^{M,n}(w_n)
+M(\theta-\theta_n) P^{M,n+1}(w_{n+1}): \\
&& \hspace{3cm}
M(\theta_{n+1}-\theta)w_n+M(\theta-\theta_n)w_{n+1}
= z\Big\}. 
\end{eqnarray}
\end{lemma} 
\noindent We mention that in view of growth condition \eqref{crescitasotto} the minimum in \eqref{formula-n} is achieved.

\begin{proof} Up to scaling, we suppose $m_M=1$ for notational convenience. Since Lemma \ref{boundary-Q} holds, for $u\in\mathcal A(Nq;z)$, if $F_1$ is the non-scaled functional given by \eqref{defF1}, we can estimate $F_1(u; [0,Nq])$ as 
\begin{eqnarray*}
F_1(u; [0,Nq])&=& \sum_{i=1}^{Nq} f(z_i)+2m_1\sum_{i=1}^{Nq}(z_i)^2+\sum_{i,j=0,\ |i-j|=M}^{Nq}(u_i-u_j)^2
\\
&\geq& \sum_{j=1}^{Nq}f_{2m_1}(z_j)+\frac{M}{Nq}\sum_{i\in M\mathbb Z\cap[M,Nq]}\Big(\sum_{j=i-M+1}^i z_j\Big)^2 +o(1)_{N\to+\infty}, 
\end{eqnarray*} 
where  and $z_i=u_i-u_{i-1}$. 

It is not restrictive to assume $Nq\in M\mathbb N$.  
For any $i\geq M$ we define 
\begin{eqnarray*}
J^{+}(i)&=&\{j\in\{i-M+1,\dots, i-1,i\}: z_j \geq z^\ast \}\\
J^{-}(i)&=&\{j\in\{i-M+1,\dots, i-1,i\}: z_j < z^\ast \}.  
\end{eqnarray*}
Moreover, for any $k=0,\dots, M$ we set
$$\mathcal I_k=\{i\in M\mathbb Z\cap[M,Nq]: \# J^+(i)=k\},$$ 
and we denote the cardinality of $\mathcal I_k$ by $I_k$.  
Note that 
\begin{equation}\label{teta}
M \sum_{k=0}^MI_k=Nq,  \ \ \sum_{k=0}^M(M-k)I_k=(1-\theta)Nq \ \ \hbox{ and } \ \ 
\sum_{k=0}^MkI_k=\theta Nq.
\end{equation}
Let $\psi_{-1}$ and $\psi_1$ 
denote the restrictions of $f_{2m_1}$ 
to $(-\infty,z^\ast)$ and $
[z^\ast, +\infty)$ respectively. Then, 
by separating the contributions in each $\mathcal I_k$, thanks to the convexity of $\psi_{-1}$, $\psi_1$ and of the square we have 
\begin{eqnarray}\label{disikk}\nonumber
&&\hspace{-5mm}\sum_{j=1}^{Nq}f_{2m_1}(z_j)+\frac{M}{Nq}\sum_{i\in M\mathbb Z\cap[M,Nq]}\Big(\sum_{j=i-M+1}^i z_j\Big)^2\\
&&=\sum_{k=0}^{M} \nonumber
\sum_{i\in \mathcal I_k}\Big(\sum_{j\in J^-(i)} \psi_{-1}(z_j)+\sum_{j\in J^{+}(i)} \psi_1(z_j) \Big) 
+ M \sum_{k=0}^{M}
\sum_{i\in \mathcal I_k}\Big( \sum_{j\in J^-(i)} z_j+
\sum_{j\in J^{+}(i)} z_j \Big)^2\Big)
\\
&&\geq
\sum_{k=0}^{M} 
I_k\Big((M-k)( \psi_{-1}(w_k^-)+k \psi_1(w_k^+)+ M \big((M-k)w_{k}^-+k w_k^+\big)^2\Big)
\end{eqnarray}
where $w_M^-=w_0^+=0$ and 
\begin{eqnarray*}
&&w_k^-=\frac{1}{(M-k) I_k}\sum_{i\in \mathcal I_k}\sum_{j\in J^-(i)}z_j, \ \ \ \ 
w_k^+=\frac{1}{k I_k}\sum_{i\in \mathcal I_k}\sum_{j\in J^+(i)}z_j
\end{eqnarray*} 
otherwise. 

We now may conclude the proof of the lower bound by applying Lemma \ref{alphaklemma} to \eqref{disikk}, regrouping the terms therein so as to compare that expression with $P^{M,n}$.
Noting that 
$$\sum_{k=0}^M(M-k)\alpha_k^n=\frac{(M\theta-n)(M-(n+1))}{M}Nq,$$ 
we get by convexity that  
\begin{eqnarray*}
&&\hspace{-15mm}\sum_{k=0}^{M} 
\alpha_k^n\Big((M-k)( \psi_{-1}(w_k^-)+k \psi_1(w_k^+)+ M \big((M-k)w_{k}^-+k w_k^+\big)^2\Big)\\
&\geq& \Big(\sum_{k=0}^{M} (M-k)\alpha_k^n\Big) \psi_{-1} (z_{n+1}^-)+
\Big(\sum_{k=0}^{M}k\alpha_k^n\Big) \psi_1 (z_{n+1}^+)\\
&&+M\Big(\sum_{k=0}^{M}\alpha_k^n\Big) 
\Big(\frac{\big(\sum_{k=0}^{M} (M-k)\alpha_k^n\big) z_{n+1}^-
+\big(\sum_{k=0}^{M} k\alpha_k^n\big) z_{n+1}^+}{\sum_{k=0}^{M}\alpha_k^n}\Big)^2\\
&\geq& \frac{(M\theta-n)}{M}Nq 
\Big((M-(n+1)) \psi_{-1} (z_{n+1}^-) + 
(n+1) \psi_1 (z_{n+1}^+)\\
&&+M \big((M-(n+1)) z_{n+1}^- + (n+1)z_{n+1}^+\big)^2\Big),
\end{eqnarray*}
where 
$$z_{n+1}^-=\frac{\sum_{k=0}^{M}(M-k) \alpha_k^n w_k^-}{\sum_{k=0}^{M} (M-k)\alpha_k^n}, 
\quad z_{n+1}^+=\frac{\sum_{k=0}^{M}k \alpha_k^n w_k^+}{\sum_{k=0}^{M} k\alpha_k^n}.$$
Hence, 
\begin{eqnarray*}
&&\hspace{-15mm}\sum_{k=0}^{M} 
\alpha_k^n\Big((M-k)( \psi_{-1}(w_k^-)+k \psi_1(w_k^+)+ M \big((M-k)w_{k}^-+k w_k^+\big)^2\Big)\\
&\geq& (M\theta-n)Nq P^{M,n+1}\Big(\Big(1-\frac{n+1}{M}\Big)z_{n+1}^-+\frac{n+1}{M}z_{n+1}^+\Big). 
\end{eqnarray*}
Correspondingly we obtain 
\begin{eqnarray*}
&&\hspace{-15mm}\sum_{k=0}^{M} 
(I_k-\alpha_k^n)\Big((M-k)( \psi_{-1}(w_k^-)+k \psi_1(w_k^+)+ M \big((M-k)w_{k}^-+k w_k^+\big)^2\Big)\\
&\geq& 
(n+1-M\theta)Nq P^{M,n}\Big(\Big(1-\frac{n}{M}\Big)z_n^-+\frac{n}{M} z_n^+\Big), 
\end{eqnarray*} 
where $$z_{n}^-=\frac{\sum_{k=0}^{M}(M-k) (I_k-\alpha_k^n) w_k^-}{\sum_{k=0}^{M} (M-k)(I_k-\alpha_k^n)}, 
\quad z_{n}^+=\frac{\sum_{k=0}^{M}k(I_k-\alpha_k^n) w_k^+}{\sum_{k=0}^{M} k(I_k-\alpha_k^n)}.$$
Noting that 
\begin{eqnarray*}
&&(n+1-M\theta)\Big((M-n)z_n^-+n z_n^+\Big)\\
&&\hspace{2cm}+(M\theta-n)\Big((M-(n+1))z_{n+1}^-+(n+1)z_{n+1}^+\Big)=
Mz, 
\end{eqnarray*}
for $\theta\in\big[\frac{n}{M},\frac{n+1}{M}\big]$ we then have, up to a negligible term,  
\begin{eqnarray*}
F_1(u; [0,Nq])&\geq&
\min\Big\{(n+1-M\theta) P^{M,n}(w_n)+(M\theta-n) P^{M,n+1}(w_{n+1}): \\ 
&&\hspace{1cm} (n+1-M\theta)w_n+(M\theta-n)w_{n+1} 
= z\Big\}
\end{eqnarray*}
which concludes the proof of the lower bound in \eqref{formula-n}. 

\bigskip 
 \begin{figure}[h!]
\centerline{\includegraphics[width=0.7\textwidth]{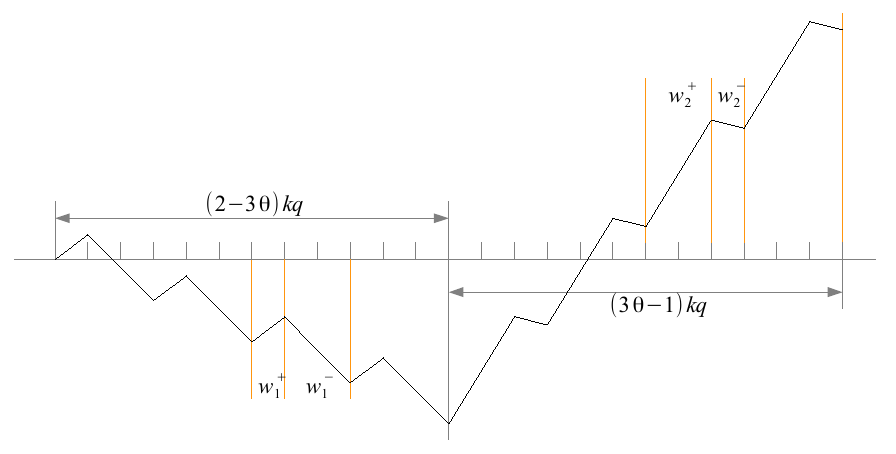}}
\caption{construction of the upper bound for $M=3$ and $n=1$.}
\label{figure_upperboundMneigh}
\end{figure} 
As for the upper bound, let $\theta=\frac{p}{q}\in[\frac{n}{M},\frac{n+1}{M}]$,  
$z\in\mathbb R$ be fixed and $(w_n,w_{n+1})$ be a minimizer of \eqref{formula-n}.
For all $k\geq 1$ we define a test function $u\colon[0,kMq]\cap\mathbb Z\to\mathbb R$ constructed as follows. 
Let $w_n^\pm$ be a minimizer of the problem defining $P^{M,n}(w_n)$ in \eqref{PMn}, and 
let $w_{n+1}^\pm$ be a minimizer of the corresponding problem defining $P^{M,n+1}(w_{n+1})$. 
We set $u_0=0$, and 
\begin{eqnarray*}
u_i-u_{i-1}&=&\begin{cases}
w_n^+ & \hbox{ if } i\in\{1,\dots, n\} \hbox{ mod } M \\
w_n^- & \hbox{ if } i\in\{n+1,\dots, M\} \hbox{ mod } M
\end{cases} \ \ \ \ \ \hbox{for }\ i\leq kMq (\theta_{n+1}-\theta)\\ 
u_i-u_{i-1}&=&\begin{cases}
w_{n+1}^+ & \hbox{ if } i\in\{1,\dots, n+1\} \hbox{ mod } M\\
w_{n+1}^- & \hbox{ if } i\in\{n+1,\dots, M\} \hbox{ mod } M
\end{cases} \ \ \, \hbox{for } i>kMq (\theta_{n+1}-\theta) 
\end{eqnarray*}
(see Fig.~\ref{figure_upperboundMneigh}). Note that 
$u(kMq)=kMqz$ and $u\in\mathcal V(kMq;\theta)$, so that $u$ is an admissible test function for the computation of $\widehat Q_{\bf m}f(\theta,z)$, and the upper bound follows. 
\end{proof}

\begin{remark}[Non-differentiability at locking states]\label{nondiff}\rm 
From formula \eqref{formula-n} we deduce that for all $z$ the function 
$\theta\mapsto Q_{\bf m}f(\theta,z)$ 
is differentiable at any $\theta\not\in\{\theta_1,\dots,\theta_{M-1}\}$, whereas instead 
$$\frac{\partial(Q_{\bf m}f)}{\partial\theta}(\theta_n^+,z)\neq \frac{\partial(Q_{\bf m}f)}{\partial\theta}(\theta_n^-,z)$$ except possibly for some critical values of $z$. Indeed, in the computation of the left-hand side derivative of $Q_{\bf m}f$ at $\theta=\theta_n$ we use $P^{M,n-1}$ while for the right-hand side we use $P^{M,n+1}$, whose values are generically different at the minimum points of \eqref{formula-n}.   
\end{remark}

\subsection{Computation of $Q_{\bf m}f$ for prototypical non-convex energies} 
We now apply Theorem \ref{structureqm} to some prototypical $f$; namely, truncated quadratic potential and double-well potential. 

\subsubsection{Truncated quadratic potential}\label{fracture1M} 
We consider a special case of the truncated convex potentials introduced in Example \ref{trconintro} with $\tilde f(z)=z^2$ and $z^\ast=1$; that is, let $f\colon\mathbb R\to\mathbb R$ be defined by 
\begin{equation}\label{frattura}
f(z)=\begin{cases} 
 z^2 & \hbox{ if } z\leq 1\\
1 & \hbox{ if } z > 1, 
 \end{cases}
 \end{equation}
and let $A=[1,+\infty)$. Note the growth assumption \eqref{crescitasotto} implies that $m_1>0$.

In this case, we have 
\begin{equation}\label{qfM-1} 
Q_{\bf m}f(z)=\begin{cases}
\displaystyle z^2& \hbox{\rm if } \ z\leq s^+_0\\
\displaystyle r^{M,n}(z)-2(m_1+m_M M^2)z^2 &  \hbox{\rm if } \ s_n^+ \leq z \leq s_{n+1}^-
\\
\displaystyle\frac{2m_1(1-\theta_n)}{2m_1+\theta_n}z^2+\theta_n& \hbox{\rm if } \ 
s_n^- \leq z \leq s_{n}^+
\\
\displaystyle 1 & \hbox{\rm if } \ s_M^- \leq z,   
\end{cases}
\end{equation} 
where the points $s_n^+$ and $s_n^-$ in Theorem \ref{structureqm} 
are  
\begin{equation}\label{tnsn}
s_n^{\pm}=s_n^{\pm}(m_1,m_M)
=\frac{2m_1+\theta_n}{\sqrt{2m_1(2m_1+1)}}\ \sqrt{\frac{m_1(2m_1+1)+m_M M^2(2m_1+\theta_{n})\pm m_M M}{m_1(2m_1+1)+m_M M^2(2m_1+\theta_{n})}} 
\end{equation}
and 
$r^{M,n}$ is the affine interpolating function in Remark \ref{rem-shapeM}. 
The formula for $Q_{\bf m}f$ is obtained by explicitly computing the functions $P^{M,n}(z)$ (see Appendix B). 

\begin{figure}[h!]
\centerline{\includegraphics[width=1\textwidth]{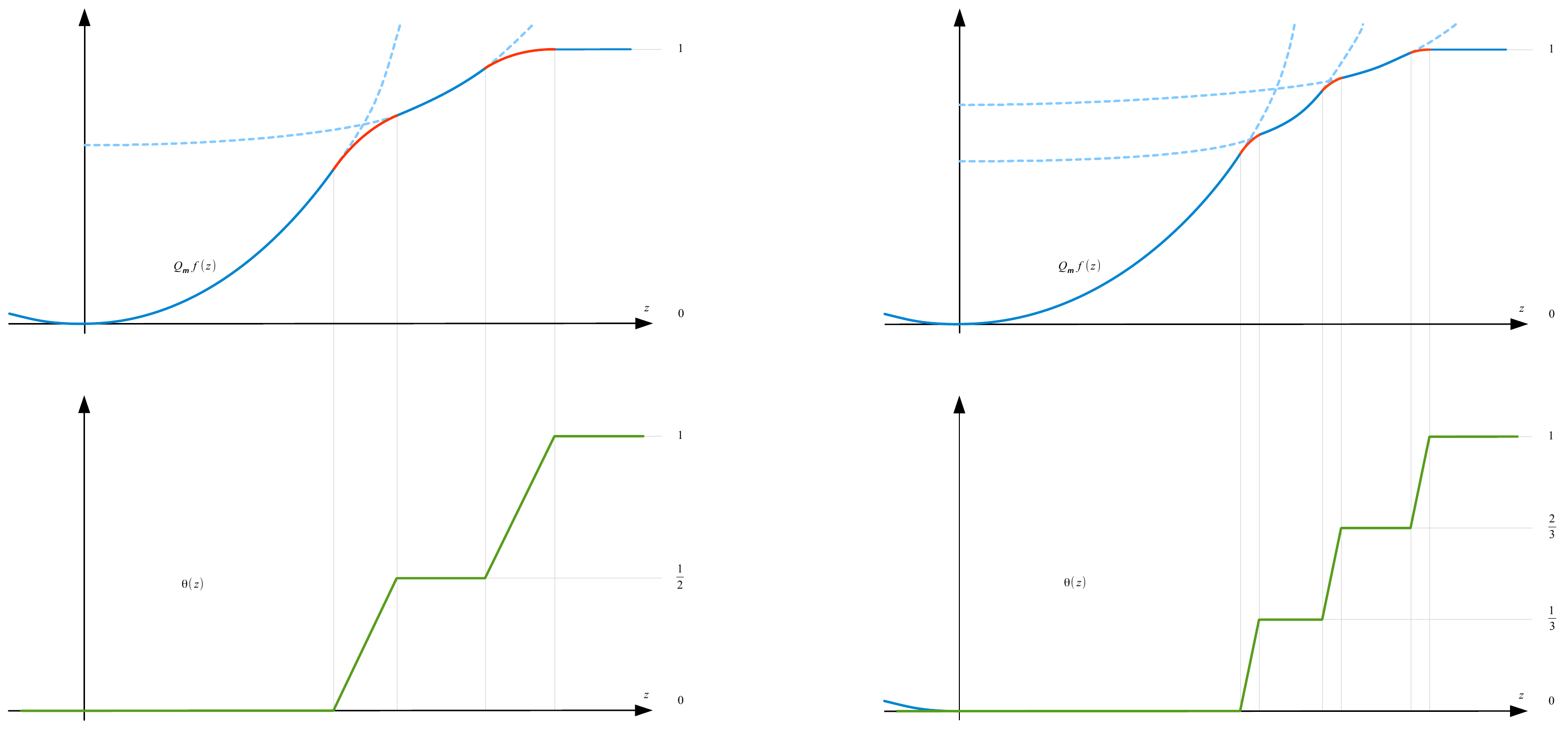}}
\caption{$Q_{\bf m}f(z)$ and $\theta(z)$ in the cases $M=2$ (a) and 
$M=3$ (b).}
\label{Q23frattura}
\end{figure} 

In Figure \ref{Q23frattura} (a)--(b), we show the structure of the functions  $Q_{\bf m}f(z)$ and $\theta(z)$ in the cases $M=2$ and $M=3$, respectively. Note that in the first case $\theta_1=\frac{1}{2}$ corresponds to periodic minimizers of period $2$ and in the second case  $\theta_1=\frac{1}{3}$ and $\theta_2=\frac{2}{3}$ correspond to the two possible periodic minimizers of period $3$. In the affine regions, we have mixtures of two periodic solutions, corresponding to neighbouring locking states. 

\begin{remark}[Degenerate case with $m_1$=0]\label{casem10}\rm 
The computation of $Q_{\bf m}f$ for the truncated quadratic potential $f$ can be performed also in the degenerate case where the growth hypothesis \eqref{crescitasotto} does not hold; that is, supposing $m_1=0$. Note that in this case there is no coercivity on the nearest-neighbour interactions. 

\begin{figure}[h!]
\centerline{\includegraphics[width=.5\textwidth]{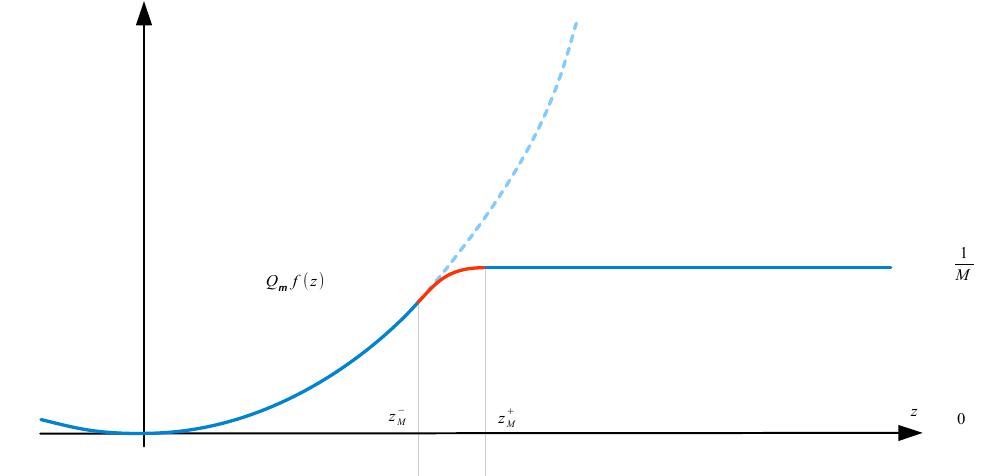}
\includegraphics[width=.5\textwidth]{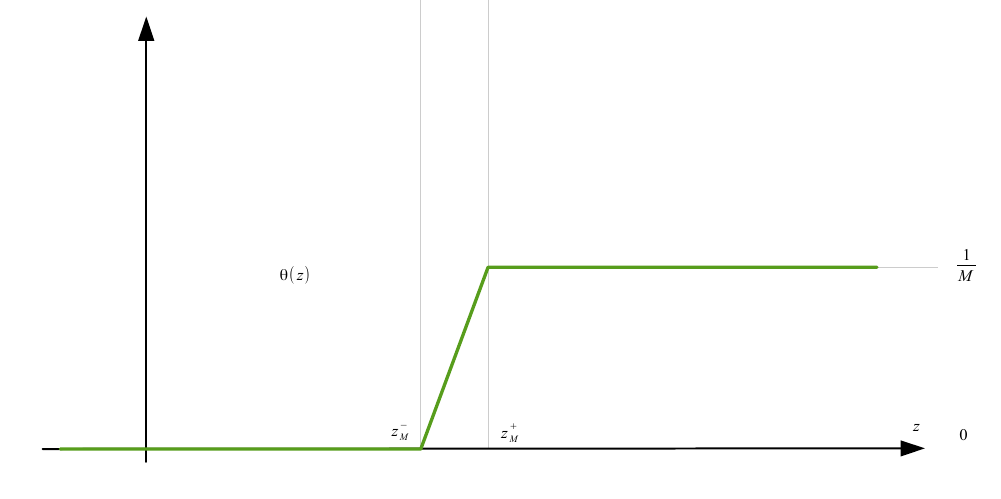}}
\caption{$Q_{\bf m}f$ and $\theta$ in a degenerate case.}
\label{casodeg}
\end{figure} 

The construction in Theorem \ref{structureqm} becomes degenerate, and 
we obtain the formula 
\begin{equation}\label{qfMm10} 
Q_{\bf m}f(z)=\begin{cases}
\displaystyle z^2&\displaystyle \hbox{\rm if } \ z\leq z_M^-
\\
\displaystyle 2\sqrt{2m_MM(1+2m_MM^2)}z-2m_MM -2m_MM^2 z^2 &\displaystyle  \hbox{\rm if } \ 
z_M^-\leq z \leq 
z_M^+
\\
\displaystyle\frac{1}{M}&\displaystyle \hbox{\rm if } \ 
z_M^+\leq z, 
\end{cases}
\end{equation} 
where 
$$z_M^-=\sqrt{\frac{2m_MM}{1+2m_MM^2}} \ \ \ \hbox{\rm and }\ \  z_M^+=\sqrt{\frac{1+2m_MM^2}{2m_MM^3}}.$$
The corresponding phase function is then given by $\theta(z)=0$ if $z\leq z_M^-$, $\theta(z)=\frac{1}{M}$ if $z\geq 
z_M^+$ and affine otherwise, so that the locking states are $\theta=0$ and $\theta=\frac{1}{M}$. 
Hence, $Q_{\bf m}f(z)$ is obtained as the convex envelope of the minimum of $P^{M,1}(z)$ and $P^{M,0}(z)$ only.

As for the description of $\theta$ as in \eqref{tetaM}, note that 
$$
\lim_{m_1\to 0}s_0^+(m_1,m_M)=z_M^-
,\ \ \  \lim_{m_1\to 0}s_1^-(m_1,m_M)=z_M^+,
$$
while we have that as $m_1\to 0$ then $s_n^+(m_1,m_M)\to +\infty$ for any $n\geq 1$ and $s_n^-(m_1,m_M)\to +\infty$ for any $n\geq 2$. 
This corresponds to the fact that the sets of $z$ where $\theta(z)>1/M$ tend to $+\infty$ as $m_1\to0$.
In Fig.~\ref{casodeg} we picture $Q_{\bf m}f$ and $\theta$.
\end{remark}

\begin{remark}[Asymptotic analysis as $M\to+\infty$]\label{mfracturerem}\rm 
In this remark we highlight the dependence of $\theta=\theta^M$ and $Q_{\bf m}f=Q^Mf$ on $M$.  
We show that the limit of the functions $\theta^M$ as $M\to+\infty$ 
is the phase function of $f$ when the only not vanishing coefficient is $m_1$, and correspondingly for $Q^Mf(z)$.  
\begin{figure}[h!]
\centerline{\includegraphics[width=1\textwidth]{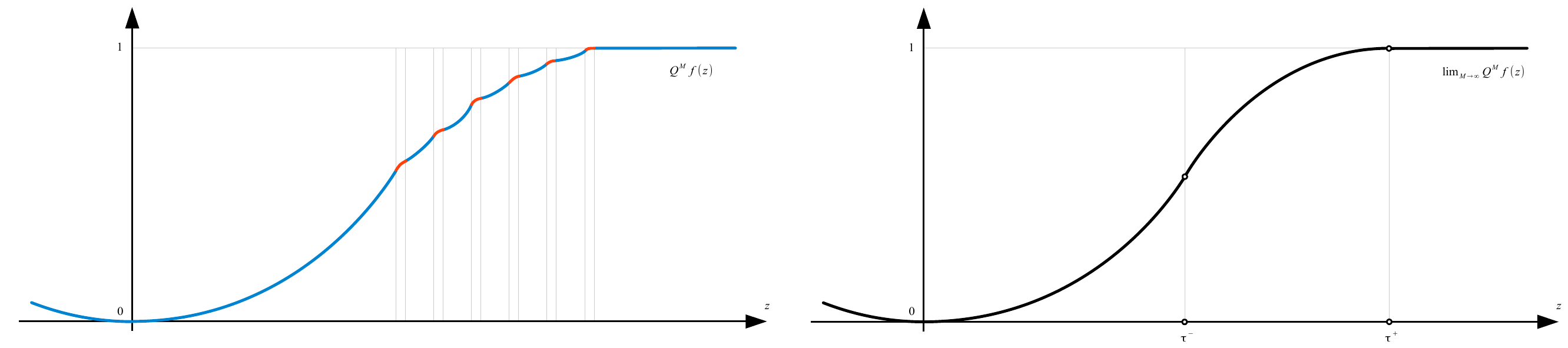}}
\caption{Graph of $Q^Mf(z)$ (for $M=6$) and of the limit function.}
\label{QMfrattura}
\end{figure} 

Indeed, the following estimates hold 
\begin{eqnarray*} 
&&s_n^+\leq \frac{n+2 m_1M+\frac{1}{2}}{\sqrt{2 m_1(1+2 m_1)} M}=:\widetilde s_n^+, \ \ \ 
s_n^-\geq \frac{n+2m_1M-1}{\sqrt{2m_1(1+2m_1)} M}=:\widetilde s_n^-, 
\end{eqnarray*}
so that we can define two piecewise-constant functions by setting
\begin{equation*}
\overline \theta^{M}(z)=
\begin{cases}
0 & \hbox{\rm if }\ z\leq  \widetilde s_0^+\\
\theta^{M}(s_{n-1}^+) & \hbox{\rm if }\ z\in (\widetilde s_{n-1}^+,\widetilde s_n^+]
\\
1 & \hbox{\rm if }\ \widetilde s_{M}^+< z 
\end{cases}\ \ \ \hbox{\rm and }\ \ \ 
\underline \theta^{M}(z)=\begin{cases}
0 & \hbox{\rm if }\ z\leq  \widetilde s_0^-\\
\theta^{M}(s_{n}^-) & \hbox{\rm if }\ z\in (\widetilde s_{n-1}^-,\widetilde s_n^-]
\\
1 & \hbox{\rm if }\ \widetilde s_{M}^-< z, 
\end{cases}
\end{equation*}
obtaining that $\underline \theta^{M}(z)\leq \theta^M(z)\leq\overline \theta^{M}(z)$. 
The claim follows noting that  
$$
\lim_{M\to+\infty}\overline \theta^{M}(z)=
\begin{cases} 
0 & \hbox{\rm if } \ z\leq \sqrt{\frac{2m_1}{1+2m_1}}
\\
 2m_1\Big(z
\sqrt{\frac{1+2m_1}{2m_1}}-1\Big)& \hbox{\rm if } \ 
\sqrt{\frac{2m_1}{1+2m_1}}
\leq z \leq \sqrt{\frac{1+2m_1}{2m_1}}
\\
1 &  \hbox{\rm if } \ 
\sqrt{\frac{1+2m_1}{2m_1}}
\leq z, 
\end{cases}$$ 
and the same for $\underline \theta^{M}(z)$. 
Correspondingly 
$$\lim_{M\to+\infty}Q^Mf(z)=\begin{cases} 
 z^2 & \hbox{\rm if } \ z\leq \sqrt{\frac{2m_1}{1+2m_1}}\\
-2m_1\Big(z^2-2z\sqrt{\frac{1+2m_1}{2m_1}}+1\Big)&\hbox{\rm if } \ 
\sqrt{\frac{2m_1}{1+2m_1}}\leq z \leq \sqrt{\frac{1+2m_1}{2m_1}}\\
 1 &\hbox{\rm if } \ \sqrt{\frac{1+2m_1}{2m_1}}\leq z  
\end{cases}$$
(see Figure \ref{QMfrattura}). In particular, we note that 
$$\lim_{M\to+\infty}Q^Mf(z)=(f_{2m_1})^{\ast\ast}(z)-2m_1z^2=Q_{\bf m^\prime}f(z),$$
where ${\bf m^\prime}=\{m_1, 0,\dots\}$. 
\end{remark}

\begin{figure}[h!]
\centerline{\includegraphics[width=0.6\textwidth]{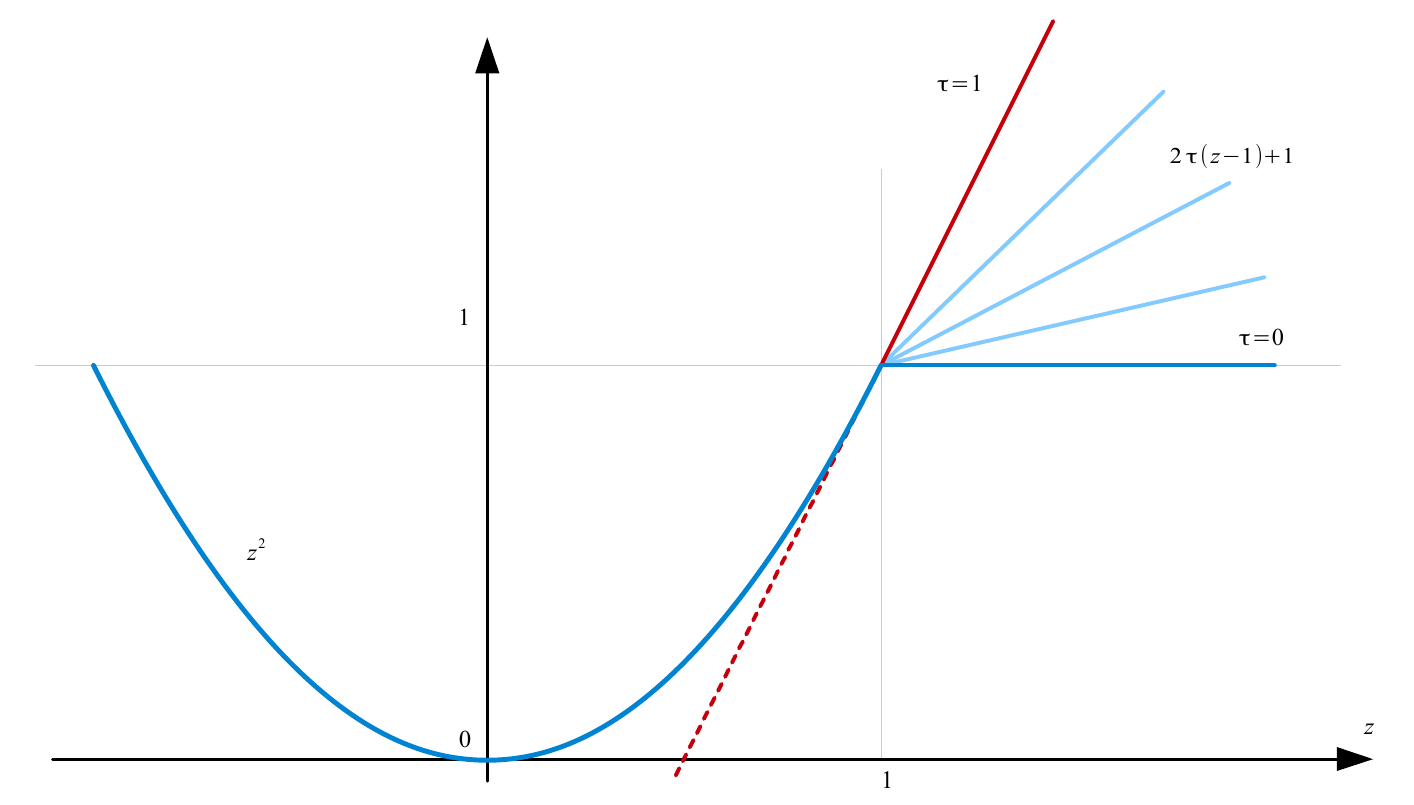}}
\caption{example of convex-affine non-convex potentials.}
\label{figure_convexaffine}
\end{figure} 
\begin{example}[convex-affine potentials as perturbations of truncated potentials]\label{plast-mneigh}\rm 
We con\-si\-der the functions $\ell^\tau$ introduced in \eqref{quadraticalfa} in the non-convex case $0\leq\tau<1$, as pictured in Figure \ref{figure_convexaffine}, with nearest and next-to-nearest neighbour interactions; that is, with $M=2$. To simplify the computations, we fix $m_1=\frac{1}{2}$ and $m_2=\frac{1}{4}$.  
The computation of $Q_{\bf m}\ell^\tau(z)$ involves the values $Q_{\bf m} \ell^\tau(\theta,z)$ in the three locking states $\theta_0=0,$ $\theta_1=\frac{1}{2}$ and $\theta_2=1$; more precisely, it is sufficient to consider  $Q_{\bf m}\ell^\tau(0,z)=\ell^\tau(z)$ for $z\leq 1$, $Q_{\bf m}\ell^\tau(1,z)=\ell^\tau(z)$ for $z\geq 1$ and 
 
$$
Q_{\bf m}\ell^\tau\Big(\frac{1}{2},z\Big)=\frac{1}{3}z^2+\frac{4\tau}{3}z+\frac{3-6\tau-\tau^2}{6}
$$
for $\frac{3}{4}\leq z\leq \frac{3}{2}$. 
Hence 
$$ 
Q_{\bf m}\ell^\tau(z)=\begin{cases}
Q_{\bf m}\ell^\tau(0,z) & \hbox{\rm if } z\leq s^{\tau,+}_0\\
r^\tau_{1}(z)-3z^2& \hbox{\rm if } s^{\tau,+}_0\leq z\leq s^{\tau,-}_1\\
Q_{\bf m}\ell^\tau(\frac{1}{2},z)& \hbox{\rm if } s^{\tau,-}_1\leq z\leq s^{\tau,+}_1\\
r^\tau_{2}(z)-3z^2& \hbox{\rm if } s^{\tau,+}_1\leq z\leq s^{\tau,-}_2\\
Q_{\bf m}\ell^\tau(1,z) & \hbox{\rm if } z\geq s^{\tau,-}_2
\end{cases} 
$$
where 
$r^\tau_1(z)$ is the common tangent (in $s^{\tau,+}_0$ and $s^{\tau,-}_1$) to the parabolas $\widehat Q_{\bf m}\ell^\tau(0,z)$ and $\widehat Q_{\bf m}\ell^\tau(\frac{1}{2},z)$, and correspondingly  
$r^\tau_2(z)$ is the common tangent (in $s^{\tau,+}_1$ and $s^{\tau,-}_2$) to the parabolas $\widehat Q_{\bf m}\ell^\tau(\frac{1}{2},z)$ and $\widehat Q_{\bf m}\ell^\tau(1,z)$. 
\begin{figure}[h!]
\centerline{\includegraphics[width=.9\textwidth]{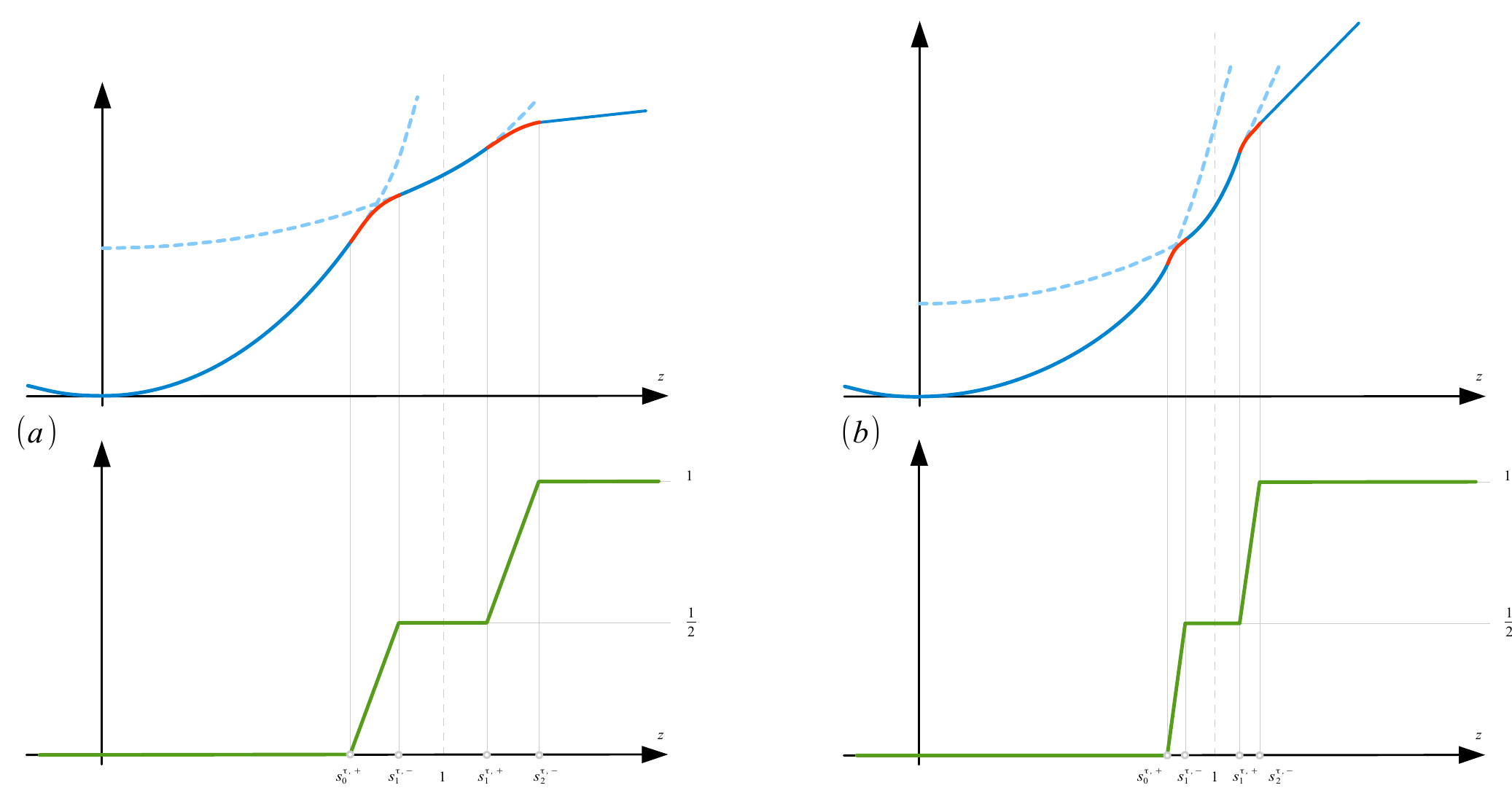}}
\caption{$Q_{\bf m}\ell^\tau$ and corresponding phase functions for increasing values of $\tau\in(0,1)$.}
\label{figure_non-convexaffine}
\end{figure} 

In Fig.~\ref{figure_non-convexaffine} we represent $Q_{\bf m}\ell^\tau$ for two different values of $\tau$, also showing the three energies $\widehat Q_{\bf m}\ell^\tau(\theta,z)$ when $\theta\in \{0,\frac12,1\}$, and the corresponding phase function $\theta$. The value of $\tau$ in (b) is larger than that in (a).   
Note in particular that if $\tau\to 1$ then $s_2^{\tau,-}-s_0^{\tau,+}\to 0$; that is, the locking state $\theta=\frac{1}{2}$ progressively disappears, and we recover the convex case (see Example \ref{plasticity-ex2}), while for $\tau=0$ we recover the case of the truncated quadratic potential with $M=2$. 
\end{example}

\subsubsection{Double-well bi-quadratic potential}\label{doublewell1M}
Let $f\colon\mathbb R\to\mathbb R$ be defined by $f(z)=(1-|z|)^2$,
and let $A=[0,+\infty)$. By explicitly computing the functions $P^{M,n}$ (see Appendix B), 
we obtain for $Q_{\bf m}f(z)$ the formula  
\begin{equation*} 
Q_{\bf m}f(z)=\begin{cases}
(1+z)^2& \hbox{\rm if } \ z\leq s_0^+\\
r^{M,n}(z)-2(m_1+m_M M^2) z^2 &  \hbox{\rm if } \ s_n^+ \leq z \leq s_{n+1}^-
\\
\displaystyle z^2+2(1-2\theta_n)z+1-\frac{4\theta_n(1-\theta_n)}{1+2m_1}& \hbox{\rm if } \ s_n^- \leq z \leq s_{n}^+
\\ 
(1-z)^2 & \hbox{\rm if } \ s_M^- \leq z, 
\end{cases}
\end{equation*} 
where 
$$
s_n^\pm=s_n^\pm(m_1,m_M)=\frac{2\theta_n-1}{1+2m_1}\pm 
\frac{2m_M M}{(1+2m_1)(1+2m_1+2m_M M^2)}
$$
and 
$r^{M,n}$ is the interpolating affine function given in Remark \ref{rem-shapeM}.

\smallskip

\begin{figure}[h!]
\centerline{\includegraphics[width=1\textwidth]{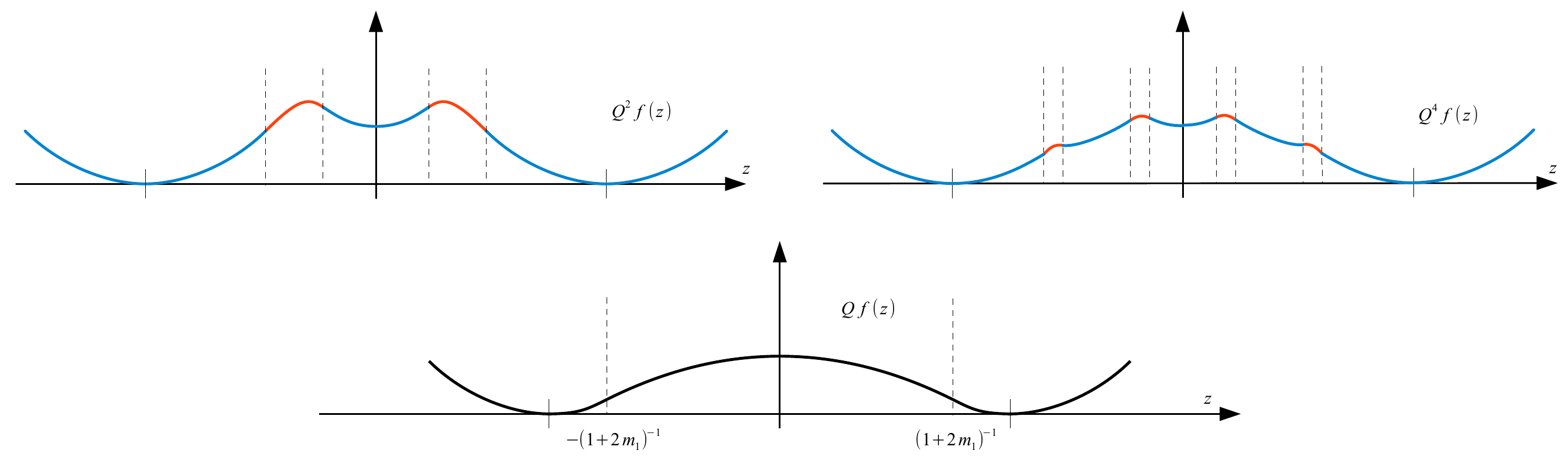}}
\caption{The function $z\mapsto Q^Mf(z)$ for different values of $M$ and the limit function.}
\label{QMphase}
\end{figure} 
\begin{remark}[Asymptotic analysis as $M\to+\infty$]\rm
As in Remark \ref{mfracturerem}, we highlight the dependence on $M$ by writing $\theta(z)=\theta^M(z)$ and $Q_{\bf m}f(z)=Q^Mf(z)$. 
We show that also in this case the limit of  $\theta^M(z)$ as $M\to+\infty$ 
is the phase function of $f$ when the only not vanishing coefficient is $m_1$, and correspondingly for  $Q^Mf(z)$. 
Indeed, since the distribution of $s_n^+$ and $s_n^-$ is uniform, 
we can directly deduce that 
$$\lim_{M\to+\infty}\theta^{M}(z)=\begin{cases} 
0 & 
\hbox{\rm if } \ z\leq -\frac{1}{1+2m_1} \\
\frac{(1+2m_1) z+1}{2}& \hbox{\rm if } \ 
|z| \leq 
\frac{1}{1+2m_1}\\
1 & \hbox{\rm if } \ z\geq \frac{1}{1+2m_1}.
\end{cases}$$
Correspondingly 
$$\lim_{M\to+\infty}Q^Mf(z)=\begin{cases} 
(1+z)^2 & \hbox{\rm if } \ z\leq -\frac{1}{1+2m_1}\\
-2m_1z^2+\frac{2m_1}{1+2m_1}& \hbox{\rm if } \ 
|z| \leq \frac{1}{1+2m_1}\\
(1-z)^2 & \hbox{\rm if } \ z\geq \frac{1}{1+2m_1}  
\end{cases}$$
(see Figure \ref{QMphase}). Again, we note that 
$\lim\limits_{M\to+\infty}Q^Mf(z)=Q_{\bf m^\prime}f(z)$, 
where ${\bf m^\prime}=\{m_1, 0,\dots\}$. 
\end{remark}

\subsubsection{Analysis of $Q_{\bf m}f(\theta,z)$} \label{regfrac}
\begin{figure}[h!]
\centerline{\includegraphics[width=1\textwidth]{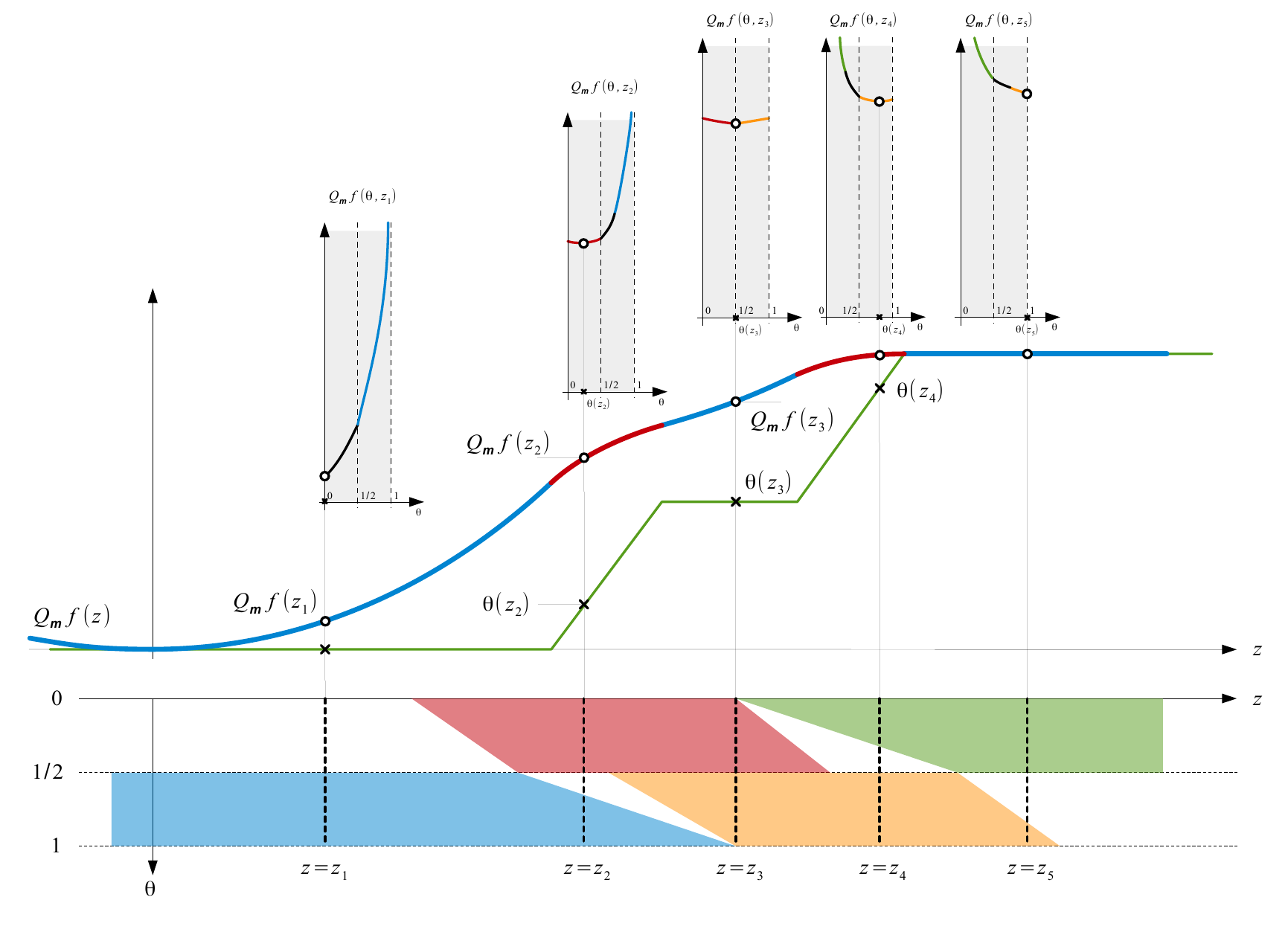}}
\caption{analysis of $\theta\mapsto Q_{\bf m}f(\theta,z)$ for different values of $z$ 
in the truncated quadratic case.}
\label{figure_domainfracture}
\end{figure} 
Examining \eqref{formula-n}, which gives the values of $Q_{\bf m}f(\theta,z)$ as interpolations between neighbouring locking states, we note that $Q_{\bf m}f$ is given by different formulas 
in different regions of the plane ($\theta, z$). 
We briefly examine some feature of this dependence in the simplest meaningful case $M=2$ (see also Fig. \ref{Q23frattura}(a) and Fig. \ref{QMphase} for a comparison).  

In Figures \ref{figure_domainfracture} (truncated quadratic potential) 
and \ref{dom-phase} (double-well potential),  
we highlight 
zones with qualitatively different behaviour, distinguished by colouring.  
In the same pictures, the graphs of $\theta\mapsto Q_{\bf m}f(\theta,z)$ are shown for some values of $z$ in the regions of qualitatively different behaviour. Note that for any fixed $z$ the function 
 $\theta\mapsto Q_{\bf m}f(\theta,z)$ 
is differentiable everywhere  (including the points  where there is a change of the analytical expression), 
except for the point corresponding to the locking state $\theta_1=\frac{1}{2}$, where the left and right derivative are not equal. 

\begin{figure}[h!]
\centerline{\includegraphics[width=1\textwidth]{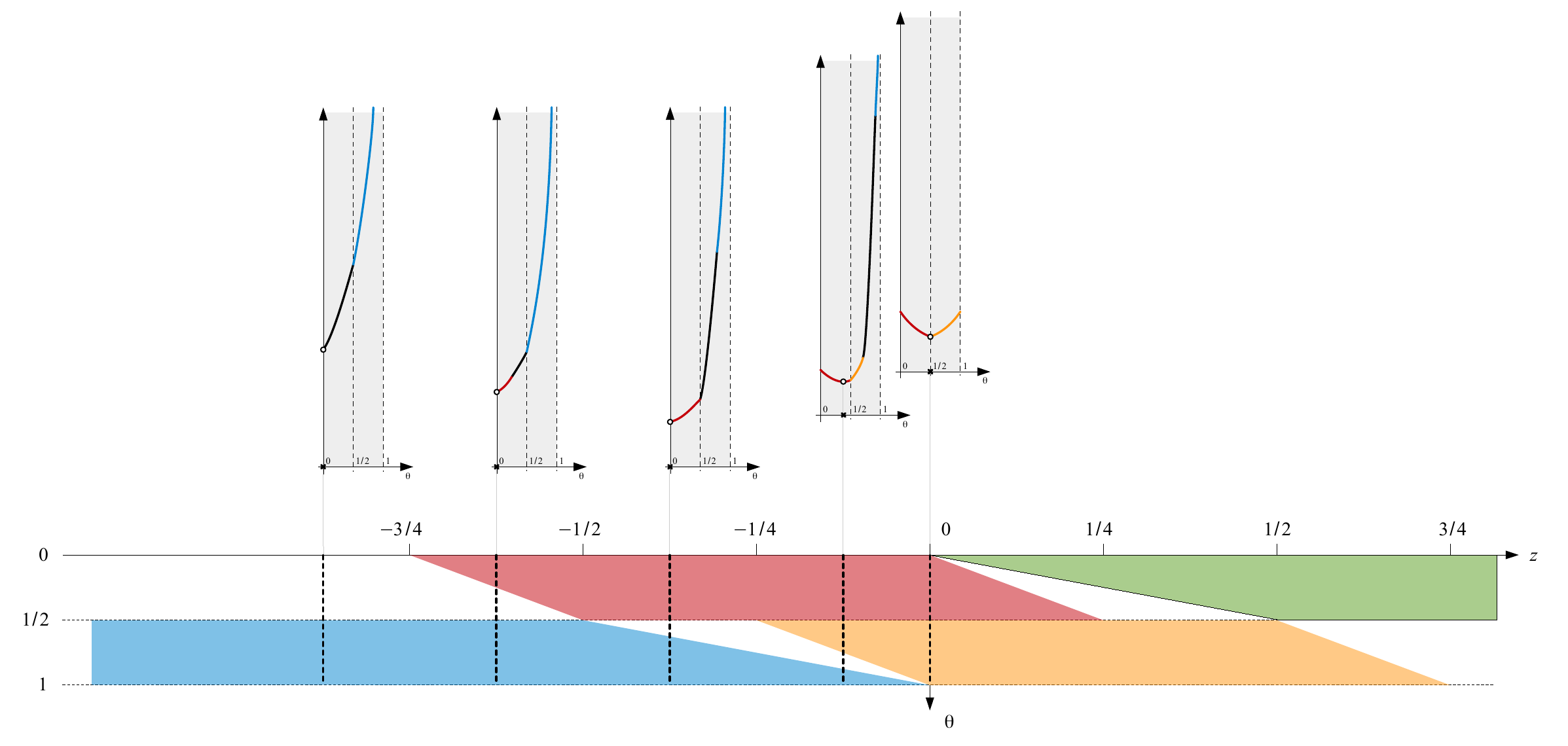}}
\caption{The function $\theta\mapsto Q_{\bf m}f(\theta,z)$ for different values of $z$ (double-well potential).}
\label{dom-phase}
\end{figure} 

For the reader's convenience, in the case of double-well potential we include 
an explicit formula which is particularly simple  
thanks to the symmetry of $Q_{\bf m}f(\theta,z)$ with respect to $(\frac{1}{2},0)$.  
We fix 
$m_1=\frac{1}{2}, m_2=\frac{1}{4}$, obtaining 
$$
Q_{\bf m}f(\theta,z)= 
\begin{cases} 
\frac{3z^2}{1-\theta}+2z+1&\hbox{\rm if }\ z\leq \theta-1 \\
\frac{2z^2}{1-\theta}+\theta&\hbox{\rm if }\ \theta-1<z\leq \frac{\theta-1}{2}\\
z^2-2(2\theta-1)z+\theta^2-\frac{\theta}{2}+\frac{1}{2}&\hbox{\rm if }\ \frac{\theta-1}{2}<z\leq \frac{2\theta+1}{4}\\
\frac{12z^2}{2\theta+1}-2z+1&\hbox{\rm if }\ \frac{2\theta+1}{4}<z. 
\end{cases}
$$

\subsubsection{Dependence on the scale parameter $\sigma$}\label{sigmafrac}
As in Remark \ref{sigmasing}, we introduce a dependence of the concentrated kernel $\bf m$ on the parameter $\sigma$ by setting 
$m_1^\sigma=\frac{m_1}{\sigma}$ and $m_M^\sigma=\frac{m_M}{\sigma}$, for which we have 
\begin{equation}\label{limitssigma}
\lim_{\sigma\to 0^+} Q_{\bf m^\sigma}f(z)=\overline f(z) \ \ \hbox{\rm and }\ \ \lim_{\sigma\to+\infty} Q_{\bf m^\sigma}f(z)=f^{\ast\ast}(z)
\end{equation}  
for any $f$. 

\begin{figure}[h!]
\centerline{\includegraphics[width=.5\textwidth]{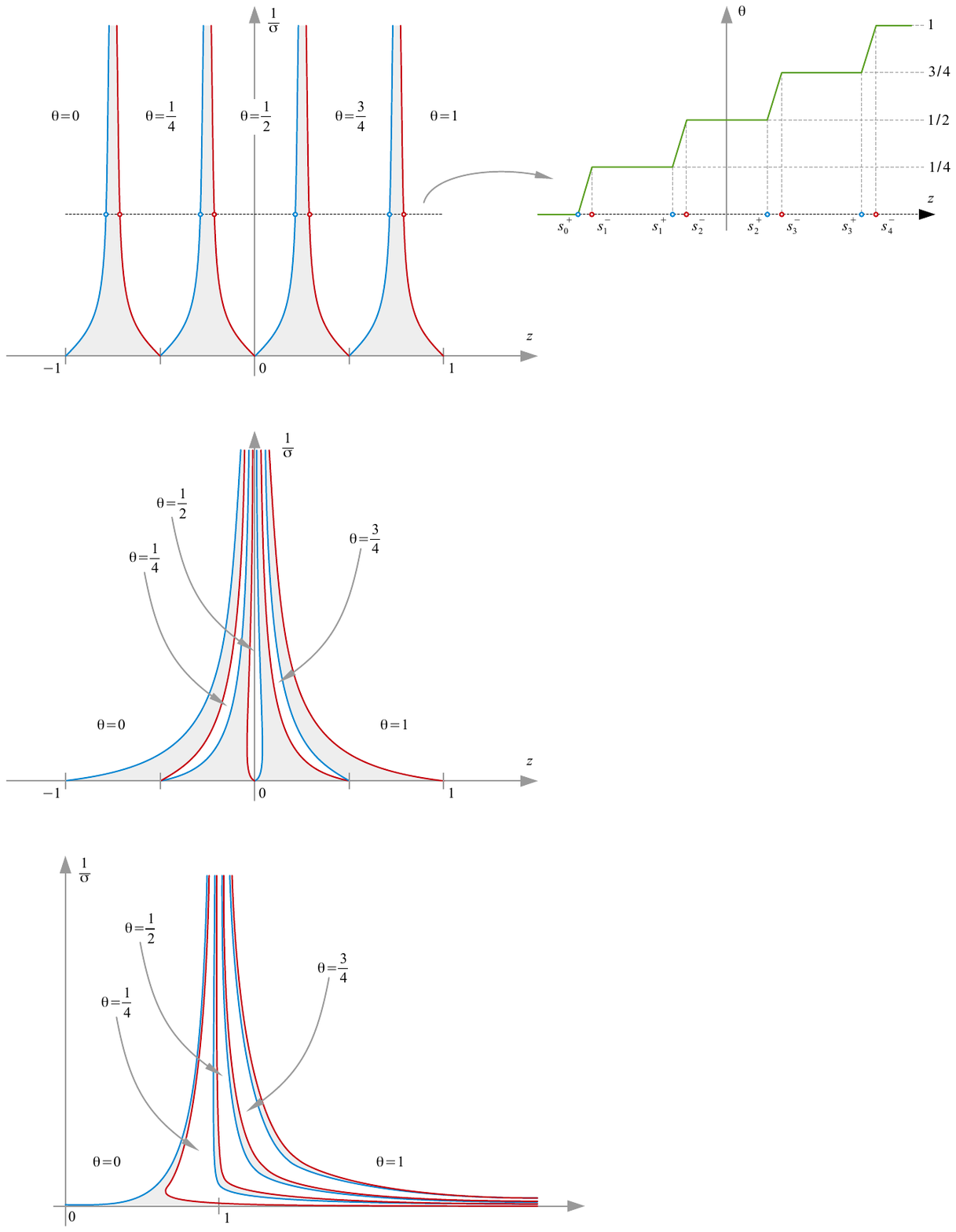}}
\caption{Representation of constancy sets of $\theta$ in the $z$-${1\over\sigma}$ plane.}
\label{trunc-tongue}
\end{figure}

In the case the truncated quadratic function $f$ defined by \eqref{frattura} and analyzed in Section 
\ref{fracture1M}, the first limit can be also checked directly noticing that 
$s_n^+(m_1^\sigma,m_M^\sigma)\to 1$ as ${1\over\sigma}\to+\infty$ for any $n$, where $s_n^+(\cdot, \cdot)$ is defined in \eqref{tnsn}. 
Note that if $\theta\in (0,1)$ then $Q_{\bf m^\sigma}f(\theta, z)\to+\infty$ as ${1\over\sigma}\to 0^+$.  
Moreover, for any $\theta\in(0,1)$ and for any $z$, 
$$\lim_{\sigma\to+\infty} Q_{{\bf m}^\sigma}f(\theta,z)=Q_{{\bf 0}}f(\theta,z)=\begin{cases} 
\theta+\frac{(z-\theta)^2}{1-\theta} & \hbox{\rm if }\ z\leq \theta\\
\theta & \hbox{\rm if }\ z\geq \theta.  
\end{cases}$$
In Fig.~\ref{trunc-tongue} we picture in the $z$-${1\over\sigma}$ plane the zones where $\theta(z)=\theta_n$ for some $n\in\{0,\ldots, M\}$ and those where $\theta(z)$ is affine for fixed $\sigma$ (in grey)  for $M=4$.

As for the double-well potential,  
if the coefficient $m_1$ does not vanish, then 
we re-obtain the first limit in \eqref{limitssigma} by noting that 
\begin{equation}\label{tnsndwsigma}
\lim_{\sigma\to 0^+}s_n^+(m_1^\sigma, m_M^\sigma)=\lim_{\sigma\to 0^+}s_n^-(m_1^\sigma, m_M^\sigma)=0,  
\end{equation} 
where $s_n^+$ and $s_n^-$ are defined in \eqref{tnsndw}. 
\begin{figure}[h!]
\centerline{\includegraphics[width=.55\textwidth]{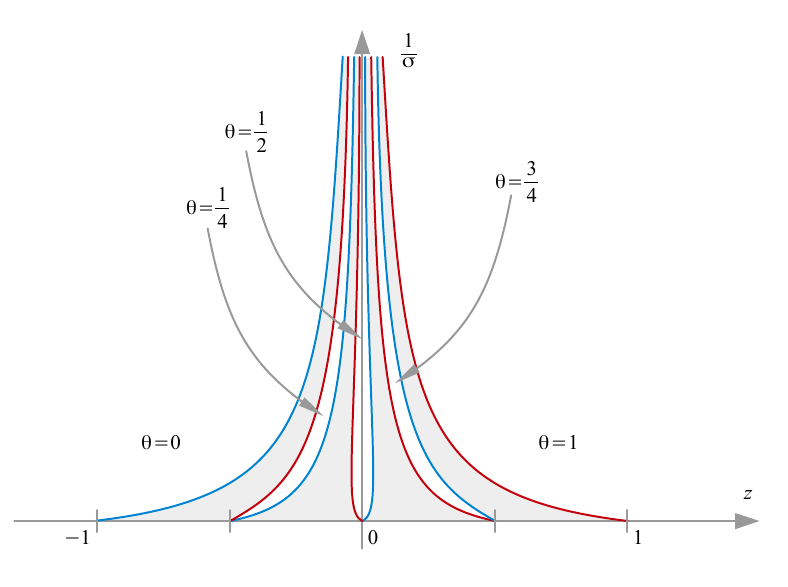}}
\caption{Representation of constancy sets of $\theta$ in the $z$-${1\over\sigma}$ plane}
\label{tongue-M1}
\end{figure} 

In Fig.~\ref{tongue-M1} we picture in the $z$-${1\over\sigma}$ plane the zones where $\theta(z)=\theta_n$ for some $n\in\{0,\ldots, M\}$ and those where $\theta(z)$ is affine for fixed $\sigma$ (in grey)  for $M=4$. 

\begin{remark}\label{casem10dw}\rm 
\begin{figure}[h!]
\centerline{\includegraphics[width=0.5\textwidth]{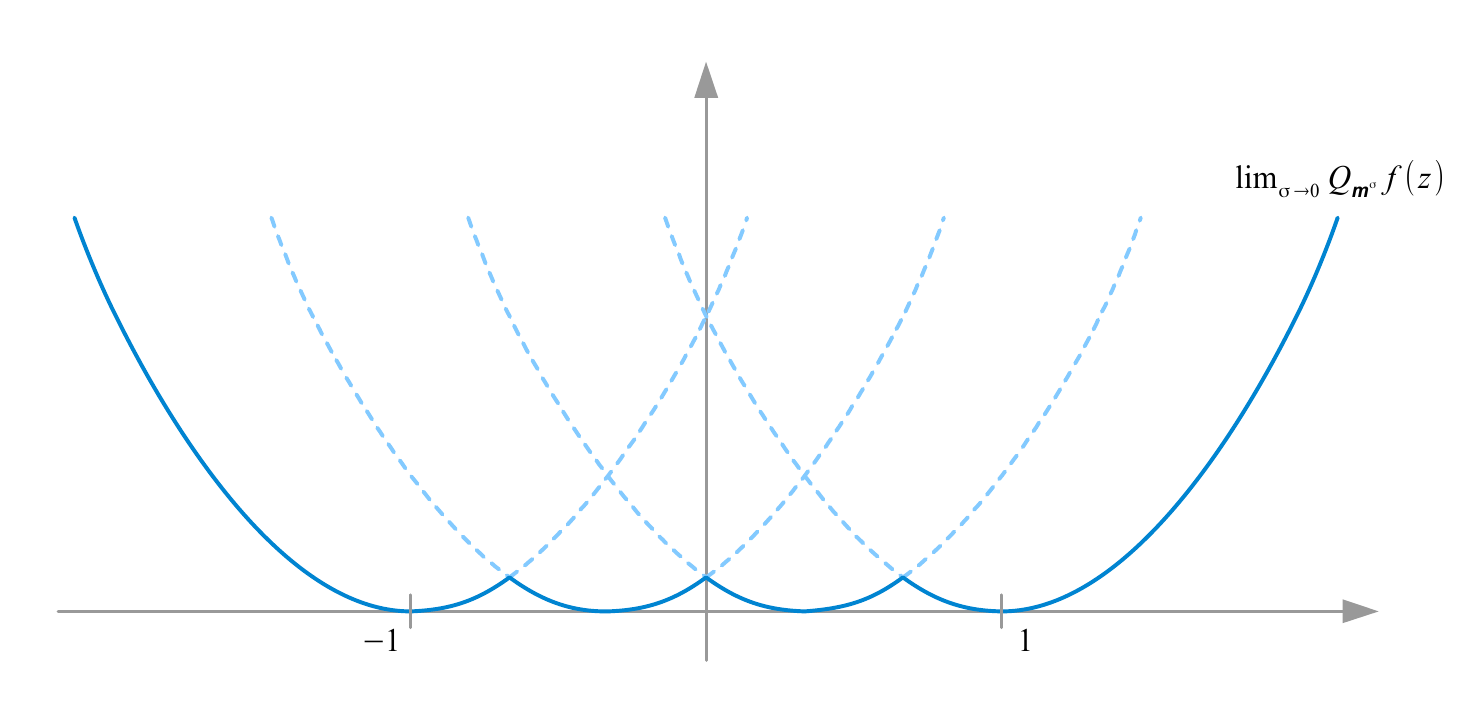}}
\caption{the limit of $Q_{{\bf m}^\sigma}f$ for $\sigma\to0$ in the case $m_1=0$.}
\label{lim-m10}
\end{figure} 

If $m_1=0$, Remark \ref{sigmasing} does not apply. 
Taking the limit for $\sigma\to 0^+$, in this case we obtain 
\begin{equation}\label{tnsndwsigma0}
\lim_{\sigma\to 0^+}s_n^+(m_1^\sigma, m_M^\sigma)=z_n, 
\ \ \ \hbox{\rm and }\ \ \ 
\lim_{\sigma\to 0^+}s_n^-(m_1^\sigma, m_M^\sigma)=z_{n-1}, 
\end{equation} 
where we set 
$$z_n=\frac{2n+1-M}{M}.$$
The limit function is then given by  
$$\lim_{\sigma\to 0^+}Q_{{\bf m}^\sigma}f(z)=
\begin{cases}
(1+z)^2& \displaystyle \hbox{\rm if } \ z\leq z_0\\
\displaystyle \big(z+(1-2\theta_n)\big)^2
& \displaystyle\hbox{\rm if } \ 
z_{n-1}\leq z \leq z_{n}\\
 (1-z)^2 & \displaystyle \hbox{\rm if } \ z_M\leq z, 
\end{cases}
$$
or, equivalently, 
\begin{equation*}
\lim_{\sigma\to 0^+}Q_{{\bf m}^\sigma}f(z)=\min_{0\leq n\leq M}
\big\{ \big(z+(1-2\theta_n)\big)^2\big\}
 =\min_{0\leq n\leq M}\{Q_{\bf m}f(\theta_n,z)\}. 
\end{equation*} 

\begin{figure}[h!]
\centerline{\includegraphics[width=.9\textwidth]{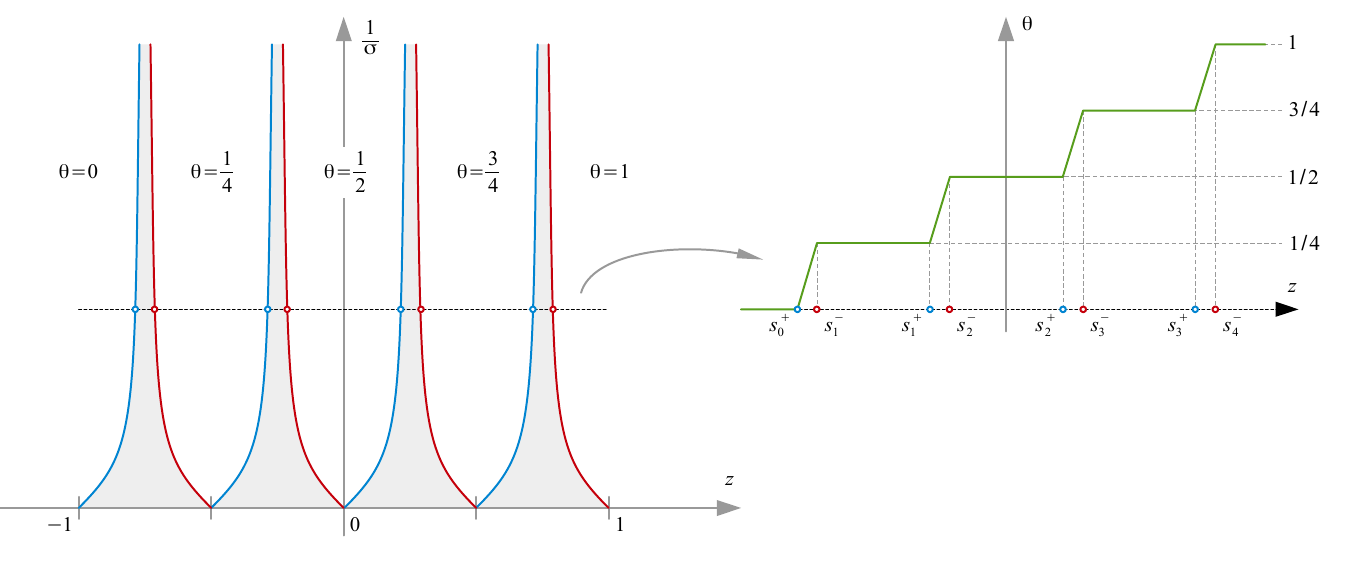}}
\caption{Representation of $\theta$ in the $z$-${1\over\sigma}$ plane for $M=4$ (case $m_1=0$)}
\label{tongue-M10}
\end{figure} 
Note that in this case the limit differs from $f$ but coincides with the minimum among $P^{M,n}(z)-2m_M M^2 z^2$ (see Fig.~\ref{lim-m10}), whose convexification still equals $f^{**}$. 

In Fig.~\ref{tongue-M10} we picture in the $z$-${1\over\sigma}$-plane the zones where $\theta(z)=\theta_n$ for some $n\in\{0,\ldots, M\}$ and those where $\theta(z)$ is affine for fixed $\sigma$ (in grey). \end{remark}

\section{Relaxation with exponential-kernel penalization}\label{exp-sec} 
The case of concentrated kernels studied in the previous section allowed us to highlight some properties of $Q_{\bf m}f$,
in particular  we were able to characterize the locking states using explicit formulas. 
Now, we analyze the effect of the superposition of spatially distributed long-range interactions,  
which bring additional complexity to the structure of $Q_{\bf m}f$.

In Section \ref{hidise} we sketch a method for obtaining bounds for a general kernel $\bf m$ via higher-dimensional embeddings. This method is optimal in the case when 
the non-local term 
 $\sum_{i,j}m_{|i-j|}(u_i-u_j)^2$ 
depending on the given kernel $\bf m$ 
can be obtained by integrating out the variable $v$ from the simplest 
additive energy depending on two variables $u$ and $v$; that is, 
 $a\sum_{i}(v_i-v_{i-1})^2+b\sum_{i}(u_i-v_i)^2$.  
To have this, 
we note that the kernel $\bf m$ must be exponential. Hence, the study of general exponential kernels will constitute the main goal of this section. 
The idea of rewriting the problems defining $\widehat Q_{\bf m}f$ as additive problems in terms of an auxiliary variable has been already used implicitly in the case of concentrated kernels. Indeed, in that case we introduced coarse-grained energies 
depending only on $M$-neighbour interactions $u_{i+M}-u_i$  
through the functions $P^{M,n}$. 

\subsection{Higher-dimensional embeddings for general $\bf m$} \label{hidise}
In this section we discuss the possibility of simplifying the quadratic penalty term in Definition \ref{defQcap} for an arbitrary kernel $\bf m$ by introducing auxiliary variables. 
This will be later applied to the exponential kernel defined in \eqref{defexpker}. 
The idea is to view the long-range interactions parameterized by an arbitrary $\bf m$ as a projection of short-range interactions operating in a higher-dimensional space. In other words, we now suppose that  the kernels $\bf m$  can be viewed as the Green's functions of some higher-dimensional local problems.  Note however that the locality of the corresponding higher-dimensional problem can be expected only for kernels $\bf m$ with sufficiently fast rate of decay. To highlight the ideas, we  discuss in detail  only the simplest class of projections, where the dimension of the extended  configurational space is doubled. As a result, the nonlocal \emph{scalar} problem is transformed into a local \emph{vector} problem. 

For each fixed $k\in\mathbb N$, we define a quadratic form depending on two variables as follows.   
Let $A$ be a $(k+1)\times (k+1)$ matrix and let $s\in\mathbb R$ be a scalar parameter.  
We set 
\begin{equation}\label{defjkm1} 
H^k[A,s](u,v)=
2s\langle Av,v\rangle+ 2s\langle u-v, u-v\rangle, 
\end{equation} 
where $u,v\colon \{0,\dots, k\}\to \mathbb Z$ and    $\langle \cdot ,\cdot \rangle$ denotes the scalar product in $\mathbb R^{k+1}$.  

The following result restates the definition of $\widehat Q_{\bf m}f$ as a minimum problem involving a quadratic form of type \eqref{defjkm1}. 

\begin{theorem}[higher-dimensional equivalent formulation]\label{teo-equivalence}
Let ${\bf m}$ satisfy \eqref{propm} and be such that 
the function $n\mapsto m_n$ is not increasing for $n$ large enough.
Then, there exist a $(k+1)\times (k+1)$-dimensional matrix $A^k_{\bf m}$ and a scalar $s_{\bf m}$ such that, setting 
$H^k_{\bf m}=H^k[A^k_{\bf m}, s_{\bf m}]$ in \eqref{defjkm1},
the following equality holds 
\begin{equation}\label{equiv-2v}
\widehat Q_{\bf m}f(z)=\lim_{k\to+\infty}\frac{1}{k}\min\Big\{\sum_{i=1}^kf(u_i-u_{i-1})+
H^k_{\bf m}(u,v): u,v\in
\mathcal A(k,z)\Big\} 
\end{equation}
for all $f\colon\mathbb R\to [0,+\infty)$ satisfying  growth conditions \eqref{crescita} and \eqref{crescitasotto}. 
\end{theorem}

The proof of Theorem \ref{teo-equivalence} is based on Lemma \ref{2v-lemma} which implies that asymptotically the quadratic part of the energies in  the definition of $\widehat Q_{\bf m} f$ can be viewed as projections of functions of the form \eqref{defjkm1}. 

To shorten the notation, we introduce the quadratic function 
\begin{equation}\label{defjkm}
J_{\bf m}^k(u)=\sum_{i,j=0}^{k} m_{|i-j|}(u_i-u_j)^2,
\end{equation} 
defined on $u\colon \{0,\dots, k\}\to \mathbb Z$. 

To quantify the relation between $J_{\bf m}^k$ and the corresponding $H^k_{\bf m}$, we introduce a notion of $L^2$ norm for   $u\colon \{0,\dots, k\}\to \mathbb Z$ by setting 
$$\|u\|_k^2=\frac{1}{k}\sum_{i=1}^{k}(u_i)^2,$$
which coincides with the $L^2$ norm of the piecewise-constant function $\tilde u\colon(0,1)\to\mathbb R$ defined by $\tilde u(t)=u_i$ in $(\frac{i-1}{k}\frac{i}{k}]$. 

\begin{lemma}[projection of the quadratic part of the energies]\label{2v-lemma}
Let ${\bf m}$ satisfy \eqref{propm} and be such that the function 
$n\mapsto m_n$ is not increasing for $n$ large enough. Let $J_{\bf m}^k$ be as in \eqref{defjkm}. 
Then, there exist a $(k+1)\times (k+1)$-dimensional matrix $A^k_{\bf m}$ and a scalar $s_{\bf m}$ such that 
\begin{equation}\label{minG}
\min\{H^k_{\bf m}(u,v): v\colon\{0,\dots,k\}\to\mathbb R\}=J_{\bf m}^k(u)+\|u\|_k^2\ o\Big(\frac{1}{k}\Big) 
\end{equation}
for all $u\colon\{0,\dots, k\}\to\mathbb R$, where $H_{\bf m}^k$ is defined in Theorem \ref{teo-equivalence}.
\end{lemma}
\begin{proof} 
We introduce the $(k+1)\times (k+1)$ matrix $M^k_{\bf m}=(m_{ij})$ given by $m_{ij}=m_{|i-j|}$, $i,j=0,\dots, k$. 
Note that the functional $J^k_{\bf m}$ is independent of the choice of $m_0$, so that we can choose the value of $m_0$ arbitrarily. We assume that this value is such that the matrix $M^{k}_{\bf m}$ is invertible. 

As a first step, we write the functional $J_{\bf m}^k$, up to an infinitesimal term, 
as the sum of a suitable quadratic form depending on 
the whole series of $m_n$ and a residual boundary term. 
By Lemma \ref{boundarycond} (see Appendix A), up to a change of variables with $L=1$ and $\e=1/k$,
we can suppose that $u$ is constant in $[0,k^\alpha]$ and in $[k-k^\alpha,k]$ with a fixed $\alpha\in(\frac{3}{\beta},1)$, where $\beta$ is the decay parameter of $\bf m$ given by \eqref{propm}.   
Up to translations, we can assume $u_0=0$ and hence $u_{i}=0$ for $i\leq k^{\alpha}$.
Setting  
$$
s_{\bf m}=m_0+2\sum_{n=1}^{+\infty}m_n\ \ \ \hbox{\rm and } \ \ \ 
s_{\bf m}^i=\sum_{j=0}^{k}m_{ij},
$$ 
we get 
$$s_{\bf m}^i-s_{\bf m}=-\sum_{n=i+1}^{+\infty}m_n-\sum_{n=k-i+1}^{+\infty}m_n,$$
so that, using the decay condition $m_n=o(n^{-\beta})$, we obtain 
\begin{eqnarray*}
J_{\bf m}^k(u)&=&2s_{\bf m}\langle u-\frac{1}{s_{\bf m}} M_{\bf m}^ku, u\rangle +2 \sum_{i=0}^{k} (s_{\bf m}^i-s_{\bf m})(u_i)^2\\
&=&2s_{\bf m}\langle u-\frac{1}{s_{\bf m}} M_{\bf m}^ku, u\rangle -2 t_{\bf m}(u_{k})^2+\sum_{i=0}^k(u_i)^2 \ o(k^{1-\alpha\beta})\\
&=&2s_{\bf m}\langle u-\frac{1}{s_{\bf m}} M_{\bf m}^ku, u\rangle -2 t_{\bf m}(u_{k})^2+\|u\|_k^2 \ o(k^{2-\alpha\beta}), 
\end{eqnarray*}
where 
$t_{\bf m}=\sum_{n=0}^{+\infty}n\,m_n$. Note that $2-\alpha\beta<-1$ since $\alpha>\frac{3}{\beta}$. 

The matrix $A^k_{\bf m}$ will be obtained by modifying the matrix $s_{\bf m} \big(M_{\bf m}^k\big)^{-1}-I$, which gives a minimum  for $H^k_{\bf m}$ in $v=\frac{1}{s_{\bf m}}M_{\bf m}^ku$, so as to take into account the boundary contribution. This is done by changing the values $(A^k_{\bf m})_{11}$ and $(A^k_{\bf m})_{kk}$ in such a way that they compensate the boundary terms. 
We set  
\begin{equation}\label{formulaA}
A^k_{\bf m}=s_{\bf m}
\begin{pmatrix} c_{\bf m} & 0 & \dots & \dots & 0\\
0& 1 & 0 & \dots & \dots \\
\dots & \dots & \dots & \dots & \dots\\
\dots & \dots  & 0 & 1 & 0\\
\dots & \dots  & 0 & 0 & c_{\bf m} 
\end{pmatrix}\big(M_{\bf m}^k\big)^{-1}-I,\ \ \hbox{\rm with }\ \ c_{\bf m}=\frac{s_{\bf m}+m_0}{2t_{\bf m}+s_{\bf m}+m_0}.\end{equation}
We can write 
$$A^k_{\bf m}=s_{\bf m}\big(M_{\bf m}^k\big)^{-1}-I-\frac{2t_{\bf m} s_{\bf m}}{2t_{\bf m}+s_{\bf m}+m_0}(e_0\otimes e_0+e_{k}\otimes e_{k}) \big(M_{\bf m}^k\big)^{-1},$$ and we prove that  
the minimum of $H^k_{\bf m}(u,v)$ coincides, up to an infinitesimal term, with $J_{\bf m}^k(u)$. 
This minimum is attained for $v^{k,\rm min}$ given by 
\begin{equation}\label{vsol}
v^{k, \rm min}=(A^k_{\bf m}+I)^{-1}u=
\frac{1}{s_{\bf m}}M_{\bf m}^ku+\frac{2t_{\bf m}}{s_{\bf m}(s_{\bf m}+m_0)}u_{k}\begin{pmatrix}m_{k}\\
m_{k-1}\\
\dots\\
m_0
\end{pmatrix}.
\end{equation}
Then, recalling the decay assumption on $m_n$, we get 
\begin{eqnarray}\nonumber
H^{k}_{\bf m}(u,v^{k, \rm min})&=&
2s_{\bf m}\langle u - v^{k,\rm min},u\rangle\\ \nonumber
&=&2s_{\bf m}\langle u - \frac{1}{s_{\bf m}}M_{\bf m}^ku,u\rangle
-2 (u_{k})^2 t_{\bf m} 
+ 
|u_k|\sqrt{k}\|u\|_k
\ o(k^{-\alpha\beta})\\ \label{accab}
&=&2s_{\bf m}\langle u - \frac{1}{s_{\bf m}}M_{\bf m}^ku,u\rangle
-2 (u_{k})^2 t_{\bf m} 
+ 
\|u\|_k^2\ o(k^{\frac{1}{2}-\alpha\beta}), 
\end{eqnarray}
concluding the proof of \eqref{minG} since $\alpha>\frac{3}{\beta}$. 
\end{proof}

\begin{remark}\label{ossv}\rm
Let $u^k$ be constant on $[0, k^\alpha]$ and $[k-k^\alpha, k]$. Then the corresponding $v^{k,\min}$ given by \eqref{vsol} 
satisfies $|v^{k,\min}_0-u^k_0|+|v^{k,\min}_k-u^k_k|= o(k^{{1\over 2}-\alpha\beta})\|u^k\|_k$. Hence it 
can be modified so as to obtain $\widehat v^{k}$ equal to $u^k$ in $0$ and $k$ and 
$|v^{k,\min}_i-v^k_i|=o(k^{{1\over 2}-\alpha\beta})\|u^k\|_k$ for all $i$. 
By \eqref{accab} we can estimate 
$$
H^{k}_{\bf m}(u,\widehat v^k)\le H^{k}_{\bf m}(u,v^{k, \rm min})+ 
\|u^k\|^2_k\ o(k^{1-\alpha\beta}).
$$
If $\|u^k\|_k$ are equibounded, then the last term is $o(\frac{1}{k})$ since $\alpha>{3\over\beta}$.
Note that we may also construct $\widehat v^{k}$ so that $\widehat v^{k}_i=u^k_0$  for $i\le k^{\alpha^\prime}$ and $\widehat v^{k}_i=u^k_k $ if $i\ge k-k^{\alpha^\prime}$ with $\alpha^\prime<\alpha$. 
\end{remark}

\begin{proof}[Proof of Theorem {\rm\ref{teo-equivalence}}] 
We write 
\begin{equation}\label{equiv-2v-1}
\widehat Q_{\bf m}f(z)=\lim_{k\to+\infty}\frac{1}{k}\min\Big\{\sum_{i=1}^kf(u_i-u_{i-1})+J_{\bf m}^k(u):
u\in
\mathcal A(k,z)\Big\}\,.
\end{equation}
Let $u^k$ denote a minimizer of the problem above, and note that $\|u^k\|_k$ are equibounded in view of the growth condition on $f(z)+m_1z^2$. Note that thanks to Lemma \ref{boundarycond} we may suppose that the function $u^k$ is constant on $[0, k^\alpha]$ and $[k-k^\alpha, k]$. Then, applying Lemma \ref{2v-lemma} and Remark \ref{ossv}, we obtain the desired result. \end{proof}

In general, the advantage of the rewriting in Theorem \ref{teo-equivalence} is not clear. However, thanks to the two-variable formulation, we 
can obtain some general lower bound in suitable hypotheses. 
In the next section, we 
will see that for exponential kernels functionals $H^{k}_{\bf m}$ can be rewritten as nearest-neighbour energies, which will allow to make these bounds sharp.
\smallskip 

\begin{remark}[lower bounds with additive vector energies]\label{remlowadd}\rm 
Suppose that there exists $C>0$ such that for all $v\in \mathcal A(k;z)$
\begin{equation}\label{stiCA}
\langle A^k_{\bf m} v,v\rangle \ge C\sum_{i=1}^k (v_i-v_{i-1})^2 +\|v\|_k^2 \ o(1)_{k\to+\infty}.
\end{equation}
Then, by \eqref{equiv-2v}, we can bound $\widehat Q_{\bf m} f(z)$ from below with limits of scaled minimum problems for energies of the form 
$$
\sum_{i=1}^k f(u_i-u_{i-1})+  2s_{\bf m} C\sum_{i=1}^k (v_i-v_{i-1})^2+2s_{\bf m}\sum_{i=1}^k (u_i-v_i)^2.
$$ 
We will see in the next section that this holds with some particular choices of the kernel $\bf m$; namely, the exponential kernels.
\end{remark} 

In view of Remark \ref{remlowadd}, we now focus on bounds for problems involving energies of the form
$$
E(u,v; [0,k])=\sum_{i=1}^k f(u_i-u_{i-1})+  a\sum_{i=1}^k (v_i-v_{i-1})^2+b\sum_{i=1}^k (u_i-v_i)^2
$$
with $a,b>0$.

We suppose that there exist $z^*$ and $\eta$ such that $f$ is convex for $z\le z^*$ and $f(z)\ge \eta$  for $z>z^*$. 
For any $N\geq 1$ we define
\begin{eqnarray}\label{gNgen}\nonumber && g_N(z)=
{1\over N}\Bigl(\min\Bigl\{\sum_{i=2}^{N} f(u_i-u_{i-1})+ a \sum_{i=1}^N  (v_i-v_{i-1})^2+b\sum_{i=1}^N (u_i-v_i)^2\\
&& \hskip3cm
v_0=0, v_N=Nz, u_i-u_{i-1}\le z^* \hbox{ for } i\ge 2\Bigr\}+ \eta\Bigr),
\end{eqnarray}
where we limit the interactions  $v_i-v_j$ only to nearest neighbours, and we allow  $u_i-u_{i-1}> z^* $ only for $i=1$. Note that if $N=1$ then $g_1(z)= a z^2+\eta$. 

We also set $g_\infty(z)= f(z)+ a z^2$ with domain $z\le z^*$, which corresponds to minimal states with $u_i-u_{i-1}\le z^* $ for all $i$.

\begin{proposition}[lower bound with nearest-neighbour energies]\label{lobaloba}
We have
\begin{equation}\label{lobbi}
\lim_{k\to+\infty}{1\over k}\min\Bigr\{E(u,v; [0,k]): u_k-u_0=v_k-v_0=kz\Bigl\} \ge \Bigl(\inf_N g_N(z)\Bigr)^{**}.
\end{equation}
\end{proposition}

\begin{proof}
The proof is obtained giving a lower bound for the minima
$$
{1\over k}\min\Bigl\{\sum_{i=1}^{k} f_{\eta}(u_i-u_{i-1})+  a\sum_{i=1}^k  (v_i-v_{i-1})^2+ b\sum_{i=1}^k (u_i-v_i)^2: u_k-u_0=v_k-v_0=kz\Bigr\},
$$
where 
$$
f_{\eta}(z)=\begin{cases} f(z) & \hbox{ if } z\le z^*\\
\eta & \hbox{ if } z> z^*.
\end{cases}
$$
Consider a minimizer $u$ for such problem. 
If $u_i-u_{i-1}\le z^*$ for all $i$ then by the convexity of $f$ this minimum equals the value $g_0(z)$.
If otherwise $u_i-u_{i-1}>z^*$ for some $i$, note that we can always suppose that this holds for $i=1$, by splitting the discrete interval $\{0,\ldots, k\}$ into subsets $\{i_{k_{j-1}}\ldots, i_{k_{j}}\}$, $j=1,\dots, r$, in which $u_i-u_{i-1}>z^*$ only for $i=i_{k_{j-1}}+1$, we obtain a lower estimate with 
$$
\sum_{j=1}^r {N_j\over k} g_{N_j} (z_j)
$$
where $N_j= k_j-k_{j-1}$ and $z_j= {u_{k_j}- u_{k_{j-1}}\over N_j}$, so that we have the convex combination
$$
\sum_{j=1}^r {N_j\over k}z_j = z\,.
$$
From this estimate \eqref{lobbi} follows.
\end{proof} 

We will prove general properties of the functions $g_N$ in Section \ref{truco:sect}, which will allow to describe the structure of their convex envelope and their optimality in computing $\widehat Q_{\bf m}f$.

\subsection{Reduction to a local problem for the exponential kernel}  
We now introduce some notation for the exponential kernels. We define 
\begin{equation}\label{defexpker}
\mathbf m =\mathbf m^\sigma=\{m^\sigma_n\}=\{e^{-\sigma n}\},
\end{equation}
where 
 $\sigma>0$ is a given constant. 
Highlighting the dependence on the parameter $\sigma$, we set 
 \begin{equation}\label{def-tsigmahat} 
\left. \begin{array}{ll}
\hspace{-3mm}\displaystyle \widehat Q_\sigma f(z)=\!\lim_{k\to+\infty}\frac{1}{k}\inf\Big\{\sum_{i=1}^k f(u_i-u_{i-1})+
\sum_{i,j=0}^k e^{-|i-j|\sigma}(u_i-u_j)^2: 
 u \in \mathcal A(k; z)\Big\}, 
\end{array}
\right.
\end{equation} 
and introduce the corresponding $\bf m^\sigma$-transform of $f$
 \begin{equation}\label{def-tsigma} Q_\sigma f(z) = \widehat Q_\sigma f(z)- a_{{\bf m}^\sigma}z^2
 =\widehat Q_\sigma f(z)-\frac{2e^{-\sigma}(1+e^{-\sigma})}{(1-e^{-\sigma})^3}z^2.
 \end{equation} 

Let $F_\e^\sigma$ denote the non-local functionals of the type defined in \eqref{ } with exponential kernel $m_n=e^{-\sigma n}$; that is, 
\begin{equation}\label{def-fe-sigma}
F^\sigma_\e(u;I)=\e\sum_{i\in \mathcal I^\ast_\e(I)}f\Big(\frac{u_i-u_{i-1}}{\e}\Big)+
\e \sum_{i,j\in \mathcal I_\e(I)}
e^{-\sigma|i-j|}\Big(\frac{u_i-u_j}{\e}\Big)^2, 
\end{equation} 
where   
$\mathcal I_\e(I)=\{i\in\mathbb Z: \e i\in I\}$, $\mathcal I^\ast_\e(I)=\{i\in\mathbb Z: \e i,\e(i-1)\in I\}$ and the function $u$ belongs to $\mathcal A_\e(I)=\{u\colon \e \mathcal I_\e(I)\to\mathbb R\}$ as defined in \eqref{def-ind}. 
Following the general approach formulated in Section \ref{hidise}, given 
$a,b>0$ we define the local two-variable energies  
\begin{equation}\label{def-2var}
E_\e(u,v;I)=\e\!\!\sum_{i\in \mathcal I^\ast_\e(I)}\!\!f\Big(\frac{u_i-u_{i-1}}{\e}\Big)+
\frac{a}{\e}\sum_{i\in \mathcal I^\ast_\e(I)}\!\!(v_i-v_{i-1})^2+
\frac{b}{\e}\sum_{i\in \mathcal I^\ast_\e(I)}\!\!(u_i-v_i)^2 
\end{equation}
for $u,v\in\mathcal A_\e(I)$. 
We will prove an asymptotic equivalence result between $F_\e^\sigma$ and $E_\e$; more precisely, that the $\Gamma$-limits 
of the two sequences are the same for a suitable choice of $a=a_\sigma$ and $b=b_\sigma$.  The $\Gamma$-limit of $E_\e$ is computed with respect to the convergence $u^\e,v^\e\to u$ defined as the convergence in $L^2(I)$ of the piecewise-constant extensions of $u^\e$ and $v^\e$ to the function $u\in H^1(I)$. 
The result is obtained, in the spirit of Section \ref{hidise}, by explicitly integrating out the variable $v$. 
\begin{theorem}[asymptotic equivalence]\label{teo-equivalence-exp}
Let 
\begin{equation}\label{aeb}
a_\sigma=a_{{\bf m}^\sigma}=\frac{2(1+e^{-\sigma}) e^{-\sigma}}{(1-e^{-\sigma})^3}, \quad \quad 
b_\sigma=
\frac{2(1+e^{-\sigma})}{(1-e^{-\sigma})}, 
\end{equation} 
and set $E^\sigma_\e=E_\e$ as defined in \eqref{def-2var} with $a=a_\sigma$ and $b=b_\sigma$. 
Then the sequence $E^\sigma_\e$ $\Gamma$-converges  to the same $\Gamma$-limit as the sequence $F^\sigma_\e$.  
\end{theorem} 

\begin{remark}[asymptotic behaviour controlled by $\sigma$]\rm
We can interpret the extremal regimes of strong and weak additivity in terms of the parameters of the two-parameter energies \eqref{def-2var}. Let $a_\sigma, b_\sigma$ be given by \eqref{aeb}. As $\sigma\to 0$ we have both $a_\sigma\to +\infty$ and $b_\sigma\to +\infty$, with an increasing strength of the effect of the term involving the distance of $u$ from the affine function $zi$. Conversely, when $\sigma\to+\infty$ we have $a_\sigma\to 0$, and the role of this distance term gradually diminishes. 
\end{remark}

\begin{remark}[equivalence with arbitrary coefficients]\label{absigma}\rm 
The equivalence result in Theorem \ref{teo-equivalence-exp} can be extended to arbitrary pairs $a,b>0$ up to considering 
the non-local functionals with kernel $m_n=\varrho e^{-\sigma n}$; that is, the functionals given by 
$$
F^{\varrho,\sigma}_\e(u;I)=\e\sum_{i,i-1\in \mathcal I_\e(I)}f\Big(\frac{u_i-u_{i-1}}{\e}\Big)+
\e \, \varrho\sum_{i,j\in \mathcal I_\e(I)}e^{-\sigma |i-j|}
\Big(\frac{u_i-u_j}{\e}\Big)^2, 
$$
with the choices 
\begin{equation}\label{absigmarho}
\sigma=\sigma_{a,b}=2\sinh^{-1}\Big(\frac{1}{2}\sqrt{\frac{b}{a}}\Big) \quad \hbox{ and } \quad \varrho=\varrho_{a,b}=\frac{b^2}{4a\sinh(\sigma_{a,b})}.
\end{equation}
Indeed, with this definition we get 
\begin{eqnarray*}
\frac{a}{\varrho_{a,b}}=\frac{2(1+e^{-\sigma_{a,b}})e^{-\sigma_{a,b}}}{(1-e^{-\sigma_{a,b}})^3}
=a_\sigma
\ \ \hbox{ and }\ \ \frac{b}{\varrho_{a,b}}=\frac{2(1+e^{-\sigma_{a,b}})}{1-e^{-\sigma_{a,b}}}
=b_\sigma, 
\end{eqnarray*}
so that we can apply Theorem \ref{teo-equivalence-exp} obtaining the equivalence between $\frac{1}{\varrho}F_\e^{\varrho,\sigma}$ and $\frac{1}{\varrho}E^{\sigma}_\e$.  
The corresponding (trivial) generalization of  $Q_\sigma f$ in \eqref{def-tsigma} 
can be obtained by defining
\begin{equation}\label{def-tsigmar} 
\widehat Q_{\sigma,\varrho} f(z)=\!\lim_{k\to+\infty}\frac{1}{k}\inf\Big\{\sum_{i=1}^k f(u_i-u_{i-1})+
\varrho\sum_{i,j=0}^k e^{-|i-j|\sigma}(u_i-u_j)^2: 
 u \in \mathcal A(k; z)\Big\}, 
 \end{equation} \
and setting $Q_{\sigma,\varrho} f(z)=\widehat Q_{\sigma,\varrho} f(z)-a_\sigma \varrho z^2$, with $a_\sigma$ as in \eqref{aeb}.  \end{remark}

The proof of Theorem \ref{teo-equivalence-exp} is based on the following lemma, which 
allows to integrate out the variable $v$ by applying the general result of 
Lemma \ref{2v-lemma} to the case of exponential kernels.

\begin{lemma}\label{lemma-equivalence} 
Let $L>0$ and $k_\e=\lfloor\frac{L}{\e}\rfloor$. We fix $\alpha\in(0,1)$ and set $n_\e=\lfloor (k_\e)^\alpha \rfloor$.
Let $F^\sigma_\e$ be given by \eqref{def-fe-sigma} and $E^\sigma_\e$ be given by \eqref{def-2var} 
with $a_\sigma,b_\sigma$ as in \eqref{aeb} and $I=[0,L]$.
Then, if $u^\e\in\mathcal A_\e=\mathcal A_\e([0,L])$ satisfies $u^\e_i=u^\e_0$ for $i\leq n_\e$, $u^\e_i=u^\e_{k_\e}$ for $i\geq k_\e-n_\e$, we have
\begin{equation}\label{formula-lemma}
\min\{E^\sigma_\e(u^\e,v;[0,L])\!:  v\in 
\mathcal A_\e^{\#}(u^\e)\}=F^\sigma_\e(u^\e; [0,L])+\|u^\e\|^2_{L^2}\ o(1)_{\e\to 0}
\end{equation}
where $\mathcal A_\e^{\#}(u^\e)=\{v\in 
\mathcal A_\e: v_0=v_1=u^\e_0, \ v_{k_\e}=v_{k_\e-1}=u^\e_{k_\e}\}$. 
\end{lemma}

\begin{proof}
For $u,v\in \mathcal A_\e$, we set 
\begin{eqnarray} 
&&\hskip-1.2cm 
H_\e(u,v)=\frac{a_\sigma}{\e}\sum_{i=1}^{k_\e}(v_i-v_{i-1})^2+
\frac{b_\sigma}{\e}\sum_{i=1}^{k_\e}(u_i-v_i)^2
=E_\e^\sigma(u,v;[0,L])-\e\sum_{i=1}^{k_\e}f(u_i-u_{i-1})\nonumber\\
&&\hskip-1.2cm 
J_\e(u)=\frac{1}{\e}\sum_{i,j=0}^{k_\e} e^{-\sigma |i-j|}(u_i-u_j)^2
=F_\e^\sigma(u,v;[0,L])-\e\sum_{i=1}^{k_\e}f(u_i-u_{i-1}).\nonumber
\end{eqnarray}
Up to translations, we can assume $u^\e_0=0$ (and hence $u^\e_{i}=0$ for $i\leq L\e^{-\alpha}$).
We introduce the $(k_\e+1)\times (k_\e+1)$ matrix $M^\e_\sigma=(m_{ij})$ 
given by $m_{ij}=m^\sigma_{|i-j|}=e^{-\sigma |i-j|}$, $i,j=0,\dots, k_\e$. 
Note that $m_n^\sigma=e^{-\sigma n}$ satisfies $m_n^\sigma=o(n^{-\beta})$ for any $\beta$ and in particular for $\beta>\frac{3}{\alpha}$.  
In order to apply Lemma \ref{2v-lemma}, 
we compute $s_\sigma=s_{\bf m^\sigma}$ and the matrix $A_\sigma^\e=A^{k_\e}_{\bf m}$ given by formula \eqref{formulaA}, obtaining   
\begin{equation}\label{formulaAs}
s_{\sigma}=m^\sigma_0+2\sum_{n=1}^{+\infty}m^\sigma_n=\frac{1+e^{-\sigma}}{1-e^{-\sigma}}
\ \ \ 
\hbox{\rm and } \ \ \ A_{\sigma}^\e=D^\e_\sigma (M^\e_\sigma)^{-1}-I, 
\end{equation}
where $D^\e_\sigma$ is the $(k_\e+1)\times(k_\e+1) $ diagonal matrix with diagonal $\{1+e^{-\sigma}, s_\sigma, \dots, s_\sigma, 1+e^{-\sigma}\}$. 
Moreover, in this case we can compute the inverse of the matrix $M^\e_\sigma$, which is the tridiagonal $(k_\e+1)\times (k_\e+1)$ matrix 
given by 
\begin{equation}\label{inversa}
(M^\e_\sigma)^{-1}=
\frac{1}{1-e^{-2\sigma}}\begin{pmatrix}
1&-e^{-\sigma}&0&\dots&&0\\
-e^{-\sigma}&1+e^{-2\sigma}&-e^{-\sigma}&0&\dots&0\\
0&-e^{-\sigma}&1+e^{-2\sigma}&-e^{-\sigma}&0&\dots\\
\dots&\dots&\dots&\dots&\dots&\dots\\
0&\dots&\dots&0&-e^{-\sigma}&1
\end{pmatrix}.\end{equation} 

Now, to each $u\in \mathcal A_\e$ we associate the 
corresponding function defined on $\{0,\dots, k^\e\}$ by $i\mapsto u(\e i)$; with a slight abuse of notation, we still denote this function by $u$. 
Setting  
$$H^{k_\e}_\sigma(u,v)=
\frac{2s_{\sigma}}{\e}\langle A^\e_{\sigma} v,v\rangle+\frac{2s_{\sigma}}{\e}\langle u-v, u-v\rangle$$
for $u,v\colon \{0,\dots, k^\e\}\to\mathbb R$, 
we can then apply Lemma \ref{2v-lemma} 
with $k=k^\e$, obtaining  
\begin{eqnarray}\min\{H^{k_\e}_\sigma(u^\e,v):\  v\colon\{0,\dots,k_\e\}\to\mathbb R\}
=J_\e(u^\e)+\|u^\e\|_{L^2}^2\ o(1)_{\e\to 0}. \label{Jmin}
\end{eqnarray}
We conclude by proving that, up to an infinitesimal term, the minimum of $H_\sigma^{k_\e}(\tilde u^\e,\cdot)$ 
on $\mathcal A_\e$ coincides with the minimum of $H_\e(u^\e,\cdot)$ 
on $\mathcal A^{\#}_\e$. 
Indeed, given $u,v\in\mathcal A_\e$ we can write 
\begin{eqnarray}\label{hh}
H_\sigma^{k_\e}(u,v)
&=&-\frac{s_\sigma}{\e} \sum_{i,j=0}^{k_\e}(A_\sigma^\e)_{ij}(v_i-v_j)^2 
+\frac{2s_\sigma}{\e}\sum_{i=0}^{k_\e}\Big(\sum_{j=0}^{k_\e}(A_\sigma^\e)_{ij}\Big)v_i^2 +
\frac{2s_\sigma}{\e}\sum_{i=0}^{k_\e}(u_i-v_i)^2\nonumber \\
&=&\frac{2(1+e^{-\sigma}) e^{-\sigma}}{\e (1-e^{-\sigma})^3} \sum_{i=1}^{k_\e}(v_i-v_{i-1})^2 
+
\frac{2(1+e^{-\sigma})}{\e(1-e^{-\sigma})}\sum_{i=0}^{k_\e}(u_i-v_i)^2\nonumber\\ 
&=& H_\e(u,v),  
\end{eqnarray} 
since $\sum_{j=0}^{k_\e}(A_\sigma^\e)_{ij}=0$ for any $i$ 
by \eqref{formulaAs} and \eqref{inversa}.  
This formula in particular implies 
\begin{equation*}
\langle A^\e_{\sigma} v,v\rangle
=\frac{e^{-\sigma}}{(1-e^{-\sigma})^2} \sum_{i=1}^{k_\e}(v_i-v_{i-1})^2;  
\end{equation*} 
that is, estimate \eqref{stiCA} with $C=\frac{e^{-\sigma}}{(1-e^{-\sigma})^2}$, which in this case is an equality.  

Finally, recalling Remark \ref{ossv} we obtain  
\begin{equation*}
\min\{H^{k_\e}_{\sigma}(u^\e,v): v\colon\{0,\dots, k_\e\}\to\mathbb R\}=\min \{H_\e(u^\e,v): v\in \mathcal A_\e^{\#}(u^\e)\}+\|u^\e\|^2_{L^2}\ o(1)_{\e\to 0} 
\end{equation*}   
and the claim follows by \eqref{Jmin}.  
\end{proof}

\begin{proof}[Proof of Theorem {\rm\ref{teo-equivalence-exp}}] 

\medskip 

\noindent {\it Upper estimate.}  
Let $F^\sigma(u;[0,L])$ be the $\Gamma$-limit of the sequence $F_\e^\sigma$. 
Let $u\in L^2(0,L)$ be such that $F^\sigma(u;[0,L])<+\infty$ and let $u^\e\in \mathcal A_\e$ be a recovery sequence for the $\Gamma$-limit $F^\sigma(u;[0,L])$. 
Let $\hat u^\e$ be the sequence given by Lemma \ref{boundarycond} and $v^{\e, {\rm min}}$ be obtained by minimization of the minimum problem in \eqref{formula-lemma} with $u^\e=\hat u^\e$.  
Recalling Lemma \ref{lemma-equivalence}, we get
\begin{eqnarray*} 
\limsup_{\e\to 0} E^\sigma_\e(\hat u^{\e}, v^{\e, {\rm min}};[0,L])
&\leq&
\limsup_{\e\to 0}F^\sigma_\e(\hat u^\e,[0,L])\\
&\leq& \limsup_{\e\to 0} F^\sigma_\e(u^{\e};[0,L]).  
\end{eqnarray*}
This gives the upper estimate for the $\Gamma$-limit of $E^\sigma_\e$. 

\medskip 

\noindent {\it Lower estimate.} 
Let $u\in H^1(0,L)$ and let $u^\e, v^\e$ converge to $u$ in $L^2(0,L)$ and be such that 
$\sup E^\sigma_\e(u^\e,v^\e; [0,L])\leq S<+\infty.$ Let $\hat u^\e, \hat v^\e$ be the sequences given by Lemma \ref{boundarycond}(B).
Hence 
\begin{equation}\label{limsup-G-costanti}
\liminf_{\e\to 0}E^\sigma_\e(\hat u^\e,\hat v^\e; (0,L))\leq \liminf_{\e\to 0}E^\sigma_\e(u^\e,v^\e; (0,L)).
\end{equation}
Applying Lemma \ref{lemma-equivalence} we obtain 
\begin{eqnarray*}
\liminf_{\e\to 0}E^\sigma_\e(u^\e,v^\e; [0,L])&\geq& \liminf_{\e\to 0}E^\sigma_\e(\hat u^\e,v^{\e,{\rm min}}(\hat u^\e); [0,L])\\
&\geq& \liminf_{\e\to 0}F^\sigma_\e(\hat u^\e; [0,L]).
\end{eqnarray*}
This concludes the proof. 
\end{proof}

By the results in Section \ref{hidise} we can use the equivalence above to give a useful characterization of  $\widehat Q_\sigma f$.

\begin{remark}[representation of  $\widehat Q_\sigma f$ in terms of local functionals]\label{localQ}\rm
Formula \eqref{equiv-2v} in Theorem \ref{teo-equivalence} 
and equality \eqref{hh} prove the following formula for the function $\widehat Q_{\sigma}f$ defined in \eqref{def-tsigmahat}:  
\begin{equation}\label{trs}
\widehat Q_{\sigma}f(z)=\lim_{N\to+\infty}\frac{1}{N}\min\{E_1^\sigma(u,v;[0,N]): u_0=v_0=0, u_N=v_N=Nz\},
\end{equation}
where $E_1^\sigma$ is defined by \eqref{def-2var} with $\e=1$ and $a=a_\sigma,b=b_\sigma$ satisfying \eqref{aeb}.  
\end{remark}

\begin{remark}[representation of the constrained relaxation in terms of local functionals]\label{polig}
\rm 
Formula \eqref{trs} can be extended to constrained problems; namely, we have 
\begin{equation}\label{trs-con} 
\widehat Q_{\sigma}f\Big(\frac{p}{q},z\Big)=
\liminf_{k\to+\infty}\frac{1}{kq}\min\Big\{E_1^\sigma(u,v;[0,kq]): u, v\in\mathcal A(kq;z), u\in \mathcal V\Big(kq;\frac{p}{q}\Big)\Big\},
\end{equation}
where, accordingly with the notation above, $\widehat Q_\sigma f(\theta,z)$ denotes the constrained relaxation 
 $\widehat Q_{{\bf m}_\sigma} f(\theta,z)$, and 
$\mathcal V(kq;\frac{p}{q})$ is the set of admissible constrained functions defined in \eqref{vtheta}. 
Indeed, we note that Theorem \ref{teo-equivalence} also holds for constrained relaxation, 
since we can apply Lemma \ref{2v-lemma} to $u$ satisfying a volume constraint  
(see Lemma \ref{boundary-Q}).  
\end{remark}

\begin{remark}[non-exponential kernels]\label{top1}\rm  For a general kernel $\bf m$ the matrix  
$M^k=(m_{ij})_{i,j=0}^k$ is a symmetric {\em Toeplitz matrix}. Under decay conditions on $m_n$ we can apply the arguments in Section \ref{hidise}. However, since $(M^k)^{-1}$ now is not of the form \eqref{inversa} (for some insight on the problem of the inversion of a general symmetric Toeplitz matrix we refer, e.g., to \cite{BG}), the resulting functional $H^k_\e$ does not depend on nearest neighbours only and the argument showing the optimality of the bounds can not be completed as above. 
However, for particular  classes of kernels $\bf m$ the resulting functionals $H^k_\e$  may be still amenable to  analysis, even if they involve next-to-nearest-neighbour interactions and beyond. The analytical transparency  of such functionals   will  then allow one to extract useful  information on the form of the corresponding $\widehat Q_{\bf m}f$.
\end{remark}

\subsection{Truncated convex potential}\label{truco:sect}
In this section we show some properties of 
$\widehat Q_\sigma f$ and of the corresponding phase function $\theta$ if $f$ is a general truncated convex function; that is,  
\begin{equation}\label{f-frattura}
f(z)=\begin{cases}
\tilde f(z) & \hbox{\rm if } z\leq z^\ast\\
\tilde f(z^\ast) & \hbox{\rm if } z>z^\ast,
\end{cases}
\end{equation} 
where $z^\ast>0$ and $\tilde f\colon\mathbb R\to [0,+\infty)$ is strictly convex and such that $\tilde f(0)=0$.
Note that we can suppose that $\tilde f$ satisfies the growth condition 
$$\tilde f(z)\geq c_1z^2-c_2 $$
in $[0,+\infty)$ 
for some $c_1,c_2>0$. Using the notation of Section \ref{constraint-sec}, we set $A=[z^\ast,+\infty)$.   

\begin{remark}[more general $f$]\label{generalcond}\rm
Note that the condition $\tilde f(0)=0$ can be substituted by the hypothesis that $\tilde f$ has a minimum point $z_{\rm min}<z^*$, since affine changes of variables are compatible with the definition of $\widehat Q_{\bf m}f$ by Remark \ref{properties}.
\end{remark}

\subsubsection{Characterization of $\widehat Q_\sigma f$ in terms of periodic arrangements}
Given  the local form of the problem \eqref{trs} formulated in terms of the two-variable functional $E_1^\sigma(u,v;[0,N])$, the relaxed energy  $\widehat Q_{\sigma} f$ can be obtained by optimizing the location of `broken bonds'; that is, of indices $i$ such that $u_i-u_{i-1}\in A$, similarly to what done in the case of concentrated kernels. 
The fact that these bonds can be always considered as either isolated 
or organized in a `broken island' makes the structure of oscillations 
(microstructure) compatible with the lattice. 
This makes the problem analytically tractable.

Note first that on the complement of the broken bonds  the energy coincides with its `convex part', defined as follows. 
Given $a,b>0$, for a bounded interval $I$ and $u,v\in \mathcal A_\e(I)$ 
we introduce the functional $\tilde E_\e$ given by 
\begin{equation}\label{def-etilden}
\tilde E_\e (u,v;I)=\e\sum_{i\in \mathcal I^\ast_\e(I)}\tilde f\big(\frac{u_i-u_{i-1}}{\e}\big)+
\frac{a}{\e}\!\sum_{i\in \mathcal I^\ast_\e(I)}\!(v_i-v_{i-1})^2
+\frac{b}{\e}\!\sum_{i\in \mathcal I_\e(I)}\!(u_i-v_i)^2,
\end{equation}
where we recall that $\mathcal I_\e^\ast=\{i\in\mathbb Z: \e i,\e(i-1)\in I\}$. 
Note that, since these energies will be used to compute minimum problems with Dirichlet boundary conditions, we consider the last term of the sum in the whole $\mathcal I_\e(I)=\{i\in\mathbb Z: \e i\in I\}$.  

In view of Section \ref{hidise}, for all $N\geq 2$ we can write the functions $g_N$ introduced in \eqref{gNgen} 
with $\eta$ replaced by $\tilde f(z^\ast)$ as  
\begin{equation}\label{def-gm} 
g_N^{a,b}(z)=g_N(z)
=\frac{1}{N}\big(\tilde f(z^\ast)+
\min\big\{av_1^2+\tilde E_1(u,v;[1,N]): v_N=Nz\big\}\big).
\end{equation}
They represent  the minimal energy of an array of $N$ bonds, of which the first one is broken, with given average gradient. 
By uniformity of notation, we also set 
\begin{equation}\label{g1ginfty}
g_1(z)=\tilde f(z^\ast)+az^2\quad \hbox{\rm and} \quad g_\infty(z)=\tilde f(z)+az^2.
\end{equation}
If $a=a_\sigma$ and $b=b_\sigma$ are given by \eqref{aeb}, then we set 
$$g_{N}^\sigma(z)=g_{N}(z) \ \ \ \hbox{\rm and }\ \ \ \tilde E_\e^\sigma(u,v;I)=\tilde E_\e(u,v;I).$$ 
Note that, by using $u_i=v_i=zi$ as test function in the definition of $g^\sigma_N(z)$, we get 
$$\lim_{N\to+\infty}g^\sigma_N(z)\leq\tilde f(z)+a_\sigma z^2.$$

In the following proposition, based on the analysis of the distribution of broken bonds in minimizers, 
we show that $\widehat Q_{\sigma}f(z)$, considered as the infimum of the corresponding constrained functions, can be described by only using the values $\theta=\frac{1}{N}$, which will be proved to be the locking states. The full description of this structure 
will be given in Proposition \ref{enwellfract}, after a delicate analysis of the general properties of $g_N$.  

\begin{proposition}[characterization of $\widehat Q_\sigma f$ in terms of periodic arrangements]\label{celle} 
Fixed $\sigma>0$, let $a=a_\sigma$ and $b=b_\sigma$ be given by \eqref{aeb}.  
If $f$ is a truncated convex potential as in \eqref{f-frattura}, then 
\begin{equation}
\widehat Q_{\sigma}f(z)=\Big(\inf_{N\in \mathbb N} \{g^\sigma_N\}\Big)^{\ast\ast}(z)
\,.
\end{equation} 
\end{proposition}
\begin{remark}
\rm Note that, recalling Remark \ref{absigma}, Proposition \ref{celle} holds  for any $a,b>0$ with $g^{a,b}_N$ in place of  $g^\sigma_N$ and $a$ in place of $a_\sigma$, up to substituting $\widehat Q_{\sigma}f$ with 
$\widehat Q_{\sigma_{a,b}, \varrho_{a,b}}f$ as defined in 
\eqref{def-tsigmar}, with $\sigma_{a,b}$ and $\varrho_{a,b}$ given by \eqref{absigmarho}.     
\end{remark} 

\begin{proof}[Proof of Proposition {\rm\ref{celle}}]
The lower bound is a consequence of Proposition \ref{lobaloba}.
To conclude the proof we show that 
$\widehat Q_\sigma f(z)\leq(\inf_{n\in \mathbb N} \{g^\sigma_n\})^{\ast\ast}(z)$. Since $\widehat Q_\sigma f$ is convex, it is sufficient to prove that $\widehat Q_\sigma f(z)\leq\inf_{n\in \mathbb N} \{g^\sigma_n(z)\}$. 

We fix $\delta>0$. 
For $z\in \mathbb R$ there exists $\overline n\in \mathbb N$ such that 
$g^\sigma_{\overline n}(z)\leq \inf_{n\in \mathbb N} \{g^\sigma_n(z)\}+\delta$. 
If $\overline n=1$, then 
we can take as test functions $u,v$ given by $u_i=v_i=iz$. 
For any $N\geq 1$ we get  
\begin{equation*}
\frac{1}{N}E^\sigma_1(u, v; [0,N])\leq \frac{1}{N}\tilde E^\sigma_1(u, v; [0,N])=\tilde f(z)+a_\sigma z^2
= g^\sigma_1(z)
\end{equation*}
and the result follows by taking the limit for $N\to+\infty$.
Otherwise, 
let $\overline u,\overline v\in\mathcal A_1([1,\overline n])$ be such that $\overline v_{\overline n}=\overline n z$ and  
$$\tilde f(z^\ast)+a_\sigma\overline v_1^2+\tilde E^\sigma_1(\overline u,\overline v;[1,\overline n])=\overline n\, g^\sigma_{\overline n}(z).$$
We extend $\overline u$ and $\overline v$ in $0$ by setting $\overline u_0=\overline n z-\overline u_{\overline n}$ and $\overline v_0=0$. It follows that    
\begin{eqnarray*}
E^\sigma_1(\overline u,\overline v;[0,\overline n])&\leq& E^\sigma_1(\overline u,\overline v; [1,\overline n])+b_\sigma(\overline u_1-\overline v_1)^2+a_\sigma\overline v_1^2+\tilde f(z^\ast)\\
&\leq&\tilde E^\sigma_1(\overline u,\overline v; [1,\overline n])+a_\sigma\overline \lambda^2+\tilde f(z^\ast)\\
&=& \overline n \,g^\sigma_{\overline n}(z). 
\end{eqnarray*}
For any $N\geq 1$ we choose $u^N$ and $v^N$ as test functions in $[0,\overline n N]$ 
 defined by setting $u^N_i$ equal to $(j-1)\overline n z+\overline u_{i-(j-1)\overline n}$ in each $[(j-1)\overline n, j\overline n)$, $j\in\{1,\dots, N-1\}$ and in $[(N-1)\overline n, N\overline n]$ and correspondingly $v^N_i$.  
We get 
\begin{eqnarray*}
\frac{1}{\overline n N}E^\sigma_1(u^N,v^N; [0,\overline n N])=\frac{1}{\overline n N} N E^\sigma_1(\overline u,\overline v;[0,\overline n])
\leq g^\sigma_{\overline n}(z)\leq \inf_{n\in \mathbb N} \{g^\sigma_n(z)\}+\delta.
\end{eqnarray*}
Letting $N\to+\infty$ the claim follows by the representation formula for $\widehat Q_\sigma f$ given in  \eqref{trs}.  
\end{proof}

\begin{figure}[h!]
\centerline{\includegraphics[width=0.3\textwidth]{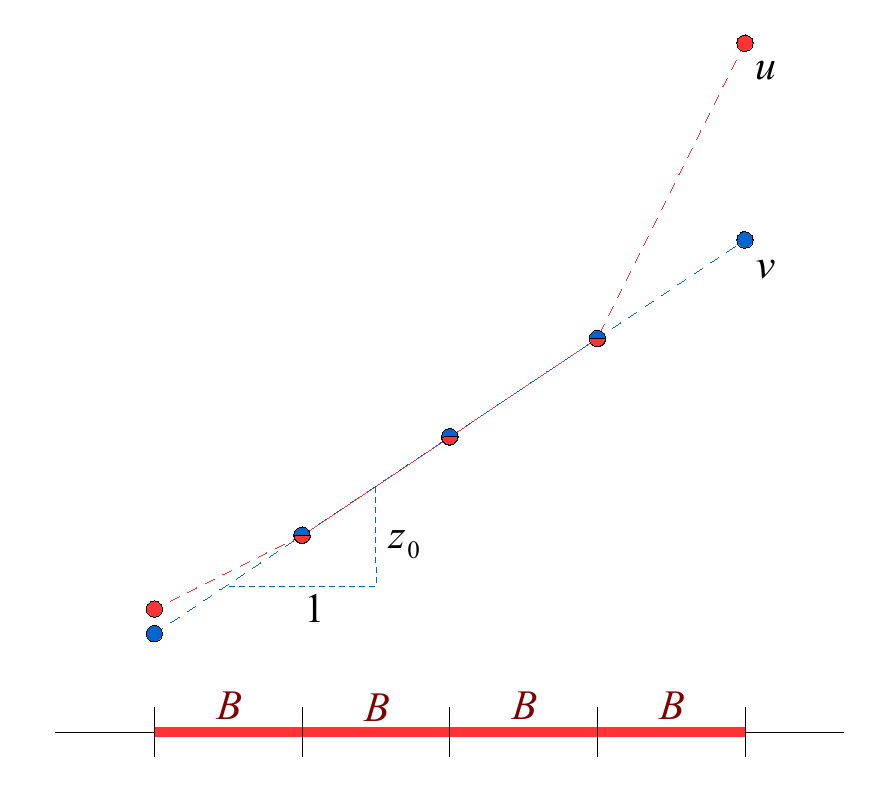}}
\caption{Shape of a minimizer of (\ref{min-e1}) in a `broken island'.}
\label{affine_equal}
\end{figure} 

\begin{remark}[simplification of the minimal configurations]\rm
Given $u\in \mathcal A_1([0,N])$, we
say that $i\in\{1,\dots, N\}$ belongs to $\mathcal B(u)$ (the set of {\it broken} indices of $u$) 
if $u_i-u_{i-1}>z^\ast$. 

For future reference we show that the solutions of 
\begin{equation}\label{min-e1}
\min\big\{E^\sigma_1(u,v; [0,N])\!: 
v_0\!=\!0, v_N\!=\!Nz, \#\mathcal B(u)=n\big\}
\end{equation} 
can be regrouped and rearranged.
Let $(u,v)$ solve (\ref{min-e1}). Note that in the union of the non-isolated `broken intervals' we can assume that $u$ and $v$ are affine and equal. More precisely, the convexity of the square and a translation argument allow to prove that there exists $z_0$ such that if $\underline i +k+1\in \mathcal B(u)$
for $k\in\{0,\dots, \underline k\}$, with $\underline k\geq 1$ then 
\begin{eqnarray*}
&&v(\underline i+k)=v(\underline i)+z_0 k \ \hbox{ for } \ k=0, \dots, ..., \underline k+1  \\
&&v(\underline i+k)=u(\underline i+k) \ \hbox{ for } \ k=1, \dots, \underline k
\end{eqnarray*}
(see Figure \ref{affine_equal}). 
As a second step, we show that if $(u,v)$ solves \eqref{min-e1}     
we can assume that there is at most one `broken zone' for $u$ with length greater than $1$. 
To this end, we extend $u$ and $v$ by periodicity by setting $u(N+j)=u(j)+Nz$ and $u(-j)=u(N-j)-Nz$
for $j=1,\dots, N$,  and correspondingly for $v$. 

\begin{figure}[h!]
\centerline{\includegraphics[width=0.9\textwidth]{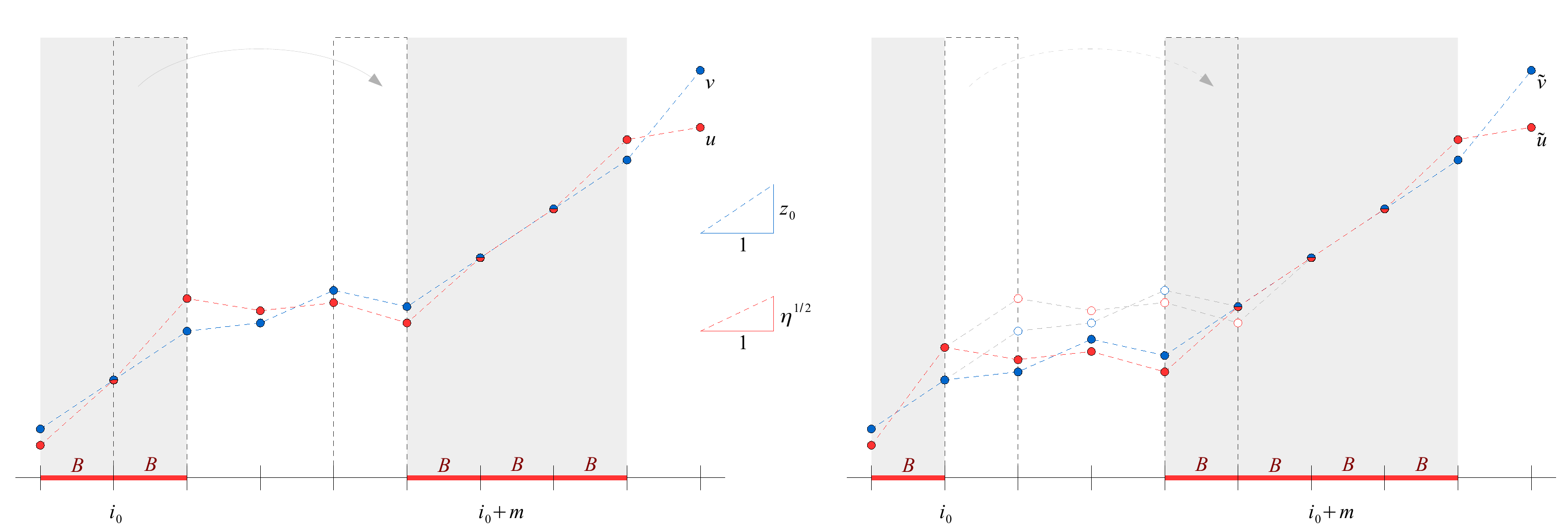}}
\caption{Construction of $(u,v)$ with isolated broken bonds.}
\label{nonbroken_islands}
\end{figure} 
Now we show that the minimum is attained at $(u,v)$ such 
that if $i\in \mathcal B(u)$, then $i-1\not\in \mathcal B(u)$ and $i+1\not\in \mathcal B(u)$, or
$j\in \mathcal B(u)$ for all $j\in \{i+1,\dots, N\}$. 
To show this, we suppose that $i_0,i_0+1$, $i_0+k$ and $i_0+k+1$ belong to $\mathcal B(u)$
for some $i_0\geq 1$, $k\geq 2$ and $i_0+k\leq N$, while $i_0+j+1\not\in \mathcal B(u)$ for $j\in\{1,\dots, k-2\}$. 

We modify $u$ and $v$ by setting for $j=0,\dots, k-1$
\begin{equation*}
\tilde u(i_0+j)=u(i_0+j+1)-z_0 \ \hbox{ and }\ \tilde v(i_0+j)=v(i_0+j+1)-z_0 
\end{equation*}
(see Figure \ref{nonbroken_islands}).
With this definition 
$$
E^\sigma_1(\tilde u,\tilde v; [0,N])\leq E^\sigma_1(u, v; [0,N]).
$$
\begin{figure}[h!]
\centerline{\includegraphics[width=0.9\textwidth]{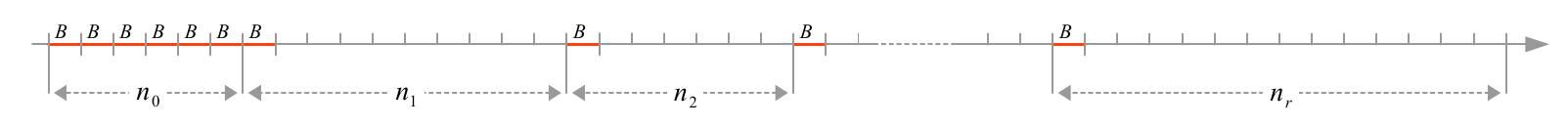}}
\caption{Distribution of broken bonds. }
\label{islands}
\end{figure} 
Thanks to the periodic extension of $u$ and $v$, 
this proves 
that in minimum problem~(\ref{min-e1}) 
we can assume that there exist $n_0,n_1,..., n_r\in \mathbb N$ with 
$n_{l}>1$ for any $l\in\{1,\dots, r\}$, $r+n_0=n(N,z)$ and $n_0=N-\sum_{l=1}^r n_l$,
such that
\begin{equation}\label{rotti-bene}
\displaystyle i\in \mathcal B(u) \ \hbox{ for all } \ i \in\{1, \dots, n_0\}\ 
\hbox{ and } \sum_{l=1}^jn_l+1\in \mathcal B(u) \hbox{ for all }\ j\in\{1,\dots, r\} 
\end{equation}
 (see Figure \ref{islands}).

This reduces the problem of the computation of the minimum value (\ref{min-e1}) to the solution of the minimum problem on each (translated) island $[0,n_j],$ $j\in\{1,\dots, r\}$,
$$\min\Big\{E^\sigma_1(u,v; [0,n_j]): v_0\!=\!0, v_{n_j}\!=\!z_jn_j,\ \mathcal B(u)=\{1\}\Big\}
$$
and in the broken island $[0,n_0]$, where  
\begin{equation}\label{m0}
\min\Big\{E^\sigma_1(u,v; [0,n_0]): v_0\!=\!0, v_{n_0}\!=\!z_0n_0,\ \#\mathcal B(u)=n_0\Big\}=n_0g_1(z_0), 
\end{equation}
with suitable boundary conditions $z_j$ satisfying $\sum_{j=0}^r n_jz_j=Nz$.

Since $E^\sigma_1(u,v;[0,n-1])=\tilde E^\sigma_1(u,v;[0,n-1])-b_\sigma(u_0-v_0)^2$ if $\#\mathcal B(u)=0$, for $n>1$ and $z\in\mathbb R$ we have 
\begin{eqnarray*}
&&\hspace{-1cm}\min\Big\{E^\sigma_1(u,v; [0,n]) : v_0\!=\!0, v_{n}\!=\!nz,\ \mathcal B(u)=\{1\}\Big\}\\
&& \geq \min_{w\in\mathbb R}
\Big\{\min\Big\{\tilde E^\sigma_1(u,v; [1,n]): v_1=w, v_{n}=nz\Big\}+ a_\sigma w^2+\tilde f(z^\ast)\Big\}
\\
&&=n g^\sigma_n(z).
\end{eqnarray*}
\end{remark}

\subsubsection{General properties of the periodic bounds $g_N$} 
In order to relate the constrained relaxation $\widehat Q_\sigma f(\theta, z)$ to $g^\sigma_N(z)$ 
and to characterize the locking states of $f$, we analyze the properties of $g^\sigma_N(z)$ in dependence on both $N$ and $z$. Note that in the following results we may consider general values of $a,b>0$ and not limit to $a_\sigma,b_\sigma$, so that the results of this section hold for a general $g_N$ as defined in \eqref{gNgen}.

\begin{proposition}[convexity of $g_N$]\label{convinN}
The functions $g_N$ are uniformly strictly convex. More precisely, we have
 \begin{equation}\label{unifconv}
 {1\over2} g_N(z)+{1\over2} g_N(z')\ge g_N\Bigl({z+z'\over 2}\Bigr)+a\Bigl({z-z'\over 2}\Bigr)^2
 \end{equation}
 for all $z,z'\in\mathbb R$ and $N\in\mathbb N$.
 \end{proposition}
\begin{proof}
If $u,v$ and $u',v'$ are minimizers for $g_N(z)$ and $g_N(z')$ we can use the functions ${1\over2}(u+u'),{1\over2}(v+v')$ as test functions for $g_N({1\over2}(z+z'))$. Using the convexity of $\tilde f$ and the quadraticity of the other terms;
 more precisely, that  for all $i$ we have (after setting $v_0=0$)
 $$
 a(v_i-v_{i-1})^2+a(v'_i-v'_{i-1})^2
 ={a\over 2}((v_i+v'_i)-(v_{i-1}+v'_{i-1}))^2+{a\over 2}((v_i-v_{i-1})-(v'_i-v'_{i-1}))^2,
 $$
  we get
 \begin{eqnarray*}
 {1\over2} g_N(z)+{1\over2} g_N(z') &\ge&  g_N\Bigl({z+z'\over 2}\Bigr)+{1\over N}{a\over 4} \sum_{i=1}^N((v_i-v_{i-1})-(v'_i-v'_{i-1}))^2\\
 &\ge&  g_N\Bigl({z+z'\over 2}\Bigr)+a \Bigl({1\over 2}{1\over N}\sum_{i=1}^N((v_i-v_{i-1})-(v'_i-v'_{i-1}))\Bigr)^2
 \\
 &=&  g_N\Bigl({z+z'\over 2}\Bigr)+a \Bigl({z-z'\over 2}\Bigr)^2\,,
  \end{eqnarray*}
 as desired. 
\end{proof}

\begin{remark}\label{min-a}\rm
From the previous proposition we deduce that $g_N''(z)\ge 2a$ at all $z$ where $g_N$ is twice differentiable.
In particular, we obtain that $g_N(z)\ge {\tilde f(z^\ast)\over N} + az^2$ for all $N\ge 1$.
\end{remark}

\begin{remark}[symmetry of solutions]\label{gnprop}\rm
The solutions $u,v$ of the minimum problem
\begin{equation}\label{pro-si}
\min\bigg\{\sum_{i=2}^N\tilde f ({u_i-u_{i-1}} )+
a\sum_{i=2}^N(v_i-v_{i-1})^2
+b\sum_{i=1}^N(u_i-v_i)^2
: v_1=v^1, v_N=v^N\bigg\}
\end{equation}
are symmetric with respect to the centre of the interval, in the sense that
\begin{equation}\label{symm-1}
v_{j+1}-v_j= v_{N-j+1}- v_{N-j}, \qquad u_{j+1}-u_j= u_{N-j+1}- u_{N-j}
\end{equation}
for $1\le j\le N-1$. Furthermore, if $N=2M+1$ is odd then
\begin{equation}\label{symm-2}
v_{M+1}=u_{M+1}= {v^N+v^1\over 2}
\end{equation}
while, if $N=2M$ is even then 
\begin{equation}\label{symm-3}
{v_{M+1}+v_{M}\over 2}={u_{M+1}+u_{M}\over 2}= {v^N+v^1\over 2}.
\end{equation}

Indeed, first note that we may state the boundary condition equivalently as $v_N-v_1= V:=v^N-v^1$.
Then, condition \eqref{symm-1} is a direct consequence of the strict convexity of the energy and is obtained using
\begin{equation}\label{symm-4}
\overline v_i= {v_i-v_{N+1-i}\over 2},\qquad \overline u_i= {u_i-u_{N+1-i}\over 2}
\end{equation}
as test functions. 
To check, e.g., \eqref{symm-2}, note that from \eqref{symm-1}
\begin{eqnarray*}
v_{M+1}&= &v_1+\sum_{j=1}^{M}(v_{j+1}-v_j)=v_1+\sum_{j=1}^{M}(v_{N-j+1}- v_{N-j})\\
&=&v_1+\sum_{k=M+1}^{N-1}(v_{k+1}- v_{k}) = v_N-v_{M+1}+v_1,
\end{eqnarray*}
from which the first equality in \eqref{symm-2} follows. 
To check the second one, note that from \eqref{symm-1} we obtain 
$v_i+v_{2M+2-i}-2v_{M+1}=u_i+u_{2M+2-i}-2u_{M+1}=0$ for all $i$,
from which 
\begin{equation}\label{symm-5}
{1\over N}\sum_{i=1}^Nu_i= u_{M+1},\qquad {1\over N}\sum_{i=1}^Nv_i= v_{M+1}.
\end{equation}
Now, considering in place of $u_i$ the function
$$
\overline u_i= u_i+{v^1+v^N\over 2}-u_{M+1},
$$
as test functions, the only change in the problem in \eqref{pro-si} is in the last sum,
for which, using \eqref{symm-5} and the already proved equality in \eqref{symm-2} for $v$, we have
\begin{eqnarray*}
\sum_{i=1}^N(\overline u_i-v_i)^2=
 \sum_{i=1}^N(u_i-v_i)^2-N\Bigl({v^1+v^N\over 2}-u_{M+1}\Bigr)^2,
\end{eqnarray*}
which contradicts the minimality of $u,v$ if the second equality in \eqref{symm-2} does not hold.
The proof of \eqref{symm-3} follows the same line with minor modifications.
\end{remark}

\begin{proposition}[convexity properties with respect to $N$ with given parity]\label{gin1n2}
For all $N_1,N_2\ge1$ such that $N_1+N_2$ is even and $N_1\neq N_2$, for all $z_1,z_2\in\mathbb R\setminus \{0\}$ we have
\begin{equation}\label{conf-gn}
{N_1\over N_1+N_2}g_{N_1}(z_1)+ {N_2\over N_1+N_2}g_{N_2}(z_2)\ > \ g_N(z),
\end{equation}
where 
$N=\frac{N_1+N_2}{2} \ \ \hbox{\rm and }\ z=\frac{N_1z_1+N_2z_2}{N_1+N_2}$. 
 In particular, we have the convexity property in $N$
 \begin{equation}\label{conf-gnz}
{N_1\over N_1+N_2}g_{N_1}(z)+ {N_2\over N_1+N_2}g_{N_2}(z)\ > \ g_N(z),\hbox{ where }N={N_1+N_2\over2} \hbox{ and } z\neq 0.
\end{equation}
\end{proposition}
\begin{proof} We consider the case of $N_1$ and $N_2$ odd, the case of $N_1$ and $N_2$ even following the same line with minor modifications. 
\begin{figure}[h!]
\centerline{\includegraphics[width=0.9\textwidth]{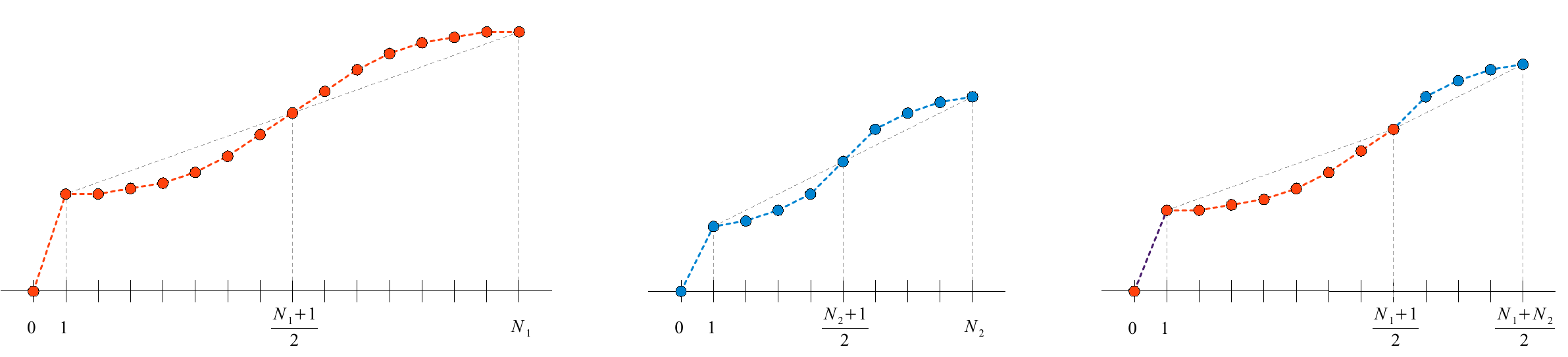}}
\caption{construction of the test function $\overline v$.}
\label{figure_regroupconv}
\end{figure} 
Let $u^1,v^1$ be minimizers for $g_{N_1}(z_1)$ and let $u^2,v^2$ be minimizers for $g_{N_2}(z_2)$. 
We  define $\overline u,\overline v$ by setting 
$$
\overline v_i=\begin{cases} v^1_i +{v^2_1-v^1_1\over 2}&\hbox{ if } 1\le i\le {N_1+1\over 2}\\
v^2_{i+\frac{N_2-N_1}{2}}+\frac{1}{2}(N_1z_1+N_2z_2) &\hbox{ if } i\ge {N_1+1\over 2},\end{cases}
$$
$$
\overline u_i=\begin{cases} u^1_i +{v^2_1-v^1_1\over 2}&\hbox{ if } 1\le i\le {N_1+1\over 2}\\
u^2_{i+\frac{N_2-N_1}{2}}+\frac{1}{2}(N_1z_1+N_2z_2) &\hbox{ if } i\ge {N_1+1\over 2}
\end{cases}
$$
(see Fig.~\ref{figure_regroupconv}). 
Thanks to Remark \ref{gnprop} this is a good definition, $\overline v_{(N_1+1)/ 2}=\overline u_{(N_1+1)/ 2}$, and we have $\overline v_N=\frac{1}{2}(N_1z_1+N_2z_2)$, so that
these are test functions for $g_N(z)$.
Again, by the symmetry properties of $v^1$ and $v^2$ in Remark \ref{gnprop} we obtain \eqref{conf-gn}. Note the strict inequality, which is proved by noting that $
\overline v_i,
\overline u_i$ do not satisfy the properties of minimizers in Remark \ref{gnprop}. 
\end{proof}

From Proposition \ref{gin1n2} we deduce a general convexity property which holds also if $N_1$ and $N_2$ have different parity. Note that this implies that fractures will be equidistributed up to oscillations of a unit, due to incommensurability phenomena.  
\begin{corollary}[convexity properties with respect to arbitrary $N$]\label{conv-disp} 
Let $k, N\geq 2$ be integers, and $w_k,w_0\in\mathbb R\setminus \{0\}$. Then  
\begin{equation}\label{equispace}
(N+k)g_{N+k}(w_k)+Ng_{N}(w_0)>(N+k-1)g_{N+k-1}(w_{k-1})+(N+1)g_{N+1}(w_1)
\end{equation}
for some $w_{k-1}, w_1$ such that $ (N+1)w_{1}+(N+k-1)w_{k-1}=Nw_{0}+(N+k)w_k$. 
Moreover $w_{k-1}, w_1$ belong to the interval with endpoints $w_0$ and $w_k$.
\end{corollary}
\begin{proof}
Let $(w_1,\dots,w_{k-2})$ be the solution of the linear system given by 
the equations 
$$(N+h)w_h+(N+h-2)w_{h-2}=2(N+h-1)w_{h-1}$$ 
for $h=2,\dots, k$. 
We can repeat the application of 
\eqref{conf-gn} to each pair $N_1=N+h$, $N_2=N+h-2$ with $h=2,\dots, k$, by fixing at each step $z_1=w_{h},$ $z_2=w_{h-2}$, obtaining 
\begin{eqnarray*}
&&(N+k)g_{N+k}(w_k)-(N+k-1) g_{N+k-1}(w_{k-1})\\ 
&&\hspace{3cm} > \ (N+k-1) g_{N+k-1}(w_{k-1})-(N+k-2)g_{N+k-2}(w_{k-2})\\
&&\hspace{3cm} > \ (N+1) g_{N+1}(w_1)-Ng_{N}(w_{0}).
\end{eqnarray*}
The last part of the claim follow by induction.
\end{proof}
Now we can show an ordering property of the functions $g_N$ which allows to describe the structure of $\widehat Q_\sigma f$ in terms of the locking states. 

\begin{remark}\label{monniN}\rm
If we define the auxiliary functions $\widetilde g_N(z)= g_N(z)-{\eta\over N}$, then we have $\widetilde g_N(z)<\widetilde g_{N+1}(z)$ for all $N\ge1$ and $z>0$. This is proved by induction using Proposition \ref{gin1n2} with $N_1=N-1$, $N_2=N+1$ and $z_1=z_2=z$, after noting that for $N=1$ the inequality $\widetilde g_1(z)<\widetilde g_{2}(z)$ is implied by Proposition \ref{convinN} since $\widetilde g_1(z)=az^2$.
\end{remark}

\subsubsection{Characterization of locking states}
The convexity properties of $g_N(z)$ allow to characterize the locking states of the function $f$ and to give a description of $Q_\sigma f(z)$. 

\begin{theorem}[locking states of $Q_\sigma f$]\label{enwellfract}
Let $f$ be as in \eqref{f-frattura} and let $m_n^\sigma=e^{-\sigma n}$. Then the set of locking states of $Q_\sigma f$ is
given by 
$$\Big\{\frac{1}{N}: N\in \mathbb N, N\geq 1\Big\}\cup\{0\}.$$
\end{theorem}

\begin{proof} \ 

\noindent
{\em Step $1$}.
We prove by induction the monotonicity of the sequence $g_N(z)$ for $z$ large enough. 
By Proposition \ref{gin1n2} we obtain that if $g_N(z)\ge g_{N-1}(z)$ then 
 \begin{equation*}
{N-1\over 2N}g_{N}(z)+ {N+1\over 2N}g_{N+1}(z)\ \ge \ 
{N-1\over 2N}g_{N-1}(z)+ {N+1\over 2N}g_{N+1}(z)\ > \ g_N(z);  
\end{equation*}
hence,
 \begin{equation*}
 {N+1\over 2N}g_{N+1}(z)\ > \ \Big(1-{N-1\over 2N}\Big)g_{N}(z) \ = \ {N+1\over 2N}g_{N}(z)\,.
\end{equation*}
Hence, iterating this argument, we get that the sequence $k\mapsto g_k(z)$ is not decreasing for $k\ge N-1$ and  strictly increasing for $k\ge N$.

\smallskip
\noindent
{\em Step $2$}.
Now we show that for $z$ large enough then $g_2(z)\geq g_1(z)$. 
By the growth hypothesis $\tilde f(z)\geq c_1z^2-c_2$ we get 
\begin{eqnarray*}
g_2(z)&\geq& \frac{\tilde f(z^\ast)}{2}-\frac{c_2}{2}+\frac{1}{2}\min\Big\{c_1(u_2-u_1)^2+a(2z-v_1)^2+a(v_1)^2\\
&&\hspace{1cm}+b(u_2-2z)^2+b(u_1-v_1)^2: u_1, u_2, v_1\in\mathbb R\Big\}.
\end{eqnarray*}
By computing the minimum, we obtain 
\begin{eqnarray*}
g_2(z)&\geq& \frac{\tilde f(z^\ast)}{2}-\frac{c_2}{2}+\frac{1}{2} 
\Big(\frac{a(2c_1+b)+bc_1}{2c_1+b}(2z-v_1)^2+a(v_1)^2\Big)
\end{eqnarray*}
with 
$$v_1=\frac{a(4c_1+2b)+2bc_1}{a(4c_1+2b)+bc_1}z.$$
Hence for $z$ large enough
\begin{eqnarray*}
g_2(z)&\geq& \frac{\tilde f(z^\ast)}{2}-\frac{c_2}{2}+a 
\Big( 1+\frac{bc_1(a(4c_1+2b)+bc_1)}{(a(4c_1+2b)+bc_1)^2} \Big)z^2\\
&>& \tilde f(z^\ast)+az^2 \ =\ g_1(z).  
\end{eqnarray*}
From this property and Remark \ref{min-a} we deduce that there exists a unique $z_1$ such that $g_2(z_1)=g_1(z_1)$, and hence $g_1(z)=\min_N g_N(z)$ in $[z_1,+\infty)$ by Step 1.

\smallskip
\noindent
{\em Step $3$}. By Step 1 $g_N(z_1)>g_2(z_1)=g_1(z_1)$ for all $N\ge 3$. Let $[z_2,z_1]$ be the maximal interval containing $z_1$ where $g_2(z)=\min_{N\ge 1} g_N(z)= \min_{N\ge 2} g_N(z)$. Since in particular $g_3> g_2$ in the interval $(z_2,z_1]$
by Remark \ref{monniN}, we have $g_N>g_4> g_3$ for all $N>4$ in the closed interval $[z_2,z_1]$ always by Step 1. This implies that $g_3(z_2)=g_2(z_2)$. Moreover, note that $g_4(z_2)>g_2(z_2)$, since otherwise we would have $g_3(z_2)<g_2(z_2)$ by \eqref{conf-gnz} with $z=z_2$, $N_1=2$ and $N_2=4$. 

\smallskip
\noindent
{\em Step $4$}. We define $z_3=\max\{z: g_4(z)\le\min\{g_3(z), g_2(z), g_1(z)\}$. This is well defined since $g_4(0)<\min\{g_3(0), g_2(0), g_1(0)\}$ and we have $z_3<z_2$. Note that in $(z_4,z_3)$ we have $\min\{ g_N(z): z\in\mathbb N\}\in\{ g_2(z),g_3(z)\}$. We then define iteratively $z_n=\max\{z: g_{n+1}(z)\le\min\{g_k(z): k\le n\}$. Again, this is a good definition and $z_n<z_{n-1}$. In $(z_n,z_{n-1})$ we have that $\min\{ g_N(z): N\in\mathbb N\}\in\{ g_n(z),g_{n-1}(z)\}$. Indeed, by Corollary \ref{conv-disp} if $g_k(z)=g_\ell(z)$ at some $z$ then $|k-\ell|\le 1$. Since $\min\{ g_N(z_{n-1}): N\in\mathbb N\}=g_n(z_{n-1})$ and we cannot have $g_n(z)=g_{n+1}(z)$ if $z\in (z_n,z_{n-1})$, the claim follows.

\smallskip

\noindent
{\em Step $5$}. Inequality \eqref{conf-gnz} shows that the graph of $g_{N}$ lies below the graph of the convex envelope of  the minimum between $g_{N-1}$ and $g_{N+1}$ in an open interval. By Proposition \ref{celle} this proves that $1\over N$ is a locking state. 
\end{proof}

\begin{figure}[h!]
\centerline{\includegraphics[width=0.9\textwidth]{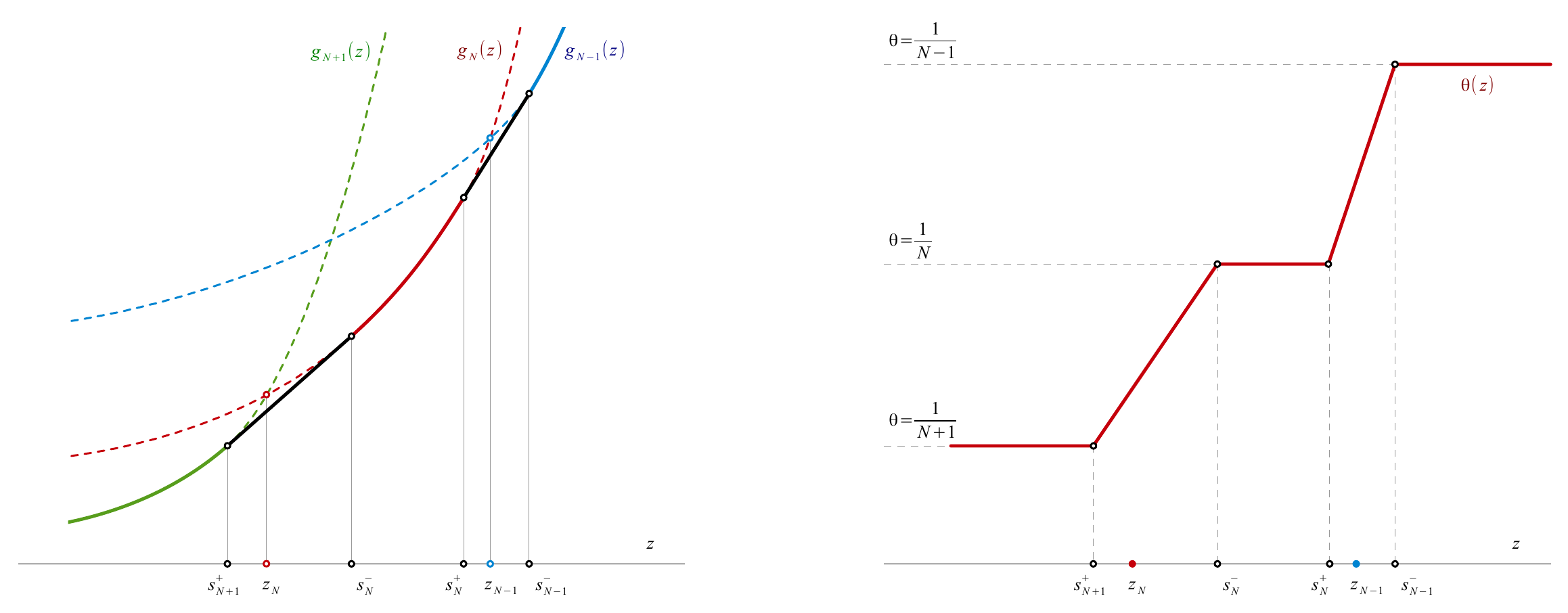}}
\caption{pictorial description of Theorem \ref{enwellfract} for a single choice of $\sigma$ (shape of $\widehat Q_\sigma f$ and $\theta$, not to scale)}
\label{figure_inf_fracture1}
\end{figure}

In order to highlight the dependence on $\sigma$, for any $\sigma>0$ and for any $N\geq 1$, in the sequel $z_N(\sigma)$ will denote the corresponding value $z_N$ given by Theorem \ref{enwellfract}.  Moreover, for any $\sigma$ we set $z_0(\sigma)=+\infty$.  

\begin{figure}[h!]
\centerline{\includegraphics[width=0.9\textwidth]{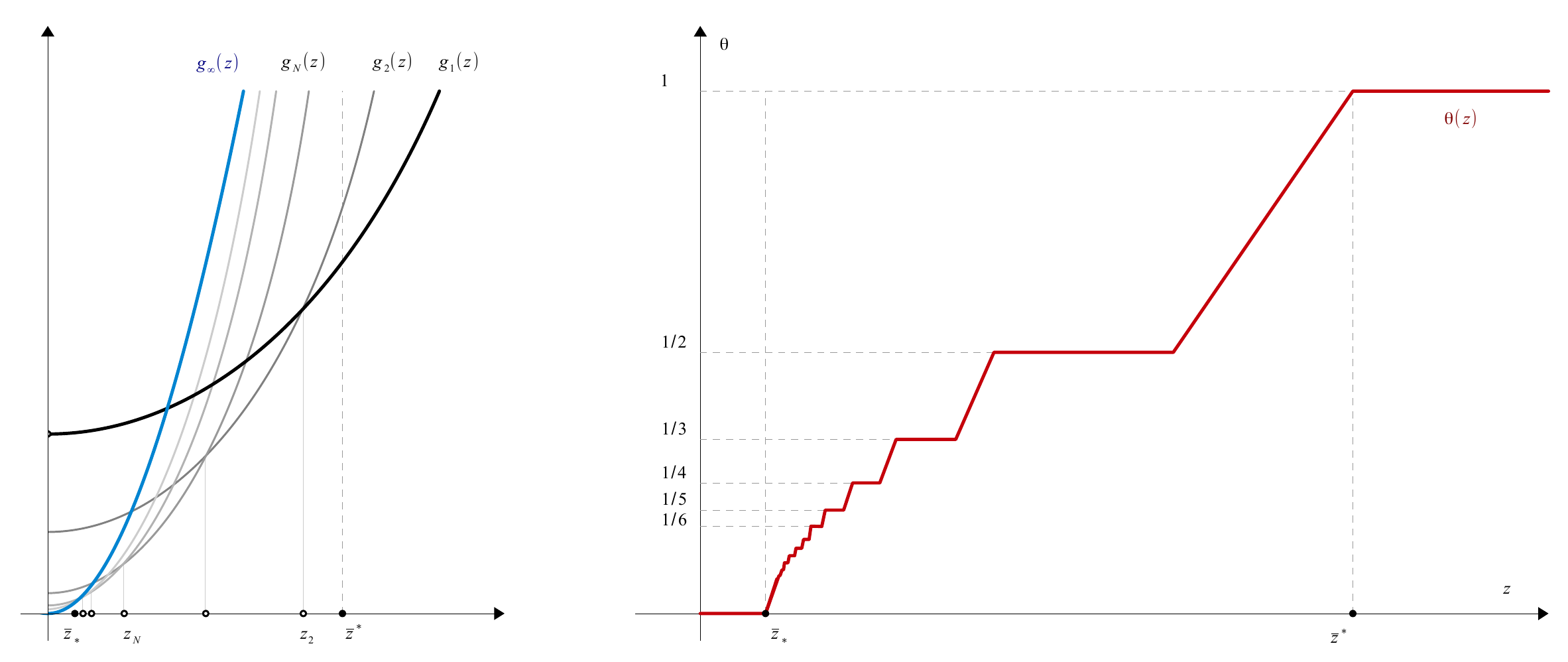}}
\caption{relative behaviour of $g^\sigma_N$ and the final resulting $\theta$.}
\label{figure_inf_fracture2}
\end{figure} 

\begin{remark}[shape of $Q_\sigma f(z)$ and $\theta(z)$]\rm 
The graph of the function $Q_\sigma f(z)$ possesses infinitely many 
concave parabolic arcs, corresponding to the intervals where $\widehat Q_\sigma f(z)$ is affine, which accumulates in  $\overline z_\ast(\sigma)=\inf_N z_N(\sigma)>0$. Correspondingly, the phase function 
$\theta(z)$ is affine, interpolating between consecutive values $1/N$ (see Fig.~\ref{figure_inf_fracture2}). 
 
Summarizing, the behaviour of the penalized energy $Q_\sigma f(z)$ in terms of the macroscopic gradient $z$ has the following features:   

$\bullet$ (`unfractured zone') for $z\le \overline z_\ast(\sigma)$ optimal sequences take into account only the convex part of $f$; i.e., there are no broken bonds;

$\bullet$ (`completely microfractured zone') there exists $\overline z^\ast(\sigma)=s_1^-(\sigma)>z_1(\sigma)$ such that for $z\ge \overline z^\ast(\sigma)$ (that is, in $I_1(\sigma)=[\overline z^\ast(\sigma),+\infty)$) the part of the energy involving the function $f$ is identically $\tilde f(z^\ast)$; i.e., we have broken bonds for all values of the index $i$;

$\bullet$ (increasingly segmented  behavior of the relaxed energy) for values of the macroscopic gradient between $\overline z_\ast(\sigma)$ and $\overline z^\ast(\sigma)$ the energy  $\widehat Q_{\sigma}f$ behaves as a superposition of infinitely many `damaged materials' indexed by the parameter $N$ representing the microscopic optimal spacing of broken bonds. For the values $z$ where $\widehat Q_{\sigma}f(z)$ is affine, optimal sequences mix the damaged materials parameterized by $N$ and $N-1$. The point $\overline z_\ast(\sigma)$ is an accumulation point for the different behaviors as $N\to+\infty$. 
\end{remark}

\begin{remark}[limit behaviours of the damaged zones]\rm 
By Proposition \ref{interpolation-sigma}, highlighting the dependence on the parameter $\sigma$, 
we deduce that 
\smallskip 

{\rm (i)} $\displaystyle\lim_{\sigma\to 0} \overline z_\ast(\sigma)=\lim_{\sigma\to 0} \overline z^\ast(\sigma)=z^\ast$,  
corresponding to the extreme non-additivity case, 

\smallskip 

{\rm (ii)} $\displaystyle\lim_{\sigma\to +\infty} \overline z_\ast(\sigma)=0$ and $\displaystyle\lim_{\sigma\to +\infty} \overline z^\ast(\sigma)=+\infty$, corresponding to full additivity.  
\end{remark}

\begin{remark}[Generic non differentiability]\label{nondiff-exp}\rm 
Note the generic non differentiability of $\widehat Q_\sigma f(\theta,z)$ with respect to $\theta$ at the locking states. This is due to the different definitions of this function in left and right neighbourhoods of each locking state $\frac{1}{N}$. 
Indeed, the definition of $\widehat Q_\sigma f(\theta,z)$ uses 
$g^\sigma_N(z), g^\sigma_{N+1}(z)$ in a left neighbourhood and $g^\sigma_{N-1}(z), g^\sigma_{N}(z)$ in a right neighbourhood of $\theta=\frac{1}{N}$, respectively, in analogy with the case of concentrated kernels, as seen in Section \ref{mneigh-sec} (see Remark \ref{nondiff}). 
\end{remark}

\subsection{Properties of optimal microstructures}
In the previous section we have shown that $\theta$ of the form $1\over N$ with $N\in\mathbb N$ are locking states. We now show that such values correspond to energy wells, and characterize all $\widehat Q_\sigma f(\theta,\cdot)$.

\subsubsection{Microstructures as interpolations of energy meta-wells} 
The following proposition reinterprets $g^\sigma_N$  as the energy of periodic minimizers for $\theta_N= 1/N$. 

\begin{proposition}[$g^\sigma_N$ as an energy meta-well]\label{gequalphi} 
Let $g_N^\sigma$ be defined as in \eqref{def-gm} with $a=a_\sigma$ and $b=b_\sigma$ satisfying \eqref{aeb}. 
The following equality holds for any $N\in\mathbb N$ and $z\in\mathbb R$:  
$$\Phi^N_{\bf m}f\Big(\frac{1}{N},z\Big)=g^\sigma_N(z)$$
where $\Phi^{N}_{\bf m}f$ is defined in \eqref{deffin} with $A=[z^\ast, +\infty)$, $f_{-1}=\tilde f$, 
$f_1=\tilde f(z^\ast)$ if $z\in A$ and $+\infty$ otherwise, and $m_n=e^{-\sigma n}$. 
\end{proposition}

\begin{proof}
We first observe  that $\Phi^N_{\bf m}f(\frac{1}{N},z)=\widehat R^N_{\bf m}f(\underline s_N, z)$ where 
$\underline s_N=(1,-1,-1,\ldots,-1)$ and $\widehat R^N_{\bf m}f$ is defined in 
\eqref{phiminus} with
$$ 
F^\#(u,\underline s; [0,N])=\sum_{i=1}^Nf_{s_i}(u_i-u_{i-1})+\sum_{i=1}^N\sum_{j\in\mathbb Z}e^{-\sigma|i-j|}(u_i-u_j)^2.$$ 
By extending $\underline s_N$ by $N$-periodicity, we have 
\begin{eqnarray*}
\widehat R^N_{\bf m}f(\underline s_N, z)&=& \frac{1}{N}\min\{ F^\#(u,\underline s_N; [0,N]): u_i-zi 
\ N\hbox{-periodic}\}\\
&=& \lim_{k\to+\infty}\frac{1}{kN}\min\{ F^\#(u,\underline s_{N}; [0,kN]): u_i-zi 
\ N\hbox{-periodic}\}\\
&=& \lim_{k\to+\infty}\frac{1}{kN}\min\Big\{ 
\tilde f(z^\ast) k+\sum_{r=1}^k\sum_{l=2}^N \tilde f(u_{N(r-1)+l}-u_{N(r-1)+l-1})\\
&&+
a_\sigma\sum_{i=1}^{kN}(v_i-v_{i-1})^2+b_\sigma\sum_{i=1}^{kN}(u_i-v_i)^2: u_i-zi, v_i-zi  
\ N\hbox{-periodic}\Big\},
\end{eqnarray*}
the last equality being a consequence of \eqref{minG}, the equivalence result of Lemma \ref{lemma-equivalence} and the characterization of the minima given by \eqref{vsol}, which ensures that also the minimizing $v$ can be chosen periodic. 
Hence by the periodicity we get 
\begin{eqnarray*}
&&\Phi^N_{\bf m}f\Big(\frac{1}{N},z\Big)= \frac{1}{N}\min\Big\{ 
\tilde f(z^\ast)+\sum_{i=2}^N \tilde f(u_{i}-u_{i-1})+
a_\sigma\sum_{i=1}^{N}(v_i-v_{i-1})^2+b_\sigma\sum_{i=1}^{N}(u_i-v_i)^2:  \\
&&\hspace{7cm} u_i-zi, v_i-zi  
\ N\hbox{-periodic}\Big\}.
\end{eqnarray*}
Finally, noting that we can remove the periodicity condition on $u$ and that we can rewrite 
the condition on $v$ as a boundary condition, we get the claim. 
\end{proof}

Let $I_N=I_N(\sigma)=\{z\in \mathbb R: \widehat Q_\sigma f(z)=g^\sigma_N(z)\}$.
Note that Remark \ref{polig} implies that $\widehat Q_\sigma f(\theta,\cdot)$ can be described in terms of the convex 
combination of the functions $g^\sigma_N(z)$. In particular, by the convexity of $g^\sigma_N(z)$
with respect to $N$, 
we have 
\begin{equation}\label{qsigmagn}
\widehat Q_\sigma f\big(\frac{1}{N}, z\big)=g^\sigma_N(z)
\end{equation}
in the whole $\mathbb R$.

We are now in the position to 
characterize $\widehat Q_\sigma f(\theta,\cdot)$ as an interpolation between consecutive energy meta-wells (corresponding to the locking states), as in Lemma \ref{formula-n-prop} for the concentrated kernels. 
 
\begin{proposition}[interpolation between energy wells] 
Given  $\sigma>0$, suppose that  $a=a_\sigma$ and $b=b_\sigma$ are as in  \eqref{aeb}. 
Then, for any $\theta\in\mathbb Q\cap(0,1)$ and for any $z\in \mathbb R$ the following equality holds: 
\begin{equation}\label{aubrygen}
\widehat Q_{\sigma}f(\theta,z)= \min\Big\{  
t(\theta) g^\sigma_{N_\theta}(z^\prime)+
(1-t(\theta)) g^\sigma_{N_\theta+1}(z^{\prime\prime}): 
  t(\theta)z^\prime+(1-t(\theta))z^{\prime\prime}=z\Big\}, 
\end{equation}
where 
\begin{equation}\label{nthetattheta}
N_\theta=\Big\lfloor \frac{1}{\theta}\Big\rfloor \quad \hbox{\rm and} \quad 
t(\theta)=N_\theta\Big(\theta (N_\theta+1)-1\Big). 
\end{equation} 
\end{proposition}

\begin{proof} 
We divide the proof in two steps. 

\smallskip 
\noindent{\em Step $1$: $\theta=\frac{1}{N}$.} 
In this case, the claim becomes \eqref{qsigmagn} for all $z\in\mathbb R$. We note that for each $N$ the formula is proved for $z\in I_N$. Moreover, for arbitrary $z$  it can be further simplified as follows. 
Let $k\in\mathbb N$ be fixed and let $(u,v)$ be a minimizer in \eqref{trs-con} with $p=1$ and $q=N$.
Since $\widehat Q_\sigma f(\frac{1}{N},z)$ can be expressed as in \eqref{trs-con}, it is sufficient to  
show that for all $k$ 
\begin{equation*}
\frac{1}{kN}E^\sigma_1(u,v;[0,kN])\geq g^\sigma_N(z).
\end{equation*}
It is not restrictive to suppose that $u_1-u_0\geq z^\ast$. 
By grouping the interactions, we estimate 
$$E^\sigma_1(u,v;[0,kN])\geq \sum_{j=1}^k N_j g^\sigma_{N_j}(z_j)$$
where 
$\sum_{j=1}^kN_j=kN$ and $\sum_{j=1}^kN_j z_j=kN z$.  
By Proposition \ref{gin1n2}, we infer that all even $N_j$ are equal to some $N_{\rm e}$, 
and the corresponding $z_j$ coincide with some $z_{\rm e}$,  
and the same holds for odd $N_j$ with $N_{\rm o}$ and corresponding $z_j$ with $z_{\rm o}$, so that there exist integers $k_{\rm e}$ and $k_{\rm o}$ such that 
$$E^\sigma_1(u,v;[0,kN])\geq k_{\rm e} N_{\rm e} g^\sigma_{N_{\rm e}}(z_{\rm e})+
k_{\rm o} N_{\rm o} g^\sigma_{N_{\rm o}}(z_{\rm o})$$
where 
$$k_{\rm e}N_{\rm e}+k_{\rm o}N_{\rm o}=kN \quad \hbox{\rm and } 
\quad k_{\rm e}N_{\rm e}z_{\rm e}+k_{\rm o}N_{\rm o}z_{\rm o}=kN z.$$ 
Since $u\in \mathcal V(kN,\frac{1}{N})$, we also have 
$k_{\rm e}+k_{\rm o}=k$.  
By \eqref{equispace} we deduce that $|N_{\rm e}-N_{\rm o}|=1$, and this is only possible if either $k_{\rm e}$ or $k_{\rm o}$ vanishes, from which we conclude. 

 \smallskip 

\noindent{\em Step $2$: general case.} 
We fix $\theta=\frac{p}{q}$ with $p$ and $q$ coprime integers satisfying $1<p<q$. 
Let $k\in\mathbb N$ be fixed and let $(u,v)$ be a minimizer in \eqref{trs-con}.
By grouping the interactions as in the case $\theta=\frac{1}{N}$, thanks to \eqref{equispace} we obtain that there exists $N\in\mathbb N$ such that 
\begin{equation}\label{mincondN}
k_1+k_2=kp, \quad k_1N+k_2(N+1)=kq 
\end{equation}
for some $k_1,k_2\in\mathbb N$, and  
$$
E^\sigma_1(u,v;[0,kq])\geq k_1 N g^\sigma_{N}(z^\prime)+
k_2 (N+1) g^\sigma_{N+1}(z^{\prime\prime})
$$ 
where $z^\prime,z^{\prime\prime}$ satisfy $k_1N z^\prime+k_2(N+1)z^{\prime\prime}=kqz$.
Since \eqref{mincondN} implies 
$\frac{q}{p}\geq N>\frac{q}{p}-1$,  
we deduce that $N=N_\theta$ is the unique integer solution of the equation (with 
$k_1=k(p(N_\theta+1)-q)>0$ and $k_2=k(q-pN_\theta)>0$).  
Hence 
\begin{equation}\label{minE1}
E^\sigma_1(u,v;[0,kq])\geq k_1 N_\theta g^\sigma_{N_\theta}(z^\prime)+
k_2 (N_\theta+1) g^\sigma_{N_\theta+1}(z^{\prime\prime}).
\end{equation}
Noting that 
$$\frac{k_1 N_\theta}{kq}=t(\theta) \quad 
\hbox{\rm and } \quad \frac{k_2 (N_\theta+1)}{kq}=1-t(\theta),$$ 
since $\widehat Qf_\sigma(\theta,z)$ can be expressed as in \eqref{trs-con}  
we obtain, by using \eqref{minE1},   
\begin{eqnarray*}
\widehat Q_{\sigma}f(\theta,z)\geq \min\Big\{  
t(\theta) g^\sigma_{N_\theta}(z^\prime)+
(1-t(\theta)) g^\sigma_{N_\theta+1}(z^{\prime\prime}): 
  t(\theta)z^\prime+(1-t(\theta))z^{\prime}=z\Big\}.
\end{eqnarray*}

\smallskip 

The opposite inequality follows by the equality $g^\sigma_N(z)=\widehat Q_{\sigma}f(\frac{1}{N},z)$ proved in the case $\theta=\frac{1}{N}$ and by the convexity of $\widehat Q_{\sigma}f(\theta,z)$. Indeed, 
noting that 
$$\frac{t(\theta)}{N_\theta}+\frac{1-t(\theta)}{N_\theta+1}=\theta,$$
for all pairs $(z^\prime,z^{\prime\prime})$ such that 
$t(\theta)z^\prime+(1-t(\theta))z^{\prime\prime}=z$, 
we have 
\begin{eqnarray*}
t(\theta) g^\sigma_{N_\theta}(z^\prime)+
(1-t(\theta)) g^\sigma_{N_\theta+1}(z^{\prime\prime})&=&
t(\theta) \widehat Q_\sigma f\Big(\frac{1}{N_\theta},z^\prime\Big)+
(1-t(\theta)) \widehat Q_\sigma f\Big(\frac{1}{N_\theta+1},z^{\prime\prime}\Big)\\
&\geq&
 \widehat Q_\sigma f\Big(\frac{t(\theta)}{N_\theta}+\frac{1-t(\theta)}{N_\theta+1},t(\theta)z^\prime+
 (1-t(\theta))z^{\prime\prime}\Big)\\
 &\geq&
 \widehat Q_\sigma f(\theta,z)  
\end{eqnarray*}
as desired. \end{proof}

\subsubsection{A canonical optimal microstructure uniform at all scales}\label{aubryrem} 
The description of $\widehat Q_\sigma f$ that we have obtained in terms of $g^\sigma_N$ highlights a number of equivalent minimizers. However, in this class we can define a set of canonical ground states.  These states are characterized by the corresponding distribution of spins, or, equivalently, the distribution of broken bonds. 
Similar sets have independently appeared  in the study  of related dynamical systems  \cite{aubry,Mather}.

In order to describe this optimal distribution of broken bonds, 
for a given $\theta\in [0,1]$ we define the set of integers 
$$A(\theta)=\{k\in \mathbb Z: \lfloor k\theta\rfloor \neq \lfloor (k+1)\theta\rfloor\}.$$ 
A characteristic property of the set $A(\theta)$ is its `uniformity at all scales'; that is, the property that for each $M\in \mathbb N$ each interval of length $M$ contains either $\lfloor M\theta\rfloor$ or  $\lfloor M\theta\rfloor+1$ elements of $A(\theta)$. The set $A(\theta)$ can be described as the most uniformly distributed among sets with such property (up to translations).
Note, for instance,  that if $\frac{1}{N+1}<\theta<\frac{1}{N}$ then the difference between two consecutive elements of $A(\theta)$ is either $N$ or $N+1$. The set $A(\theta)$ is periodic if and only if $\theta$ is rational; otherwise it follows a pattern reminiscent of quasiperiodic functions (see e.g.~\cite{Besicovitch,LZ}). 

The following proposition states that in the computation of $\widehat Q_\sigma f(z)$ we can consider the corresponding minimum problems only on functions $u$ whose broken sites coincide with $A(\theta(z))$.  
\begin{proposition}[optimality of $A(\theta)$] 
Let $f$ be as in \eqref{f-frattura}. Then, for any $\sigma>0$ and $z\in\mathbb R$, the following equality holds: 
\begin{equation*}
\widehat Q_\sigma f(z)=\liminf_{\substack{k\to+\infty\\k\in A(\theta(z))}}
\frac{1}{k}\min\{E^\sigma_1(u,v;[0,k]): 
v_0=0,v_{k}=zk, u_i-u_{i-1}\geq z^\ast \Leftrightarrow i\in A(\theta(z))\}.
\end{equation*}
\end{proposition} 
\begin{proof}
For each $N$, we can suppose that the set where $\widehat Q_\sigma f(z)=g^\sigma_N(z)$ is an interval $I_N=[s_N^-, s_N^+]$. Let $z\in (s_{N+1}^+,s_N^-)$. Then, writing 
\begin{equation}\label{defztN}
z=t s_N^-+(1-t)s_{N+1}^+,
\end{equation} we have that 
\begin{equation}\label{qeg}
\widehat Q_\sigma f(z)=r^\sigma_{N+1}(z)=t g^\sigma_N(s_N^-)+(1-t)g^\sigma_{N+1}(s_{N+1}^+).
\end{equation}
Recalling the definition of the phase function $\theta(z)$ (see Definition \ref{phasefunctiondef}) and the fact that $\theta(z)$ is affine in each open interval where $\widehat Q_\sigma f$ is affine, as stated in Proposition
\ref{phasefunctionaffine}, we deduce 
$$\widehat Q_\sigma f(z)=\widehat Q_\sigma f(\theta(z),z) \ \ \hbox{and } \ \ \theta(z)=t \frac{1}{N}+(1-t)\frac{1}{N+1},$$ 
where the link between $z,t$ and $N$ is given by \eqref{defztN}. 
Hence, using the local representation given by \eqref{trs-con}, for all $k\in A(\theta(z))$ we can split the minimum
$$\min\{E^\sigma_1(u,v;[0,k]): 
v_0=0,v_{k}=zk, u_i-u_{i-1}\geq z^\ast \Leftrightarrow i\in A(\theta(z))\}$$
into the sum of the minima 
$$M_j=\min\{E_1(u,v;[i_{j-1},i_j]): 
v_{i_{j-1}}=0,v_{i_j}= (i_j-i_{j-1}) z_j, 
u_i-u_{i-1}\geq z^\ast \Leftrightarrow i=i_j\},$$
where $A(\theta(z))\cap [0,k]=\{i_0, i_1,\dots, i_{n_k}\}$ with $0=i_0<\dots <i_{n_k}=k$, 
and $z_j$ are such that $\sum_{j=1}^{n_k}(i_j-i_{j-1})z_j=kz$.  

Furthermore, noting that $M_j=(i_j-i_{j-1})g^\sigma_{i_j-i_{j-1}}(z_j)$, we 
obtain by convexity \begin{eqnarray}\label{minimaubry}
&&\nonumber\frac{1}{k}\min\{E_1(u,v;[0,k]): 
v_0=0,v_{k}=zk, u_i-u_{i-1}\geq z^\ast \Leftrightarrow i\in A(\theta(z))\} \\
&&\hspace{1cm}\geq \frac{1}{k}\sum_{j=1}^{n_k} (i_j-i_{j-1})g^\sigma_{i_j-i_{j-1}}(z_j)
\geq \frac{1}{k}\sum_{j=1}^{n_k} (i_j-i_{j-1})\widehat Q_\sigma f(z_j)\geq 
\widehat Q_\sigma f(z). 
\end{eqnarray}

Conversely, fixed $k\in A(\theta(z))$, let $\mathcal I_N=\{j\leq n_k: i_j-i_{j-1}=N\}$ and 
$\mathcal I_{N+1}=\{j\leq n_k: i_j-i_{j-1}=N+1\}$ and $z_k^\pm$ be such that 
$$N\#\mathcal I_N z_k^-+(N+1)\#\mathcal I_{N+1} z_k^+=kz$$ 
and $z_k^-\to s_N^-, z_k^+\to s_{N+1}^+$ as $k\to+\infty$. 
Then, using the minimizers of $g^\sigma_N(z_k^-)$ and of $g^\sigma_{N+1}(z_k^+)$ to test 
the minimum problem in \eqref{minimaubry}, we get the upper bound  
$$N\#\mathcal I_N g^\sigma_N(z_k^-)+(N+1)\#\mathcal I_{N+1} g^\sigma_{N+1}(z_k^+).$$
Taking the limit as $k\to+\infty$, by 
\eqref{qeg} 
we obtain 
the claim. 
\end{proof} 

\begin{remark}[optimality of $A(\theta)$ for the constrained relaxation]\rm 
The same proof shows that for any $\theta$ 
\begin{eqnarray*}
\widehat Q_\sigma f(\theta, z)=\liminf_{\substack{k\to+\infty\\k\in A(\theta)}}
\frac{1}{k}\min\{E^\sigma_1(u,v;[0,k]): 
v_0=0,v_{k}=zk, u_i-u_{i-1}\geq z^\ast \Leftrightarrow i\in A(\theta)\}.
\end{eqnarray*} 
\end{remark}

\begin{figure}[h!]
\centerline{\includegraphics[width=0.8\textwidth]{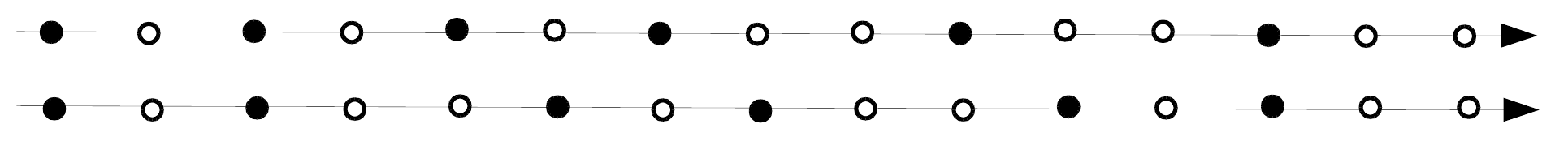}}
\caption{representation of two periodic minimizers}
\label{Canonical_minimizers_1}
\end{figure}
For the sake of illustration, in Fig.~\ref{Canonical_minimizers_1} we represent two periodic minimizers (the black dots representing broken bonds) for $\theta=2/5$. In the first case we have a 15-periodic minimizers, the second array is the `canonical' one, alternating broken bonds at distance two and three.

 \begin{remark}[the $M$-th neighbour case]\rm
In the case of $M$-th only interactions, we  focus first on $\theta=\theta_k={k\over M}$, with $k\in\{0,\ldots, M\}$ , that is, on  locking states, or, equivalently, on energy wells.  The construction in Proposition \ref{formula-n-prop} shows that all periodic spin configurations  with period a submultiple of $M$ compatible with $\theta_k$,  correspond to optimal laminates. Indeed, the only requirement on minimizers is that for all intervals of length $M$ we have an equal number of spins of either type (which is trivially true). Note in particular that  we may choose minimizers with $u_i-u_{i-1}>z^*$ exactly for $i\in A(\theta)$ since this set is $M$-periodic. Now if  $\theta$ is not of the form $k/M$, we do not have periodic optimal minimizers. This is in contrast to the exponential case, where we do have periodic minimizers for all $\theta\in \mathbb Q$.
\end{remark}

\begin{figure}[h!]
\centerline{\includegraphics[width=0.8\textwidth]{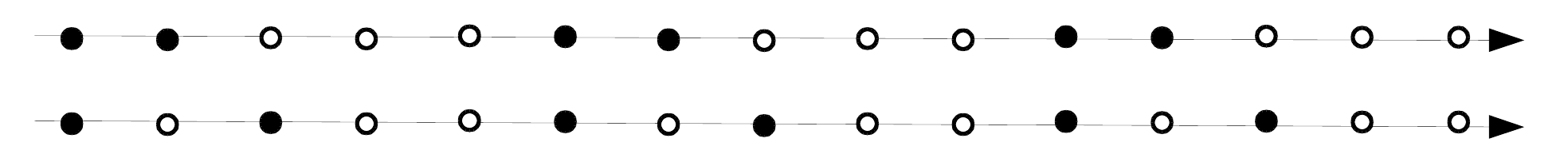}}
\caption{representation of two periodic minimizers}
\label{Canonical_minimizers_2}
\end{figure}
In Fig.~\ref{Canonical_minimizers_2} we represent two 5-periodic minimizers (the black dots representing the elongations larger than $z^*$) for $M=5$ and $\theta=2/5$. The second array is the `canonical' one, alternating broken bonds at distance two and three.

\bigskip
We note that in some of our examples illustrating periodic minimizers with `global' properties,   the canonical periodic microstructures,  epitomizing a generalized Cauchy-Born (GCB) states,  are unique.   This is true, for instance,  in  the case of the exponential kernel ${\bf m}$.  Instead,  for concentrated 
kernels we may have more than one minimal (GCB-type) microstructure.  Note also that in the case of exponential kernels,  outside the  special regimes where the minimizers are periodic,  we can mix GCB states and, since different GCB states do not interact, the mixing process is bringing arbitrariness. In particular, GCB states could be mixed canonically, even though  in the examples of interest in this paper this does not bring any advantages. However, this is not the general case and when different GCB states interact, their mixtures can become suboptimal, as in the case of concentrated kernels. We argue that in such `strongly non-additive' cases the non-periodic GCB states with the properties of our canonical microstructures can  become the preferred ones if interaction happens at all scales (which is not the case for concentrated kernels). 

\subsection{Explicit constructions}\label{esp:co}
In this section we explicitly compute $Q_{\sigma}f$ in a meaningful case, using the general results of the previous section. This also allows us to treat some classes of energies more general than truncated potentials. 

\subsubsection{The Novak-Truskinovsky model}\label{truquapo-exa}
Let $f$ be the truncated quadratic potential  
defined as in \eqref{f-frattura} with $\tilde f(z)=z^2$; that is, 
\begin{equation}\label{fNT}
f(z)=
\begin{cases}
z^2 & \hbox{\rm if }\ z\leq \sqrt{\eta}\\
\eta & \hbox{\rm if }\ z\geq \sqrt{\eta}\\
\end{cases}
\end{equation}
with $\eta>0$ fixed. 
By using the computations in \cite{NT2016} and the results of this section, we obtain an explicit formula for $g^\sigma_N(z)$, and hence $Q_{\sigma}f(z)$. 

\begin{remark}[explicit computation of minima]\rm 
Let $\tilde E_1$ be defined as in \eqref{def-etilden} with $\tilde f(z)=z^2$ and $a,b>0$. 
Then, by the computations in \cite[Sec.~3]{NT2016} we get 
\begin{equation*}
\min\{\tilde E_1(u,v;[0,N]): v_0=0, v_N=N\}=\frac{N^2a(a+1)}{Na+\tanh((N+1)\zeta)\coth(\zeta)-1} 
\end{equation*} 
where 
\begin{equation}\label{def-zeta}
\zeta=2\sinh^{-1}\!\Big(\frac{1}{2}\sqrt{\frac{b(a+1)}{a}}\Big).
\end{equation}
By using \eqref{def-gm}, we obtain   
\begin{equation}\label{def-fdhom}
g_N(z)=c_Nz^2+\frac{\eta}{N},
\end{equation}
where 
\begin{equation}\label{def-cn} 
c_N=\frac{Na(a+1)}{Na+\tanh(N\zeta)\coth(\zeta)}
\end{equation}
and $\zeta$ as in \eqref{def-zeta}. 
\end{remark}
Since we are interested in the analysis of $Q_{\sigma}f$, if $a=a_\sigma$ and $b=b_\sigma$ 
satisfy \eqref{aeb} we write $g^\sigma_N, c^\sigma_N$ and $\zeta_\sigma$ in place of $g_N,c_N$ and $\zeta$, respectively. 
The interval where $\widehat Q_\sigma f(z)=g^\sigma_N(z)$ is given by $I_N(\sigma)=[s_N^-,s_N^+]$,  
where  
\begin{equation}\label{tNsN}
\left.
\begin{array}{ll}
\displaystyle s_N^+=s_N^+(\sigma)=\displaystyle \sqrt{\frac{\eta}{N(N-1)(c^\sigma_N-c^\sigma_{N-1})}}\sqrt{\frac{c^\sigma_{N-1}}{c^\sigma_N}} \ \ \hbox{\rm if } \ N\geq 2; 
\ s_1^+=s_1^+(\sigma)=+\infty\\
\displaystyle s_N^-=\displaystyle s_N^-(\sigma)=\displaystyle \sqrt{\frac{\eta}{(N+1)N(c^\sigma_{N+1}-c^\sigma_N)}}\sqrt{\frac{c^\sigma_{N+1}}{c^\sigma_N}} \ \ \hbox{\rm if } \ N\geq 1. 
\end{array}
\right.
\end{equation}
Hence,  
\begin{equation}\label{zestasta}
\overline z_\ast(\sigma)=\lim_{N\to+\infty}s_N^\pm=\sqrt{\frac{a_\sigma\eta}{(a_\sigma+1)\coth(\zeta_\sigma)}}
 \ \ \hbox{\rm and } \  \overline z^\ast(\sigma)=s_1^-=\sqrt{\frac{\eta(2a_\sigma+b_\sigma(a_\sigma+1))}{a_\sigma b_\sigma}}.
 \end{equation} 
Note that $\overline z_\ast(\sigma)>\sqrt{a_\sigma\eta}
$.
Concluding, we have 
\begin{equation}\label{quatsigmaquad}
Q_\sigma f(z)=\begin{cases}
z^2
& \hbox{if } z\leq \overline z_\ast(\sigma)\\
g^\sigma_N(z)-a_\sigma z^2 & \hbox{if } s_N^-\leq z\leq s_N^+ \ \hbox{ for some $N\geq 2$}\\
r^\sigma_{N+1}(z)-a_\sigma z^2 & \hbox{if } s_{N+1}^+\leq z\leq s_{N}^- \ \hbox{ for some $N$}\\
\eta 
& \hbox{if } z\geq \overline z^\ast(\sigma), 
\end{cases}\end{equation}
where $r^\sigma_{N+1}(z)$ is the common tangent to $g^\sigma_{N+1}(z)$ and  $g^\sigma_{N}(z)$. 

\begin{figure}[h!]
\centerline{\includegraphics[width=1\textwidth]{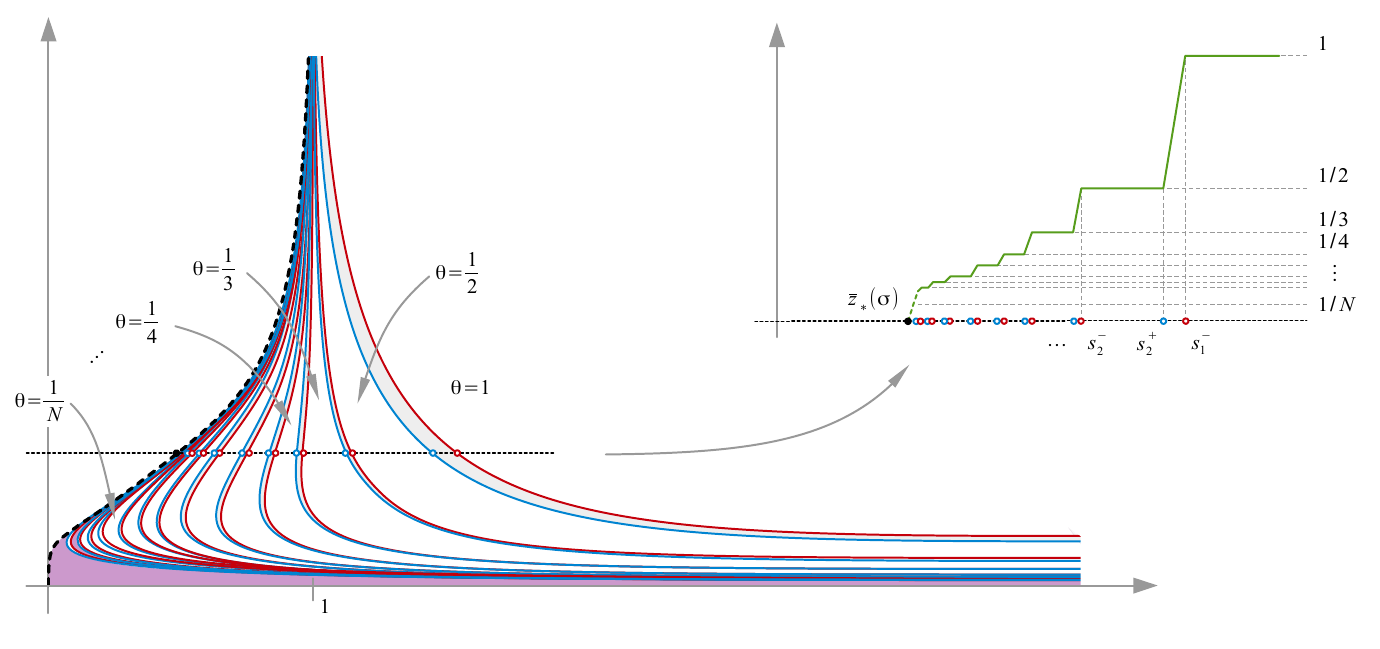}}
\caption{representation of $\theta$ in the $z$-${1\over\sigma}$ plane and a cross section at fixed $\sigma$.}
\label{tonguesexp}
\end{figure} 

The phase function $\theta$ corresponding to this example is pictured in Fig.~\ref{tonguesexp}, where the grey zones between pair of curves denote the pairs in the $z$-${1\over\sigma}$ plane in which $\theta$ is affine for fixed $\sigma$ between consecutive value of the form ${1\over N}$.

\subsubsection{Interpolation between varying degrees of non convexity} 
In this setting it is also of interest to consider a broader class of non convex convex-affine functions $f$ which includes  the convex-constant functions as particular cases. More specifically, consider the functions $
\ell_f^\tau$ defined by
\begin{equation}\label{effealfa}
\ell_f^\tau(z)=\begin{cases}
f(z) & \hbox{\rm if }\ z\leq z^*\\
f(z^*)+ \tau f'(z^*)(z-z^*)& \hbox{\rm if }\ z>z^* 
\end{cases}\end{equation}
with $0<\tau<1$. In this way we construct  an interpolation between the constrained relaxation of the truncated-convex potential and of the convex potential which is obtained if beyond  $z^*$ we smoothly extend $f$ in an affine way.  Accordingly, in \eqref{effealfa} we have  the truncated-convex potential as above at  $\tau=0$, while at $\tau=1$ the function $\ell_f^1$ is convex.

We can write $\ell_f^\tau(z)= \Phi^\tau(z)+ \Gamma^\tau(z)$, where
$$
 \Gamma^\tau(z)= f(z^*)+\tau f'(z^*)(z-z^*)
 $$ and 
$$
\Phi^\tau(z)=\begin{cases}
f(z)-\tau f'(z^*)(z-z^*) & \hbox{\rm if }\ z\leq z^*\\
f(z^*)& \hbox{\rm if }\ z>z^*. \end{cases}
$$
The function $\Phi^\tau$ is a truncated convex potential to which we can apply the results above, while, by Remark \ref{properties}(iii) we have $$Q_\sigma \ell_f^\tau= Q_\sigma(\Phi^\tau+ \Gamma^\tau)=Q_\sigma(\Phi^\tau)+ \Gamma^\tau.$$
 
\smallskip
We can carry on this computation for the quadratic-affine functions $
\ell^\tau$ defined in \eqref{quadraticalfa}; that is, $\ell_f^\tau$ with $f(z)=z^2$ and $z^\ast=\sqrt\eta$.  
Note that we can equivalently rewrite
$\ell^\tau(z)=  \widetilde\Phi^\tau(z)+ \widetilde\Gamma^\tau(z)$, where $\widetilde \Gamma^\tau(z)= 2\tau z-\tau^2$ and 
$$
\widetilde \Phi^\tau(z)=\begin{cases}
(z-\tau)^2 & \hbox{\rm if }\ z\leq 1\\
(1-\tau)^2& \hbox{\rm if }\ z>1,
\end{cases}
$$
which can be seen as a translation by $\tau$ of the function $\Psi^\tau$ given by
$$
\Psi^\tau(z)=\begin{cases}
z^2 & \hbox{\rm if }\ z\leq 1-\tau\\
(1-\tau)^2& \hbox{\rm if }\ z>1-\tau.
\end{cases}
$$
The latter is exactly of the form considered in Example \ref{truquapo-exa} with $\eta=\eta^\tau=(1-\tau)^2$.
Its constrained relaxation is then described in \eqref{quatsigmaquad}, and we eventually have
$$
Q_\sigma \ell^\tau(z)= (Q_\sigma\Psi^\tau)(z-\tau)+ 2\tau z-\tau^2.
$$
Note that by \eqref{zestasta} the endpoints of the interval where the corresponding $\theta(z)$ is not $0$ or $1$ are
$$
\overline z_{\ast,\tau}(\sigma)=\tau+(1-\tau)\sqrt{\frac{a_\sigma}{(a_\sigma+1)\coth(\zeta_\sigma)}}
 \ \ \hbox{\rm and } \  \overline z^{\ast,\tau}(\sigma)=\tau+(1-\tau)\sqrt{\frac{2a_\sigma+b_\sigma(a_\sigma+1)}{a_\sigma b_\sigma}},
$$
with $a_\sigma,b_\sigma, \zeta_\sigma$ as in Example \ref{truquapo-exa}.  Note that $\overline z_{\ast,\tau}(\sigma)<1<\overline z^{\ast,\tau}(\sigma)$, and
$\lim\limits_{\tau\to 1^-}\overline z_{\ast,\tau}(\sigma)=\lim\limits_{\tau\to 1^-} \overline z^{\ast,\tau}(\sigma)=1$.

\section{Asymptotically equivalent continuum models}\label{cont-cons}
The  goal of the relaxation of the  discrete problems discussed in this paper was to obtain a homogenized continuum model. We have seen that generically   the presence of  nonlocal interactions prevents even the simplest non-convex  1D   problem from being   fully characterized by a bulk continuum energy. It follows from our analysis that the  exceptions,  when the `local' description also has `global' features and the generalized Cauchy-Born rule is applicable, are extremely rare. Then the question arises regarding  the very nature of the continuum model which could be considered as asymptotically equivalent to   a discrete model carrying both non-convexity and incompatibility induced by nonlocal interactions. In this section we present an explicit example showing  that the answer to this question may be nontrivial.  While our analysis here will not be exhaustive, it points towards  a new class of hybrid discrete-continuum variational problems which may be of a considerable  interest {\em per se}.

In the interest of analytical transparency  we focus  on  the specific homogenization problem for energies $E_\e$ with the truncated quadratic potential $f$ given by \eqref{fNT}; that is, the NT model analyzed in Example \ref{truquapo-exa}.   
Our goal will be to find  a  continuum analog of this problem allowing one to approximate both the minimal energy and the optimal microstructure. More specifically we search for the continuum problem which will be  asymptotically $\Gamma$-equivalent to $E_\e$ in the sense of \cite{BT}.  In other words,  the challenge is to construct a quasi-continuum problem still carrying some elements of the `lost'  discreteness of the original problem.  
 
 To show that the task of constructing  such a problem is nontrivial we first present a naive approach to `continualization'  in this setting which has been proposed phenomenologically and studied extensively  in applications \cite{BBMM}.  We show the shortcomings of such an approach  and  then correct it to match the exact solution of the discrete problem  presented  in Section~\ref{esp:co}.

\subsection{Naive construction}

We recall that the original problem is defined on a bounded interval $I$ and involves two  functions  $u,v\in \mathcal A_\e(I)$. We can write 
the corresponding energy function  in the form of a sum
\begin{eqnarray}\label{def-etilden1}
E_\e (u,v;I)=E_\e ^*(u;I)+E_\e^{**} (u,v;I)
\end{eqnarray}
where 
\begin{eqnarray}\label{def-etilden2}
E_\e^*(u;I)=\e\sum_{i\in \mathcal I^\ast_\e(I)}f\big(\frac{u_i-u_{i-1}}{\e}\big)
\end{eqnarray}
with  $\mathcal I_\e^\ast=\{i\in\mathbb Z: \e i,\e(i-1)\in I\}$    and 
\begin{eqnarray}\label{def-etilden3}
E_\e^{**}(u,v;I)=\frac{a}{\e}\!\sum_{i\in \mathcal I^\ast_\e(I)}\!(v_i-v_{i-1})^2
+\frac{b}{\e}\!\sum_{i\in \mathcal I_\e(I)}\!(u_i-v_i)^2,
\end{eqnarray}
Assuming now that $I$ is a bounded interval and $\e>0$, we can construct for each of the entries in the sum \eqref{def-etilden1}, viewed   \emph{independently}, the   asymptotically $\Gamma$-equivalent functionals, defined, respectively,  for $u\in SBV(I)$ and $v\in H^1(I)$.  This equivalence can be interpreted as a uniform (with respect to boundary data)  approximation  up to order $\e$ of problems with fixed boundary data for $E_\e ^*$ and $E_\e ^{**}$ by the corresponding problems for some functionals $G_\e ^*$ and $G_\e ^{**}$, respectively. 
 
A natural  choice for such \emph{independently } equivalent  functionals (see \cite{BT} for details) is 
\begin{equation}\label{def-fe-1d}
G^{*}_\e(u;I)=\int_I\gamma(u^\prime)^2 \, dt+\eta\e \#S(u),
\end{equation}
and 
\begin{equation}\label{def-fe-1d1}
G^{**}_\e(u,v;I)=\int_I\Big( \alpha(v^\prime)^2+\beta\big(\frac{u-v}{\e}\big)^2\Big)\, dt 
\end{equation}
for suitable $\alpha,\beta, \gamma, \eta>0$.
We recall that here   $u$ is a piecewise-Sobolev function with jump set denoted by $S(u)$. Given \eqref{def-fe-1d} and  \eqref{def-fe-1d1}  it seems natural to assume that the functional 
\begin{equation}\label{def-fe-1d2}
G_\e(u,v;I)=\int_I\Big(\gamma(u^\prime)^2+\alpha(v^\prime)^2+\beta\big(\frac{u-v}{\e}\big)^2\Big)\, dt+ \eta\e \#S(u)
\end{equation}
represents the desired (quasi) continuum analog of the original problem. 

We recall the convergence result proved in \cite{BCS2017}. 
\begin{remark}[asymptotic behaviour of the energies $G_\e$]\label{BCS-remark}\rm 
The $\Gamma$-limit of $G_\e$ with respect to the convergence $u_\e,v_\e\to v$ in $L^2(I)$ is given by 
$$G_{\rm hom}(v)=\int_I g_{\rm hom}(v^\prime)\, dt.$$ 
The integrand $g_{\rm hom}$ is characterized as 
\begin{equation}\label{def-gchom}
g_{\rm hom}(z)=\inf_{S>0} \Big\{\lambda_S z^2+\frac{\eta}{S}\Big\}
\end{equation}
where 
\begin{equation}\label{def-lambda}
\lambda_S=\frac{(\alpha+\gamma)\frac{\omega S}{2}}{\frac{\omega S}{2}+\frac{\gamma}{\alpha} \tanh(\frac{\omega S}{2})} \quad \hbox{ and } \quad \omega^2=\frac{(\alpha+\gamma)\beta}{\alpha\gamma}. 
\end{equation}
The function $g_{\rm hom}(z)$ is strictly convex, and the following properties hold: 
\smallskip 

{\rm (i)} $\displaystyle g_{\rm hom}(z)=(\alpha+\gamma) z^2$  in $ [0,z_c]$, 
where $z_c=\sqrt \frac{2\eta\omega \alpha}{4\gamma(\alpha+\gamma)}$;

\smallskip 

{\rm (ii)} 
$\displaystyle g_{\rm hom}(z)\sim 
\alpha z^2 +Cz^{2/3}$ as $z\to+\infty$,
where $C>0$ depends only on $\alpha,\beta, \gamma,\eta$.
\end{remark}

\subsection{Lattice induced interdependence of $E_\e ^*(u;I)$ and $E_\e^{**} (u,v;I)$}
Now we show that using the above approach, we obtain   the   discontinuous function $u$  which provides  only  formal approximations  for the  `jump sets' of the original discrete problems.

\begin{figure}[h!]
\centerline{\includegraphics[width=0.7\textwidth]{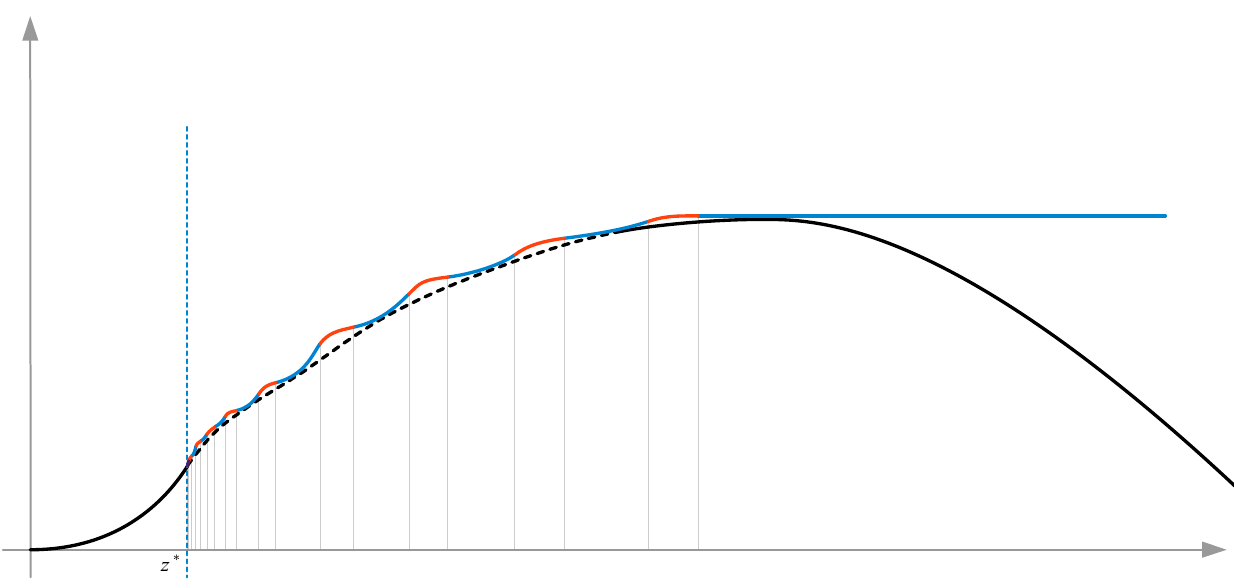}}
\caption{comparison between the graph of the function $g_{\rm hom}$ (below) and that of $Q_{\bf m}f$ after subtraction of the quadratic part}
\label{comparison}
\end{figure} 
\begin{remark}[non-equivalent scaling behavior]\rm 
Note first that the critical value $\overline z_\ast$ in the NT discrete model, defined in \eqref{zestasta}, is different from the corresponding critical value in the continuum problem discussed above.
Indeed, if we choose  $\gamma=1$ as in the discrete case, in order for the discrete and continuous energies to be equivalent up to $\overline z_\ast$ 
we need to `correct' the continuum fracture energy by substituting $\eta$ with 
an {\it effective fracture toughness} $\frac{\eta}{\cosh\zeta}$ with $\zeta$ given by \eqref{def-zeta}. 
However, such a correction will not extend the equality of the energy functions beyond the threshold. In particular, note  the different scaling behavior of the two models as $z$ diverges, see Fig. \ref{comparison}.
\end{remark}

It is clear that the proposed  \emph{lattice-independent } approximation of $E_\e ^*(u;I)$ and $E_\e^{**} (u,v;I)$ fails  because  in general separate uniform approximations of minima for two functionals does not provide a uniform approximation for the minimum of the sum. More specifically, in our case functionals $G^{*}_\e$ favor the onset of (at most) one jump point of $u$, while functionals $G^{**}_\e$, not involving jump sets, allow for an unbounded number of jumps. While in the correspondingly tailored  regimes we can have good \emph{separate}  approximations, the sum of the two energies in $E_\e$  optimizes the number and location of jumps accounting for the lattice induced  interaction between $E_\e ^*(u;I)$ and $E_\e^{**} (u,v;I)$ and therefore in a different way than $G_\e$ which does not account for such lattice induced interaction. 

Note that while in the discrete case we have interaction constrained by the lattice discreteness, in the naive continuum problem such interaction is lattice-unconstrained, which allows in principle for a richer class of microstructures. That is why we can obtain in this way at most a lower bound.

\subsection{A lattice-compatible  construction}

As we have seen above,  the limit of the energies defined in \eqref{def-fe-1d} when $\e\to0$ has different properties from those of its discrete counterpart and the failure of this approach is related to the discrete-to-continuum transition-induced loss of the constraint on the location of the jumps.

To construct the asymptotically equivalent \cite{BT} continuum theory the approach should be more subtle because  the corresponding  relaxation procedure should involve  a delicate interplay between continuum limit  and discrete energy minimization, which are tightly coupled.  

Indeed, as we have seen above decoupling discrete-to-continuum transition from the relaxation of a non-convex energy  gives rise to a quantitatively and qualitatively incorrect  asymptotic behavior. Apparently the  discrete-to-continuum limit  and   the  incompatibility-constrained non-convex minimization do not commute and by performing the former   independently of the latter  we at best underestimate the relaxed energy. In other words,  by neglecting the discrete constraint we may be  able to construct lower bounds (using the  naive approximation). We do not systematically analyze this issue here.

To get an insight on how to fix the problem, it is instructive to compare  \eqref{def-gchom}  with formulas \eqref{def-fdhom} and \eqref{def-cn}.  Note, in particular, that in the latter the parameter $N$ is discrete while in the former the parameter $S$ is continuous. This highlights that the discreteness, fundamental in the construction of the ${\bf m}$-relaxation in the original problem, is underestimated  in the computation of $g_{\rm hom}$. In other words, the internal physical scale and the lattice scale tend to zero simultaneously but the value of their ratio is not remembered  in the limit.

With this remark in mind, we now look for a modification of the `naive'  continuum energies which corrects the  non-equivalent behavior, while maintaining the relevant features associated with the discreteness in the original functional $E_\e$. Since the energies defined in \eqref{def-fe-1d} cannot be equivalent to $E_\e$ mainly because of the discrete location of the jump points, it is natural to add the constraint that the jump set $S(u)$ be contained in $\e\mathbb Z$. 

As we  show below,  this simple modification is indeed sufficient to obtain equivalence.  Here we imply  that the  energies depending on three parameters $\alpha$, $\beta$ and $\gamma$ (instead of $a$, $b$ and $1$, respectively), can  be tuned appropriately  to  construct  the correct limiting energy.

More specifically, for any $\e>0$ we define for $u\in SBV(I)$ and $v\in H^1(I)$ the functional  
\begin{equation}\label{def-fc-1d}
G^{\mathbb Z}_\e(u,v;I)=\begin{cases}
G_\e(u,v;I) & \hbox{ if } S(u)\subset \e\mathbb Z\cr 
+\infty & \hbox{ otherwise.}
\end{cases}
\end{equation}
By the general homogenization theorem \cite[Th.~3]{BCS2017} we get the following 
$\Gamma$-convergence result.
\begin{proposition}\label{salti-interi}
The sequence $G^{\mathbb Z}_\e(u,v;I)$ $\Gamma$-converges with respect to the convergence 
$u_\e,v_\e\to v$ in $L^2(I)$ to 
\begin{equation}
G^{\mathbb Z}_{\rm hom}(v)=\int_I g^{\mathbb Z}_{\rm hom}(v^\prime)\, dt
\end{equation}
where 
\begin{equation}\label{def-echom} 
g^{\mathbb Z}_{\rm hom}(z)=
\displaystyle\lim_{N\to+\infty}
\frac{1}{N}\inf\{G_1^{\mathbb Z}(u,v; (0,N))\!: \ \! u(0)=v(0)=0, \ \! u(N)=v(N)=Nz\}.
\end{equation}
\end{proposition} 
The proof of Proposition \ref{salti-interi} can be obtained by following the steps of the proof of \cite[Theorem 3]{BCS2017}. Indeed, in the blow-up procedure the jump set $S(u_\e)$ is not modified, and the $\liminf$ inequality follows.  
Concerning the upper estimate, by density we can consider a piecewise-affine target function $v$ such that $S(v^\prime)\subset \mathbb Q$; then, the construction of the recovery sequence can be done by following the same steps as in the the proof of \cite[Theorem 3]{BCS2017}, and the scaling argument gives $u_\e$ such that $S(u_\e)\subset \e\mathbb Z$.  
Note that the function $g_{\rm hom}^{\mathbb Z}$ is convex. 
\bigskip

Now we will show that the sequence $G^{\mathbb Z}_\e(u,v;I)$ has the same $\Gamma$-limit as the discrete sequence $E_\e$ 
for a suitable choice of the parameters $\alpha,\beta, \gamma$.
We define  
\begin{equation}\label{def-psi}
g(N,z)=\displaystyle\frac{1}{N}\min\{\tilde G_1(u,v; (0,N))\!: \ u,v\in H^1(0,N), \ v(0)=0, \ v(N)=Nz\},  
\end{equation}
where, in analogy with (\ref{def-etilden}), we denote by $\tilde G_1$ the (non scaled) functional given by 
\begin{equation*}
\tilde G_1(u,v;I)=\int_I\Big(\gamma(u^\prime)^2+\alpha(v^\prime)^2+\beta(u-v)^2\Big)\, dt.  
\end{equation*}
By solving the Euler-Lagrange equations for $\tilde G_1$ and minimizing on the boundary values of $u$, 
it follows that 
\begin{equation}\label{euler}
g(N,z)=\displaystyle\lambda_N z^2
\end{equation}
with $\lambda_N$ defined in (\ref{def-lambda}). 
Note that the (unique) solution $(u_N,v_N)$ of the minimum problem defining $g(N,z)$ satisfies the symmetry property 
$u(\frac{N}{2})=v(\frac{N}{2})=\frac{N}{2}z$. 
\begin{proposition}
For any $z\in\mathbb R$ the following equality holds: 
\begin{equation}\label{uguali-discreto}
g^{\mathbb Z}_{\rm hom}(z)=\Big(\inf_{N\in\mathbb N}\Big\{\lambda_N z^2+\frac{\eta}{N}\Big\}\Big)^{\ast\ast}.
\end{equation}
\end{proposition}
\begin{proof}
We fix $z\in\mathbb R$ and $N\in\mathbb N$; let $(u_N, v_N)$ be the solution of the minimum problem defining $\psi(N,z)$. 
We define $\tilde u_N\in SBV(0,2N)$ by setting 
\begin{equation}\label{estensione}
\tilde u_N(t)=\left\{ 
\begin{array}{ll} 
2u_N(\frac{t+N}{2})-Nz & \hbox{ if } \ t\in(0,N)\\
2u_N(\frac{t-N}{2})+Nz & \hbox{ if } \ t\in(N,2N)
\end{array}
\right.  
\end{equation}
and correspondingly $\tilde v_N\in H^1(0,2N)$. Since $u_N(\frac{N}{2})=v_N(\frac{N}{2})=\frac{N}{2}z$, 
then $S(\tilde u_N)=\{N\}$, $\tilde u_N(0)=\tilde v_N(0)=0$ and $\tilde u_N(2N)=\tilde v_N(2N)=2Nz$; by construction 
$$\frac{1}{2N} \tilde G_1(\tilde u_N,\tilde v_N;(0,2N))=
\frac{1}{N}\tilde G_1(u_N,v_N;(0,N))=\displaystyle
\lambda_N z^2.$$ 
Let $k\in\mathbb N$. We define $\tilde u$ in $(0,2kN)$ by setting  
\begin{equation*}
\tilde u(t)=
 \tilde u_N(t-2jN)+2jNz \quad \hbox{ in } \quad (2jN, 2(j+1)N), \ \ \  j=0,\dots, k-1 
\end{equation*}
and in the same way we define $\tilde v$. 
By construction, $S(\tilde u)\subset \mathbb N$ and $\# S(\tilde u)=k-1$; hence, 
since the boundary conditions for $\tilde u$ and $\tilde v$ hold, we have
\begin{eqnarray*} 
\displaystyle
\lambda_N z^2+\frac{\eta}{N}&=&
\frac{1}{2kN} \tilde G_1(\tilde u,\tilde v;(0,2kN))+\frac{\eta (k-1)}{kN}+\frac{\eta}{kN}\\
&=&\frac{1}{2kN} G_1^{\mathbb Z}(\tilde u,\tilde v;(0,2kN))+\frac{\eta}{kN}\\
&\geq&\frac{1}{2kN} \inf\Bigl\{G_1^{\mathbb Z}(u,v;(0,2kN)): \\
&& \hspace{1cm} u(0)=v(0)=0, 
u(kN)=v(kN)=2kNz \Bigr\}+\frac{\eta}{ k N}, 
\end{eqnarray*}
and, by taking the 
limit as $k\to+\infty$, 
\begin{eqnarray*}
\displaystyle
\lambda_N z^2+\frac{\eta}{N}\geq g_{\rm hom}^{\mathbb Z}(z).
\end{eqnarray*}
Hence, since $g_{\rm hom}^{\mathbb Z}$ is convex, 
\begin{eqnarray*}
\Big(\inf_{N\in\mathbb N}\Big\{\displaystyle
\lambda_N z^2+\frac{\eta}{N}\Big\}\Big)^{\ast\ast}\geq g_{\rm hom}^{\mathbb Z}(z).
\end{eqnarray*}
Next we need to  prove  the opposite inequality. 
Let $u\in SBV(0,N)$ and $v\in H^1(0,N)$ be such that the boundary conditions 
$u(0)=v(0)=0$, $u(N)=v(N)=Nz$ 
hold and $S(u)\subset \mathbb N$. We denote the jump points of $u$ by $N_i$, $i=1,\dots k$, with 
$N_i<N_{i+1}$ for any $i=1,\dots, k-1$.  
Setting $N_0=0$ and $N_{k+1}=N$, we define 
$$n_i=N_i-N_{i-1} \quad \hbox{ and } \quad z_i=\frac{v(N_i)-v(N_{i-1})}{n_i}$$ 
for $i=1,\dots k+1$.
We then have 
$$\frac{1}{n_i}\tilde G_1(u,v; (N_{i-1}, N_{i}))\geq g(n_i,z_i)=
\lambda_{n_i}z_i^2$$ 
 for any $i$, so that 
\begin{eqnarray*}
\frac{1}{N}G_1^{\mathbb Z}(u,v; (0,N))&\geq&\sum_{i=1}^{k+1} 
\frac{n_i}{N} \lambda_{n_i}z_i^2 +
\frac{\eta k}{N}
\ = \ \sum_{i=1}^{k+1} 
\frac{n_i}{N} \big(\lambda_{n_i}z_i^2 +
\frac{\eta }{n_i}\big)\\
&\geq& \sum_{i=1}^{k+1} 
\frac{n_i}{N} \inf_{n\in\mathbb N}\Big\{
\lambda_{n}z_i^2 +
\frac{\eta }{n}\Big\}
\end{eqnarray*}
Since $\sum_{i=1}^{k+1}n_i=N$ and $\sum_{i=1}^{k+1}n_iz_i=Nz$, an application of Carath\'eodory's Theorem 
gives  
\begin{eqnarray*}
\frac{1}{N}G_1^{\mathbb Z}(u,v; (0,N))
&\geq& \Big( \inf_{n\in\mathbb N}\Big\{\lambda_{n}z^2 +
\frac{\eta }{n}\Big\}\Big)^{\ast\ast}.
\end{eqnarray*}
Taking the $\inf$ over the admissible functions and the limit for $N\to+\infty$ we get the inequality 
\begin{eqnarray*}
g_{\rm hom}^{\mathbb Z}(z)
&\geq& \Big( \inf_{n\in\mathbb N}\Big\{\lambda_{n}z^2 +
\frac{\eta }{n}\Big\}\Big)^{\ast\ast}
\end{eqnarray*}
concluding the proof.
\end{proof}

Now, if we choose 
\begin{equation}\label{parametri}
\displaystyle\alpha=\displaystyle\frac{a(a+1)}{a+\zeta\coth(\zeta)}, \quad
\displaystyle\beta=\displaystyle\frac{4a(a+1)\zeta^3\coth(\zeta)}{(a+\zeta\coth(\zeta))^2}, \quad 
\displaystyle\gamma=\displaystyle\frac{(a+1)\zeta\coth(\zeta)}{a+\zeta\coth(\zeta)}
\end{equation}
it follows that $\omega=2\zeta,$ 
where $\zeta$ is defined in \eqref{def-zeta}, and for any $N$ the following equality holds 
$$
\lambda_N=c_N=\frac{N(a+1)a}{aN+\tanh(N\zeta)\coth(\zeta)}.
$$
We can then state the following equivalence result, whose proof follows from the equivalence between $E_\e$ and $F_\e$ (Theorem \ref{teo-equivalence-exp} and Remark \ref{absigma}) and the results above. 

\begin{theorem}[equivalence with the Novak-Truskinovsky model] 
Choosing the coefficients as in \eqref{parametri}, the sequence $G_\e^{\mathbb Z}$ defined in \eqref{def-fc-1d} 
$\Gamma$-converges with respect to the $L^2$-convergence to the same $\Gamma$-limit of the sequence of discrete functionals $E_\e$ in the truncated quadratic case.  
\end{theorem}

\noindent We reiterate that in general, the above result  can be viewed as a cautionary tale, showing that relaxation and homogenization (discrete-to-continuum limit) do not always commute.

\section{Conclusions}
\label{conclusions}

In this paper, we systematically explored   the possibility of  using some auxiliary  `local' considerations to obtain minimizers with `global' features for  nonlocal  variational  boundary-value problems on lattices. Having in mind  some known cases when asymptotically (i.e. in continuum limit) such boundary-value problems exhibit   periodic minimizers, we associated  the   possibility of `local' description with  applicability of the GCB rule and posed the  question of the pertinence of such a rule for a generic variational problems in our class.  It is clear that the GCB rule is  not applicable in general, for instance,  it clearly fails  in the case of minimization with  concentrations, appearing in non-coercive problems of fracture mechanics.    Here  we extended the known class  of non-GCB problems  by incorporating into the analysis some general  non-convex energy densities  with quadratic growth. 

More specifically, we used the simplest examples of functionals with quadratically penalized non-convexity, we demonstrated  various facets  of  frustration and incompatibility in one-dimensional  discrete variational problems computed on an increasing and diverging number of nodes. In the chosen class of non-convex lattice problems with energy density $f$, linear long-range interactions were introduced  through an infinite matrix $\bf m$. We studied relaxation of such problems with given boundary conditions on intervals with a large number of nodes. This operation can be interpreted as a discrete-to-continuum $\bf m$-transform of the function $f$ and we studied the  dependence of such a transform on the  parameter $z$ describing boundary conditions. 

We addressed  the  question   whether the minimizers for a given functional are close to functions with `global' properties, for instance, to periodic functions, where closeness can be understood as having the same energy up to an asymptotically negligible quantity as the number of nodes diverges.  The answer is in general negative, for example, this is not true in the case of  minimizers  describing transitions between two energy wells, when the parameter $z$ lies in some intervals.  Still, we were able to identify interesting cases when  the knowledge of the minimizers, that are asymptotically of a `global' form, are sufficient to determine the whole $\bf m$-transform of the function $f$ through some form of convexification. 

Outside our general considerations, we mostly focused on potentials $f$ with a bi-convex form; i.e., which have a convex restriction to two complementary phase sets. For boundary-value problems involving such potentials  and prescribed $z$ it is natural to   define   phase functions $\theta(z)$. We have shown that  of particular interest are values of $\theta$ for which the set $\{z: \theta(z)=\theta\}$ contains a non-degenerate interval (locking states). We studied the main properties of both, the functions $\theta(z)$ and of locking states, and showed that for  some combinations of $f$ and $\bf m$ the  minimizers representing the  locking states are periodic and hence of a `global' (or GCB)  nature in the sense that they  determine the whole $\bf m$-transform of the function $f$. We also showed that the optimal periodic minimizers whose  structure may depend delicately   on $f$ and $\bf m$  are not necessarily   unique. Among different optimal minimizers we identified universal periodic microstructures, which exist for all values of $\theta$ and have fascinating analogs  in the  theory of dynamical systems.

 The concept of  $\bf m$-transform, introduced in this paper for the first time,  was shown to be  rather rich. The complexity of the ensuing transformations  suggests that even in   scalar one-dimensional problems,  the interplay  of long-range interactions, non-convexity and discreteness  can be highly nontrivial. We presented several examples where the $\bf m$-transform of a given non-convex function could be either computed explicitly or narrowly bounded. Some of the  obtained  $\bf m$-transforms were shown to be singular  exhibiting   the  `devilish' features with  locking on some but not all \emph{rational} microstructures. 

The analytical accessibility of the $\bf m$-transforms in the presented examples,  as well as the associated  non-unique\-ness of the optimal micro-structures, hint towards a certain \emph{degeneracy} of the chosen problems.  We can  associate such a degeneracy with the absence of `strong'  geometrical frustration representing some fundamental incommensuration  between the non-convexity, the long range interactions  and the discreteness. It is clear that  more complex optimal minimizing sequences, not reducible to periodic states or combinations of periodic states,  can be expected  in cases when such  incommensuration is present. 

The `strong'  frustration of this type may be driven, for instance,  by the  competing interactions inside  the kernel {\bf m}, for instance, by the combination of ferromagnetic and antiferromagnetic interactions acting on incommensurate scales.  The frustration can be also  `strong' even in the apparently simple case   when   different scales are `favored' by antiferromagnetic   interaction involving  the  first  and the third nearest neighbors.  `Strong'  frustration may also be brought by the structure of the non-convex function $f$   carrying the  `characteristic strain' which is  incompatible with the strain emerging through the interplay between the loading and the long-range interaction  kernel, see for instance \cite{NT2016} where a  `complete devil staircase'  emerges in a   problem  involving  a  non-degenerate bi-quadratic potential and an exponential kernel.  

In a separate paper  we will show  that the presence of `strong' frustration may  eliminate   the degeneracy and bring  the uniqueness to the problem of finding  the optimal microstructure. More generally,  our preliminary   analysis of  problems  with  `strong' frustration reveals  an even  deeper  link between    lattice variational problem and the  discrete nonlinear mappings  where the analog of  constructing the $\bf m$-transform turns out to be  the problem of classifying  all quasi-periodic  trajectories.

\subsection*{Acknowledgments}
AC and MS acknowledge the projects `Fondo di Ateneo per la Ricerca 2019' and `Fondo di Ateneo per la Ricerca 2020', funded by the University of Sassari. This work has been supported by PRIN 2017  `Variational methods for stationary and evolution problems with singularities and interfaces'. AB and MS are members of GNAMPA, INdAM, AC is member of GNSAGA, INdAM. The authors acknowledge the MIUR Excellence Department Project awarded to the Department of Mathematics, University of Rome Tor Vergata, CUP E83C18000100006. The work of LT was supported by   the grant ANR-10-IDEX-0001-02 PSL.

\goodbreak
\appendix

\section{Appendix: variations of boundary data} 
\nobreak
In this appendix we state and prove some technical results which allow the modification of boundary values of test functions for the minimum problems used in various characterization of $Q_{\bf m} f$. In particular, these results allow to assume that test functions be constant close to the endpoints of the domain.
\smallskip

Let ${\bf m}=\{m_n\}_n$ be such that $m_n\geq 0$ for any $n$, and there exists $\overline n$ such that 
$m_n$ is not increasing for $n\geq \overline n$. Moreover, we assume the decay condition 
$m_n=o(n^{-\beta})_{n\to+\infty}$ for some $\beta>2$. 

Let $F_\e$ be defined as in \eqref{def-fe-prima}; that is, 
$$F_\e(u;I)=\sum_{\e i,\e(i-1)\in  I}\e\, f\Big(\frac{u_{i}-u_{i-1}}{\e}\Big)+
\sum_{\e i,\e j\in  I}\e\, m_{|i-j|} \Big(\frac{u_{i}-u_{j}}{\e}\Big)^2 
$$
for $I$ interval and $u\in\mathcal A_\e(I)$.   

\begin{lemma}\label{boundarycond} 
Let $L>0$ and $N_\e=\lfloor\frac{L}{\e}\rfloor$. 
Let $\alpha\in(\frac{2}{\beta},1)$. 
Assume that $u\in L^2(0,L)$ and $u^\e\in \mathcal A_\e=\mathcal A_\e(0,L)$ be such that (the piecewise-affine extension of) 
the sequence $u^\e$ converges to $u$ in $L^2(0,L)$, and $\sup_\e (F_\e(u^\e;[0,L])+\|u^\e\|^2_{L^2})=S<+\infty$. Then, 
there exists $\hat u^{\e}\in \mathcal A_\e$ converging to $u$ such that 
\begin{enumerate}
\item[{\rm(i)}] $\hat u^\e_i=\hat u^\e_0$ for $i\leq \e^{-\alpha}$,  $\hat u^\e_i=\hat u^\e_{N_\e}$ 
for $i\geq N_\e-\e^{-\alpha}$; 
\item[{\rm(ii)}] $F_\e(\hat u^\e;[0,L])\leq F_\e(u^\e;[0,L])+r(\e)$, where the remainder $r$ depends only on $S$ and $f(0)$, and 
$r(\e)\to 0$ as $\e\to 0$.  
\end{enumerate}
\end{lemma}

\begin{proof}We choose $\alpha^\prime\in(0,1-\alpha)$ and define 
$\lambda_\e=\e^{\alpha^\prime}$ and $M_\e=\lfloor \e^{\alpha+\alpha^\prime-1}\rfloor-1$. 
For $\e$ small enough we divide $(0,\lambda_\e]$ and $[L-\lambda_\e,L)$ in $M_\e+1$ intervals by setting 
$$I^k_\e=\Big(\frac{k\lambda_\e}{M_\e+1},\frac{(k+1)\lambda_\e}{M_\e+1}\Big], \ \ 
J^k_\e=\Big[L-\frac{(k+1)\lambda_\e}{M_\e+1},L-\frac{k\lambda_\e}{M_\e+1}\Big), \ \ k\in\{0,\dots, M_\e\}.$$

Since 
$$\frac{1}{\e}\sum_{k=1}^{M_\e} \sum_{\e i\in I_\e^k, \e j\in I_\e^{k-1}} m_{|i-j|}(u^\e_i-u^\e_j)^2\leq F_\e(u^\e;[0,L])\leq S,$$
then there exists $k_\e^-\in\{1,\dots, M_\e\}$ such that 
\begin{equation}\label{stimaM1}
\frac{1}{\e} \sum_{\e i\in I_\e^{k_\e^-}, \e j\in I_\e^{k_\e^--1}} m_{|i-j|}(u^\e_i-u^\e_j)^2
\leq \frac{S}{M_\e}.
\end{equation}
The same argument allows to find $k_\e^+\in\{1,\dots, M_\e\}$ 
such that 
the same inequality holds for $\e i\in J_\e^{k_\e^+}, \e j\in J_\e^{k_\e^+-1}$.
Setting $j^-_\e=\min\{j: \e j\in I_\e^{k_\e^-}\}$ and $j^+_\e=\max\{j: \e j\in J_\e^{k_\e^+}\}$, we define $\hat u^{\e}$ by setting 
\begin{equation}\label{def-costante}
\hat u^{\e}_i=\left\{ \begin{array}{ll}
u^\e_{j^-_\e} & \hbox{ if } i\leq  j^-_\e \\
u^\e_i &  \hbox{ if } j^-_\e \leq i\leq j^+_\e  \\
u^\e_{j^+_\e} &  \hbox{ if } i\geq j^+_\e.
\end{array}
\right. 
\end{equation}
Since $j_\e^{-}\geq L\e^{-\alpha}$ and $j_\e^{+}\leq N_\e-L\e^{-\alpha}$, then $\hat u^\e$ satisfies claim (i). Moreover, $\hat u^\e\to u$ as $\e\to 0$. To prove this, for simplicity we suppose that $m_n$ is not increasing for $n\geq 1$. Then,  
\begin{eqnarray*}\e\sum_{i=1}^{j^-_\e}(u^\e_i-\hat u^\e_{i})^2&=&\e\sum_{i=1}^{j^-_\e}(u^\e_i-u^\e_{j^-_\e})^2\leq \e\sum_{i=1}^{j^-_\e} j^-_\e \!\!\sum_{j=i+1}^{j^-_\e}\!(u^\e_i-u^\e_{i-1})^2\\
&\leq& \frac{S}{m_1}\e^2 (j^-_\e)^2\leq \frac{S}{m_1}\lambda_\e^2,\end{eqnarray*}
and correspondingly $\e\sum_{i=j^+_\e}^{\lfloor L/\e\rfloor}(u^\e_i-\hat u^\e_{i})^2\leq \frac{S}{m_1}\lambda_\e^2$.
Setting, $n_\e=\lfloor\frac{\lambda_\e}{\e (M_\e+1)}\rfloor$, since 
$$\sum_{|i-j|\geq n_\e} m_{|i-j|}(u^{\e}_i-u^{\e}_j)^2 \leq 
\frac{2}{\e}  m_{n_\e} \|u^\e\|^2_{L_2}\leq 
\frac{2}{\e}  m_{\lfloor\e^{-\alpha}\rfloor} \|u^\e\|^2_{L_2},$$
and recalling  \eqref{stimaM1}, we obtain 
\begin{eqnarray*} 
F_\e(\hat u^{\e};[0,L])\leq F_\e(u^{\e};[0,L])+ 2\lambda_\e f(0) + \frac{C}{\e^2} 
m_{\lfloor\e^{-\alpha}\rfloor}+\frac{C}{M_\e},
\end{eqnarray*}
where 
$C$ denotes a constant depending only on $\sup_\e F_\e(u^\e;[0,L])$ and $\sup_\e\|u^\e\|_{L^2}$.  
Setting 
$$r(t)=2f(0) t^{\alpha^\prime}+Ct^{\alpha\beta-2}+C t^{1-\alpha-\alpha^\prime},$$
we conclude the proof since $m_n=o(n^{-\beta})$ and $\alpha>\frac{2}{\beta}$.  
\end{proof}

Let $a,b>0$. We define the functional $E_\e(u,v;I)$ by setting 
\begin{equation}\label{def-ee-gen} 
E_\e(u,v;I)=\sum_{\e i,\e(i-i)\in I}\e\, f\Big(\frac{u_i-u_{i-1}}{\e}\Big)+
\frac{a}{2}\sum_{\e i,\e(i-i)\in I}\e\, \Big(\frac{v_{i}-v_{i-1}}{\e}\Big)^2+\frac{b}{2\e}\sum_{\e i\in I}(u_i-v_i)^2 
\end{equation} 
for $I$ interval and $u,v\in\mathcal A_\e(I)$.   
\begin{lemma}\label{boundarycond2} 
Let $L>0$ and $N_\e=\lfloor\frac{L}{\e}\rfloor$. 
Let $\alpha\in(\frac{2}{\beta},1)$. 
Assume that $u^\e, v^\e\in \mathcal A_\e$ be such that (the piecewise-affine extensions of) 
$u^\e$ and $v^\e$ converge to $u$ in $L^2(0,L)$ and $\sup_\e (E_\e(u^\e;[0,L])+\|u^\e\|^2_{L^2})=S<+\infty$. Then there exist $\hat u^{\e}, \hat v^\e \in \mathcal A_\e$ converging to $u$ such that 
\begin{enumerate}
\item[{\rm(i)}] $\hat u^\e_i=\hat v^\e_i=\hat u^\e_0$ for $i\leq \e^{-\alpha}$,  $\hat u^\e_i=\hat v^\e_i=\hat u^\e_{N_\e}$ 
for $i\geq N_\e-\e^{-\alpha}$; 
\item[{\rm(ii)}] $E_\e(\hat u^\e, \hat v^\e;[0,L])\leq E_\e(u^\e, v^\e;[0,L])+r(\e)$, where the remainder $r$ depends only on $S$ and $f(0)$, and $r(\e)\to 0$ as $\e\to 0$.  
\end{enumerate}
\end{lemma}

\begin{proof}
We choose $\lambda_\e$ and $M_\e$ as in the proof of Lemma \ref{boundarycond}, 
and divide $(0,\lambda_\e]$ and $[L-\lambda_\e,L)$ in $M_\e+1$ intervals, denoted by $I^k_\e$ and $J^k_\e$ respectively, as above. 
Then, there exist $k_\e$ and $h_\e$ in $\{1,\dots, M_\e\}$ such that 
\begin{equation}\label{stimaM3}
\frac{1}{2\e} \sum_{\e i\in I_\e^{k_\e}\cup J_\e^{h_\e}} \big(a(v^\e_i-v^\e_{i-1})^2+b(u^\e_i-v^\e_{i})^2 \big)
\leq \frac{S}{M_\e}.
\end{equation}
Setting $j^-_\e=\min\{j: \e j\in I_\e^{k_\e}\}$ and $j^+_\e=\max\{j: \e j\in J_\e^{h_\e}\}$, we define 
\begin{equation*}
\hat u^{\e}_i=\left\{ \begin{array}{ll}
u^\e_{j_{\e}^-} & \hbox{ if } i\leq  j_{\e}^- \\
u^\e_i &  \hbox{ if } j_{\e}^-< i<  j_{\e}^+ \\
u^\e_{j_{\e}^+} &  \hbox{ if } i\geq  j_{\e}^+
\end{array}
\right. 
\ \ 
\hbox{ and } \ \ \ \
\hat v^{\e}_i=\left\{ \begin{array}{ll}
u^\e_{j_{\e}^-} & \hbox{ if } i\leq j_{\e}^- \\
v^\e_i &  \hbox{ if } j_{\e}^-< i <  j_{\e}^+ \\
u^\e_{j_{\e}^+} &  \hbox{ if } i\geq j_{\e}^+,
\end{array}
\right. 
\end{equation*}
so that $\hat u^\e$ and $\hat v^\e$ converge to $u$ in $L^2$, 
and satisfy (i). Recalling \eqref{stimaM3}, we get 
in particular that 
\begin{eqnarray*}
\frac{a}{2\e} (\hat v^\e_{j_{\e}^-+1}-\hat v^\e_{j_{\e}^-})^2\leq 
\frac{a}{\e} (v^\e_{j_{\e}^-+1}-v^\e_{j_{\e}^-})^2+ \frac{a}{\e} (v^\e_{j_{\e}^-}-u^\e_{j_{\e}^-})^2 
\leq \frac{C}{M_\e},
\end{eqnarray*}
where $C$ denotes a positive constant depending only on $a,b$ and $S$. 
The same bound holds for $\frac{a}{2\e}(\hat v^\e_{j_{\e}^+}-\hat v^\e_{j_{\e}^+-1})^2$. 
Hence 
\begin{eqnarray*}
E_\e(\hat u_\e,\hat v_\e; [0,L])&\leq&
2\lambda_\e f(0)
+E_\e(u^\e,v^\e;(0,L)) \\
&&
+\frac{a}{2\e} (\hat v^\e_{j_{\e}^-+1}-\hat v^\e_{j_{\e}^-})^2
+
\frac{a}{2\e}(\hat v^\e_{j_{\e}^+}-\hat v^\e_{j_{\e}^+-1})^2
\\
&\leq&
2\lambda_\e f(0)
+E_\e(u^\e,v^\e;[0,L]) + \frac{2C}{M_\e},
\end{eqnarray*}
concluding the proof as above. 
\end{proof}

\begin{remark}\label{rem-bound}\rm In the hypotheses of Lemma \ref{boundarycond2}, 
if there exists $\alpha\in (0,1)$ such that 
$u^\e_i=\hat u^\e_0$ for $i\leq \e^{-\alpha}$ and $u^\e_i=\hat u^\e_{N_\e}$ for $i\geq N_\e-\e^{-\alpha}$ for some $\alpha>0$, 
then the function $\hat v^\e$ can be chosen such that 
it coincides with $u^\e$ for $i\leq \e^{-\alpha^{\prime\prime}}$ and for $i\geq N_\e-\e^{-\alpha^{\prime\prime}}$ with $\alpha^{\prime\prime}<\alpha$.   
\end{remark}

\section{Appendix: formulas for $P^{M,n}$ in the concentrated case} 
In this appendix we include some explicit computations of the functions $P^{M,n}$ defined in \eqref{pieMMeenne}, which are the energies of the locking states $n\over M$ in the concentrated case. The formulas of these functions have been used in Sections \ref{fracture1M} and \ref{doublewell1M} to highlight the structure of $Q_{\bf m}f(z)$ in the truncated-parabolic and double-well case, respectively. Here, we include the corresponding computations.

\paragraph{Truncated-parabolic case.}
Let $f$ be given by \eqref{frattura}.
In view of  \eqref{PMn},  
the domains of $P^{M,0}$ and $P^{M,M}$ are $\{z\leq 1\}$ and $\{z\geq 1\}$, respectively. We recall that here   
\begin{equation*}
P^{M,0}(z)=z^2+2(m_1+m_M M^2)z^2 
\ \ \hbox{\rm and }\ \ 
P^{M,M}(z)=1+2(m_1+m_M M^2)z^2. 
\end{equation*} 
For $n=1,\dots, M-1$, we can also write  
\begin{equation}\label{pmnformula}
P^{M,n}(z)=\begin{cases} \vspace{2mm}
\displaystyle \frac{2m_1+1}{1-\theta_n}\big(z^2-\theta_n(2z-1)\big)+2m_M M^2z^2
& 
\displaystyle\hbox{\rm if }\ z\leq 
T_n^{-}\\ \vspace{2mm}
\displaystyle \theta_n+\frac{2m_1(2m_1+1)}{2m_1+\theta_n}z^2
+2m_M M^2 z^2 
& \displaystyle\hbox{\rm if }\  T_n^{-}\leq z \leq T_n^{+}\\
\displaystyle  1+\frac{2m_1}{\theta_n}\big((z-1)^2+\theta_n(2z-1)\big)
+2m_M M^2 z^2
& \displaystyle\hbox{\rm if }\ z\geq T_n^{+}, 
\end{cases}
\end{equation} 
where 
$$T_n^{-}=\frac{2m_1+\theta_n}{2m_1+1} \ \ \hbox{\rm and }\ \  T_n^{+}=\frac{2m_1+\theta_n}{2m_1}.$$

\begin{figure}[h!]
\centerline{\includegraphics[width=.5\textwidth]{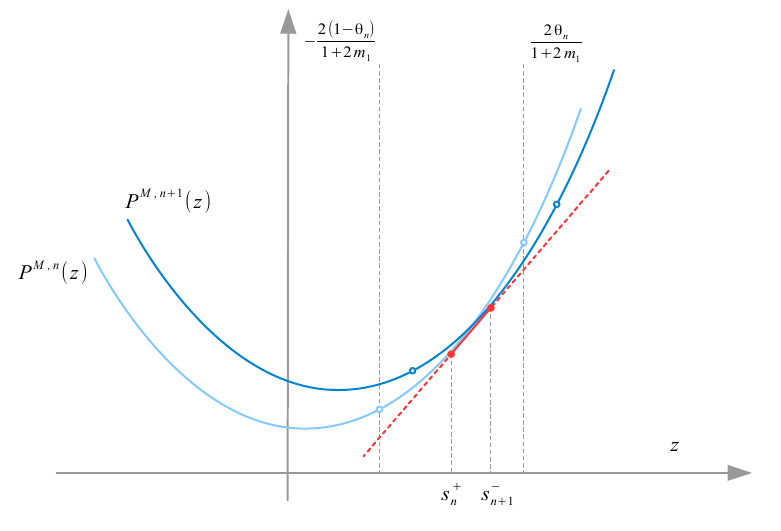}}
\caption{Envelope of two consecutive functions $P^{M,n}(z)$}
\label{figure_PMn}
\end{figure} 
Note that while the formula defining $P^{M,n}$ changes form at $z=T_n^{-}$  and $z=T_n^{+}$,   the computation of the common tangent points of $P^{M,n}$ and $P^{M,n+1}$ involves only the central formula in \eqref{pmnformula}. 
Consequently, the points $s_n^+$ and $s_n^-$ in Theorem \ref{structureqm} are  
\begin{equation}\label{tnsnapp}
\left. 
\begin{array}{ll}
&\displaystyle s_n^+=s_n^+(m_1,m_M)
=\frac{2m_1+\theta_n}{\sqrt{2m_1(2m_1+1)}}\ \sqrt{\frac{m_1(2m_1+1)+m_M M^2(2m_1+\theta_{n+1})}{m_1(2m_1+1)+m_M M^2(2m_1+\theta_{n})}} \vspace{3mm}\\
&\displaystyle s_n^-=s_n^-(m_1,m_M)
= \frac{2m_1+\theta_n}{\sqrt{2m_1(2m_1+1)}}\ \sqrt{\frac{m_1(2m_1+1)+m_M M^2(2m_1+\theta_{n-1})}{m_1(2m_1+1)+m_M M^2(2m_1+\theta_{n})}}. 
\end{array}
\right. 
\end{equation}
In Fig.~\ref{figure_PMn} we illustrate  the envelope of two consecutive functions $P^{M,n}(z)$, bridging energies of consecutive locking states with an affine function. 

Finally, since $s_n^+\geq T_n^-$ and $s_n^-\leq T_n^+$, we have the following formula 
\begin{equation}\label{qfM} 
Q_{\bf m}f(z)=\begin{cases}
\displaystyle z^2& \hbox{\rm if } \ z\leq s^+_0\\
\displaystyle r^{M,n}(z)-2(m_1+m_M M^2)z^2 &  \hbox{\rm if } \ s_n^+ \leq z \leq s_{n+1}^-
\\
\displaystyle\frac{2m_1(1-\theta_n)}{2m_1+\theta_n}z^2+\theta_n& \hbox{\rm if } \ 
s_n^- \leq z \leq s_{n}^+
\\
\displaystyle 1 & \hbox{\rm if } \ s_M^- \leq z,   
\end{cases}
\end{equation} 
where $r^{M,n}$ is the affine function \begin{eqnarray*}
r^{M,n}(z)=P^{M,n}(s_n^+)+\frac{2(z-s_n^+)}{M(s_{n+1}^--s_n^+)}.
\end{eqnarray*}

\paragraph{Bi-quadratic double-well case.}
Let $f$ be given by $f(z)=(1-|z|)^2$. By using \eqref{PMn}  
the domains of $P^{M,0}$ and $P^{M,M}$ are $\{z\leq 0\}$ and $\{z\geq 0\}$, respectively, 
where 
\begin{equation*}
P^{M,0}(z)=
(1+z)^2+
2(m_1+m_MM^2)z^2 \ \ \hbox{\rm and } \ \ 
P^{M,M}(z)=
(1-z)^2+
2(m_1+m_MM^2)z^2. 
\end{equation*} 
For $n=1,\dots, M-1$  
\begin{equation*}
P^{M,n}(z)=\begin{cases} \vspace{2mm}
\displaystyle \Big(\frac{1+2m_1}{1-\theta_n}+2m_M M^2\Big)z^2+2z+1 & 
\displaystyle\hbox{\rm if }\ z\leq T_n^-\\ 
\displaystyle (1+z)^2
+\displaystyle 2(m_1+m_M M^2)z^2-4\theta_n\Big(z+\frac{1-\theta_n}{1+2m_1}\Big)& \displaystyle\hbox{\rm if }\  T_n^-\leq z\leq T_n^+\\
\displaystyle \Big(\frac{1+2m_1}{\theta_n}+2m_M M^2\Big)z^2-2z+1
& \displaystyle\hbox{\rm if }\ z\geq T_n^+, 
\end{cases}
\end{equation*} 
where in this case the points $T_n^-$ and $T_n^+$ where the formula changes 
are given by 
$$T_n^-=-\frac{2(1-\theta_n)}{1+2m_1} \ \ \hbox{\rm and }\ \ T_n^+=\frac{2\theta_n}{1+2m_1}.$$
Consequently,  
\begin{equation}\label{tnsndw}
\left. \begin{array}{ll}
&\displaystyle s_n^+(m_1,m_M)=s_n^+=
\frac{2m_M M}{(1+2m_1)(1+2m_1+2m_M M^2)}+\frac{2\theta_n-1}{1+2m_1}\\ 
&\displaystyle s_n^-(m_1,m_M)=s_n^-=-\frac{2m_M M}{(1+2m_1)(1+2m_1+2m_M M^2)}+\frac{2\theta_n-1}{1+2m_1}.
\end{array}
\right. 
\end{equation}
Since $s_n^+\geq T_n^-$ and $s_n^-\leq T_n^+$, we obtain 
\begin{equation*} 
Q_{\bf m}f(z)=\begin{cases}
(1+z)^2& \hbox{\rm if } \ z\leq s_0^+\\
r^{M,n}(z)-2(m_1+m_M M^2) z^2 &  \hbox{\rm if } \ s_n^+ \leq z \leq s_{n+1}^-
\\
\displaystyle z^2+2(1-2\theta_n)z+1-\frac{4\theta_n(1-\theta_n)}{1+2m_1}& \hbox{\rm if } \ s_n^- \leq z \leq s_{n}^+
\\
 (1-z)^2 & \hbox{\rm if } \ s_M^- \leq z, 
\end{cases}
\end{equation*} 
where $r^{M,n}$ is the affine function 
$$
r^{M,n}(z)=P^{M,n}(s_n^+)+\frac{M(1+2(m_1+m_M M^2))}{2}\big(P^{M,n+1}(s_{n+1}^-)-P^{M,n}(s_n^+)\big)(z-s_n^+).
$$

\end{document}